\documentclass[10pt,a4paper,reqno]{amsart}
\usepackage[a4paper,margin=2.75cm]{geometry}
\usepackage{graphicx} 
\usepackage[T1]{fontenc}
\usepackage[utf8]{inputenc}
\usepackage{amsfonts,amsmath,amsthm,amssymb}
\usepackage{mathtools}
\usepackage{hyperref}
\usepackage[capitalise]{cleveref}
\usepackage{bm}
\usepackage{enumitem}
\usepackage{varwidth}
\usepackage{booktabs,subfig} 
\usepackage{caption}
\usepackage{subcaption}
\usepackage[dvipsnames]{xcolor}
\usepackage{verbatim}
\usepackage{tikz-cd}
\usepackage{chemfig}
\usepackage{multicol}
\usepackage[toc,page]{appendix}

\numberwithin{equation}{section}

\newcommand{\ie}{\emph{i.e.}\ }

\theoremstyle{definition}

\newtheorem{definition}{Definition}[section]
\newtheorem{remark}[definition]{Remark} 
\newtheorem{theorem}{Theorem}[section]
\newtheorem{lemma}[theorem]{Lemma}
\newtheorem{corollary}{Corollary}[theorem]

\newtheorem{assumption}{Assumptions} 

\newcommand*\circled[1]{\tikz[baseline=(char.base)]{
            \node[shape=circle,draw,inner sep=1pt] (char) {#1};}}
\def\half{\frac{1}{2}}
\newcommand{\bigo}{{\mathcal O}}

\newcommand\bfa{{\mathbf a}}
\newcommand\bfb{{\mathbf b}}
\newcommand\bfd{{\mathbf d}}
\newcommand\bfe{{\mathbf e}}
\newcommand\bfeta{{\bm{\eta}}}
\newcommand\bff{{\bm f}}

\newcommand\bfg{{\mathbf g}}

\newcommand\bfn{{\mathbf n}}
\newcommand\bfng{{\mathbf{n}_{\Gamma}}}
\newcommand\bfm{{\mathbf m}}

\newcommand\bfchi{{\bm{\chi}}}

\newcommand\bfu{{\mathbf{u}}}
\newcommand\bfv{{\mathbf{v}}}
\newcommand\bfw{{\mathbf w}}
\newcommand\bfx{{\mathbf x}}

\newcommand\bfy{{\mathbf y}}
\newcommand\bfz{{\mathbf z}}
\newcommand\bfA{{\mathbf A}}
\newcommand\bfB{{\mathbf B}}

\newcommand\bfH{\bm{H}}

\newcommand\bfI{{\bm I}}
\newcommand\bfF{{\mathbf F}}
\newcommand\bfG{{\mathbf G}}

\newcommand\bfM{{\mathbf M}}

\newcommand\bfU{{\mathbf U}}
\newcommand\bfV{{\mathbf V}}
\newcommand\bfP{{\mathbf{P}}}

\newcommand{\eu}{\bfe_\bfu}

\newcommand{\ep}{\bfe_p}
\newcommand{\el}{\bfe_{\lambda}}

\newcommand{\Ga}{\Gamma}
\newcommand{\Gah}{\Gamma_h}

\newcommand{\Gat}{\Gamma(t)}
\newcommand{\Gaht}{\Gamma_h(t)}

\newcommand{\Gahn}{\Gamma_h^n}
\newcommand{\GahnN}{\Gamma_h^{n-1}}
\newcommand{\ds}{d\sigma}

\newcommand{\dsh}{d\sigma_h}
\newcommand{\nbg}{\nabla_{\Gamma}}
\newcommand{\nbgcov}{\nabla_{\Gamma}^{cov}}
\newcommand{\nbgcovh}{\nabla_{\Gamma_h}^{cov}}

\newcommand{\nbgh}{\nabla_{\Gamma_h}}
\newcommand{\nbglin}{\nabla_{\Gamma_h^{(1)}}}

\newcommand{\divg}{\textnormal{div}_{\Gamma}}
\newcommand{\divgh}{\textnormal{div}_{\Gah}}

\newcommand{\diver}{\text{div}}

\newcommand{\mat}{\partial^{\bullet}}
\newcommand{\matn}{\partial^{\circ}}
\newcommand{\matd}{\partial^{\circ}_{h}}
\newcommand{\matdt}{\partial^{\circ}_{h,\tau}}
\newcommand{\matdtl}{\partial^{\circ}_{\ell,\tau}}
\newcommand{\matdl}{\partial^{\circ}_{\ell}}

\newcommand{\nb}{\nabla}

\newcommand{\R}{\mathbb{R}}

\def \t {(t)}

\def \to {\rightarrow}
\newcommand{\Th}{\mathcal{T}_h}

\newcommand{\Ttilde}{\Tilde{T}} 
\newcommand{\Ih}{\widetilde{I}_h}
\newcommand{\Ihz}{\widetilde{I}_h^{SZ}}

\newcommand{\Ihl}{{I}_h}

\newcommand{\ru}{\widetilde{r}_u}

\newcommand{\vhl}{\mathbf{v}_h^{\ell}}
\newcommand{\vh}{\mathbf{v}_h} 

\newcommand{\vhn}{\mathbf{v}_h^n}
\newcommand{\vhnN}{\mathbf{v}_h^{n-1}}

\newcommand{\uh}{\mathbf{u}_h}

\newcommand{\uhl}{\mathbf{u}_h^{\ell}}
\newcommand{\uhn}{\mathbf{u}_h^n}
\newcommand{\uhnN}{\mathbf{u}_h^{n-1}}

\newcommand{\wh}{\mathbf{w}_h}
\newcommand{\whl}{\mathbf{w}_h^{\ell}}
\newcommand{\whn}{\mathbf{w}_h^n}
\newcommand{\whnN}{\mathbf{w}_h^{n-1}}

\newcommand{\zh}{\mathbf{z}_h}

\newcommand{\qh}{q_h}
\newcommand{\qhl}{q_h^\ell}

\newcommand{\qhn}{q_h^n}

\newcommand{\ph}{p_h}

\newcommand{\phn}{p_h^n}

\newcommand{\lh}{\lambda_h}

\newcommand{\lhn}{\lambda_h^n}

\newcommand{\xih}{\xi_h}

\newcommand{\xihn}{\xi_h^n}

\newcommand{\thbfu}{\theta_{\bfu}}
\newcommand{\thp}{\theta_{p}}
\newcommand{\thl}{\theta_{\lambda}}
\newcommand{\nh}{\bfn_h}

\newcommand{\nhtil}{\Tilde{\bfn}_h}

\newcommand{\muh}{\mu_h}

\newcommand\bfPg{{\mathbf{P}_{\Ga}}}
\newcommand\bfPh{{\mathbf{P}_h}}

\newcommand\Bhg{{\bm B_h}}

\newcommand{\FlVel}{\bfV_{\Ga}^N}
\newcommand{\TrVel}{\bfV_{\Gah}^N}
\newcommand{\TrVelL}{\bfv_{\Gah}^N}

\newcommand{\mb}{\bfm}
\newcommand{\gb}{\bfg}
\newcommand{\ahat}{a_S}
\newcommand{\ab}{a}
\newcommand{\db}{d}
\newcommand{\ddb}{\bfd}

\newcommand{\bLb}{b^L}
\newcommand{\btil}{\Tilde{b}}
\newcommand{\Grm}{\mathrm{G}}

\newcommand{\mh}{\bfm_h}
\newcommand{\gh}{\bfg_h}
\newcommand{\ddh}{d_h}
\newcommand{\bfdh}{\bfd_h}

\newcommand{\ah}{a_h}
\newcommand{\ahhat}{a_{S,h}}

\newcommand{\bhtil}{b_h^{L}}
\newcommand{\bhtilda}{\tilde{b}_h}

\newcommand\norm[1]{||#1||}

\begin{document}
\author{Charles M. Elliott, Achilleas Mavrakis}
\title[Fully Discrete ESFEM for the Navier-Stokes equations on Evolving Surfaces]{A Fully Discrete Surface Finite Element Method for the Navier--Stokes equations on Evolving Surfaces with prescribed Normal Velocity}
\address{Mathematics Institute, Zeeman Building, University of Warwick, Coventry CV4 7AL, UK}
\email{\href{mailto:c.m.elliott@warwick.ac.uk}{c.m.elliott@warwick.ac.uk}, \href{mailto:Achilleas.Mavrakis@warwick.ac.uk}{Achilleas.Mavrakis@warwick.ac.uk}}
\date{}
\maketitle

\addtolength\leftmargini{-0.15in}
\begin{quote}
\footnotesize
\textsc{Abstract.}  We analyze two fully time-discrete numerical schemes for the incompressible Navier-Stokes equations posed on evolving surfaces in $\mathbb{R}^3$ with prescribed normal velocity using the evolving surface finite element method (ESFEM). We employ generalized Taylor-Hood finite elements
$\mathrm{\mathbf{P}}_{k_u}$-- $\mathrm{P}_{k_{pr}}$-- $\mathrm{P}_{k_\lambda}$, $k_u=k_{pr}+1 \geq 2$, $k_\lambda\geq 1$, for the spatial discretization, where the normal velocity constraint is enforced weakly via a Lagrange multiplier $\lambda$, and a backward Euler discretization for the time-stepping procedure. Depending on the approximation order of $\lambda$ and weak formulation of the Navier-Stokes equations, we present stability and error analysis for two different discrete schemes, whose difference lies in the geometric information needed. We establish optimal velocity $L^{2}_{a_h}$-norm error bounds ($a_h$ an energy norm) for both schemes when $k_\lambda=k_u$, but only for the more information intensive one when $k_\lambda=k_u-1$, using \emph{iso-parametric} and \emph{super-parametric} discretizations, respectively, with the help of a newly derived surface Ritz-Stokes projection. Similarly, stability and optimal convergence for the pressures is established in an $L^2_{L^2}\times L^2_{H_h^{-1}}$-norm ($H_h^{-1}$ a discrete dual space) when $k_\lambda=k_u$, using a novel \emph{Leray time-projection} to ensure weakly divergence conformity for our discrete velocity solution at two different time-steps (surfaces). Assuming further regularity conditions for the more information intensive scheme, along with an almost weak divergence conformity result at two different time-steps, we establish optimal $L^2_{L^2}\times L^2_{L^2}$-norm pressure error bounds when $k_\lambda=k_u-1$,  using \emph{super-parametric} approximation. Simulations verifying our results are provided, along with a comparison test against a penalty approach.
\end{quote}

\section{Introduction}
Analysis of fluid equations on \emph{evolving surfaces} has been a growing field in recent years both in an analytic and numerical setting,  as they are important in describing many phenomena such as liquid films, bubbles, foams and lipid bilayers \cite{ArroyoSimone2009,hu2007continuum,Reuther2018}. Due to surface evolution, the flow now consists of both in-plane (tangential) and out-of-plane (normal) components, the latter leading to shape changes that are generally unknown \cite{TorresMilArroyo2019,Miura2020,LipMembKeber2014}. For example, lipid bilayer membranes can be described as a two-dimensional surface whose tangential fluid-like motions can be prescribed by the evolving surface Navier-Stokes equations (eNSE) \cite{jankuhn2018incompressible,Koba2016Energy,Krause2023ASurfaces}, while elastic forces or bending energies act in the normal direction \cite{TorresMilArroyo2019,BarGarNurFluidMemb2014,barrett2016stable}, or two phase flows where the interface can be modeled by (eNSE); see \cite{garckeStrPresTwoPhase2025,garcke2023structure,barrett2015stable,agnese2020fitted,li2025optimal} where the latter is treated numerically. 
We also refer to \cite{BarGarNur19} for a general review and additional literature.

A simplified variant can be analyzed by specifying the evolution (normal dynamics) of the surface, which leads to the evolving surface Navier-Stokes equations with prescribed out-of-plane flow (normal velocity) as modeled in \cite{jankuhn2018incompressible,Miura2020}. Considering this assumption, it only remains to solve for the in-plane fluidic part. 
Such a simplified model was studied in \cite{olshanskii2022tangential}, where well-posedness was addressed for the tangential surface Navier-Stokes equations  (TSNSE)  posed on a passively evolving surface embedded in $\mathbb{R}^3$. For our numerical purposes, instead of splitting the surface Navier-Stokes equations into in-plane and out-of-plane contributions, we instead treat the equations as vector-valued PDE for the unknown surface velocity $\bfu$, whose normal component is prescribed via a Lagrange multiplier \cite{brandner2022derivations,Miura2020}:

For an evolving two-dimensional closed, compact and oriented surface $\{\Ga(t)\}_{t\in[0,T]} \subset \mathbb{R}^3$ we want to consider a fully discrete evolving surface Taylor-Hood finite element scheme (ESFEM) for the \emph{evolving surface Navier-Stokes equations with extra Lagrange multiplier}:
\begin{align}
\begin{cases}
\label{eq: NS Lagrange begin ENS}
\tag{eNSEL}
\bfu \cdot \bfng = V_{\Ga}\\
\rho\mat \bfu = -\nbg p + 2\mu\divg(E(\bfu)) + \bff - \lambda^{dir} \bfng\\
\divg \bfu =0.
\end{cases}
\end{align}
where $\bfu(\cdot,t)  : \Gat\times[0,T] \to \R^3$ the velocity field, $p(\cdot,t)  : \Gat\times[0,T] \to \R$ the surface pressure  with $\int_{\Ga\t} p \, ds =0$, acting as a Lagrange multiplier for the incompressibility condition and $\lambda^{dir}(\cdot,t): \Gat\times[0,T] \to \R$ an extra Lagrange multiplier associated with the constraint of the prescribed normal velocity $\bfu\cdot\bfng = V_{\Ga}$ (we explain the superscript $dir$ later in \cref{sec: var form ENS}). In addition, $\mat$ denotes the material derivative, $\rho$ is the density distribution and $\mu$ stands for the viscosity coefficient, while $\bff \in ({L}^2(\Gat))^3$ is an external vector forcing term. Here   the   rate-of-strain (deformation) tensor $E(\cdot)$ \cite{Miura2020}, is defined for an arbitrary vector field $\bfv$, by 	
\begin{equation}
		\begin{aligned}\label{eq: Deformation tensor}
			E(\bfv)  = \half (\nbgcov \bfv + \nbg^{cov,t}\bfv).
		\end{aligned}
	\end{equation}
where  $ \nbgcov \bfv := \bfPg \nbg \bfv $ (with transpose $\nbg^{cov,t}\bfv=(\nbg \bfv)^t\bfPg$) and  $\bfPg$ denotes the projection to the tangent space of $\Gamma$.

The system \eqref{eq: NS Lagrange begin ENS} was derived in \cite{Miura2020}, as a thin film limit of an annular three dimensional classical Navier-Stokes system. Other derivations for surface Navier-Stokes equations on evolving surfaces, include \cite{jankuhn2018incompressible}, where the authors use conservation laws of surface mass and momentum quantities as physical principles and \cite{Koba2016Energy} where instead   such equations were derived by variational
energy principles; we also refer the reader to \cite{brandner2022derivations} for a review regarding such derivations. See also \cite{djurdjevac2021non} regarding well-posedness of Oseen equations on moving domains.

In this work, we want to consider a finite element discretization of \eqref{eq: NS Lagrange begin ENS} using \emph{Taylor-Hood} evolving surface finite elements. Surface finite elements methods (SFEM) have been studied extensively in the literature; see \cite{DziukElliott_acta, EllRan21,DziukElliott_SFEM} for elliptic and parabolic problems on a stationary and evolving surfaces $\Ga$. There have also been extensions to vector or tensor-valued functions \cite{hansbo2020analysis,Hardering2022}. Recently, there have been numerous studies around the surface (Navier)-Stokes equations, including parametric (SFEM) \cite{bonito2020divergence,fries2018higher,reusken2024analysis,elliott2024sfem,demlow2024tangential,demlow2025taylor,nochetto2025surface,hardering2023parametric} and Eulerian (TraceFem-CutFem) approaches, see e.g. \cite{Burman2016CutFE,olshanskii2021inf,jankuhn2021Higherror}. Discretizations of fluid problems on surfaces pose a lot of challenges: Mainly, the simultaneous imposition of an out-of-plane (e.g. tangential) condition (motion) and conformity (continuity) of the discrete velocity vector fields used to approximate solutions in the primitive variables; see \cite{lederer2020divergence,reusken2024analysis,elliott2024sfem,hardering2023parametric,demlow2024tangential,Reuther2018} and the approximation of geometric quantities, e.g. mean and Gaussian curvature. There has been a plethora of papers and techniques dealing with these issues, each with its own advantages and disadvantages. 
For example, approaches that enforce the tangentiality condition strongly have been considered in \cite{bonito2020divergence}, where an $H(\divg)$ conforming  finite element method with an interior penalty approach to enforce $H^1-$ continuity weakly is analyzed, while similarly the authors in \cite{demlow2024tangential,demlow2025taylor} use the surface Piola transforms, instead, to ensure strong tangentiality, $H(\divg)$ conformity and also weak $H^1$-continuity, without requiring additional consistency terms as before. Furthermore, simpler existing $H^1$-conforming \emph{Taylor-Hood} finite element methods have been considered in studies, where instead, the authors enforce  the out-of-plane conditions (tangentiality) weakly either by penalization \cite{reusken2024analysis,hardering2023parametric,olshanskii2021inf,jankuhn2021Higherror} or by Lagrange multipliers \cite{elliott2024sfem,elliott2025unsteady}. We follow the latter.

As mentioned, we consider simplified models for fluid deformable surfaces, such that the deformation of the surface is prescribed. Recently, there has been considerable interest in the numerical approximations for solutions of such systems \cite{Krause2023ASurfaces,KrauseDeform_2023,Reuther2018,olshanskii2024eulerian,sahu2025arbitrarylagrangianeulerianfiniteelement}. However, to the authors' knowledge there does not seem to be much literature involving numerical (error) analysis of systems like \eqref{eq: NS Lagrange begin ENS} on evolving surfaces, especially using (Lagrangian)  parametric FEM. In \cite{olshanskii2024eulerian}, the authors study the discretizations 
of the \emph{tangential surface Navier-Stokes equations on a passively involving surfaces} (TSNSE) using an Eulerian formulation of the PDE (unfitted FEM -- TraceFem). The TSNSE arises by decoupling the velocity into a tangential and normal component \cite{olshanskii2022tangential,brandner2022derivations}. Using TraceFem, the authors were able to study a-priori discretization error bounds and find optimal velocity $L^2_{\ah}$ ($\norm{\cdot}_{\ah}$-- an energy norm) error estimates. On the other hand, due to the use of Eulerian time-stepping scheme and therefore the lack of weak divergence conformity of velocity solutions at two different time-steps, e.g.  $\uhnN\notin\bfV_h^{n,div}$,  only a weaker special $L^2_{H^1}$-norm pressure bound was be obtained, assuming also an inverse (albeit weak) CFL condition. Similar difficulties were observed in the case of Navier-Stokes on moving domains \cite{neilan2024eulerian,burman2022eulerian,von2022unfitted}; see also \cite{besier2012pressure}. Only recently in \cite{garcke2025NSParam} has there been a study involving the stability analysis of a numerical method for the Navier-Stokes equation on evolving surface (eNSE) using parametric FEM.

We investigate numerically a similar scheme to \cite{Krause2023ASurfaces} for the full eNSE, where instead of using a penalty-based approach to constraint of the normal component of the velocity, we use a 
Lagrange-based approach and therefore consider  \eqref{eq: NS Lagrange begin ENS}.
Using $H^1$-conforming \emph{Taylor-Hood} ESFEM, we prove optimal $L^2_{\ah}$-norm velocity error estimates. Additionally, compared to the Eulerian FEM approach \cite{olshanskii2024eulerian}, we are also able to prove $L^2_{L^2}$-norm stability and error bounds for the pressure $p$. This is achieved by using a novel discrete Leray \emph{time-projection} $\hat{\bfu}_h^{n-1}$ \cref{eq: discrete time Leray proj ENS} on the previous time-step $\uhnN$, so that $\hat{\bfu}_h^{n-1}$ is weakly divergence free with respect to the next surface but also ``close'' enough to $\uhnN$. To our knowledge this is the first full error analysis for the Navier--Stokes equations on an evolving surface using parametric finite elements.

\subsection{Main Results}
We analyze two fully time-discrete backward Euler numerical schemes of the incompressible surface Navier--Stokes equations on an evolving surface with prescribed normal velocity using the evolving surface finite element method (ESFEM). We employ a generalized Taylor-Hood finite element triple 
$\mathrm{\mathbf{P}}_{k_u}$ -- $\mathrm{P}_{k_{pr}}$ -- $\mathrm{P}_{k_{\lambda}}$, $k_u = k_{pr}+1 \geq 2$, $k_\lambda\geq 1$, to handle the spatial discretization, where  the normal velocity constraint is enforced weakly via an extra Lagrange multiplier $\lambda$. By $k_u$, $k_{pr}$, $k_{\lambda}$ we represent the degree of polynomials used for the velocity field $\bfu$, pressure $p$ and extra Lagrange multiplier $\lambda$ approximations, respectively. For inf-sup stability reasons we always consider $k_u = k_{pr}+1$, see \cref{Lemma: Discrete inf-sup condition Gah Lagrange ENS,lemma: L^2 H^{-1} discrete inf-sup condition Gah Lagrange ENS}. Depending on the richness of $\mathrm{P}_{k_{\lambda}}$, which involves the approximation of $\lambda$, we present a fully discrete stability and error analysis  for two different numerical schemes \eqref{eq: fully discrete fin elem approx ENS} (for $k_\lambda = k_u$) or \eqref{eq: fully discrete fin elem approx cov ENS} (both for $k_\lambda=k_u$ and  $k_\lambda=k_u-1$). The main difference between these two formulations comes down to the choice of convective term and hence the geometric information needed for the approximations; see \cref{remark: diff formulations reason ENS} for more details.

In summary, the main results are the following: 
\begin{itemize}
    \item Based on discrete weak formulations of \eqref{eq: NS Lagrange begin ENS} (see \cref{sec: problem reform ENS}) we expand on some basic results for the surface finite element method \cite{DziukElliott_L2,DziukElliott_acta,EllRan21,highorderESFEM}), involving the surface lifting, approximation of geometry and perturbation of newly introduced bilinear forms, including new discrete transport formulae in \cref{sec: Setup: evolving surface finite element method}, necessary to describe the relation between the time-derivative and time-dependent integrals.
    \item As with the stationary surface case \cite{elliott2025unsteady}, establishing optimal convergence for the $L^2_{\ah}$-norm of the velocity necessitates the introduction of a \emph{new Ritz-Stokes projection} $\mathcal{R}_h(\bfu)$ over an evolving surface. The main difference now is that the material derivative and the surface Ritz-Stokes projection do not commute, i.e. $\matd \mathcal{R}_h\bfu \neq \mathcal{R}_h\matd\bfu$ and therefore it requires further work to establish appropriate approximation properties. Unfortunately, we are only able to show the following suboptimal bounds for the $L^2$-norms  \vspace{-1mm}
    \begin{equation*}
        \begin{aligned}
        \norm{\matd(\bfu - \mathcal{R}_h(\bfu))}_{L^2(\Gaht)}&\leq ch^{r_u}\sum_{j=0}^{1}\norm{(\matn)^j\bfu}_{H^{k_u+1}(\Gat)} \quad \text{ for } k_\lambda = k_u-1,\\[-2pt]
            \norm{\matd(\bfu - \mathcal{R}_h^{\ell}(\bfu))}_{L^2(\Gaht)} &\leq ch^{\widehat{r}_u}\sum_{j=0}^{1}\norm{(\matn)^j\bfu}_{H^{k_u+1}(\Gat)}  \quad \text{ for } k_\lambda = k_u,
        \end{aligned}
    \end{equation*}
with $r_u = min\{k_u,k_g-1\}$, $\widehat{r}_u = min\{k_u,k_g\}$ (see \cref{lemma: Error Bounds der Ritz-Stokes ENS}). However these are still sufficient to establish optimal $L^2_{\ah}$ convergence for the velocity and eventually $L^2_{L^2}\times L^2_{H_h^{-1}}$ (or $L^2_{L^2}\times L^2_{L^2}$) for the pressure, using \emph{iso-parametric discretization} for $k_\lambda = k_u$ and \emph{super-parametric discretization} for $k_\lambda = k_u-1$, as is expected in the stationary case \cite{elliott2024sfem,elliott2025unsteady,fries2018higher}. The main issue lies in the fact that even if we consider the tangential solution $\bfu$, the material derivative is no longer tangential, thus a lot of the geometric perturbation results that take advantage of the said tangentiality (\cref{lemma: Geometric perturbations lagrange ENS}) become void; see \cref{remark: non tangential mat der ENS,remark: issues mat der ENS} for more details.
\item Using ESFEM we prove $L^2_{L^2}$-norm stability and optimal convergence for the pressure $p$. However, it depend on the choice of $k_\lambda$ in the following sense:

\quad -- For $k_\lambda=k_u$, considering \emph{standard relatively low regularity conditions} (see \cref{assumption: Regularity assumptions for velocity estimate ENS}, the ones used to prove $L^2_{\ah}$ convergence), we can show stability and optimal error estimates in the discrete $L^2_{L^2} \times L^2_{H_h^{-1}}$-norm for the \eqref{eq: fully discrete fin elem approx ENS} scheme (see \cref{Lemma: Pressure stab Estimate ENS,theorem: Pressures Error Estimate ENS}), by introducing a novel  discrete Leray \emph{time-projection} \eqref{eq: discrete time Leray proj ENS}, to ensure weakly divergence conformity for our discrete velocity solution at two different time-steps (surfaces),  along with the help of a discrete inverse Stokes operator \eqref{eq: Discrete inverse Stokes ENS}, a newly developed dual energy norm estimate \Cref{lemma: dual estimate ENS} and the $L^2\times H_h^{-1}$ discrete \textsc{inf-sup} condition \eqref{eq: L^2 H^{-1} discrete inf-sup condition Gah Lagrange ENS} (therefore only establishing error bounds in this weaker dual norm for $\lh$). Similarly, we can prove results for the second scheme \eqref{eq: fully discrete fin elem approx cov ENS}, where one uses the $L^2_{L^2} \times L^2_{L^2}$ \textsc{inf-sup} condition \cref{Lemma: Discrete inf-sup condition Gah Lagrange ENS} instead  to obtain pressure results in the $L^2_{L^2}\times L^2_{L^2}$-norm. The difference between the two schemes regards the geometric information needed, i.e. the approximation of the men curvature $\bfH_h$; see \cref{remark: diff formulations reason ENS}.

\quad -- For $k_\lambda=k_u-1$, on the other hand, we show discrete $L^2_{L^2}\times L^2_{L^2}$-norm stability and optimal convergence results (suboptimal w.r.t. the geometric approximations; see \cref{theorem: Pressures Error Estimate HR ENS}, hence the need of \emph{super-parametric discretization}) only for the second scheme \eqref{eq: fully discrete fin elem approx cov ENS} (it is not possible to prove optimal convergence for \eqref{eq: fully discrete fin elem approx ENS}), by considering \emph{additional regularity conditions} (see \cref{assumption: Regularity assumptions for velocity estimate 2 ENS}), and finding an estimate for the time-derivative approximations in a stronger norm (see \cref{sec: Pressure a-priori estimates for kl= ku-1 ENS}), via an ``almost'' weak divergence conformity result for solutions at two different time-steps (surfaces), along with the $L^2_{L^2} \times L^2_{L^2}$ \textsc{inf-sup} condition \cref{Lemma: Discrete inf-sup condition Gah Lagrange ENS}.
\end{itemize}

\subsection{Outline}
The rest of the chapter is organized as follows. Basic notation, e.g. regarding the functions spaces and evolving surfaces is introduced in \cref{sec: Differential geometry on Surfaces}. The continuous evolving surface Navier-Stokes problems, which we will analyze numerically, are introduced and set-up in \cref{sec: var form ENS}, together with new key Transport formulae for moving surfaces, appropriate for our case. 
In \cref{sec: Setup: evolving surface finite element method} we present the evolving surface finite element method, adapted to our situation, introduce the \emph{Taylor-Hood} elements and develop the relative notation regarding discrete surfaces, velocities, material derivative etc. We also recall and \emph{expand} some basic results involving surface lifting, approximation of geometry and perturbation of new bilinear forms (including new discrete Transport formulae). A new surface Ritz-Stokes projection is introduced in \cref{sec: Ritz-Stokes Projection ENS}, along with error bounds involving the material derivative, which are necessary leading up to the a-priori error analysis for our schemes. In \cref{sec: The Fully Discrete Scheme Main ENS} we finally state our two fully discrete numerical schemes, which depends on the choice of the approximation space of the extra Lagrange multiplier $\lambda$, i.e. $k_\lambda$, and the convective term.
The main stability arguments are also carried out in this section, where we make use of a novel \emph{Leray time-projection}, to ensure weakly divergence conformity for our discrete velocity solution at two different time-steps (surfaces). The main convergence results are presented in \cref{sec: error analysis ENS}, where in the first part we establish optimal $L^2_{\ah}$ velocity error estimates for the respective numerical schemes under \emph{standard regularity assumptions}. Then, we establish optimal pressure error estimates in standard spaces. Specifically, for $k_\lambda=k_u$, with the previous \emph{regularity assumptions} in mind, we prove optimal error pressure bounds in an $L^2_{L^2}\times L^2_{H_h^{-1}}$-norm, while for $k_\lambda=k_u-1$, we prove optimal $L^2_{L^2}\times L^2_{L^2}$-norm error pressure bounds, only after considering
 \emph{additional regularity assumptions}. Finally, numerical results are presented in \Cref{sec: Numerical results ENS}, which support the theoretical results.

\section{Differential geometry on surfaces}\label{sec: Differential geometry on Surfaces}
We recall some fundamental notions and tools concerning surface vector-PDEs. We mainly follow definitions in \cite{hansbo2020analysis,jankuhn2018incompressible,elliott2024sfem}. 

\subsection{The closed smooth surface}\label{Sec: The closed smooth surface ENS}
We consider surfaces $\Ga \in \mathbb{R}^3$ to be a closed, oriented, compact $C^m$  two-dimensional hypersurface embedded in $\mathbb{R}^3$. We will require $m\geq 4$ in our analysis, see \cite[Remark 4.1]{elliott2024sfem}. Note that  since  $\Gamma$ is the boundary of an open set we choose the orientation by setting  $\bfng$ to be the unit outward pointing normal to $\Ga$. Let $d(\cdot)$ denote the signed distance functions such that $d\equiv 0 $ on $\Ga$. Then, by the smoothness of $\Ga$, there exists a small enough $\delta$ such that  $d(\cdot) : U_{\delta} \to \mathbb{R}$ is globally Lipchitz-continuous, where 
$U_\delta \subset \mathbb{R}^3$ is the tubular neighborhood $U_{\delta} = \{x \in \R^3\ : |d(x)| <\delta\}$. 
Then we may define the closest point projection map to $\Ga$ as 
\begin{equation}\label{eq: closest point proje ENS}
    \pi(x) = x - d(x)\bfng(\pi(x)) \in \Ga
\end{equation}
for each $x \in U_\delta$, and for $\delta>0$ small enough, see also \cite{GilTrud98,DziukElliott_acta}, where $\nb d(x) = \bfng(\pi(x))$ for $x \in U_\delta$. We can then see that $d(\cdot)$ and $\pi(\cdot)$ are of class $C^m$ and $C^{m-1}$ on $\overline{U_{\delta}}$. Using the projection $\pi$ we may extend functions $\bfv : \Ga \to \mathbb{R}^3$  on $\Ga$ to the tubular neighbourhood $U_{\delta}$, by 
\begin{equation}
\label{eq: smooth extension ENS}
    \bfv^e(x) = \bfv(\pi(x)), \ \ \ x \in U_{\delta}.
\end{equation}
This extension is constant in the normal direction of $\Ga$ and thus contains certain useful additional properties compared to other regular extensions. We define the extension of the normal to the neighborhood $U_{\delta}$ by $\bfn = \bfng^{e} = \bfng\circ \pi$. Furthermore, define the orthogonal projection operator onto the tangent plane $\bfPg (x) = \bfI - \bfng(x) \otimes \bfng(x)$ for $x \in \Ga$ that satisfies $\bfPg^T=\bfPg^2=\bfPg$, $|\bfPg|_{Fr}=2$ and $\bfPg \cdot \bfng=0$, with its extension to  $U_{\delta}$ given by $\bfP = \bfPg^{e} = \bfI - \bfn \otimes \bfn$. We can readily see that $\bfn|_{\Ga} = \bfng$ and $\bfP|_{\Ga} = \bfPg$.

\subsection{Scalar functions and vector fields}\label{sec: scalar-vector func ENS}

We start by defining useful quantities for \emph{scalar} functions following \cite{DziukElliott_acta}. For a function $f \in C^1(\Ga)$ we define the \emph{tangential (surface) derivative} of $f$ as 
\begin{equation}
    \nbg f(x) = \bfPg(x) \nb f^e(x), \ \ \text{with } \ \underline{D}_if(x) = 
    P_{ij}\partial_jf^e(x), \ \ \ x\in\Ga, 
\end{equation}
where we use the Einstein summation as defined in \Cref{appendix: differential operators ENS} and where  $(\nb f^e)_i = \partial_i(f^e)$, $1 \leq i \leq 3$ is the Euclidean derivative. Here $f^e$ is the extension to $U_{\delta}$ where $f^e|_{\Ga} = f$ as in \eqref{eq: smooth extension ENS}. Note also that  we consider $\nbg f(x) = (\underline{D}_1f(x),$ $ \underline{D}_2f(x),$ $\underline{D}_3f(x))^t$ to be a column vector.

Now, since $\bfPg$ is the orthogonal projection, which is symmetric, we have that 
\begin{equation*}
    \nbg f(x) = \bfPg(x) \nb f^e(x) =  \nb^t f^e(x)\bfPg(x), \ \ \ \ \bfPg \nbg f = \nbg f  \ \ \text{on } \Ga.
\end{equation*}
As noted in \cite[Lemma 2.4]{DziukElliott_acta} the definition of the tangential gradient is independent of the extension.

We now introduce surface derivatives for \emph{vector fields} e.g. $\bfv : \Ga \to \mathbb{R}^3$; for more details see \cite{elliott2024sfem} and \Cref{appendix: differential operators ENS}. Using the Einstein summations described in  \Cref{appendix: differential operators ENS}, the \emph{tangential (directional) derivative} is defined via the derivatives in $\mathbb{R}^3$ as
\begin{equation}\label{eq: tangential derivative ENS}
    \nbg \bfv = (\nbg \bfv)_{ij} = (\nb \bfv^e \bfPg)_{ij} =  \partial_l v^e_iP_{jl}  ,
\end{equation}
which is a $3\times3$ matrix and where we note that $\bfv^e = \bfv \circ \pi$, the normal extension. With this in mind, we see that the tangential derivative turns out to be the Euclidean derivative corresponding to this extended vector field i.e. $\nbg \bfv = \nb \bfv^e|_{\Ga}$. 

For general vector fields in $\mathbb{R}^3$, we readily see that the surface gradients above are not truly tangential. Consequently, we define the \emph{covariant derivatives} as follows (see also \cite{brandner2022derivations,fries2018higher}) 
\begin{equation}\label{eq: covariant derivative ENS}
    \nbgcov \bfv = (\bfPg \nbg \bfv)_{ij} = (\bfPg \nb \bfv^e|_{\Ga} \bfPg)_{ij} = P_{jk}\partial_k v^e_lP_{li},
\end{equation}
which is again a $3\times3$ matrix. Now it is clear that, compared to the tangential derivative defined in \eqref{eq: tangential derivative ENS}, the covariant derivative \eqref{eq: covariant derivative ENS} is also a tangential tensor field. 

We may also define the Weingarten map $\bfH$, to quantify the curvature, as
\begin{equation}\label{eq: Weingarten map}
    \bfH := \nbgcov \bfn, \ \text{ with } \ \bfH \bfn=0, \ \ \ \bfH \bfP = \bfP \bfH =\bfH,
\end{equation}
where we observe that $\bfH$ is a tangential (in-plane) tensor and also $C^{m-2}$ and thus bounded on $\Ga$, see \cite{DziukElliott_acta, GilTrud98}. We also denote by $\kappa = tr(\bfH)$ twice the \emph{mean curvature}.

Now, the surface divergence on $\Ga$ is given by 
\begin{equation}\label{eq: divg vector definition ENS}
\divg(\bfv) = tr(\nbgcov \bfv) = tr(\bfPg \nabla \bfv^e \bfPg)= P_{ik}\partial_kv_l^eP_{li} = tr(\nabla \bfv^e \bfPg) =  tr(\nbg \bfv ). 
\end{equation}
One can always split a vector field $\bfv$ into a tangent and a normal component \ie  $\bfv = \bfv_T + \bfv_n$ with $\bfv_T = \bfPg \bfv$ and $\bfv_n = (\bfv \cdot \bfng)\, \bfng$, where also the following useful formulae holds 
\begin{equation}
    \begin{aligned}\label{eq: split cov ENS}
       \nbgcov \bfv = \nbgcov \bfv_T + \bfH v_n.
       \end{aligned}
\end{equation}
We also define, via \eqref{eq: covariant derivative ENS}, the surface divergence of a tensor function $\bfF : \Ga \to \mathbb{R}^{3 \times 3}$:
\begin{equation}
    \begin{aligned}
       \divg(\bfF) :=  \begin{bmatrix}
\divg(\bfF_{1,j}) \\
\divg(\bfF_{2,j}) \\
\divg(\bfF_{3,j})
\end{bmatrix}, \quad j=1,2,3.
       \end{aligned}
\end{equation}
By $(\bfv\cdot\nbgcov)\bfv$ and $(\bfv\cdot\nbgcov)\bfv$ we denote the following vectors
\begin{equation}
    \begin{aligned}
        \big((\bfv\cdot\nbg)\bfv\big)_{i} = v_jP_{jk}\partial_kv_i^e = \big(\nbg\bfv\bfv\big)_{i} = \nbg\bfv \bfv_T,\\
        \big((\bfv\cdot\nbgcov)\bfv\big)_{i}  = v_jP_{jk}\partial_kv_l^eP_{li} = \big(\nbgcov\bfv\bfv_T\big)_{i} = \bfPg\nbg\bfv \bfv_T,
    \end{aligned}
\end{equation}
Finally, for compact surfaces, i.e. $\Ga$, we present the following integration by parts rule (surface divergence Theorem) (see \cite[Theorem 2.10]{DziukElliott_acta}) for smooth enough $\bfv$ and $\xi$:
\begin{equation}
    \begin{aligned}\label{eq: integration by parts cont ENS}
        \int_{\Ga}\bfv\cdot \nbg\xi \, \ds = -\int_{\Ga}\xi \divg\bfv \, \ds + \int_{\Ga}\kappa\xi(\bfv\cdot \bfng) \, \ds.
    \end{aligned}
\end{equation}

\subsection{Function spaces}\label{sec: Function spaces ENS}
We now define function spaces on the surface $\Ga$. By $L^p(\Ga)$, $p \in [1,\infty]$, we define the function space that consists of functions $\zeta : \Ga \to \mathbb{R}$ that are measurable w.r.t. the surface measure $\ds$, endowed with the standard $L^p$-norms, \ie $\norm{\cdot}_{L^p(\Ga)}$; see \cite{DziukElliott_acta}. By $(\cdot,\cdot)_{L^2(\Ga)}$ we denote the usual $L^2(\Ga)$ inner product with norm $\norm{\cdot}_{L^2(\Ga)} = (\cdot,\cdot)^{1/2}_{L^2(\Ga)}$ and by $\norm{\cdot}_{L^\infty(\Ga)}$ we denote the $L^{\infty}(\Ga)$ norm. Also, $W^{k,p}(\Ga)$ denotes the standard function Sobolev spaces as presented in \cite{DziukElliott_acta}, while by $\mathbf{W}^{k,p}(\Ga)=(W^{k,p}(\Ga))^3$ we write the natural extension to vector-valued functions. Specifically, for $p=2$ we write $(H^{k}(\Ga))^n$ with $n=1,3$, which are in fact Hilbert spaces. So, for a function $\bm{z} : \Ga \to \mathbb{R}^n$, with $n=1,3$, the corresponding norm is given by
\begin{equation}
    \begin{aligned}
    \norm{\bm{z}}_{H^k(\Ga)}^2 = \sum_{j=0}^m \norm{\nbg^j \bm{z}}^2_{L^2(\Ga)},
    \end{aligned}
\end{equation}\vspace{-7mm}

\noindent where $\nbg^j$ denote all the weak tangential derivative of order $j$, i.e. $\nbg^j \bm{z} = \overbrace{\nbg \cdot \cdot \cdot \nbg}^{\text{ j times}} \bm{z}$. For $k=0$ we get the usual $L^2-$ norm $\norm{\cdot}_{L^2(\Ga)}$ for scalar or vector-valued functions.

\noindent For vector fields $\bm{z}  : \Ga \to \mathbb{R}^3$ the space $\bfH^1(\Ga) = (H^1(\Ga))^3$ is equipped with the following norm
\begin{equation}
    \begin{aligned} \label{eq: H1 norm definition}
    \norm{\bm{z}}_{H^1(\Ga)}^2 = \norm{\bm{z}}_{L^2(\Ga)}^2 + \norm{\nbg \bm{z}}_{L^2(\Ga)}^2.
    \end{aligned}
\end{equation}
\noindent In general with \emph{\textbf{bold}} we define vector-valued space, unless specified otherwise. We also set
$$L^2_0(\Ga) :=\{\zeta \in L^2(\Ga) | \ \int_\Ga \zeta \, \ds =0\},$$
to be $L^2$ functions with zero mean value, equipped with the standard $L^2$-norm, while the subspace of tangential vector fields is given by
$$\bfH^1_T := \{\bfv \in \bfH^1(\Ga) \: | \: \bfv \cdot \bfng =0\},$$
endowed with the $H^1$-norm \eqref{eq: H1 norm definition}.
\subsection{Evolving surfaces}\label{sec: Evolving surface cont ENS}
Let $\Ga_0\in \mathbb{R}^3$ be a closed, compact $C^k$ surface, $k \geq 3$. In order to define a collection of closed and compact \emph{evolving hypersurfaces} $\{\Gat\}_{t\in[0,T]}\in \mathbb{R}^3$ for each time $t\in[0,T]$ we assume the existence of a smooth $C^k$ trajectory (or flow) $\Phi(\cdot,t) : \Ga_0 \mapsto \Gat$ of a material points (induced by a velocity field; see \eqref{eq: flow ENS} below) such that
\begin{equation*}
    \Gat = \{\Phi(\bfz,t)\, |\, \bfz \in \Ga_0\}.
\end{equation*}
We also write $\Phi(\cdot,t) = \Phi_t(\cdot)$. This map induces the spacetime manifold  
\begin{equation*}
    \mathcal{G}_T = \bigcup_{t\in[0,T]}\Gat\times\{t\}.
\end{equation*}
We also define $\bfng =\bfng(\bfx,t)$ to be the unit normal of $\Gat$ for $(\bfx,t)\in \mathcal{G}_T$ and $\pi(\bfx,t)$ the closest point projection of $\Gat$, see also \eqref{eq: closest point proje 2 ENS}, defined by
\begin{equation}\label{eq: closest point proje 2 ENS}
    \pi(\bfx,t) = \bfx - d(\bfx,t)\bfng(\pi(\bfx,t),t) \ \, \in \Gat.
\end{equation}
In order to describe the evolution of the surface, we need to specify the velocity of $\Gat$. In our setting, as we shall see in \cref{sec: problem reform ENS}, since the normal component of the velocity $\bfu$ is \textbf{known} \eqref{eq: NS Lagrange begin ENS}, we have $\bfu\cdot\bfng = V_{\Ga}$ and therefore our surface evolves with the \emph{normal material velocity (flow)} 
\begin{equation}
    \label{eq: normal vel flow phd}
    \FlVel = V_{\Ga}\bfng \in C^1([0,T];C^k).
\end{equation}
We may then, define the normal trajectory mapping $\Phi_t(\cdot)$ as followed.
\begin{definition}[Normal trajectory]
Let $\Phi_t(\cdot): [0,T]\times \Ga_0 \mapsto \Gat$ be a purely $C^{k}$ geometric diffeomorphism satisfying  
\begin{align}
\begin{cases}\label{eq: flow ENS}
   &\dfrac{d}{dt} \Phi_t(\bfz) = \FlVel\big(\Phi_t(\bfz),t\big) \qquad \text{ for $\bfz \in \Ga_0$ and $t \in (0,T]$},\\
&\Phi_0(\bfz) = \bfz, \qquad\qquad\qquad\qquad\quad\ \ \,  \text{ for $\bfz \in \Ga_0$},
\end{cases}
\end{align}
where $\FlVel$ the normal material velocity.
\end{definition}
\noindent This map exists from standard ODE theory. We also define its inverse mapping $\Phi_{-t}(\cdot) := \Phi^{-1}(\cdot,t).$ By the assumed smoothness, it is clear that $\norm{\FlVel}_{C^{k}(\Gat)} \leq C_{V_{\Ga}}$ for all $t \in [0,T]$.

\subsubsection{Evolving Banach/Bochner Spaces}\label{eq: Evolving Banach/Bochner Spaces ENS}
We now shall briefly introduce  Bochner-type function spaces for evolving hypersurfaces; for more details see \cite{alphonse2015abstract,alphonse2015some}.

First, let us define the \emph{evolving Banach spaces}, which we denote by $X(t)$. For this purpose, it is necessary to use suitable pushforward (pullback) maps to (and from) standard Banach spaces. This is achieved with the help of the \emph{normal trajectory map} \eqref{eq: flow ENS}. So, if consider 
$\eta \in X_0=X(0)$, a standard Banach space, and $\chi \in  X(t)$ for $t\in[0,T]$, e.g. $X(t) = W^{k,p}(\Gat)$, we may define the pushforward map of $\zeta$ and pullback of $\chi$ by
\begin{equation*}
    \begin{aligned}
        \phi_t\zeta = \zeta\circ\Phi_{-t}\in   X(t), \qquad \phi_{-t}\chi = \chi\circ\Phi_t\in   X_0.
    \end{aligned}
\end{equation*}
This naturally extends to vector-valued functions; see \cite{olshanskii2022tangential}. So, the pairs $\big(W^{k,p}(\Gat),\phi_t\big)_{t\in [0,T]}$ are shown to be compatible in the sense of \cite{alphonse2015abstract,alphonse2015some}; that is, classical results carry over to the evolving setting.

\noindent With these definitions in mind, we can define \emph{evolving Bochner spaces} for $p \geq 1$, as in \cite{alphonse2015abstract,alphonse2015some}, by
\begin{align*}
    L^p_X :=\big\{\chi:[0,T]\to \bigcup X(t)\times \{t\}, \, t \mapsto (\overline{\chi}(t),t)\big| \phi_{-t}\overline{\chi}(t) \in L^p([0,T];X_0)  \big\},
\end{align*}
where we identify $\chi$ with $\overline{\chi}$. This space is endowed with the norm
\begin{align*}
    \norm{\chi}_{L^p_{X}}^p = \int_{0}^{T} \norm{\chi}^p_{X(t)}\, dt.
\end{align*}
If in fact $X(t)$ are Hilbert spaces and $p=2$, the above norm is induced by the $L^2$-inner product
\begin{align*}
    (\chi,v)_{L^2_X} = \int_{0}^{T} (\chi(t),v(t))_{X(t)}\,dt \qquad \forall \ \chi,\,v \in L^2_X.
\end{align*}
Following \cite{alphonse2023function,alphonse2015abstract} we see that we can make the identification $(L^2_X)^* \simeq L^2_{X^*}$, where by $X^*$ we denote the dual space. In case where $p=\infty$ the space is equippes with the norm
\begin{align*}
    \norm{\chi}_{L^{\infty}_{X}} =\mathrm{ess \, sup}_{t\in[0,T]} \norm{\chi}^p_{X(t)}.
\end{align*}

Assuming that $\{\Gat\}_{t\in[0,T]}$ \emph{evolving hypersurfaces} evolve with normal velocity field $\FlVel$, then taking into account the trajectory of the material points \eqref{eq: flow ENS}, we define an appropriate time-derivative known as \emph{normal material derivative} (This time for a vector-valued function).
\begin{definition}[Normal material derivative]
Let the hypersurface $\{\Gat\}_{t\in[0,T]}$ with flow defined by \eqref{eq: flow ENS}. For $\bm{\chi} : \mathcal{G}_T \to \mathbb{R}^3$ a sufficiently smooth vector field, we define the normal material derivative as 
\begin{equation}\label{eq: stong mat deriv ENS}
    \matn\bm{\chi} = \Big(\frac{d}{dt}\bm{\chi}\circ \Phi_t\Big) \circ \Phi_{-t} = \phi_t\Big(\frac{d}{dt}\phi_{-t}\bm{\chi}\Big).
\end{equation}
\end{definition}
It is clear that $\matn(\phi_t\bm{\chi})=\matn(\bm{\chi}\circ\Phi_{-t}) =0$, for $\bm{\chi} \in \Ga_0$. It is also clear that, using the normal extension \eqref{eq: smooth extension ENS}, the normal material derivative can be expressed as $\matn\bm{\chi}= \bm{\chi}_{t}^e(\bfx,t) + (\FlVel(\bfx,t)\cdot\nb)\bm{\chi}^e(\bfx,t)$. This strong definition can be generalized to define a weak material derivative in the sense of \cite[Def. 2.28]{alphonse2015abstract}.

\section{Variational formulation}\label{sec: var form ENS}
\subsection{Problem Set-up and reformulation}\label{sec: problem reform ENS}
In our setting, the normal velocity component $\bfu_n = (\bfu\cdot\bfng)\bfng$ is prescribed in \eqref{eq: NS Lagrange begin ENS} by $\bfu\cdot\bfng = V_{\Gamma}$ and so  determines the material velocity (out-of-plane flow) of our surface $\Gat$. In other words, the normal component of $\bfu$ also describes the purely geometric (shape) evolution of the surface $\Gat$, i.e. $\FlVel = \bfu_n = V_{\Ga}\bfng$. Of course, this is just a simplification assumption, and a further physical arbitrary tangential velocity, e.g. $\bfw_T$, may be present for the in-plane flow of the material points of the surface $\Gat$; see \cite{alphonse2015some}. Nevertheless, the tangent part of the velocity in \eqref{eq: NS Lagrange begin ENS} governs the free lateral motion (in-plane motion) of the material surface \cite{jankuhn2018incompressible,Miura2020}.

Now, notice that the material derivative depends on the full velocity $\bfu$. However, as mentioned, the surface velocity (geometric evolution) of our surface $\Gat$, i.e. $\FlVel$, is completely determined by the normal component of $\bfu$, that is $\bfu_n = V_{\Ga}\bfn = \FlVel$ on $\Gat$; see \eqref{eq: flow ENS}. So, since the trajectories of the material points of $\Gat$ are completely determined by $\FlVel$, we rewrite the material derivative in \eqref{eq: NS Lagrange begin ENS} with the help of the \emph{normal time derivative} \eqref{eq: stong mat deriv ENS} as
\begin{equation}
    \matn\bfu = \mat\bfu - (\bfu_T\cdot\nbg)\bfu. 
\end{equation}

It is clear that only the tangential component of $\bfu$, which describes the free lateral motions, is now  \emph{unknown}. So, considering \eqref{eq: NS Lagrange begin ENS} the problem may be formulated as:

Find velocity $\bfu  : \Ga \to \R^3$,   surface pressure $p  : \Ga \to \R$ with $\int_{\Ga\t} p ds =0$, and  new Lagrange multiplier associated with the constraint of the prescribed normal velocity $\lambda^{dir}: \Ga \to \R$, such that
\begin{align}
\begin{cases}
\label{eq: NS Lagrange ENS}
\tag{eNS}
\bfu \cdot \bfng = V_{\Ga}\\
\rho\big(\matn \bfu + (\bfu_T\cdot\nbg)\bfu\big) - 2\mu\divg(E(\bfu)) + \nbg p + \lambda^{dir} \bfng =  \bff\\
\divg \bfu =0,
\end{cases}
\end{align}
which now represents a non-linear problem with only \emph{unknown} the tangent part of the velocity $\bfu(\bfx,t)$ and the pressures $p,\,\lambda$, with initial velocity $\bfu(0) =  \bfu^0 = \bfu_n(0) + \bfu_T(0) = V_{\Gamma}(0)\bfng + \bfu_T^0$. Now $\lambda^{dir}$ refers to the fact that the convective term is written with respect to the directional derivative $\nbg(\cdot)$.

Furthermore, for numerical analysis purposes, we also need to consider a different formulation for \eqref{eq: NS Lagrange begin ENS}. Using \eqref{eq: identity dir to cov} in \cref{appendix: differential operators ENS} and defining a new pressure as 
\begin{equation}\label{eq: new pressure lambda ENS}
\lambda^{cov} = \lambda^{dir} + \rho\big(\nbg(V_{\Ga}) \cdot\bfu_T +(\bfH\bfu_T) \cdot \bfu_T\big), 
\end{equation}
we can rewrite  \eqref{eq: NS Lagrange begin ENS} as  follows: 

Find velocity $\bfu  : \Ga \to \R^3$,   surface pressure $p  : \Ga \to \R$ with $\int_{\Ga\t} p ds =0$, and new Lagrange multiplier associated with the constraint of the prescribed normal velocity $\lambda^{cov} : \Ga \to \R$, such that
\begin{align}
\begin{cases}
\label{eq: NS Lagrange new ENS}
\tag{eNSc}
\bfu \cdot \bfng = V_{\Ga}\\
\rho\big(\matn \bfu + (\bfu_T\cdot\nbgcov)\bfu \big) - 2\mu\divg(E(\bfu)) + \nbg p + \lambda^{cov} \bfng =  \bff\\
\divg \bfu =0,
\end{cases}
\end{align}
with initial velocity $\bfu(0)= \bfu^0 = \bfu_n(0) + \bfu_T(0) = V_{\Gamma}(0)\bfng + \bfu_T^0$. Once again, only the tangential component of $\bfu$ is \emph{unknown}. We denote the new Lagrange multiplier (pressure) by $\lambda^{cov}$ since the convective term is written with respect to the covariant derivative $\nbgcov(\cdot)$. We also note that this system \eqref{eq: NS Lagrange new ENS} reduces to the system studied in \cite{elliott2025unsteady} when $V_{\Ga}=0$, i.e. when the surface is stationary.

From now on, for convenience, we set the \emph{density distribution} $\rho=1$ and the \emph{viscosity parameter} $\mu=1/2$. Furthermore, from now on, we will simply denote both $\lambda^{dir}$ of \eqref{eq: NS Lagrange ENS} and $\lambda^{cov}$ of \eqref{eq: NS Lagrange new ENS} as $\lambda$, since in both cases it acts as a Lagrange multiplier for the normal constraint. Lastly, we see that the incompressibility condition implies the following inextensibility condition
\begin{equation}
    \begin{aligned}
        \frac{d}{dt}|\Gat| = \int_{\Gat} \divg \bfu=0,
    \end{aligned}
\end{equation}
and knowing that $\int_{\Gat} \divg (\bfPg\bfu) = 0$ from the integration by parts formula \eqref{eq: integration by parts cont ENS}, we also have
\begin{equation}
    \int_{\Gat} \divg \bfu_n = \int_{\Gat} \divg (V_{\Gamma}\bfng) =\int_{\Gat} \kappa V_{\Ga} =0.
\end{equation}

\subsection{Bilinear forms}
We now introduce an abstract notation for the bilinear forms that will be used throughout the rest of the text.
\begin{definition}\label{def: bilinear forms ENS}
    For $\bfw,\bfv,\bfz \in \bfH^{1}(\Gat)$, $q,\xi \in H^1(\Gat)$ we define the following bilinear forms
\begin{align*}
    \mb(t;\bfw, \bfv) &:= \int_{\Gat} \bfw \cdot \bfv \, \ds, \qquad \qquad   \  \gb(t;\FlVel;\bfw, \bfv) := \int_{\Gat} \divg(\FlVel) \bfw \cdot  \bfv \, \ds,\\
    \ahat(t;\bfw, \bfv) &:= \int_{\Gat} E(\bfw):E(\bfv)\, \ds, \qquad \quad \ \ab(t;\bfw,\bfv) := \int_{\Gat} E(\bfw):E(\bfv) \, \ds + \int_{\Gat} \bfw \cdot  \bfv \, \ds,\\
    c^{cov}(t;\bfz; \bfw,\bfv) &:= \int_{\Gat} ((\bfz \cdot \nbgcov)\bfw ) \bfv \, \ds, \qquad \  \!\!\!  c(t;\bfz; \bfw,\bfv) := \int_{\Gat} ((\bfz \cdot \nbg)\bfw ) \bfv \, \ds ,\\
    \db(t;\FlVel;\bfw,\bfv) &:= \int_{\Gat} \nbg \bfw : \nbgcov \bfv \, \Tilde{\mathcal{B}}_{\Ga}(\FlVel,\bfPg) \,\ds+ \int_{\Gat} \nbgcov \bfw : \nbg \bfv \, \Tilde{\mathcal{B}}_{\Ga}(\FlVel,\bfPg) \,\ds.\\
    \bLb(t;\bfw,\{q,\xi\}) &:= \int_{\Gat} \bfw \cdot \nbg q  \, \ds + \int_{\Gat} \xi \bfw\cdot \bfng \, \ds,\\
     \btil(t;\FlVel;\bfw,\{q,\xi\}) &:= m(\xi,\bfw\cdot\matn\bfng)  + g(\FlVel;\xi,\bfw\cdot \bfng) + \int_{\Gat}  \bfw \cdot \mathcal{B}_{\Ga}^{\mathrm{div}}(\FlVel) \nbg q \, \ds,
\end{align*}
where $\mathcal{B}_{(\cdot)}^{(\cdot)}(\bullet,\bullet)$ the deformation tensors are defined in \Cref{lemma: Transport formulae app ENS,lemma: Transport formulae app II ENS}. Furthermore, we define $\bfd(t;\bullet;\bullet,\bullet)$ as in \eqref{eq: bfd formula ENS}.
\end{definition}
We will usually omit the arguments with (t) for clarity reasons, unless we state otherwise.  Now, the analysis over an evolving surface necessitates a form of transport theorem. With the help of \cref{appendix: Transport formulae ENS} we have the following.
\begin{lemma}[Transport Theorem]\label{lemma: Transport formulae cont ENS}
Let $\mathcal{M}\t$ be an evolving surface with velocity $\FlVel$. For function $f:\Ga \to \mathbb{R}$ sufficiently smooth we we have 
\begin{equation}
     \label{eq: Transport formulae 1 ENS}
        \frac{d}{dt} \int_{\Gat} f \,\ds = \int_{\Gat} \matn f + f\, \divg(\FlVel) \,\ds.
\end{equation}
Furthemore, for $\bfw,  \bfv, \matn\bfw , \matn\bfv \in \bfH^1(\Gat)$ and $q, \xi , \matn q, \matn \xi \in H^1(\Ga\t)$, the following transport formulae hold:
\begin{align}
    \label{eq: Transport formulae 2 ENS}
    \frac{d}{dt}\mb(\bfw,\bfv) &= \mb(\matn\bfw,\bfv) + \mb(\bfw,\matn\bfv) + \gb(\FlVel;\bfw,\bfv),\\
    \label{eq: Transport formulae 3 ENS}
    \frac{d}{dt}\bLb(\bfw,\{q,\xi\}) &= \bLb(\matn\bfw,\{q,\xi\}) + \bLb(\bfw,\matn\{q,\xi\})+ \btil(\FlVel;\bfw,\{q,\xi\}),  \\
    \label{eq: Transport formulae 3.5 ENS}
     \!\!\!\!\!\!\!\!\frac{d}{dt}(\nbgcov \bfw,\nbgcov \bfv)_{L^2(\Ga \t)} &= (\nbgcov \matn\bfw,\nbgcov \bfv)_{L^2(\Ga\t)}+ (\nbgcov \bfw,\nbgcov \matn\bfv)_{L^2(\Ga\t)} + \db(\FlVel;\bfw,\bfv),  \\
    \label{eq: Transport formulae 4 ENS} 
     \frac{d}{dt}\ahat(\bfw,\bfv) &= \ahat(\matn\bfw,\bfv) + \ahat(\bfw,\matn\bfv) + \ddb(\FlVel;\bfw,\bfv),\\
     \label{eq: Transport formulae 5 ENS}
     \frac{d}{dt}\ab(\bfw,\bfv) &= \ab(\matn\bfw,\bfv) + \ab(\bfw,\matn\bfv) + \gb(\FlVel;\bfw,\bfv) + \ddb(\FlVel;\bfw,\bfv). 
     \end{align}
\end{lemma}
We note that an explicit expression for $\bfd(\bullet;\bullet,\bullet)$ is also given in \cref{eq: bfd formula ENS}.

\subsection{Variational formulation}\label{sec: Variational formulation ENS}
We wish to derive a weak formulation for both \eqref{eq: NS Lagrange ENS} and \eqref{eq: NS Lagrange new ENS}. For reasons of brevity, we use the shorthand notation $c^{(\cdot)}(\bullet;\bullet,\bullet)$ to denote either $c(\bullet;\bullet,\bullet)$ or $c^{cov}(\bullet;\bullet,\bullet)$ and, as we mentioned, we use a common notation for the extra Lagrange multiplier associated with the normal constraint $\lambda$. This way, we can define one concise weak formulation.

\noindent So, testing \eqref{eq: NS Lagrange ENS} and \eqref{eq: NS Lagrange new ENS}  with appropriately smooth test functions $\bfv$ and $\{q,\xi\}$, using a common notation for $\lambda$, the integration by parts \eqref{eq: integration by parts cont ENS} (see also \cite[Eq. (2.3)]{fries2018higher}), and setting
\begin{align}\label{eq: function f ENS}
    \bm{\eta} = (\eta_1,\eta_2) =(-\kappa V_{\Ga},V_{\Ga}),
\end{align}
we define the \emph{weak formulation of the evolving surface Navier-Stokes with extra Lagrange multiplier}:
\noindent {\bf (eNSW)}: Given $\{\bff, \bm{\eta}\} \in L^2_{\mathbf{L}^2} \times  (L^2_{L^2_0}\times L^2_{L^2})$ we seek solution $\bfu \in L^{\infty}_{L^2}\cap L^2_{\bfH^1}\cap H^1_{\bfH^{-1}}$, and $\{p,\lambda\} \in  L^2_{L^2_0} \times  L^2_{L^2}$ satisfying the initial condition $ \bfu(\cdot,0)= \bfu_0 \in L^2(\Ga(0))$ such that,
\begin{align}
\begin{cases}
    \label{weak lagrange hom NV ENS}
    \tag{eNSW}
        \mb(\matn\bfu,\bfv) + \ahat(\bfu,\bfv) \ + c^{(\cdot)}(\bfu;\bfu,\bfv)\ + \!\!\!\!&\bLb(\bfv,\{p,\lambda\}) = \mb(\bff,\bfv),\\
        &\bLb(\bfu,\{q,\xi\})=m(\bm{\eta},\{q,\xi\}),
    \end{cases}
\end{align}
for all $\bfv\in \bfH^1(\Gat), \, \{q,\xi\}\in (L_0^2(\Gat)\times L^2(\Gat))$ and for a.e. $t\in[0,T]$, where $m(\bm{\eta},\{q,\xi\}) = m(\eta_1,q) + m(\eta_2,\xi).$ For well-posedness of \eqref{weak lagrange hom NV ENS} we can follow similar computations as in the tangential evolving surface Navier-Stokes in \cite{olshanskii2022tangential}. Other ways of proving well-posedness, may involve a Rothe-type method, by discretizing the time-derivative first, appropriately. Moreover, in our case, we instead have to use a new inf-sup that also involves the extra Lagrange $\lambda$; see for example \cite[Lemma 3.3]{elliott2024sfem} (similarly we can prove an evolving analogue, as in \cite[Lemma 3.3]{olshanskii2022tangential}). Nonetheless, we omit further details, and assume that the following a-priori estimates hold,
 \begin{equation}
     \begin{aligned}
         \norm{\matn\bfu}_{L^2_{\bfH^{-1}}} +  \norm{\bfu}_{L^2_{\bfH^{1}}}+  \norm{\{p,\lambda\}}_{L^{2}_{L^2}\times L^2_{L^2}} \leq \norm{\bfu_0}_{L^2(\Ga)} + \norm{\bff}_{L^{2}_{\mathbf{L}^2}} + \norm{\bm{\eta}}_{L^{2}_{L^2}\times L^{2}_{L^2}}.
     \end{aligned}
 \end{equation}

\begin{remark}[About the source term $\bm{\eta}$]\label{remark: about f source ENS}
Given the solution $\{\bfu,p,\lambda\}$ with data $\{\bff,\bm{\eta}\}$  the problems 
may be rewritten using the substitution $\tilde \bfu= \bfu - \eta_2\bfng - \nbg\phi$ where $\Delta_{\Ga}\phi = \eta_1 = \kappa V_{\Gamma}$. With this substitution we obtain the same equations for  $\{\tilde \bfu,\, \tilde p, \,\tilde \lambda\}$ with new data $\{\tilde \bff,  \bm{\tilde \eta}=0\}$ and an extra zero order term $V_{\Gamma}\bfH\tilde \bfu$. Since the new term can be handled easily in the numerical analysis (see e.g. \cite[Section 4.3]{olshanskii2022tangential}),  for reasons of brevity and clarity, we will not include it and simply analyse the numerical discretizations for the systems setting  $\bm{\eta}=0$. In \cref{remark: about f source discrete ENS} this is discussed further. \end{remark}

Thus, for our analysis of numerical discretization in  \cref{sec: The Fully Discrete Scheme ENS} and \cref{sec: error analysis ENS},  we study the following two \emph{solenoidal} weak formulations.

The \emph{directional weak formulations of the evolving surface Navier-Stokes with extra Lagrange multiplier} is the following:

\noindent {\textbf{(eNSW{\small d}): }} \ Determine $\bfu \in L^{\infty}_{L^2}\cap L^2_{\bfH^1}\cap H^1_{\bfH^{-1}}$, and $\{p,\lambda\} \in  L^2_{L^2_0} \times  L^2_{L^2}$ where $ \bfu(\cdot,0)= \bfu_0 \in L^2(\Ga(0))$ such that,
\begin{align}
\begin{cases}
    \label{weak lagrange hom NV dir ENS}
    \tag{eNSW{\small d}}
        \mb(\matn\bfu,\bfv) + \ahat(\bfu,\bfv) \ + c(\bfu;\bfu,\bfv)\ + \!\!\!\!&\bLb(\bfv,\{p,\lambda\}) = \mb(\bff,\bfv),\\
        &\bLb(\bfu,\{q,\xi\})=0,
    \end{cases}
\end{align}
for all $\bfv\in \bfH^1(\Gat), \, \{q,\xi\}\in (L_0^2(\Gat)\times L^2(\Gat))$ and for a.e. $t\in[0,T]$.

The \emph{covariant weak formulations of the evolving surface Navier-Stokes with extra Lagrange multiplier} is the following:

\noindent {\textbf{(eNSW{\small c}): }} \ Determine $\bfu \in L^{\infty}_{L^2}\cap L^2_{\bfH^1}\cap H^1_{\bfH^{-1}}$, and $\{p,\lambda\} \in  L^2_{L^2_0} \times  L^2_{L^2}$ where $ \bfu(\cdot,0)= \bfu_0 \in L^2(\Ga(0))$ such that,
\begin{align}
\begin{cases}
    \label{weak lagrange hom NV cov ENS}
    \tag{eNSWc}
        \mb(\matn\bfu,\bfv) + \ahat(\bfu,\bfv) \ + c^{cov}(\bfu;\bfu,\bfv)\ + \!\!\!\!&\bLb(\bfv,\{p,\lambda\}) = \mb(\bff,\bfv),\\
        &\bLb(\bfu,\{q,\xi\})=0,
    \end{cases}
\end{align}
for all $\bfv\in \bfH^1(\Gat), \, \{q,\xi\}\in (L_0^2(\Gat)\times L^2(\Gat))$ and for a.e. $t\in[0,T]$. 

 Notice that the name \emph{directional} and \emph{covariant} pertains to the different convective term $c(\bullet;\bullet,\bullet)$ and $c^{cov}(\bullet;\bullet,\bullet)$ respectively of the two weak formulations.

\begin{remark}\label{remark: skew-symmetric Ga ENS}
We can rewrite $c(\bullet;\bullet,\bullet)$, $c^{cov}(\bullet;\bullet,\bullet)$ in  \cref{def: bilinear forms ENS} in a skew-symmetric type form. Indeed, starting with \eqref{weak lagrange hom NV ENS}, using integration by parts \eqref{eq: integration by parts cont ENS}, the fact that $\divg(\bfu) = \divg(\bfu_T) + (\bfu\cdot\bfng)\kappa$ and the definition of $\bfeta$ \eqref{eq: function f ENS}, we derive the following  for $\bfu,\bfv\in \bfH^1(\Gat)$
\begin{align}\label{eq: c equiv expression ENS}
    c(\bfu;\bfu,\bfv) = \frac{1}{2}\Big(\int_{\Gat} ((\bfu \cdot \nbg)\bfu ) \bfv \, \ds -\int_{\Gat} ((\bfu \cdot \nbg)\bfv ) \bfu \, \ds \Big) -\frac{1}{2} m(\eta_1\bfu,\bfv).
\end{align}
In a similar manner, and also relying on the fact that $\nbgcov\bfu = \nbgcov\bfu_T + (\bfu\cdot\bfng)\bfH$, we derive that
\begin{align}\label{eq: ccov equiv expression ENS}
    c^{cov}(\bfu;\bfu,\bfv) = \frac{1}{2}\Big(\int_{\Gat} ((\bfu \cdot \nbgcov)\bfu_T ) \bfv \, \ds &-\int_{\Gat} ((\bfu \cdot \nbgcov)\bfv_T ) \bfu \, \ds \Big)\nonumber \\
    &\qquad\qquad + \int_{\Ga} \eta_2\bfu_T\cdot\bfH\bfv_T    
   - \frac{1}{2} m(\eta_1\bfu_T,\bfv_T).
\end{align}
Since, as mentioned in \cref{remark: about f source ENS}, we will analyze \eqref{weak lagrange hom NV dir ENS} or \eqref{weak lagrange hom NV cov ENS}, we consider $\bm{\eta}=(\eta_1,\eta_2)=0$. That means that the expression \emph{outside the large brackets} in \eqref{eq: c equiv expression ENS} and \eqref{eq: ccov equiv expression ENS} will be zero, and therefore the trilinear forms are now skew-symmetric. Its these consistent modifications that we discretize later in \cref{def: bilinear forms discrete ENS}, and \cref{sec: The Fully Discrete Scheme ENS} or eventually in the numerical results; see \cref{sec: Numerical results ENS}.
\end{remark}

\section{Setup: evolving surface finite element method}\label{sec: Setup: evolving surface finite element method}
\subsection{Recap: Evolving Triangulated Surfaces}\label{sec: Recap: Evolving Triangulated Surfaces}
We want to approximate our smooth evolving surface $\Gat$ by an appropriate high order discretization. To begin with, we approximate our initial surface $\Ga(0)$ by a polyhedral approximation (usually with the help of a $k_g$-order Lagrange interpolation), denoted $\Gah(0)$,
contained in the tubular neighborhood $U_\delta$ i.e. $\Gah(0) \subset U_\delta$, whose nodal points $\{\alpha_j(0)\}_{j=1}^{J}$ lie on the surface $\Ga(0)$.
These nodes define an admissible and conforming (see \cite[Section 6.2]{EllRan21}) triangulation $\Th(0)$ such that 
\begin{equation*}
    \Gah(0) = \bigcup_{T(0) \in \Th(0)} T(0).
\end{equation*}
See \cite{elliott2024sfem,DziukElliott_SFEM,DziukElliott_acta,Demlow2009} for specific construction. Then these nodes of the initial triangulation move along $\Gat$ with our prescribed velocity $\FlVel$, i.e. solve the ODE
\begin{equation*}
    d\alpha_j(t)/dt = \FlVel(\alpha_j(t),t) \qquad j=1,...,J \quad \text{ for }t\in[0,T],
\end{equation*}
where $\FlVel$ is the velocity field associated with the evolution of $\Gat$; see \eqref{eq: flow ENS}. Therefore, the nodes lie on $\Gat$ for all $t\in[0,T]$ and thus also define (by construction) an admissible and conforming triangulation $\Th\t$ such that
\begin{equation*}
    \Gah(t) = \bigcup_{T(t) \in \Th(t)} T(t).
\end{equation*}
See \cite[Section 9.5]{EllRan21} for an exact construction of the evolving triangulation. We denote the discrete spacetime surface as
$\mathcal{G}_{h,T} = \bigcup_{t\in[0,T]}\Gaht\times\{t\}.$

We denote by  $h_{T\t}$ the diameter of a simplex $T(t) \in \Th\t$ and set $h := \sup_{t\in[0,T]}\text{max} \{ h_{T\t} \, : \, T(t) \in \Th\t\}$. Considering the assumptions of the  smoothness of $\FlVel$, we may assume that the evolving triangulated surface is \emph{uniformly quasi-uniform} (see \cite[Section 6.2, Prop. 9.8]{EllRan21}), that is,  the velocity is such that the simplicies in $\mathcal{T}_h\t$ do not become too distorted, or more formally,  $\exists \, \rho >0$ such that for all $t\in[0,T]$ and $h\in(0,h_0)$ for sufficiently small $h_0$ we have 
\begin{equation}
    min\big\{ \rho_{T(t)} \ | \ T(t) \in \Th\t \big\} \geq \rho h,
\end{equation}
where $\rho_{T(t)}$ the standard 2-dim ball contained in each triangle $T(t) \in \Th\t$, as long as, the initial mesh $\Gah(0)$ is also. This is important for the extension of the inf-sup to evolving triangulated surfaces $\Gaht$. 

We now denote element-wise outward unit normal to $\Gah$ by $\nh$ and define the discrete projection $\bfP_h$ onto the tangent space of $\Gah$ by
\begin{equation*}
    \bfPh(x) = \bfI - \nh(x) \otimes \nh(x), \quad x \in T, \text{ where } T \in \Th.
\end{equation*}
As such, one can define the discrete tangential surface operators $\nbgh, \, \nbgcovh$ analogous to the continuous case; see \eqref{eq: Eh forms ENS}. Also, the corresponding discrete Weingarten map is $\bfH_h := \nbgcovh\nh.$
For the curved triangulation, we may also consider the set of all the edges of the triangulation $\Th\t$, denoted by  $\mathcal{E}_h\t$. 
For each edge, we define the outward pointing unit co-normals $\bm{\mu}_h^{\pm}$ with respect to the two adjacent triangles $T^{\pm}$. Notice that on discrete surfaces in general
$$[\bm{\mu}_h]|_E = \bm{\mu}_h^+ + \bm{\mu}_h^- \neq 0.$$ 
For more details on the construction of the triangulated surface see \cite{elliott2024sfem,EllRan21} and references therein.

\subsection{Finite Element Spaces and Discrete Material Derivative}\label{sec: Finite Element Spaces and Discrete Material Derivative}
Certain quantities on the discrete surface $\Gah$, e.g. projections and their derivatives, are only defined elementwise on each surface finite element $T$, therefore, it is convenient to introduce broken surface Sobolev spaces on $\Gah$, with their respected norm \cite{EllRan21}. As a discrete analogue to \cref{sec: Function spaces ENS}, the norm $\norm{\cdot}_{H^m(\Th)}$ related to the \emph{broken Sobolev space} $(H^m(\mathcal{T}_h\t))^n$, with $n=1,3$ and is defined by
\begin{equation}
    \begin{aligned}
        \norm{\bm{z}}_{H^m(\Th\t)}^2 = \sum_{T\in\Th\t}\norm{\bm{z}}^2_{H^m(T)}
    \end{aligned}
\end{equation}
for a vector field $\bm{z} \in H^1(\Th\t)^n$. When a quantity belongs to $\Gaht$ then we replace $\Th\t$ with $\Gaht$.

For every $t \in [0,T]$ and $\Gaht$ we now define the space of $H^1-$ conforming Lagrange finite elements, which also evolves with time, as
\begin{equation}\label{eq: h1 conf space ENS}
    S_{h,k_g}^{k}\t := \{v_h \in C^0(\Gah\t) : v_h|_{T\t} = \hat{v}_h\circ \hat{F}_{T\t}^{-1} \text{ for some } \hat{v}_h \in \mathbb{P}^{k}(\hat{T}), \ \text{ for all } T\t \in \Th\t \},
\end{equation}
where $\mathbb{P}^{k}(\hat{T})$ is the space of piecewise polynomials on the reference element $\hat{T}$ of degree $k$ and $\hat{F}_{T\t}$ is some reference map to some reference element $\hat{T}$ also depending on time; see \cite[Section 4]{elliott2024sfem}, \cite[Definition 6.25, Section 9.2]{EllRan21}. Note also that $S_{h,k_g}^{k}\t \subset C^0(\Gaht)\cap H^1(\Gaht)$.

This finite element subspace is spanned by continuous evolving basis functions (nodal basis) of order $k$ denoted by $\chi_{j,k_g}^{k}(\cdot,t)$, $j=1,...,J$,  that satisfy $\chi_{j,k_g}^{k}(\alpha_i(t),t) = \delta_{ij}$ for all $i,\,j = 1,...,J$, so that $S_{h,k_g}^{k}\t=\mathrm{span}\{\chi_{1,k_g}^{k}(\cdot,t),...,\chi_{J,k_g}^{k}(\cdot,t)\},$ where $J$ the number of vertices. All the above are easily extended to \textbf{vector-valued functions} where now the finite element space takes the form:
\begin{equation}
    (S_{h,k_g}^{k}\t)^3 = S_{h,k_g}^{k}\t\times S_{h,k_g}^{k}\t \times S_{h,k_g}^{k} \t\subset \bfH^1(\Gaht).
\end{equation}

For later use, we denote the velocity approximation order by $k_u$, and for the pressures $p$ and $\lambda$, we represent their approximation orders as $k_{pr}$ and $k_{\lambda}$, respectively. So, we denote the appropriate velocity and pressure spaces as
$$\bfV_h\t := (S_{h,k_g}^{k_u}\t)^3, \ Q_h\t:= S_{h,k_g}^{k_{pr}}\t, 
\ \Lambda_h\t :=S_{h,k_g}^{k_\lambda}\t,$$
where, we now  write  $\mathrm{\mathbf{P}}_{k_u}$-- $\mathrm{P}_{k_{pr}}$-- $\mathrm{P}_{k_{\lambda}}$ to denote the extended \emph{Taylor-Hood} surface finite elements. Let us also define the space of \emph{discrete weakly tangential divergence-free functions}
\begin{equation}\label{eq: discrete weakly divfree space beginning ENS}
    \bfV_h^{div}\t := \{\wh \in \bfV_h\t : \bhtil(\wh,\{\qh,\xi_h\}) =0, ~~\forall \, \{\qh,\xi_h\} \in Q_h\t\times\Lambda_h\t\}.
\end{equation}
We also need the dual space of the finite element space $\Lambda_h\t$, which we denote as 
$H_h^{-1}(\Gah\t)$. For $\ell_h \in \Lambda_h\t$  we define its dual norm as 
\begin{equation}\label{eq: H^-1h definition ENS}
    \norm{\ell_h}_{H_h^{-1}(\Gah\t)} = \sup_{\xi_h\in \Lambda_h\backslash\{0\}}\frac{<\ell_h,\xi_h>_{H^{-1},H^1}}{\norm{\xi_h}_{H^1(\Gah\t)}} = \sup_{\xi_h\in \Lambda_h\t\backslash\{0\}}\frac{(\ell_h,\xi_h)_{L^2(\Gah\t)}}{\norm{\xi_h}_{H^1(\Gah\t)}}.
\end{equation}
Later in \cref{sec: The Fully Discrete Scheme Main ENS} we will address the particular choice of parameter choices.

Using the above definitions allows us to characterize the velocity of an arbitrary material point $\bfx_h(t)$ on the discrete surface $\Gaht$, i.e. the \emph{discrete material (fluid) velocity} $\TrVel$ that governs the evolution of $\Gaht$, as the Lagrange interpolant of fluid velocity $\FlVel$  on the nodes $\{\alpha_j(t)\}_{j=1}^J$, given by
\begin{align}\label{eq: vel discrete surf ENS}
    \frac{d}{dt}\bfx_h(t) = \TrVel(\bfx_h(t),t) := \sum_{j=1}^J \dot{\alpha}_j\t \, \chi_{j,k_g}^{k_u}(\bfx_h(t),t) = \sum_{i=1}^J \FlVel(\alpha_j\t,t) \, \chi_{j,k_g}^{k_u}(\bfx_h(t),t), 
\end{align}
for $\bfx_h \in \Gaht$, where $\{\{\chi_{j,k_g}^{k_u} \bfe_i\}_{i=1}^3\}_{j=1}^J$ the nodal basis of $(S_{h,k_g}^{k_u}\t)^3 = \bfV_h\t$ ($\alpha_j\t$ is a vector-valued quantity) and where we write $\dot{\alpha}_j\t = \frac{d}{dt}\alpha_j\t$. Notice that despite $\FlVel$ being purely normal, $\TrVel$ as an interpolant is not, nonetheless we still choose to depict the discrete material velocity with a superscript $N$. 
Since $\TrVel$ is the Lagrange interpolant of $\FlVel$, it is clear that (see also \cite[Lemma 9.18]{EllRan21}),
\begin{align}\label{eq: discrete vel bound ENS}
    \sup_{t\in[0,T]}\norm{\nbgh\TrVel}_{L^{\infty}(\Gaht)} \leq c\norm{\FlVel}_{W^{1,\infty}(\Gat)}.
\end{align}

This evolution, in turn, as in the continuous case \cref{sec: Evolving surface cont ENS}, induces a discrete analogue flow map $\Phi_t^h : [0,T]\times  \Gah(0) \mapsto \Gaht$ from which we write the pushforward and pullback of $\chi_h \in \Gah(0)$ and $\zeta_h \in \Gaht$ as 
\begin{equation}
    \begin{aligned}
        \phi_t^h \chi_h = \chi_h \circ \Phi_{-t}^h \in \Gaht \qquad \phi_{-t}^h \zeta_h = \zeta_h \circ \Phi_t^h \in \Gah(0),
    \end{aligned}
\end{equation}
and so one defines the \emph{discrete strong normal time derivative} by
\begin{equation}\label{eq: discrete stong mat deriv ENS}
    \matd\zeta_h =  \Big(\frac{d}{dt}\zeta_h\circ \Phi_t^h\Big)\circ\Phi_{-t}^h = \phi_t^h\Big(\frac{d}{dt}\phi_{-t}^h\zeta_h\Big),
\end{equation}
for a sufficiently smooth discrete function $\zeta_h$.
It is also convenient to introduce the space
\begin{align*}
    S_{h,k_g}^{k,C}:=\{z_h,\, \matd z_h \in C^0(\mathcal{G}_{h,T})\big|  z_h(\cdot,t) \in S_{h,k_g}^k, \, t \in [0,T]\}.
\end{align*}

A key result of the evolving surface finite element method is that the discrete material derivative of the nodal basis functions on $\Gaht$ satisfy the \emph{transport property}, shown in \cite{DziukElliott_ESFEM,DziukElliott_L2}:
\begin{equation}\label{eq: transport property ENS}
\matd\chi_{j,k_g}^k=0, \quad \text{for } j=1,...,J.
\end{equation}
This result is crucial during the implementation of our numerical scheme. It also implies that for $z_h \in S^{k}_{h,k_g}$, such that $z_h = \sum_{j=1}^J z_j\t \chi^{k}_{j,k_g}(\cdot,t)$, where $\{z_j\t\}_{j=1}^J$ nodal values, the \emph{discrete normal time derivative} can, instead, be written as 
\begin{align}
    \matd z_h = \sum_{j=1}^J \dot{z}_j\t \chi^{k}_{j,k_g}(\cdot,t),
\end{align}
where we write $ \dot{z}_j = \frac{d}{dt} z_j(t)$. Therefore, we see that $\matd z_h \in S^{k}_{h,k_g}$ and if the nodal values $\{z_j\t\}_{j=1}^J$ are sufficiently smooth in times, then $\matd z_h \in S_{h,k_g}^{k,C}$. All of the above naturally extend to vector-valued functions e.g. for $\bm{z}_h  = \sum_{j=1}^J \dot{\bm{z}}_j\t \chi^{k}_{j,k_g}(\cdot,t) \in (S_{h,k_g}^{k}\t)^3$ with vector nodal values $\{\bm{z}_j\t\}_{j=1}^J \in \mathbb{R}^3.$

\subsection{Lifts to Exact Surface}\label{sec: Lifts to Exact Surface}
Now to relate quantities between the discrete surface $\Gaht$ and the continuous one $\Gat$ we want to make use of a lift operator defined using the time-dependent closest point operator $\pi(\cdot,t)$ \eqref{eq: closest point proje 2 ENS}. For that we assume, from now on, that $h\in(0,h_0)$ with $h_0$ sufficiently small, such that $\Gaht \in U_{\delta}(\Gat)$ for all $t\in[0,T]$.

So, for any (scalar : $n=1$  or vector-valued : $n=3$)  finite element function $\bm{z}_h: \mathcal{G}_{h,T} \to \mathbb{R}^n$, we define its \emph{lift} $\bm{z}_h^\ell : \mathcal{G}_{T} \to \mathbb{R}^n $ onto $\Gat$ by
\begin{equation}\label{eq: lift ENS}
    \bm{z}_h^\ell(\pi(x,t),t) := \bm{z}_h(x,t), \quad \text{for } x\in\Gaht.
\end{equation}
Similarly to the extension \eqref{eq: smooth extension ENS}, we also define the \emph{inverse lift} operator for a function $\bm{z}:\mathcal{G}_{T} \to \mathbb{R}^n$:
\begin{equation}\label{eq: inverse lift ENS}
    \bm{z}^{-\ell}(x,t) := \bm{z}(\pi(x,t),t), \quad \text{for } x\in\Gat.
\end{equation}
Notice that the two extensions \eqref{eq: smooth extension ENS} and \eqref{eq: inverse lift ENS} agree on the discrete surface $\Gaht$, i.e. $\bm{z}^{-\ell} = \bm{z}^e|_{\Gah}$ since we have that $\Gaht \subset U_\delta(\Gat)$. 

Again from the assumption that $\Gaht\subset U_\delta(\Gat)$ for all $t\in[0,T]$, the projection $\pi(\cdot)$ is well-defined and a bijection. Therefore, for every triangle in the triangulation $T\t \in \Th\t$ there exists an induced curvilinear triangle $T^\ell\t = \pi(T\t,t) \subset \Gat$. We further assume that the evolving triangulation is exact, that is, $\Th^{\ell}\t$ forms a conforming subdivision of $\Ga$; see \cite[Definition 7.7]{EllRan21}:
\begin{equation*}
    \Gat = \bigcup_{T^\ell\t \in \Th^\ell\t} \!\!\!\!\!\!\!T^\ell\t.
\end{equation*}
Then, we may introduce the lifted finite element spaces 
\begin{equation*}
    (S_{h,k_g}^{k}\t)^\ell = \{z_h^\ell \in H^1(\Gat) : z_h \in S_{h,k_g}^{k}\t \},
\end{equation*}
where we note that $\{(\chi_{j,k_g}^k)^{\ell}(\cdot,t)\}, $ $j=1,...,J$ forms the lifted nodal basis of $(S_{h,k_g}^{k}\t)^\ell.$ The vector-valued quantities are defined in a similar manner.\\

\noindent \textbf{Geometric Approximations.} \ The discretization of the continuous surface $\Gat$, by a discrete triangulated surface $\Gaht$ described in \cref{sec: Recap: Evolving Triangulated Surfaces}, introduces geometric errors. We now present some key geometric error estimates found in \cite{highorderESFEM,DziukElliott_L2,DziukElliott_acta,DziukElliott_SFEM,LubMans2015Wave}.  In our case, for the numerical study of the evolving surface Navier-Stokes equations, we also need to derive some new geometric estimates involving the material derivative of the discrete normal $\nh$ and projection $\bfPh$.
Define $\bfP, \bfn$ the extension of $\bfPg,\,\bfng$ to the discrete surface $\Gah$.


\begin{lemma}[Geometric Errors]\label{lemma: Geometric errors ENS}
Let $\Gat$ and $\Gaht$ be as described above. Then for sufficiently small $h$ we have
\begin{align}\label{eq: geometric errors 1 ENS}
    \norm{\partial_h^{(k)} d}_{L^\infty(\Gaht)} \leq ch^{k_g + 1}, \quad &\norm{\partial_h^{(k)} (\bfn- \nh)}_{L^\infty(\Gaht)} \leq ch^{k_g}, \quad \norm{\bfH- \bfH_h}_{L^\infty(\Gaht)} \leq ch^{k_g -1}, \\
    \label{eq: geometric errors 1b ENS}
        &\norm{[\bm{\mu}_h]}_{L^{\infty}(\mathcal{E}_h\t)} \leq ch^{k_g}, \quad \norm{\bfP[\bm{\mu}_h]}_{L^{\infty}(\mathcal{E}_h\t)} \leq ch^{2k_g}.
\end{align}
As a consequence, we also have the following
\begin{align}\label{eq: geometric errors 2 ENS}
    &\norm{\partial_h^{(k)}(\bfP- \bfPh)}_{L^\infty(\Gaht)} \leq ch^{k_g}, \quad \norm{\partial_h^{(k)}(\bfP_h\cdot\bfn)}_{L^\infty(\Gaht)} \leq ch^{k_g}, \quad \norm{1-\bfn \cdot \nh}_{L^\infty(\Gah)} \leq ch^{k_g+1}
    \\
    \label{eq: geometric errors 3 ENS}
    &\norm{\partial_h^{(k)}(\bfPh \bfP - \bfP)}_{L^\infty(\Gaht)} \leq c h^{k_g}, \qquad \text{ and } \qquad \norm{\partial_h^{(k)} (\bfP \bfPh \bfP - \bfP)}_{L^\infty(\Gaht)} \leq c h^{k_g+1},
\end{align}
where the superscript $(k)$ denotes the $k$-th discrete normal material derivative, and the constant $c>0$ is independent of $h, \,t$.
\end{lemma}
\begin{proof}
    A proof for \eqref{eq: geometric errors 1b ENS} is given in \cite{Olshanskii2014ASurfaces}. The rest of the results for $k=0$ appear in numerous literature \cite{DziukElliott_acta,Demlow2009} and references therein. For $k\geq 1$, most of the results can be found in \cite{highorderESFEM,LubMans2015Wave,DziukElliott_L2} etc. Specifically, let us first focus on the second estimate in \eqref{eq: geometric errors 1 ENS} for $k=1$, which has not been proved, but it is an immediate sequence of already known results. Then, for $k>1$ the result follows with similar arguments.
    
    As in \cite{LubMans2015Wave,DziukElliott_L2} we consider $T\t \in \Gaht$ a single element and w.l.o.g. we assume that $T \in \mathbb{R}^2\times\{0\}$. There, it was proven that
    \begin{align}
        \norm{\matd n_j}_{L^{\infty}(T)} \leq ch^{k_g} \text{ for }j=1,2,
    \end{align}
and $\matd n_3 = \bigo(h^{k_g+1})$, where $\bfn = (n_1,n_2,n_3),$. Then, since by construction $\bfn_h = \bfe_3=(0,0,1)$, we clear see that 
\begin{align}
    \norm{\partial_h^{\circ} (\bfn- \nh)}_{L^\infty(T)} \leq \norm{\partial_h^{\circ}\bfn} \leq ch^{k_g}.
\end{align}
Now using the fact that $\bfP = \bfI - \bfn \otimes\bfn$ and $\bfP_h = \bfI - \bfn_h \otimes\bfn_h$ it clear that 
\begin{align}
    \norm{\partial_h^{\circ}(\bfP- \bfPh)}_{L^\infty(T)}\leq ch^{k_g}.
\end{align}
The estimate $\norm{\partial_h^{\circ}(\bfP_h\cdot\bfn)}_{L^\infty(\Gah)} \leq ch^{k_g}$ was proven similarly in \cite[Lemma 5.4]{DziukElliott_L2}. From this, we see that
\begin{align}
   &\norm{\partial_h^{\circ}(\bfPh \bfP - \bfP)}_{L^\infty(T)}  = \norm{\partial_h^{\circ}((\nh\otimes\nh)\bfP)}_{L^\infty(T)} \nonumber \\
   &\qquad \qquad \qquad\qquad\quad \ \leq c\norm{\nh\otimes\partial_h^{\circ}(\bfP\cdot\nh)}_{L^\infty(T)} + \norm{\partial_h^{\circ}\nh \otimes(\bfP\cdot\nh)}_{L^\infty(T)} \leq c h^{k_g},  \\
    &\norm{\partial_h^{\circ}(\bfP \bfPh \bfP - \bfP)}_{L^\infty(T)} \leq \norm{\partial_h^{\circ}(\bfP\cdot\nh\otimes\bfP \cdot\nh)}_{L^\infty(T)}\leq c h^{k_g+1},
\end{align}
which completes the proof.
\end{proof}

Let $\muh(x,t)$ be the Jacobian of the transformation $\pi(x,t)|_{\Gah}$ : $\Gaht \to \Gat$ such that $\muh(x,t)\dsh = \ds$, then the following estimate also holds 
\begin{equation}
    \norm{1-\muh}_{L^{\infty}(\Gaht)} \leq ch^{k_g+1},
\end{equation}
see for instance \cite{DziukElliott_L2,Demlow2009}.

Now, let us compare functions and surface differential operators defined on the discrete surface $\Gah$ and the exact surface $\Ga$.  Set  $\Bhg = \bfPh(\bfI - d\bfH)\bfPg$, then we also relate the tangential derivatives with the help of the lift extension defined above as follows:
\begin{description}
    \item[Scalar Functions :] Using the extension of a scalar function  $\zeta_h : \Gah \to \mathbb{R}$, \eqref{eq: lift ENS} and the chain rule we have that the discrete tangential gradient of $z_h \in H^1(\Gah)$ for $x \in \Gah$ is
\begin{equation}
    \begin{aligned}
        \nbgh z_h(x) = \nbgh z_h^{\ell}(\pi(x)) = \bfPh(\bfI - d\bfH)\nbg z_h^\ell(\pi(x)) = \Bhg \nbg z_h^\ell(\pi(x)),
    \end{aligned}
\end{equation}
We also notice that $\Bhg : T_x(\Gah) \mapsto T_{\pi(x)}(\Ga)$ is invertible; see \cref{Lemma: Bh estimates ENS}, i.e. we have that the map $\Bhg^{-1}$ on the tangent space $T( \Ga)$ is given by $$\Bhg^{-1}|_{T_x \Ga} = \bfPg(\bfI - d\bfH)^{-1}(\bfI - \frac{\nh\otimes\bfn}{\nh \cdot \bfn})\bfPh,$$
thus for $z_h : \Gah \to \mathbb{R}$ and for $x\in \Gah$ we have that (see also \cite{demlow2007adaptive}),
\begin{equation}
    \begin{aligned}
        \nbg z_h^{\ell}(\pi(x)) = \bfPg(\bfI - d\bfH)^{-1}(\bfI - \frac{\nh\otimes\bfn}{\nh \cdot \bfn})\nbgh z_h(x) = \Bhg^{-1}\nbgh z_h(x).
    \end{aligned}
\end{equation}
Noticing that the matrix $\bfG = (\bfI - \frac{\nh\otimes\bfn}{\nh \cdot \bfn}) = (\bfI - \frac{\nh\otimes\bfn}{\nh \cdot \bfn}) \bfPh$ we can easily calculate that  $\Bhg^{-1}\Bhg = \bfPg$ and $\Bhg\Bhg^{-1} = \bfPh$.
\item[Vector Valued Functions :] Considering the tangential gradients of vector valued function \eqref{eq: tangential derivative ENS}, we may apply the lift extension component-wise, for each row, and with similar calculations to the scalar case we may get for $\bm{z}_h \in \bfH^1(\Gah)$, after factoring out the common map $\Bhg$, that
\begin{equation}
    \begin{aligned}\label{eq: gah to ga Bh ENS}
        \nbgh \bm{z}_h(x) = \nbgh \bm{z}_h^\ell(\pi(x)) 
        = \nbg \bm{z}_h^\ell(\pi(x))\Bhg^t.
    \end{aligned}
\end{equation}
Likewise for the inverse transformation we get
\begin{equation}
    \begin{aligned}
        \nbg \bm{z}_h^{\ell}(\pi(x)) = \nbgh \bm{z}_h(x)(\Bhg^{-1})^t.
    \end{aligned}
\end{equation}
\end{description}
Furthermore, we define $\mathcal{Q}_h = \frac{1}{\mu_h}\Bhg^t\Bhg = \frac{1}{\mu_h}(\bfI - d\bfH)\bfPg\bfPh\bfPg(\bfI - d\bfH)$; see \cite{DziukElliott_L2,demlow2007adaptive,LubMans2015Wave} and $\mathcal{W}_h = \frac{1}{\mu_h}\Bhg-\bfPg$, which will be important later on in \cref{sec: Geometric Perturbations Errors}.

\begin{lemma}\label{Lemma: Bh estimates ENS}
Let $\Gat$ and $\Gaht$ be as described above. Then for sufficiently small $h$ we have the following bounds
\begin{align}
    \label{eq: Bh stability ENS}
      \norm{\Bhg}_{L^{\infty}(\Gat)} &\leq c, \qquad\qquad\norm{\Bhg^{-1}}_{L^{\infty}(\Gaht)} \leq 1, \\
     \label{eq: Bh estimates ENS}
     \norm{\partial_h^{(k)}\mathcal{W}_h}_{L^{\infty}(\Gaht)} &\leq ch^{k_g}, \quad \norm{\partial_h^{(k)}\bfP\mathcal{W}_h}_{L^{\infty}(\Gaht)} \leq ch^{k_g+1},  \quad\norm{\bfP - \mathcal{Q}_h}_{L^{\infty}(\Gaht)} \leq ch^{k_g+1}, \nonumber\\
     &\!\!\!\!\qquad\qquad\norm{\bfP(\partial_h^{(k)}\mathcal{Q}_h)\bfP}_{L^{\infty}(\Gaht)} \leq ch^{k_g+1}
\end{align}
where the superscript $(k)$ denotes the $k$-th discrete material derivative.
\end{lemma}
\begin{proof}
For \eqref{eq: Bh stability ENS} recall   \cite{Heine2004,DziukElliott_acta,hansbo2020analysis,elliott2024sfem}. Regarding  
\eqref{eq: Bh estimates ENS} the last two estimates have been shown in \cite{highorderESFEM,LubMans2015Wave}. The first inequality in \eqref{eq: Bh estimates ENS} is an immediate consequence of the first inequality in \eqref{eq: geometric errors 3 ENS} and the second of the second inequality in \eqref{eq: geometric errors 3 ENS}. 
\end{proof}

\subsubsection{Norm equivalence}
We use $\sim$ to denote the norm equivalence independent of $h$. Recall the following lemma involving equivalence of norms on different surfaces.
\begin{lemma}[Norm Equivalence]
    \label{lemma: norm equivalence ENS}
    Let $v_h \in H^j(T),$ $j\geq 2$ for $T\in \Th$ a face of $\Gah$ and let 
    $T^{\ell} \in \mathcal{T}_h^{\ell}$ be a face of the lifted surface $\Ga$. Then for $h$ small enough we have that following equivalences
    \begin{equation}\label{eq: norm equivalence ENS}
        \begin{aligned}
            \norm{v_h^{\ell}}_{L^2(T^{\ell})}&\sim \norm{v_h}_{L^2(T)}\sim \norm{\tilde{v}_h}_{L^2(\Tilde{T})}, \\
            \norm{\nbg v_h^{\ell}}_{L^2(T^{\ell})}&\sim \norm{\nbgh v_h}_{L^2(T)}\sim \norm{\nbglin\tilde{v}_h}_{L^2(\Tilde{T})},\\
            |\nbgh^j v_h|_{L^2(T)} \leq c \sum_{k=1}^j & |\nbg^k v_h^{\ell}|_{L^2(T^\ell)},  \qquad 
             |\nbg^j v_h^{\ell}|_{L^2(T^{\ell})} \leq c \sum_{k=1}^j |\nb_{\Gah}^k v_h|_{L^2(T^\ell)}.
        \end{aligned}
    \end{equation}
\end{lemma}
\begin{proof}
    The first two set of equivalences can be found in several literature, e.g. \cite{DziukElliott_acta,Demlow2009}. Regarding the last two inequalities, see \cite[Lemma 3.3]{LarssonLarsonContDiscont}.
\end{proof}

\subsection{Material Velocity of Lifted Surface}\label{sec: Material Velocity of Lifted Surface}
In \cref{sec: Finite Element Spaces and Discrete Material Derivative} we described that the discrete surface $\Gaht$ evolves with discrete material velocity $\TrVel$. With this in mind, following \cite{DziukElliott_L2,DziukElliott_acta,EllRan21}, the associated material velocity of the \emph{lifted} material points $\bfy\t$ of $\Gat$ are defined as
\begin{align}\label{eq: lifted discrete velocity ENS}
    \TrVelL(\bfy\t,t)= \dot{\bfy}\t  = (\bfP - d\bfH)(\bfx_h,t)\TrVel(\bfx_h,t) - d_t(\bfx_h,t)\bfn(\bfx_h,t) - d(\bfx_h,t) \bfn_t(\bfx_h,t),
\end{align}
where $\bfy\t = \pi(\bfx_h\t,t)$ \eqref{eq: closest point proje 2 ENS}, with $\bfx_h\t \in \Gaht$. Then, since $d_t(\bfx_h,t)\bfn(\bfx_h,t)$ is just the normal component of the fluid velocity (in our setting it is our normal fluid velocity $\FlVel$), and the rest of the terms are tangential, it is clear that $\TrVelL - \FlVel$ \emph{ is tangent}. 

So, as before, this \emph{lifted discrete material (fluid) velocity} $\TrVelL$, induces a lifted discrete flow map $\Phi^{\ell}_t:[0,T]\times\Gat$ from which we write the pushforward and pullback of $\chi_h^{\ell} \in \Ga(0)$ and $\zeta_h^{\ell} \in \Gat$ as 
\begin{equation*}
    \begin{aligned}
        \phi_t^{\ell}  \chi_h^{\ell}  = \chi_h^{\ell}  \circ \Phi_{-t}^{\ell}  \in \Gat \qquad \phi_{-t}^{\ell}  \zeta_h^{\ell}  = \zeta_h^{\ell}  \circ \Phi_t^{\ell}  \in \Ga(0),
    \end{aligned}
\end{equation*}
with $\Phi^{\ell}_t(\bfx) = \Phi^{\ell}(\bfx,t) =\Phi^h(\bfx^{-\ell},t)^{\ell} = \Phi^{h}_t(\bfx^{-\ell})^{\ell}$. So, one may also define the \emph{lifted discrete strong normal material derivative} by
\begin{equation}\label{eq: lifted discrete stong mat deriv ENS}
    \matdl\zeta =  \Big(\frac{d}{dt}\zeta\circ \Phi_t^{\ell}\Big)\circ\Phi_{-t}^{\ell} = \phi_t^{\ell}\Big(\frac{d}{dt}\phi_{-t}^{\ell}\zeta\Big),
\end{equation}
for a sufficiently smooth discrete function $\zeta : \mathcal{G}_T \to \mathbb{R}^n$. For more details, see also \cite{EllRan21}. Therefore, it is clear that $\matdl \zeta = (\matd\zeta^{-\ell})^\ell$. Furthermore, the \emph{transport property} clearly still holds for the lifted material derivative; see \cite{DziukElliott_L2}, i.e.
\begin{equation}\label{eq: lifted transport property ENS}
\matdl(\chi_{j,k_g}^k)^{\ell}(\cdot,t)=0, \quad \text{for } j=1,...,J,
\end{equation}
where, as in \cref{sec: Finite Element Spaces and Discrete Material Derivative}, $\{(\chi_{j,k_g}^k)^{\ell}(\cdot,t)\}, $ $j=1,...,J$ is the lifted nodal basis of $(S_{h,k_g}^{k}\t)^\ell.$\\

Now notice, once again, that $\TrVel$ is the Lagrange interpolant of $\FlVel$ onto the finite element space $(S_{h,k_g}^{k_u}\t)^3$. With that in mind, the following high order bounds have been derived in \cite[Lemma 9.22]{EllRan21} or \cite[Lemma 5.4]{highorderESFEM} for the difference between the flow velocity $\FlVel$ and the lifted discrete flow velocity $\TrVelL$ \eqref{eq: lifted discrete velocity ENS}.
\begin{lemma}
For $a=0,1$ it holds that
\begin{equation}\label{eq: difference of velocities ENS}
    \norm{\partial_{\ell}^{a}(\FlVel - \TrVelL)}_{L^{\infty}(\Gat)} + h\norm{\nbg\partial_{\ell}^{a}(\FlVel - \TrVelL)}_{L^{\infty}(\Gat)} \leq c(h^{k_g+1}+h^{k_u+1}),
\end{equation}
where $c$ a positive constant depending on $\mathcal{G}_T$ but independent of $h, \, t$.  
\end{lemma}

Furthermore, these bounds allow us to relate the two material derivatives $\matn,$ $\matdl$ in the following way; see \cite[Lemma 9.25]{EllRan21}.
\begin{lemma}
Let $\bm{z} : \mathcal{G}_T \to \mathbb{R}^n$ (scalar: $n=1$, vector-valued: $n=3$) and let $\matn \bm{z}$ and $\matdl \bm{z}$ exist, then
\begin{equation}\label{eq: relate matd to matdl ENS}
\begin{aligned}
     \norm{\matn\bm{z}- \matdl \bm{z}}_{L^{2}(\Gat)} \leq c(h^{k_g+1}+h^{k_u+1})\norm{\bm{z}}_{H^1(\Gat)},  \qquad \text{for }\bm{z}\in (H^1(\Gat))^n,\\
     \norm{\nbg(\matn\bm{z} - \matdl \bm{z})}_{L^{2}(\Gat)} \leq c(h^{k_g}+h^{k_u})\norm{\bm{z}}_{H^2(\Gat)},  \qquad \text{for } \bm{z} \in (H^2(\Gat))^n,
\end{aligned}
\end{equation}
where $c$ a positive constant depending on $\mathcal{G}_T$ but independent of $h, \, t$. 
\end{lemma}
The vector-valued case is just a natural point-wise extension. For more details on these type of results see \cite{EllRan21,DziukElliott_acta,DziukElliott_L2,highorderESFEM} and references therein.

\subsection{Discrete Bilinear Forms and Transport Theorem}
Analogously to the continuous case in \cref{sec: scalar-vector func ENS}  we define similarly for the discrete case the following operators for $\wh \in \bfH^1(\Gah)$,
\begin{equation}
    \begin{aligned} \label{eq: Eh forms ENS}
        \nbgcovh \wh &:= \bfPh \nbgh \wh , \qquad 
        E_h(\wh) := \half (\nbgcovh \wh + \nb_{\Gah}^{cov,t} \wh).
     \end{aligned}
\end{equation}
Our discrete finite element methods later on \cref{sec: The Fully Discrete Scheme ENS}  relies on the following time-dependent discrete bilinear forms and inertia terms.
\begin{definition}\label{def: bilinear forms discrete ENS}
    For $\wh,\vh,\zh \in \bfH^{1}(\Gaht)$, $\qh,\xih \in H^1(\Gaht)$ we define the following bilinear forms:
\begin{align*}
    &\mh(t;\wh,\vh) := \int_{\Gah \t} \wh \cdot \vh \, \dsh, \qquad  \ \ \gh(t;\TrVel;\wh,\vh) :=  \int_{\Gah \t} \!\!\!\!\divgh(\TrVel) \wh \cdot \vh \, \dsh,\\
    &\ahhat(t;\wh,\vh) := \int_{\Gaht} E_h(\wh):E_h(\vh) \,\dsh, \quad\\
    &\ah(t;\wh,\vh) := \int_{\Gaht}\!\!\!\! E_h(\wh):E_h(\vh) \, \dsh + \int_{\Gaht} \!\!\!\!\wh \cdot \vh \, \dsh,\\
    &\ddh(t;\TrVel;\wh,\vh) := \int_{\Gaht} \!\!\!\! \nbgh \wh : \nbgcovh \vh \Tilde{\mathcal{B}}_{\Gah}(\TrVel,\bfPh) \, \dsh  + \int_{\Gaht} \!\!\!\! \nbgcovh \wh : \nbgh \vh \Tilde{\mathcal{B}}_{\Gah}(\TrVel,\bfPh) \, \dsh\\
    &c_h(t;\zh; \wh,\vh) := \frac{1}{2} \Big(\int_{\Gaht } ((\bfz_h \cdot \nbgh)\wh ) \vh \, \dsh - \int_{\Gaht } ((\bfz_h \cdot \nbgh)\vh ) \wh \, \dsh \Big), \\
    &c_h^{cov}(t;\zh; \wh,\vh) :=\frac{1}{2} \Big(\int_{\Gah } ((\bfz_h \cdot \nbgcovh)\wh ) \cdot\vh \; \dsh - \int_{\Gah } ((\bfz_h \cdot \nbgcovh)\vh ) \cdot \wh \; \dsh \Big)\\ 
      &\qquad\qquad\qquad\qquad\qquad\qquad - \frac{1}{2}\Big(\int_{\Gaht} (\wh\cdot \nh) \zh \cdot \bfH_h \vh \; \dsh - \int_{\Gaht} (\vh\cdot \nh) \zh \cdot \bfH_h \wh \; \dsh \Big),\\
    &\bhtil(t;\wh,\{\qh,\xih\}) :=\int_{\Gaht} \wh \cdot \nbgh \qh  \, \dsh + \int_{\Gaht} \xi_h \wh\cdot \nh \, \dsh,\\
     &\bhtilda(t;\TrVel;\wh,\{\qh,\xih\}) := m_h(\xih,\wh\cdot\matd\nh)  + g_h(\TrVel;\xih,\wh\cdot \nh) + \int_{\Gaht}  \wh \cdot \mathcal{B}_{\Gah}^{\mathrm{div}}(\TrVel) \nbgh \qh \, \ds,
\end{align*}
where $\mathcal{B}_{(\cdot)}^{(\cdot)}(\bullet,\bullet)$ the discrete deformation tensors are defined in \Cref{lemma: Transport formulae app ENS}. Furthermore, we define $\bfdh(\bullet;\bullet,\bullet)$ as the discrete version of \eqref{eq: bfd formula ENS}; similar to $\ddh(\bullet;\bullet,\bullet)$.
\end{definition}

\begin{remark}\label{remark: skew-symmetric Gah ENS}
Notice that we have explicitly skew-symmetrized the trilinear forms $c_h$ and $c_h^{cov}$, according to their continuous counterpart in \cref{remark: skew-symmetric Ga ENS}, when $\bfeta =0$, which is standard in the literature when using $H^1$-conforming Taylor-Hood elements; see \cref{sec: The Fully Discrete Scheme ENS} for more details. More specifically for $c_h^{cov}$ notice that we have replaced $\nbgcovh\bfPh\wh$ by the equivalent expression $\nbgcovh\wh - (\wh\cdot\nh)\bfH_h$; see \eqref{eq: split cov ENS}.
\end{remark}

 We omit the time-dependency on the bilinear forms, from now unless stated otherwise. Throughout the text we also use the following \emph{energy norm}: 
\begin{align}\label{eq: energy norm ENS}
\norm{\wh}_{\ah} = \norm{E_h(\wh)}_{L^2(\Gaht)} + \norm{\wh}_{L^2(\Gaht)}, \qquad 
\norm{\wh}_{\ah} \leq c\norm{\wh}_{H^{1}(\Gaht)}.
\end{align}

Now, with the help of the Transport theorem presented in \cref{appendix: Transport formulae ENS} the following Lemma holds.
\begin{lemma}[Transport Formula Triangulated Surface]\label{lemma: Transport formulae discrete ENS}
Let $\Gaht$ be an evolving admissible triangulation with material velocity $\TrVel$, as prescribed in \cref{sec: Finite Element Spaces and Discrete Material Derivative}. Then, for $f_h$ sufficiently smooth we have 
\begin{equation}
     \label{eq: Transport formulae 1 discrete ENS}
        \frac{d}{dt} \int_{\Gaht} f_h \,\ds = \int_{\Gaht} \matd f_h + f_h \,\divgh(\TrVel) \,\ds.
\end{equation}
For $\wh, \vh \in  \big(S_{h,k_g}^{k,C}\big)^3 $, and $\qh,\xih \in  S_{h,k_g}^{k,C}$ we have
\begin{align}
   \label{eq: Transport formulae 2 discrete ENS}
    \frac{d}{dt}m_h(\wh,\vh) &= m_h(\matd\wh,\vh) + m_h(\wh,\matd\vh) + g_h(\TrVel;\wh,\vh),\\
     \label{eq: Transport formulae 3 discrete ENS}
    \frac{d}{dt} \bhtil(\wh,\{\qh,\xi_h\}) &= \bhtil(\matd\wh,\{\qh,\xi_h\}) + \bhtil(\wh,\matd\{\qh,\xi_h\})+ \bhtilda(\TrVel;\wh,\{\qh,\xi_h\}),  \\
    \label{eq: Transport formulae 3.5 ENS}
     \!\!\!\!\!\!\!\! \frac{d}{dt}(\nbgcovh \wh,\nbgcovh \vh)_{L^2(\Gah\t)} &= (\nbgcovh \matd\wh,\nbgcovh \vh)_{L^2(\Gah\t)} \\
     &+ (\nbgcovh \wh,\nbgcovh \matd\vh)_{L^2(\Gah\t)} + \ddh(\TrVel;\wh,\vh), \nonumber \\
   \label{eq: Transport formulae 4 discrete ENS}
     \frac{d}{dt}\ahhat(\wh,\vh) &= \ahhat(\matd\wh,\vh) + \ahhat(\wh,\matd\vh) + \bfdh(\TrVel;\wh,\vh),\\
     \label{eq: Transport formulae 5 discrete ENS}
     \frac{d}{dt}\ah(\wh,\vh) &= \ah(\matd\wh,\vh) + \ah(\wh,\matd\vh) + \gh(\TrVel;\wh,\vh)+ \bfdh(\TrVel;\wh,\vh).  
     \end{align}
\end{lemma}
\begin{remark}[About $\bfdh(\bullet;\bullet,\bullet)$]\label{remark: assertion bdfh ENS}
The discrete formula of $\bfdh(\bullet;\bullet,\bullet)$ can be derived similar to the discrete formula of $\ddh(\bullet;\bullet,\bullet)$ in \cref{def: bilinear forms discrete ENS} by considering \eqref{eq: bfd formula ENS}. For our analysis, we mainly need its bound and eventually some geometric perturbations, see e.g. \cref{lemma: Geometric perturbations lagrange II ENS}. For brevity reasons we make the following compromise:

In general, we see that the structure of $\bfdh(\bullet;\bullet,\bullet)$ is similar to the $\ddh(\bullet;\bullet,\bullet)$; see \eqref{eq: bfd formula ENS}. For that reason, from now on, when we want to carry out calculations for $\bfdh(\bullet;\bullet,\bullet)$, we instead do those calculations, first for the easier case $\ddh(\bullet;\bullet,\bullet)$, and say these results then easily extend to $\bfdh(\bullet;\bullet,\bullet)$. 
\end{remark}

\begin{lemma}[Transport Formula Lifted Triangulated Surface]\label{lemma: Transport formulae lifted ENS}
    Let $\Gat$ be an evolving surface consisting of the lifted triangulation, whose elements move with velocity $\TrVelL$ as described in \cref{sec: Material Velocity of Lifted Surface}. Then for For $\bfw,  \bfv, \matn\bfw , \matn\bfv \in \bfH^1(\Gat)$ and $q, \xi , \matn q, \matn \xi \in H^1(\Ga\t)$,
\begin{align}
    \label{eq: Transport formulae 2 lift ENS}
    \frac{d}{dt}\mb(\bfw,\bfv) &= \mb(\matdl\bfw,\bfv) + \mb(\bfw,\matdl\bfv) + \gb(\TrVelL;\bfw,\bfv),\\
    \label{eq: Transport formulae 3 lift ENS}
    \frac{d}{dt}\bLb(\bfw,\{q,\xi\}) &= \bLb(\matdl\bfw,\{q,\xi\}) + \bLb(\bfw,\matdl\{q,\xi\})+ \btil(\TrVelL;\bfw,\{q,\xi\}),  \\
    \label{eq: Transport formulae 3.5 lift ENS}
     \!\!\!\!\!\!\!\!\!\!\! \frac{d}{dt}(\nbgcov \bfw,\nbgcov \bfv)_{L^2(\Ga \t)} &= (\nbgcov \matdl\bfw,\nbgcov \bfv)_{L^2(\Ga\t)}+ (\nbgcov \bfw,\nbgcov \matdl\bfv)_{L^2(\Ga\t)} + \db(\TrVelL;\bfw,\bfv),  \\
    \label{eq: Transport formulae 4 lift ENS}
     \frac{d}{dt}\ahat(\bfw,\bfv) &= \ahat(\matdl\bfw,\bfv) + \ahat(\bfw,\matdl\bfv) + \ddb(\TrVelL;\bfw,\bfv),\\
     \label{eq: Transport formulae 5 lift ENS}
     \frac{d}{dt}\ab(\bfw,\bfv) &= \ab(\matdl\bfw,\bfv) + \ab(\bfw,\matdl\bfv) + \gb(\TrVelL;\bfw,\bfv) + \ddb(\TrVelL;\bfw,\bfv). 
     \end{align}     
\end{lemma}

We note, once again, that an explicit expression for $\bfd(\bullet;\bullet,\bullet)$ is given in \eqref{eq: bfd formula ENS}.  Also, in comparison to \cref{def: bilinear forms ENS} we have that
\begin{align}\label{eq: lifted btil ENS}
    \btil(\TrVelL;\bfw,\{q,\xi\}) := m(\xi,\bfw\cdot\matdl\bfng)  + g(\TrVelL;\xi,\bfw\cdot \bfng) + \int_{\Gat}  \bfw \cdot \mathcal{B}_{\Ga}^{\mathrm{div}}(\TrVelL) \nbg q \, \ds.
\end{align}
\subsection{Properties of the discrete bilinear forms}
We now present some known uniform interpolation and Korn-type inequalities that can be found in \cite{elliott2024sfem,jankuhn2021trace} along with some key results involving \emph{discrete weakly tangential divergence-free functions}, i.e. $\wh \in \bfV_h^{div}\t$ \eqref{eq: discrete weakly divfree space beginning ENS}. We also extend some of those results using Broken Sobolev spaces.
\begin{lemma}[Interpolation and Korn-type inequalities]\label{lemma: Korn's type inequalities Lagrange ENS}
There exists $c>0$ depending on $\mathcal{G}_T$ but independent of $h, \, t$, such that he following important uniform inequalities hold,
\begin{align}
        \label{discrete Korn's inequality T nh ENS}
         \norm{\wh}_{H^1(\Gaht)} &\leq c(\norm{E_h(\wh)}_{L^2(\Gaht)}  +\norm{\wh}_{L^2(\Gaht)} + h^{-1}\norm{\wh \cdot \nh}_{L^2(\Gaht)}),
        \\
        \label{eq: coercivity and Korn's inequality Lagrange ENS}
        \norm{\wh}_{H^1(\Gaht)}^2 &\leq c h^{-2} \norm{\wh}_{\ah}^2  \qquad\qquad\qquad\qquad\quad \ \textnormal{ for all } \wh \in \bfV_h\t,\\
        \label{eq: interpolation inequality Gaglia ENS}
        \norm{\wh}_{L^4(\Gaht)} &\leq c\norm{\wh}_{L^2(\Gaht)}^{1/2}\norm{\wh}_{H^1(\Gaht)}^{1/2} \quad  \quad  \textnormal{ for all } \wh \in \bfH^1(\Gaht).
    \end{align}
\end{lemma}
\begin{proof}
   Estimates \eqref{discrete Korn's inequality T nh ENS}, \eqref{eq: coercivity and Korn's inequality Lagrange ENS} follow from \cite[Lemma 4.7]{elliott2024sfem} (in our case we make use \cref{lemma: errors of geometric pert ENS} where the constant $c$ is independent of both $h,\,t$), while \eqref{eq: interpolation inequality Gaglia ENS} holds due to \cite[Lemma 4.2]{olshanskii2024eulerian}.    
\end{proof}

\noindent The following special case regarding the discrete Korn-type inequality, with the help of the broken Sobolev spaces has been proven in \cite{elliott2025unsteady}, following similar arguments to \cite{Hardering2022}.
\begin{lemma}\label{lemma: Korn's inequality Ph ENS}
    For $\wh \in \bfV_h$, there exists $c>0$ depending on $\mathcal{G}_T$ but independent of $h, \, t$, such that
    \begin{align}
            \norm{\bfPh\wh}_{H^1(\Th\t)} &\leq c \, \ah(\wh,\wh), \label{eq: korn inequality Ph ENS}\\
             \norm{\bfPh\wh}_{L^4(\Th\t)} &\leq c\norm{\wh}^{1/2}_{L^2(\Gaht)}\norm{\wh}_{\ah}^{1/2} \leq c\norm{\wh}_{\ah} \label{eq: L4 inequality Ph ENS}.
    \end{align}
\end{lemma}
\noindent The following integration by parts formula holds; see \cite{elliott2024sfem}.
\begin{lemma}
 On the discrete surface $\Gaht$ the following integration by parts formula holds:
 \begin{equation}
     \begin{aligned}
     \label{eq: integration by parts ENS}
         \int_{\Gaht} \vh \cdot \nbgh \qh \ \dsh &= - \int_{\Gaht} \qh \divgh \vh \ \dsh + \sum_{T\in \mathcal{T}_h} \int_T (\vh \cdot \nh)\qh \divgh \nh \ \dsh \\
         &+ \sum_{E \in \mathcal{E}_h} \int_E [\mh] \cdot \vh \qh \, d\ell,
     \end{aligned}
 \end{equation}
 for all $\vh \in \bfH^1(\Gaht)$ and $\qh \in H^1(\Gaht)$.
\end{lemma}
Utilizing all the aforementioned relations and inequalities, we establish the following coercivity and continuity estimates for our discrete bilinear and trilinear forms.
\begin{lemma}[Bounds and coercivity]
\label{Lemma: discrete bounds and coercivity results ENS}
    The following bounds and coercivity estimates hold:
    \begin{align}
        \label{ah boundedness ENS}
        \ah(\wh,\vh) &\leq \norm{\wh}_{\ah}\norm{\vh}_{\ah},\\
        \label{ahhat coercivity ENS}
        \ahhat(\wh,\wh) &\geq \norm{\wh}_{\ah}^2 - \norm{\wh}_{L^2(\Gaht)}^2 ,\\
        \label{g boundedness ENS}
        \gh(t;\TrVel;\wh,\vh) &\leq c\norm{\wh}_{L^2(\Gaht)}\norm{\vh}_{L^2(\Gaht)},\\
        \label{dh boundedness ENS}
        \ddh(t;\TrVel;\wh,\vh) &\leq c\norm{\wh}_{\ah}\norm{\vh}_{\ah},\\
         \label{bfdh boundedness ENS}
        \bfdh(t;\TrVel;\wh,\vh) &\leq c\norm{\wh}_{\ah}\norm{\vh}_{\ah},\\
        \label{bhtilde boundedness ENS}
       \bhtil(\wh,\{\qh,\xi_h\}) &\leq c \norm{\wh}_{\ah} \norm{\{\qh,\xi_h\}}_{L^2(\Gaht)},\\
       \label{bhtilda boundedness ENS}
       \bhtilda(t;\TrVel;\wh,\{\qh,\xi_h\}) &\leq c \norm{\wh}_{\ah}\norm{\{\qh,\xi_h\}}_{L^2(\Gaht)},\\
       \label{ch cov boundedness ENS}
        c_h^{cov}(\zh; \wh,\vh) &\leq \norm{\zh}_{L^2(\Gaht)}^{1/2}\norm{\zh}^{1/2}_{\ah}\norm{\wh}_{\ah}\norm{\vh}^{1/2}_{L^2(\Gaht)}\norm{\vh}^{1/2}_{\ah} \\
        &\quad + \norm{\zh}_{L^2(\Gaht)}^{1/2}\norm{\zh}^{1/2}_{\ah}\norm{\vh}_{\ah}  \norm{\zh}_{L^2(\Gaht)}^{1/2}\norm{\zh}^{1/2}_{\ah}, \nonumber\\
         \label{ch boundedness ENS}
        c_h(\zh; \wh,\vh) &\leq \norm{\zh}_{L^2(\Gaht)}^{1/2}\norm{\zh}^{1/2}_{\ah}\norm{\wh}_{H^1(\Gaht)}\norm{\vh}_{H^1(\Gaht)},
    \end{align}
for all $ \zh(\cdot,t),\wh(\cdot,t),\vh(\cdot,t) \in \bfV_h\t. $ and $\{\qh(\cdot,t),\xi_h(\cdot,t)\} \in Q_h\t\times \Lambda_h\t,$ where the constant $c$ is independent of time $t$ and the mesh-parameter $h$.
\end{lemma}
\begin{proof}
The continuity estimates \eqref{ah boundedness ENS}, \eqref{bhtilde boundedness ENS}, and \eqref{ch cov boundedness ENS}, \eqref{ch boundedness ENS}  have already been proven in \cite[Lemma 5.3]{elliott2025unsteady}. Specifically for latter two, we just make use of the previous \cref{lemma: Korn's inequality Ph ENS} and the bounds on the discrete Weingarten map $\norm{\bfH_h}_{L^\infty(\Gaht)} \leq c$. From \cref{def: bilinear forms discrete ENS}, the coercivity estimate \eqref{ahhat coercivity ENS} readily follows, while a simple Cauchy-Schwarz inequality and the bounds on $\TrVel$ \eqref{eq: discrete vel bound ENS} along with the smoothness of $\FlVel$ (see \cref{sec: Evolving surface cont ENS}) give \eqref{g boundedness ENS}. 

Regarding \eqref{dh boundedness ENS}, we notice that $\bfPh$ is symmetric and therefore so is $ \Tilde{\mathcal{B}}_{\Gah}(\TrVel,\bfPh)$. So, following the proof of \Cref{lemma: Transport formulae strain tensor app ENS} gives the desired result. The bound \eqref{bfdh boundedness ENS}, can then be easily extrapolated; see \cref{remark: assertion bdfh ENS}. Finally, \eqref{bhtilda boundedness ENS} is treated similarly to \eqref{bhtilde boundedness ENS} (see again \cite[Lemma 5.3]{elliott2025unsteady}), that is, we use the integration by parts formula \eqref{eq: integration by parts ENS}, along with the fact that  $|[\bm{\mu}_h]| \leq ch^{k_g}$ \eqref{eq: geometric errors 1b ENS} and that $\TrVel$ is the interpolant of the fluid velocity $\FlVel$ and so the bounds for $\TrVel$ are as in \eqref{eq: discrete vel bound ENS}.
\end{proof}

Now we present some key properties involving the \emph{discrete weakly tangential divergence-free functions} $\wh \in \bfV_h^{div}\t$ \eqref{eq: discrete weakly divfree space ENS}, that will be used throughout. These properties, where proved in \cite[Lemma 5.4, Corollary 5.4.1, Lemma 5.5]{elliott2025unsteady}, respectively, for stationary surfaces, but can be easily expanded in the evolving setting, so we omit the proofs. 

\begin{lemma}\label{lemma: divfree L2 inner normal ENS}
Let $\wh \in \bfV_h^{div}\t$ and $\vh \in \bfV_h$, then the following inequalities hold
\begin{align}
\label{eq: divfree L2 inner normal kl=ku ENS}
    (\wh\cdot\nh,\vh\cdot\nh)_{L^2(\Gaht)} &\leq ch\norm{\wh\cdot\nh}_{L^2(\Gaht)}\norm{\vh}_{L^2(\Gaht)}, \  \text{ for } k_\lambda = k_u.\\
    \label{eq: divfree L2 inner normal kl=ku-1 ENS}
    (\wh\cdot\nh,\vh\cdot\nh)_{L^2(\Gaht)} &\leq ch\norm{\wh\cdot\nh}_{L^2(\Gaht)}\norm{\vh}_{H^1(\Gaht)}, \  \text{ for } k_\lambda = k_u-1,
\end{align}
where $c>0$ a constant depending on $\mathcal{G}_T$ but independent of $h, \, t$.
\end{lemma}

\begin{corollary}\label{corollary: divfree L2 norm ENS}
Let $\wh \in \bfV_h^{div}$, then there exists $c>0$ depending on $\mathcal{G}_T$ but independent of $h, \, t$, such that \vspace{-1mm}
\begin{align}
\label{eq: divfree L2 norm kl=ku ENS}
    \norm{\wh\cdot\nh}_{L^2(\Gaht)} &\leq ch\norm{\wh}_{L^2(\Gaht)}, \quad  \text{ for } k_\lambda = k_u,\\
    \label{eq: divfree L2 norm kl=ku-1 ENS}
    \norm{\wh\cdot\nh}_{L^2(\Gaht)} &\leq ch\norm{\wh}_{H^1(\Gaht)}, \quad  \text{ for } k_\lambda = k_u-1. 
\end{align}
\end{corollary}

\begin{lemma}[Improved $H^1$ coercivity bound]\label{lemma: improved h1-ah bound ENS}
     If $\underline{k_{\lambda} = k_u}$ and $\bfw_h \in \bfV_h^{div}\t$ then the following  holds \vspace{-2mm}
     \begin{align}\label{eq: improved h1-ah bound ENS}
         \norm{\bfw_h}_{H^1(\Gaht)}^2 &\leq c\norm{\wh}_{\ah},
     \end{align}
where $c>0$ a constant depending on $\mathcal{G}_T$ but independent of $h, \, t$.
\end{lemma}

\subsection{Geometric Perturbations Error Bounds}\label{sec: Geometric Perturbations Errors}
Let us, first, recall some geometric perturbation results found in several literature \cite{DziukElliott_L2,DziukElliott_acta,hardering2023parametric,Hardering2022,highorderESFEM,elliott2024sfem,jankuhn2021trace,hansbo2020analysis}.
\begin{lemma}\label{lemma: errors of geometric pert ENS}
For $h$ small enough and $\wh,\vh  \in \bfH^1(\Gaht)$, $\bff \in (L^2(\Gat))^3$,  we have the following perturbation bounds, with constant $c>0$ depending on $\mathcal{G}_T$ but independent of $h, \, t$
\begin{align}\label{eq: errors of domain of integration data ENS}
        |\mb(\bff,\vhl)- \mh(\bff^{-\ell},\vh)| &\leq  ch^{k_g+1}\norm{\bff}_{L^2(\Gat)}\norm{\vh}_{L^2(\Gaht)},\\[5pt]
         \label{eq: Geometric perturbations nbg ENS}
        \norm{(\nbg\whl)^{-\ell} - \nbgh\wh}_{L^2(\Gaht)} &\leq ch^{k_g}\norm{\wh}_{H^1(\Gaht)},\\[5pt]
        \label{eq: Geometric perturbations a 1 ENS}
        |(E(\bfw_h^{\ell}),E(\vhl))_{L^2(\Gat)} - (E_h(\bfw_h),E_h(\vh))_{L^2(\Gaht)} | &\leq ch^{k_g} \norm{\bfw_h}_{H^1(\Gaht)}\norm{\vh}_{H^1(\Gaht)},\\[5pt]
        \norm{E(\bfw_h^{\ell})^{-\ell} - E_h(\bfw_h)}_{L^2(\Gaht)} &\leq ch^{k_g}\norm{\bfw_h}_{H^1(\Gaht)}. \label{eq: Geometric perturbations a 2 ENS}
    \end{align}
\end{lemma}
\begin{proof}
Refer to \cite{highorderESFEM,DziukElliott_acta} for the initial two estimates, which can be easily generalized to vector-valued functions, and \cite[Lemma 4.6]{elliott2024sfem} for the remaining two (notice that the constant $c$ is now independent of both $h, \, t$).
\end{proof}
\begin{lemma}[Lemma 5.8 \cite{elliott2025unsteady}]\label{lemma: weingarten map improved ENS}
The following bound holds for all $\vh \in \bfH^1(\Gaht)$,
    \begin{equation}
        \begin{aligned}\label{eq: weingarten map improved ENS}
             \Big|\int_{\Gaht} \beta^{-\ell}\,\vh\cdot\bfH_h\bfu^{-\ell} \, \dsh - \int_{\Gat} \beta\,\vhl\cdot\bfH\bfu \, \ds\Big| \leq ch^{k_g}\norm{\beta}_{H^1(\Gat)}\norm{\bfu}_{H^1(\Gat)}\norm{\vh}_{H^1(\Gaht)}.
        \end{aligned}
    \end{equation}
\end{lemma}

We now want to compare the bilinear forms on $\Gat$ and $\Gaht$ via the lifting process. We will see that the geometric approximation order of the discrete surface $\Gaht$, i.e. $\bigo(k_g)$, appears as a discrepancy error. Before that, for brevity reasons, let us introduce some further notation. For $\wh,\vh \in \bfH^1(\Gaht)$ and $\qh,\xih \in H^1(\Gaht)$ we define the following:
\begin{align}
    &\Grm_{\mb}(\wh,\vh):= \mh(\wh,\vh) - \mb(\whl,\vhl), \qquad \Grm_{\gb}(\wh,\vh):=\gh(\TrVel;\wh,\vh)-\gb(\TrVelL;\whl,\vhl),\nonumber\\
    &\Grm_{\ab}(\wh,\vh):= \ah(\wh,\vh) -\ab(\bfw_h^{\ell},\vhl), \qquad \, \quad \Grm_{\db}(\wh,\vh):=\ddh(\TrVel;\wh,\vh) - \db(\TrVelL;\wh,\vhl),\nonumber\\
    &\Grm_{\ddb}(\wh,\vh):=\bfdh(\TrVel;\wh,\vh) - \ddb(\TrVelL;\wh,\vhl), \quad
    \Grm_{\bLb}(\wh,\{q_h,\xih\}):= \bhtil(\bfw_h,\{\qh,\xi_h\}) - b^L(\bfw_h^{\ell},\{\qhl,\xi_h^{\ell}\}), \nonumber\\
    \label{eq: The G pert ENS}
    &\Grm_{\btil}(\wh,\{q_h,\xih\}) : =\bhtilda(t;\FlVel;\wh,\{\qh,\xi_h\}) - \btil(\TrVelL;\whl,\{\qhl,\xi_h^{\ell}\}).
\end{align}

Most of the following geometric perturbation error estimates have already been proven in \cite{elliott2024sfem,hansbo2020analysis,DziukElliott_acta,Hardering2022} and references therein, for stationary surfaces, but can be expanded to the evolving setting \cite{highorderESFEM}; see also \cite{highorderESFEM,DziukElliott_L2,EllRan21} regarding the $\Grm_{\gb}$ perturbation estimate.
\begin{lemma}[Geometric Perturbations I]\label{lemma: Geometric perturbations lagrange ENS}
Let $\bfw_h,\vh \in \bfH^1(\Gaht)$ and $\{\qh,\xi_h \}\in H^1(\Gaht)\times H^1(\Gaht)$. Then we have
\begin{align}
    \label{eq: Geometric perturbations a ENS}
    |\Grm_{\ab}(\wh,\vh)| &\leq ch^{k_g} \norm{\wh}_{H^1(\Gaht)} \norm{\vh}_{H^1(\Gaht)},\\
    \label{eq: Geometric perturbations g ENS}
    |\Grm_{\gb}(\wh,\vh)| &\leq c h^{k_g+1}\norm{\wh}_{L^2(\Gaht)} \norm{\vh}_{L^2(\Gaht)},\\
    \label{eq: Geometric perturbations btilde ENS}
     |\Grm_{\bLb}(\wh,\{q_h,\xih\})| &\leq ch^{k_g} \norm{\wh}_{L^2(\Gaht)}(\norm{\qh}_{H^1(\Gaht)} + \norm{\xi_h}_{L^2(\Gaht)}) \nonumber\\
     &\leq ch^{k_g-1} \norm{\bfw_h}_{L^2(\Gaht)}(\norm{\qh}_{L^2(\Gaht)} + h\norm{\xi_h}_{L^2(\Gaht)}).
\end{align}
If  $\bfw \in \bfH_T^1\t$, then we get the following higher order bound
\begin{equation}
    \begin{aligned}\label{eq: Geometric perturbations btilde tangent ENS}
        |\Grm_{\bLb}(\bfw^{-\ell},\{q_h,\xih\})|
        \leq ch^{k_g} \norm{\bfw}_{L^2(\Gat)}\norm{\{\qh,\xi_h\}}_{L^2(\Gaht)}.
    \end{aligned}
\end{equation}
Furthermore if  $\bfw,\bfv \in \bfH_T^1\t\cap \bfH^2(\Gat)$ then we have
\begin{align}\label{eq: Geometric perturbations btilde tangent extra regularity ENS}
    |\Grm_{\bLb}(\bfw^{-\ell},\{q_h,\xih\})|&\leq ch^{k_g+1} \norm{\bfw}_{H^2(\Gat)}\norm{\{\qh,\xih\}}_{H^1(\Gaht)}, \\
    \label{eq: Geometric perturbations a tangent extra regularity ENS}
    |\Grm_{\ab}(\bfw^{-\ell},\bfv^{-\ell})&\leq ch^{k_g+1} \norm{\bfw}_{H^2(\Gat)} \norm{\bfv}_{H^2(\Gaht)},
\end{align}
where, in all the above, the constant $c>0$ is independent both of $h$ and $t$.
\end{lemma}

\begin{lemma}[Geometric Perturbations II]\label{lemma: Geometric perturbations lagrange II ENS}
    Let $\bfw_h,\vh \in \bfH^1(\Gaht)$ and $\{\qh,\xi_h \}\in H^1(\Gaht)\times H^1(\Gaht)$. Then we have
    \begin{align}
    \label{eq: Geometric perturbations d ENS}
    |\Grm_{\db}(\wh,\vh)| &\leq ch^{k_g} \norm{\wh}_{H^1(\Gaht)} \norm{\vh}_{H^1(\Gaht)},\\
    \label{eq: Geometric perturbations bfd ENS}
    |\Grm_{\ddb}(\wh,\vh)| &\leq c h^{k_g}\norm{\wh}_{H^1(\Gaht)} \norm{\vh}_{H^1(\Gaht)},\\
    \label{eq: Geometric perturbations btilda ENS}
     |\Grm_{\btil}(\wh,\{q_h,\xih\})| &\leq ch^{k_g} \norm{\wh}_{L^2(\Gaht)}(\norm{\qh}_{H^1(\Gaht)} + \norm{\xi_h}_{L^2(\Gaht)}) \nonumber\\
     &\leq ch^{k_g-1} \norm{\bfw_h}_{L^2(\Gaht)}(\norm{\qh}_{L^2(\Gaht)} + h\norm{\xi_h}_{L^2(\Gaht)}).
\end{align}
\end{lemma}
\begin{proof}
Let us start with \eqref{eq: Geometric perturbations d ENS}, then as mentioned in \cref{remark: assertion bdfh ENS}, \eqref{eq: Geometric perturbations bfd ENS} follows similarly.
Lifting to $\Gat$, using \eqref{eq: gah to ga Bh ENS} and adding and subtracting $\bfPg$ appropriately we see that
\begin{equation}
\begin{aligned}\label{eq: Geometric perturbations d  inside 1 ENS}
    &\frac{d}{dt}\int_{\Gaht} \nbgcovh \wh:\nbgcovh \vh  = \frac{d}{dt} \int_{\Gat} \frac{1}{\mu_h}\bfPh \nbg\whl:\nbg\vhl \bfB_h^t\bfB_h \\
    &= \frac{d}{dt} \int_{\Gat} (\bfPh - \bfPg)\nbg\whl:\nbg\vhl \mathcal{Q}_h^{\ell} +   \frac{d}{dt} \int_{\Gat} \nbgcov\whl:\nbgcov\vhl \mathcal{Q}_h^{\ell},\\
    \end{aligned}
\end{equation}
and using the transport formula \eqref{eq: Transport formulae 3.5 ENS} in the left-hand side and lifting once again as above, we obtain
    \begin{align}\label{eq: Geometric perturbations d  inside 2 ENS}
       \!\!\!\!\!\!\! &\frac{d}{dt}\int_{\Gaht} \nbgcovh \wh:\nbgcovh \vh = \int_{\Gaht} \nbgcovh \matd\wh:\nbgcovh \vh + \int_{\Gaht} \nbgcovh \wh:\nbgcovh \matd\vh + \ddh(\TrVel;\wh,\vh)\nonumber \\
       &=  \int_{\Gat} \nbgcov \matdl \whl:\nbgcov \vhl \mathcal{Q}_h^{\ell} + \int_{\Gat} \nbgcov \whl:\nbgcov \matdl\vhl \mathcal{Q}_h^{\ell} + \ddh(\TrVel;\wh,\vh),\nonumber\\
       &+ \int_{\Gat} (\bfPh-\bfPg)\nbg \matdl \whl:\nbg \vhl \mathcal{Q}_h^{\ell} + \int_{\Gat} (\bfPh-\bfPg)\nbg \whl:\nbg \matdl\vhl \mathcal{Q}_h^{\ell}.
    \end{align}
\noindent Using now the corresponding lifted transport formula in \cref{lemma: Transport formulae lifted ENS} we derive that
\begin{equation}
    \begin{aligned}\label{eq: Geometric perturbations d  inside 3 ENS}
       \frac{d}{dt} \int_{\Gat} \nbgcov\whl:\nbgcov\vhl \mathcal{Q}_h^{\ell} &=  \int_{\Gat} \nbgcov\matdl \whl:\nbgcov\vhl Q_h^{\ell} +  \int_{\Gat} \nbgcov \whl:\nbgcov \matdl \vhl Q_h^{\ell}\\
       &+  \int_{\Gat} \nbgcov\whl:\nbgcov\vhl \matdl \mathcal{Q}_h^{\ell}  + d(\TrVelL;\whl,\vhl \mathcal{Q}_h^{\ell}),
    \end{aligned}
\end{equation}
where $d(\TrVelL;\whl,\vhl \mathcal{Q}_h^{\ell}) = \int_{\Gat} \nbg \whl : \nbgcov \vhl \,\Tilde{\mathcal{B}}_{\Ga}(\TrVelL,\bfPg)\mathcal{Q}_h^{\ell} \,\ds + \int_{\Gat} \nbgcov \whl : \nbg \vhl \, \Tilde{\mathcal{B}}_{\Ga}(\TrVelL,\bfPg)  \mathcal{Q}_h^{\ell} \,\ds$;  see \cref{def: bilinear forms ENS} and \cref{lemma: Transport formulae app II ENS}. Furthermore, recalling also \eqref{eq: Transport formulae 3 app ENS} and\eqref{eq: mathcal b tensor} we see that
\begin{equation}
    \begin{aligned}\label{eq: Geometric perturbations d  inside 4 ENS}
       &\frac{d}{dt} \int_{\Gat}(\bfPh - \bfPg)\nbg\whl:\nbg\vhl \mathcal{Q}_h^{\ell} =  \int_{\Gat} \matdl(\bfPh - \bfPg)\nbg\whl:\nbg\vhl \mathcal{Q}_h^{\ell} \\
       &\int_{\Gat} (\bfPh - \bfPg)\nbg\matdl\whl:\nbg\vhl \mathcal{Q}_h^{\ell} + \int_{\Gat} (\bfPh - \bfPg)\nbg\whl:\nbg\matdl\vhl \mathcal{Q}_h^{\ell}   \\
       &+ \int_{\Gat} (\bfPh - \bfPg)\nbg\whl:\nbg\vhl \matdl \mathcal{Q}_h^{\ell} + \int_{\Gat}\nbg\whl:\nbg\vhl \mathcal{B}_{\Ga}(\TrVelL,\bfPh - \bfPg)\mathcal{Q}_h^{\ell}.
    \end{aligned}
\end{equation}
Combining \eqref{eq: Geometric perturbations d  inside 2 ENS}-\eqref{eq: Geometric perturbations d  inside 4 ENS} and \eqref{eq: Geometric perturbations d  inside 1 ENS}, recalling the geometric estimates in \cref{lemma: Geometric errors ENS,Lemma: Bh estimates ENS}, the bound for the discrete velocity \eqref{eq: discrete vel bound ENS} and \cref{Lemma: discrete bounds and coercivity results ENS} we deduce \vspace{-2mm}
\begin{equation*}
    \begin{aligned}
        &|\Grm_{\db}(\wh,\vh)| = \Big|  d(\TrVelL;\whl,\vhl (\mathcal{Q}_h^{\ell}-\bfPg))  + \int_{\Gat} \nbgcov\whl:\nbgcov\vhl \matdl \mathcal{Q}_h^{\ell} \\
        &+\int_{\Gat} (\bfPh - \bfPg)\nbg\whl:\nbg\vhl \matdl \mathcal{Q}_h^{\ell}+\int_{\Gat}\nbg\whl:\nbg\vhl\mathcal{B}_{\Ga}(\TrVelL,\bfPh - \bfPg) \mathcal{Q}_h^{\ell}\\
        &+\int_{\Gat} \matdl(\bfPh - \bfPg)\nbg\whl:\nbg\vhl \mathcal{Q}_h^{\ell} \Big| \leq ch^{k_g}\norm{\wh}_{H^1(\Gaht)}\norm{\vh}_{H^1(\Gaht)},
    \end{aligned}
\end{equation*}
which yields our desired estimate. 

Now moving onto \eqref{eq: Geometric perturbations btilda ENS} we can readily see from the geometric error \eqref{eq: geometric errors 1 ENS} and  geometric perturbations \eqref{eq: errors of domain of integration data ENS}, and \eqref{eq: Geometric perturbations g ENS} that the first two difference terms of $\Grm_{\btil}$, which we define as $\Grm_{\btil}^I(\wh,\{\qh,\xih\})$; see \cref{def: bilinear forms discrete ENS} or \eqref{eq: lifted btil ENS}, are bounded as desired, that is $\Grm_{\btil}^I(\wh,\{\qh,\xih\})\leq ch^{k_g}\norm{\wh}_{L^2(\Gah)}\norm{\xih}_{L^2(\Gah)}$.
We are left to bound the difference of the last terms, that is \vspace{-1mm}
\begin{align*}
    \int_{\Gaht}  \wh \cdot \mathcal{B}_{\Gah}^{\mathrm{div}}(\TrVel) \nbgh \qh \, \ds -
\int_{\Gat}  \whl \cdot \mathcal{B}_{\Ga}^{\mathrm{div}}(\TrVelL) \nbg \qhl \, \ds.
\end{align*}
We will follow similar arguments as before. Recall $\mathcal{W}_h= \frac{\bfB_h}{\mu_h} - \bfPg$ and that $\int_{\Gaht} \wh\cdot\nbgh\qh = \int_{\Gat}\whl\cdot\frac{1}{\muh}\bfB_h\nbg\qhl$. Now lifting to $\Gat$, adding and subtracting $\bfPg$ appropriately and using the transport formulae  \eqref{eq: Transport formulae 3 discrete ENS} (discrete) and \eqref{eq: Transport formulae 3 lift ENS} (lifted) (look also \eqref{eq: Transport formulae 2 app ENS}) we derive the following
    \begin{align}
        &\frac{d}{dt}\int_{\Gaht} \wh\cdot\nbgh\qh = \frac{d}{dt}\int_{\Gat} \whl\cdot\mathcal{W}_h\nbg\qhl + \frac{d}{dt}\int_{\Gat} \whl\cdot\nbg\qhl \nonumber\\
        & = \int_{\Gat} \matdl\whl\cdot\mathcal{W}_h\nbg\qhl + \int_{\Gat} \whl\cdot\mathcal{W}_h\nbg\matdl\qhl + \int_{\Gat} \whl\cdot\matdl(\mathcal{W}_h)\nbg\qhl  + \int_{\Gat}\whl\cdot\mathcal{W}_h\mathcal{B}_{\Ga}^{\diver}(\TrVelL)\nbg\qhl\nonumber\\
        &\quad + \int_{\Gat} \matdl\whl\cdot\nbg\qhl + \int_{\Gat} \whl\cdot\nbg\matdl\qhl + \int_{\Gat}\whl\cdot\mathcal{B}_{\Ga}^{\diver}(\TrVelL)\nbg\qhl \nonumber
        \\
        &= \int_{\Gaht} \matd\wh\cdot\nbgh\qh + \int_{\Gaht} \wh\cdot\nbgh\matd\qh + \int_{\Gaht}\wh\cdot\mathcal{B}_{\Gah}^{\diver}(\TrVel)\nbgh\qh.
    \end{align}
Now, since $(\matd\wh)^{\ell} = \matdl\whl$ and $\int_{\Gaht} \wh\cdot\nbgh\qh = \int_{\Gat}\whl\cdot\frac{1}{\muh}\bfB_h\nbg\qhl$ by lifting, we deduce that
\begin{align}\label{eq: Geometric perturbations btil inside last ENS}
    \Big| \int_{\Gaht}  \wh \cdot \mathcal{B}_{\Gah}^{\mathrm{div}}(\TrVel) \nbgh \qh  &-
\int_{\Gat}  \whl \cdot \mathcal{B}_{\Ga}^{\mathrm{div}}(\TrVelL) \nbg \qhl  \Big| \nonumber
\\
&\leq \Big| \int_{\Gat}\whl\cdot\mathcal{W}_h\mathcal{B}_{\Ga}^{\diver}(\TrVelL)\nbg\qhl\Big| + \Big|\int_{\Gat} \whl\cdot\matdl(\mathcal{W}_h)\nbg\qhl\Big|\nonumber\\
&\leq  ch^{k_g}\norm{\wh}_{L^2(\Gaht)}\norm{\qh}_{H^1(\Gaht)},
\end{align}
where we used the fact that $\norm{\partial_h^{(k)}\mathcal{W}_h}_{L^{\infty}(\Gat)} \leq ch^{k_g}$ from \eqref{eq: Bh estimates ENS}. Combining with the bound $\Grm_{\btil}^I(\wh,\{\qh,\xih\})\leq ch^{k_g}\norm{\wh}_{L^2(\Gaht)}\norm{\xih}_{L^2(\Gaht)}$, concludes the proof.
\end{proof}

\begin{remark}[Better Approximations]\label{remark: Better Approximations ENS}
A question may be whether it is possible to ensure higher order of approximation for the $\Grm$-terms in  \cref{lemma: Geometric perturbations lagrange II ENS}, considering higher regularity tangent vectors as in \cref{lemma: Geometric perturbations lagrange ENS}. This seems to be a bit complicated.

\noindent First, regarding an estimate for  $\Grm_{\btil}(\bfw^{-\ell},\{\qh,\xih\})$ it is not difficult to see that the answer is positive. Indeed, since with simple calculations we can see that in \eqref{eq: Geometric perturbations btil inside last ENS} we can replace $\partial_h^{(k)}(\mathcal{W}_h)$ by  $\partial_h^{(k)} (\bfPg\mathcal{W}_h)$ and therefore the second estimate in \eqref{eq: Bh estimates ENS}
would yield (after applying inverse ineq.) the better estimate
\begin{align}
  |\Grm_{\btil}(\bfw^{-\ell},\{q_h,\xih\})| &\leq ch^{k_g} \norm{\bfw}_{L^2(\Gat)}\norm{\{\qh,\xi_h\}}_{L^2(\Gaht)}.
\end{align}
Now, for a result similar to \eqref{eq: Geometric perturbations btilde tangent extra regularity ENS}, i.e.
\begin{align}\label{eq: Geometric perturbations btilda tangent extra regularity ENS}
     |\Grm_{\btil}(\bfw^{-\ell},\{q_h,\xih\})| &\leq ch^{k_g+1} \norm{\bfw}_{H^2(\Gat)}\norm{\{\qh,\xi_h\}}_{H^1(\Gaht)},
\end{align}
as seen in \cite[Lemma 6.1]{elliott2024sfem} we would need some sort of \emph{non-standard geometric approximation estimates} similar to \cite[Lemma 4.2]{hansbo2020analysis}, where now it should be w.r.t. the material derivatives, that is \vspace{-1mm}
\begin{equation}
\begin{aligned}\label{eq: non-Stand Geom ENS}
    |(\matd(\bfPh\cdot\bfn),\chi^{-\ell})_{L^2(\Gaht)}| \leq ch^{k_g+1}\norm{\chi}_{W^{1,1}(\Gat)}, \quad 
    |(\bfP\matd(\bfn-\nh),\chi^{-\ell})_{L^2(\Gaht)}| \leq ch^{k_g+1}\norm{\chi}_{W^{1,1}(\Gat)}.
\end{aligned}
\end{equation}
Such estimates have been derived in \cite{MavrakisThesis}, but since they are technical and since we will not use such high order approximations as in \eqref{eq: Geometric perturbations btilda tangent extra regularity ENS}, we do not present the proof.

\noindent Similarly, to find an estimate as in \eqref{eq: Geometric perturbations a tangent extra regularity ENS} for $\Grm_{\db}(\bfw,\bfv)$ (and therefore for $\Grm_{\ddb}(\bfw,\bfv)$) with $\bfw,\bfv \in \bfH_T^1\cap \bfH^2(\Ga)$, that is \vspace{-2mm}
\begin{align}
    |\Grm_{\db}(\bfw^{-\ell},\bfv^{-\ell})| &\leq ch^{k_g+1} \norm{\bfw}_{H^2(\Gat)} \norm{\bfv}_{H^2(\Gat)},
\end{align}
one needs the non-standard estimates \eqref{eq: non-Stand Geom ENS} again, but also very heavy computations that mimic the proof in \cite[Lemma 5.4, eq. (5.33)]{hansbo2020analysis}. Since, again, we will not use this estimate, we do not provide a proof, but do comment that it should be possible to show it. Lastly, we want to note that we can also improve \eqref{eq: Geometric perturbations a tangent extra regularity ENS} in the sense that $\bfw$ \underline{needs not to be tangent}, e.g. for $\matd\bfw^{-\ell} 
\in \bfH^1(\Gat) \cap \bfH^2(\Gat)$:
\begin{equation}
    |\Grm_{\ab}(\matd\bfw^{-\ell},\bfv^{-\ell})\leq ch^{k_g+1} \norm{\matdl\bfw}_{H^2(\Gat)} \norm{\bfv}_{H^2(\Gaht)}.
\end{equation}
This can be shown similarly to \cite[Lemma 5.4]{hansbo2020analysis}, where the improvement comes from the that we use the discrete  rate-of-strain tensor $E_h(\bfw^{-\ell})$ instead of $E_h(\bfw^{-\ell}) - (\bfw^{-\ell}\cdot\nh)\bfH_h$.
\end{remark}

\section{Ritz-Stokes Projection for Evolving Surface Finite Elements}\label{sec: Ritz-Stokes Projection ENS}
Ritz maps have been studied in the case of elliptic operators on stationary and evolving surfaces \cite{DziukElliott_L2,highorderESFEM,li2023optimal}, while Ritz-Stokes maps have been studied for stationary domains in the case of no geometric variational crimes, that is, $\Omega_h = \Omega$ \cite{AyusoPostNS2005,FrutosGradDivOseen2016,burman2009galerkin} and evolving domains \cite{rao2025optimal} using ALE FEM. In this section, we examine an extended version of the Ritz-Stokes map  $\mathcal{R}_h(\cdot)$ \cite[Section 5.4]{elliott2025unsteady} to evolving triangulated surfaces. The main difference, as we shall see, is that now the normal material derivative and the Ritz-Stokes projection do not commute, i.e. $\matd \mathcal{R}_h\bfu \neq \mathcal{R}_h\matd\bfu$. However, before we proceed, let us first present the Scott-Zhang interpolant, cf. \cite{elliott2024sfem,hansbo2020analysis}.

\subsection{Scott-Zhang Interpolant}\label{sec: SZ interpolant ENS}
To prove convergence and stability we need to make use of the Scott-Zhang interpolant, and specifically its projection and super-approximation properties. The proofs of  \cref{lemma: divfree L2 inner normal ENS}, \cref{corollary: divfree L2 norm ENS}, \cref{lemma: improved h1-ah bound ENS}, despite not stating them, rely on this projection since they follow exactly as in\cite{elliott2025unsteady}. 

Let $\Gaht$ be the interpolated $k_g$-order approximation surface of $\Gat$ cf. \cref{sec: Recap: Evolving Triangulated Surfaces}, and let $T\in\mathcal{T}_h\t$ a triangulated element. Then, for any $\bm{z} : \mathcal{G}_T\to \mathbb{R}^n$ (scalar : $n=1$  or vector-valued : $n=3$) with $\bm{z} \in L^2(\Ga)^n$, we define the Scott-Zhang interpolant to be the projection $\Ihz : (L^2(\Gaht))^n \to (S_{h,k_g}^{k}\t)^n$; see \cite[Section 4.4]{hansbo2020analysis}. Then, the following estimate holds
\begin{equation}
    \begin{aligned}\label{eq: bounds of Scott-Zhang interpolant ENS}
        \norm{\bm{z}^{-\ell} - \Ihz\bm{z}^{-\ell}}_{H^l(T\t)} \leq ch^{k+1-l}\norm{\bm{z}}_{H^{k+1}(\omega_{\Ttilde}\t)}, \quad   0\leq l \leq k+1, \ k\geq0,
    \end{aligned}
\end{equation}
where the constant $c>0$ is independent of the mesh parameter $h$ and time $t$, but depends on $\mathcal{G}_T$, and $\omega_{\Ttilde}\t$ is a union of neighbouring elements of the element $T\t$. The following stability bound also holds
\begin{equation}\label{eq: stability of Scott-Zhang interpolant ENS}
    \norm{\Ihz \bm{z}^{-\ell}}_{H^1(T)} \leq c\norm{\bm{z}^{-\ell}}_{H^1(\omega_{\Ttilde})} \leq c\norm{\bm{z}}_{H^1(\omega_{\Ttilde}^\ell)}.
\end{equation}
We omit further detail since the extension to evolving surfaces is trivial, but for further information about the construction of the Scott-Zhang interpolant on high-order surfaces, we refer to \cite{hansbo2020analysis,elliott2024sfem}.

Finally, we present super-approximation type of estimates and stability results for the Scott-Zhang interpolant. These type of results were proven in \cite{hansbo2020analysis} and made use of the projection property of $\Ihz(\cdot)$. We extend them to the evolving surface setting.
\begin{lemma}
    For finite element function $\bm{z}_h\in (S_{h,k_g}^{k}\t)^3$, $\chi \in [W^{3,\infty}(\Gat)]^3$ and $t \in [0,T]$ it holds
    \begin{align}
        \label{eq: super-approximation estimate 1 ENS}
            \sum_{T\in \Th\t}\norm{\nbgh(I - \Ihz)(\chi^{-\ell} \cdot \bm{z}_h)}_{L^2(T)} \leq c h\norm{\chi}_{W^{3,\infty}(\Gat)} \norm{\bm{z}_h}_{H^1(\Gaht)},\\[-4pt]
            \label{eq: super-approximation estimate 2 ENS}
            \sum_{T\in \Th\t}\norm{(I - \Ihz)(\chi^{-\ell} \cdot \bm{z}_h)}_{L^2(T)} \leq ch \norm{\chi}_{W^{3,\infty}(\Gat)} \norm{\bm{z}_h}_{L^2(\Gaht)},
    \end{align}
where constant $c$ is independent of $h,t$, but depends on $\mathcal{G}_T$.
The following $L^2$ stability estimate holds
    \begin{equation}
        \begin{aligned}
        \label{eq: super-approximation stability ENS}
            \norm{\Ihz(\chi^{-\ell} \cdot \bm{z}_h)}_{L^2(\Gaht)} \leq c\norm{\chi^{-\ell} \cdot \bm{z}_h}_{L^2(\Gaht)}.
        \end{aligned}
    \end{equation}
\end{lemma}

\subsection{Surface Ritz-Stokes Projection}

\begin{definition}[Modified Ritz-Stokes projection]\label{def: surface Ritz-Stokes projection ENS} 
\noindent Given $(\bfu,\{p,\lambda\}) \in \bfH^1(\Gat)\times(L^2(\Gat)\times L^2(\Gat))$ we define by $\mathcal{R}_h(\bfu)\t \in \bfV_h^{div}\t$, $\{\mathcal{P}_h(\bfu)\t,\mathcal{L}_h(\bfu)\t\} \in Q_h\t \times \Lambda_h\t$ the unique projection such that
\begin{align}
    \begin{cases}\label{eq: surface Ritz-Stokes projection ENS} 
     a_h(\mathcal{R}_h(\bfu),\vh) + \bhtil (\vh,\{\mathcal{P}_h(\bfu),\mathcal{L}_h(\bfu)\}) = a(\bfu,\vhl), \\
     \bhtil(\mathcal{R}_h(\bfu),\{\qh,\xi_h\}) = 0,
    \end{cases}
\end{align}
 for all $(\vh,\{q_h ,\xi_h \}) \in \bfV_h\t\times (Q_h\t \times \Lambda_h\t)$. The lift is then defined typically $\big(\mathcal{R}_h^{\ell}(\bfu)\t$, $\{\mathcal{P}_h^{\ell}(\bfu)\t,$ $ \,  \mathcal{L}_h^{\ell}(\bfu)\t\} \big) \in \bfV_h^{\ell}\t\times (Q_h^{\ell}\t \times \Lambda_h^{\ell}\t)$.
\end{definition}

\begin{definition}[Ritz-Stokes projection]\label{def: surface Ritz-Stokes projection std ENS} 
\noindent Given $(\bfu,\{p,\lambda\}) \in \bfH^1(\Gat)\times(L^2(\Gat)\times L^2(\Gat))$ we define by $\mathcal{R}_h^b(\bfu)\t = \mathcal{R}_h(\bfu,\{p,\lambda\})\t\in \bfV_h^{div}\t$, $\mathcal{P}^b_h(\bfu)\t = \mathcal{P}_h(\bfu,\{p,\lambda\})\t \in Q_h\t , \, \mathcal{L}^b_h(\bfu)\t  = \mathcal{L}_h(\bfu,\{p,\lambda\})\t \in \Lambda_h\t$ the unique projection such that
\begin{align}
    \begin{cases}\label{eq: surface Ritz-Stokes projection std ENS} 
     a_h(\mathcal{R}^b_h(\bfu),\vh) + \bhtil (\vh,\{\mathcal{P}^b_h(\bfu),\mathcal{L}^b_h(\bfu)\}) = a(\bfu,\vhl) + b^L(\vhl,\{p,\lambda\}), \\
     \bhtil(\mathcal{R}^b_h(\bfu),\{\qh,\xi_h\}) = 0,
    \end{cases}
\end{align}
 for all $(\vh,\{q_h ,\xi_h \}) \in \bfV_h\t\times (Q_h\t \times \Lambda_h\t)$. The lift is then defined in the standard way.
\end{definition}

\noindent As before, we omit the time-dependency of the bilinear forms and the Ritz-Stokes map, unless it is not clear from the context. We also omit, for convenience, the inverse lift extension $(\cdot)^{-\ell}$ notation \eqref{eq: inverse lift ENS}, unless specified, since it should be clear from the context if, e.g., a function $\bfu$ is defined on $\Gat$ or $\Gaht$.\\

\noindent The Ritz-Stokes maps above $\mathcal{R}_h^{(\cdot)}(\bfu)$ $\big(\!\!:= \mathcal{R}_h(\bfu)$ or $\mathcal{R}_h^{b}(\bfu)\big)$ are well-defined (see well-posedness and stability in \cite[Theorem 5.4]{elliott2024sfem} for stationary case) and the following \emph{a priori} stability estimates hold; 
\begin{equation}\label{eq: Stab estimates for Ritz-Stokes projection ENS} 
\begin{aligned}
    \norm{\mathcal{R}_h(\bfu)}_{\ah}^2 + \norm{\{\mathcal{P}_h(\bfu),\mathcal{L}_h(\bfu)\}}_{L^2(\Gaht)}^2 \leq c\norm{\bfu}_{H^1(\Gat)}^2 , \\
    \norm{\mathcal{R}^b_h(\bfu)}_{\ah}^2 + \norm{\{\mathcal{P}_h^b(\bfu),\mathcal{L}^b_h(\bfu)\}}_{L^2(\Gaht)}^2 \leq c\big( \norm{\bfu}_{H^1(\Gat)}^2 + \norm{\{p,\lambda\}}_{L^2(\Gat)}^2 \big).
    \end{aligned}
\end{equation}
The following lemma was again shown in \cite[Lemma 5.10]{elliott2025unsteady}. The extension to evolving surfaces is trivial with the use of \cref{lemma: Geometric perturbations lagrange ENS}.
\begin{lemma}[pseudo Galerkin orthogonality]\label{lemma: pseudo Galerkin orthogonality ENS}
For every $\bfu\in \bfH^1(\Ga)$ it holds that
\begin{align}
\label{eq: pseudo Galerkin orthogonality Ritz 1 UNS}
      \ah((\bfu -\mathcal{R}_h(\bfu), \vh) \leq Ch^{k_g}\norm{\vh}_{H^1(\Gah)} \quad  \text{ for all } \vh\in\bfV_h^{div},
    \end{align}
\begin{align}
   \label{eq: pseudo Galerkin orthogonality Ritz b UNS}
      \ah((\bfu -\mathcal{R}_h^b(\bfu), \vh)  + \bhtil(\vh,\{p,\lambda\} - \{\mathcal{P}^b_h(\bfu),\mathcal{L}^b_h(\bfu)\})\leq C^bh^{k_g-1}\norm{\vh}_{\ah} \quad  \text{ for all } \vh\in\bfV_h, 
\end{align}
where the constant $C = c\norm{\bfu}_{H^1(\Ga)}$ and $C^b = c(\norm{\bfu}_{H^1(\Ga)} + \norm{p}_{H^1} + \norm{\lambda}_{L^2(\Ga)})$ independent of $h,\,t$.
\end{lemma}

\begin{lemma}[Error Bounds modified Ritz-Stokes]\label{lemma: Error Bounds Ritz-Stokes ENS}
 Let  $(\bfu,\{p,\lambda\}) \in \bfH^1(\Gat)\times(H^1(\Gat)\times L^2(\Gat))$ the solution of \eqref{weak lagrange hom NV dir ENS} or \eqref{weak lagrange hom NV cov ENS} and $t \in [0,T]$. Then,  there exists a constant $c>0$ independent of $h$ and $t$, such that the following error bounds for the  Ritz-Stokes projection hold
 \begin{align}
     \label{eq: Error Bounds Ritz-Stokes ENS}
         \norm{\bfu - \mathcal{R}_h(\bfu)}_{\ah}  + \norm{\{\mathcal{P}_h(\bfu),\mathcal{L}_h(\bfu)\}}_{L^2(\Gaht)}&\leq ch^{r_u}\norm{\bfu}_{H^{k_u+1}(\Gat)},\\
         \label{eq: Error Bounds Ritz-Stokes L2 ENS}
         \norm{\bfPg(\bfu - \mathcal{R}_h^{\ell}(\bfu))}_{L^2(\Gat)} &\leq ch^{r_u+1}\norm{\bfu}_{H^{k_u+1}(\Gat)}.
 \end{align}
with $r_u = min\{k_u,k_g-1\}$ and $k_g \geq 2$ and $h \leq h_0$ with sufficiently small $h_0$.
 
\noindent If instead we assume that $\underline{k_{\lambda} = k_u}$ (that is $V_h= \Lambda_h$) we obtain the following improved estimates, 
\begin{align}
    \label{eq: Error Bounds Ritz-Stokes improved ENS}
         \norm{\bfu - \mathcal{R}_h(\bfu)}_{\ah}  + \norm{\{\mathcal{P}_h(\bfu),\mathcal{L}_h(\bfu)\}}_{L^2(\Gaht)\times H_h^{-1}(\Gaht)} &\leq ch^{\widehat{r}_u}\norm{\bfu}_{H^{k_u+1}(\Gat)},\\
         \label{eq: Error Bounds Ritz-Stokes improved L2 Lh ENS}
         \norm{\mathcal{L}_h(\bfu)}_{L^2(\Gaht)} \leq c(h^{\widehat{r}_u}+ h^{k_g-1})\norm{\bfu}_{H^{k_u+1}(\Gat)} &\leq ch^{r_u}\norm{\bfu}_{H^{k_u+1}(\Gat)},\\
          \label{eq: Error Bounds Ritz-Stokes improved H1 ENS}
        \norm{\bfPg(\bfu - \mathcal{R}_h^{\ell}(\bfu))}_{L^2(\Gat)}+ h\norm{(\bfu - \mathcal{R}_h(\bfu))}_{H^1(\Gaht)} &\leq ch^{\widehat{r}_u+1}\norm{\bfu}_{H^{k_u+1}(\Gat)},
         \\
         \label{eq: Error Bounds Ritz-Stokes L2 improved ENS}
       \norm{(\bfu - \mathcal{R}_h(\bfu))\cdot\bfng}_{L^2(\Gat)}&\leq ch^{\widehat{r}_u+1/2}\norm{\bfu}_{H^{k_u+1}(\Gat)}.
\end{align}
with $\widehat{r}_u = min\{k_u,k_g\}$, for $h \leq h_0$ with sufficiently small $h_0$. 
\end{lemma}
\begin{proof}
Under the assumption that our solution solves \eqref{weak lagrange hom NV dir ENS} or \eqref{weak lagrange hom NV cov ENS} the calculations for the proof are nearly identical to \cite[Lemma 5.11]{elliott2025unsteady}, where one uses material from the evolving setting instead when applicable.
\end{proof}

\begin{lemma}[Error Bounds Ritz-Stokes]\label{lemma: Error Bounds Ritz-Stokes std ENS}
 Let  $(\bfu,\{p,\lambda\}) \in \bfH^1(\Gat)\times(H^1(\Gat)\times L^2(\Gat))$ the solution of \eqref{weak lagrange hom NV dir ENS} or \eqref{weak lagrange hom NV cov ENS} and $t \in [0,T]$. Then, we have the following error bounds for the  Ritz-Stokes projection,
 \begin{align}
     \label{eq: Error Bounds Ritz-Stokes std ENS}
         &\norm{\bfu - \mathcal{R}_h^b(\bfu)}_{\ah}  + \norm{\{p-\mathcal{P}_h^b(\bfu),\lambda-\mathcal{L}_h^b(\bfu)\}}_{L^2(\Gaht)}\leq  c h^m \big(\norm{\bfu}_{H^{k_u+1}(\Gat)}\nonumber\\
         &\qquad\qquad\qquad\qquad\qquad\qquad\qquad\qquad\qquad\qquad\quad \ \ \ \,+ \norm{p}_{H^{k_{pr}+1}(\Gat)} + \norm{\lambda}_{H^{k_{\lambda}+1}(\Gat)}\big),\\
         \label{eq: Error Bounds Ritz-Stokes L2 std ENS}
         &\norm{\bfPg(\bfu - \mathcal{R}_h^{b,\ell}(\bfu))}_{L^2(\Gat)} \leq ch^{m+1}\big(\norm{\bfu}_{H^{k_u+1}(\Gat)}+\norm{p}_{H^{k_{pr}+1}(\Gat)} + \norm{\lambda}_{H^{k_{\lambda}+1}(\Gat)} \big),
 \end{align}
with $m=min\{r_u,k_{pr}+1,k_{\lambda}+1\}$, $r_u = min\{k_u,k_g-1\}$ and $k_g \geq 2$, for $h \leq h_0$ with sufficiently small $h_0$ and with $c$ independent of $h$ and $t$. 
\end{lemma}
\begin{lemma}[$L^{\infty}$ bound Ritz-Stokes projection; \cite{elliott2025unsteady} Lemma 5.13]\label{lemma: Linfty estimate Ritz-Stokes ENS}
    Given $\bfu \in \mathbf{W}^{2,\infty}(\Ga)$ and $\underline{k_\lambda = k_u}$ or $\underline{k_\lambda = k_u-1}$ with $k_g\geq 2$, there exists a positive constant $c$ independent of $h$ and $t$ such that the following estimate holds
    \begin{equation}\label{eq: W1infty estimate Ritz-Stokes ENS}
        \norm{\mathcal{R}_h^{(\cdot)} \bfu}_{L^{\infty}(\Gaht)} +  \norm{\nbgcovh\mathcal{R}_h^{(\cdot)} \bfu}_{L^{\infty}(\Gaht)} \leq c\norm{\bfu}_{W^{2,\infty}(\Gat)}.
    \end{equation}
\end{lemma}

\subsection{Material Derivative of the Ritz-Stokes map}\label{sec: Material Derivative of the Ritz-Stokes map}
Since our surface $\Gat$ also depends on time, we require further work to find interpolation error estimates for the material derivative, compared to \cite{elliott2025unsteady}. As mentioned before, this arises from the fact that the normal time derivative and the Ritz-Stokes projection do not commute, i.e. $\matd \mathcal{R}_h\bfu \neq \mathcal{R}_h\matd\bfu$. For convenience, we focus on the \emph{modified Ritz-Stokes} projection; then the general case follows with similar arguments.

We start by taking the time-derivative of \eqref{eq: surface Ritz-Stokes projection ENS} and use the transport formulae in \cref{lemma: Transport formulae discrete ENS,lemma: Transport formulae lifted ENS} to derive the following:
\begin{equation}
    \begin{aligned}\label{eq: time der Ritz Stokes eq ENS}
      a_h(\matd\bfu-\matd\mathcal{R}_h(\bfu),\vh) + \bhtil (\vh,\{\matd\mathcal{P}_h(\bfu),\matd\mathcal{L}_h(\bfu)\}) = \mathrm{\bfG}_1(\vh) + \mathrm{\mathbf{R}}_1(\vh)\\
      \bhtil(\matd\mathcal{R}_h(\bfu),\{\qh,\xi_h\}) = \mathrm{\bfG}_2(\{\qh,\xih\}) + \mathrm{\mathbf{R}}_2(\{\qh,\xih\}),
    \end{aligned}
\end{equation}
for $\vh \in \bfV_h$, $\{\qh,\xih\} \in Q_h\times \Lambda_h$, where recalling the notation in \eqref{eq: The G pert ENS} we have that \vspace{1mm}
    \begin{align}
        &\!\!\mathrm{\bfG}_1(\vh):= \ah(\matd\bfu,\vh)-\ab(\matdl\bfu,\vhl) + \gb(\TrVelL;\bfu,\vhl)-\gh(\TrVel;\bfu,\vh) + \ddb(\TrVelL;\bfu,\vhl)-\bfdh(\TrVel;\bfu,\vh)\nonumber\\    
        \label{eq: time der Ritz Stokes G1 ENS}
        &\!\! \qquad \quad\  \,= \Grm_{\ab}(\matd \bfu,\vh) + \Grm_{\gb}(\bfu,\vh) + \Grm_{\ddb}(\bfu,\vh),\\
        &\!\! \mathrm{\bfG}_2(\{\qh,\xih\}):= -\bhtilda(\TrVel;\bfu,\{\qh,\xi_h\}),\label{eq: time der Ritz Stokes G2 ENS}\\
        \label{eq: time der Ritz Stokes R1 ENS}
        &\!\! \mathrm{\mathbf{R}}_1(\vh):= \gh(\TrVel;\bfu-\mathcal{R}_h(\bfu),\vh) + \bfd_h(\TrVel;\bfu-\mathcal{R}_h(\bfu),\vh) - \bhtilda(\TrVel;\vh,\{\mathcal{P}_h(\bfu),\mathcal{L}_h(\bfu)\}),\\
        \label{eq: time der Ritz Stokes R2 ENS}
        &\!\! \mathrm{\mathbf{R}}_2(\{\qh,\xih\}):=  \bhtilda(\TrVel;\bfu-\mathcal{R}_h(\bfu),\{\qh,\xi_h\}).
    \end{align}  

\begin{lemma}
For sufficiently smooth $\bfu$, $\matn\bfu$, the following bounds hold:
    \begin{align}
    \label{eq: time der Ritz-Stokes inside 2 ENS}
        \!\!\mathrm{\bfG}_1(\vh) &\leq  c(h^{k_g}+h^{k_u+1})(\norm{\matn \bfu}_{H^1(\Gat)} + \norm{\bfu}_{H^1(\Gat)} )\norm{\vh}_{H^1(\Gaht)}, \\
        \label{eq: time der Ritz-Stokes inside 444 ENS}
        \!\!\mathrm{\mathbf{R}}_1(\vh) &\leq ch^{r_u}\norm{\bfu}_{H^{k_u+1}(\Gat)}\norm{\vh}_{\ah} ,\quad \mathrm{\mathbf{R}}_2(\{\qh,\xih\}) \leq  ch^{r_u}\norm{\bfu}_{H^{k_u+1}(\Gat)}\norm{\{\qh,\xih\}}_{L^2(\Gaht)}.
    \end{align}
    If furthermore we assume that $\underline{k_\lambda = k_u}$, we obtain
    \begin{align}
        \label{eq: time der Ritz-Stokes inside 4 ENS}
        \!\!\mathrm{\mathbf{R}}_1(\vh) \leq ch^{\widehat{r}_u}\norm{\bfu}_{H^{k_u+1}(\Gat)}\norm{\vh}_{\ah} ,\quad \mathrm{\mathbf{R}}_2(\{\qh,\xih\}) \leq  ch^{\widehat{r}_u}\norm{\bfu}_{H^{k_u+1}(\Gat)}\norm{\{\qh,\xih\}}_{L^2(\Gaht)}.
    \end{align}
  for all $\vh \in \bfV_h$, $\{\qh,\xih\} \in Q_h\times \Lambda_h$,  where $\widehat{r}_u = min\{k_u,k_g\}$ and $r_u = min\{k_u,k_g-1\}$.
\end{lemma}
\begin{proof}
Starting with $\mathrm{\bfG}_1(\vh)$, using the perturbation estimates for $ \Grm_{\ab}$ \eqref{eq: Geometric perturbations a ENS}, $\Grm_{\gb}$  \eqref{eq: Geometric perturbations g ENS}, and $ \Grm_{\ddb}$, \eqref{eq: Geometric perturbations bfd ENS} and the bound \eqref{eq: relate matd to matdl ENS} we obtain
\begin{equation}
\begin{aligned}
| \mathrm{\bfG}_1(\vh)|  &\leq ch^{k_g}\norm{\matdl \bfu}_{H^1(\Gat)}\norm{\vh}_{H^1(\Gaht)} + ch^{k_g}\norm{\bfu}_{H^1(\Gaht)}\norm{\vh}_{H^1(\Gaht)}\\
&\leq c(h^{k_g}+h^{k_u+1})(\norm{\matn \bfu}_{H^1(\Gat)} + \norm{\bfu}_{H^1(\Gat)} )\norm{\vh}_{H^1(\Gat)}.
\end{aligned}
\end{equation}
Now, focusing on \eqref{eq: time der Ritz-Stokes inside 444 ENS}, using the bounds for the bilinear forms  \eqref{g boundedness ENS}, \eqref{bfdh boundedness ENS}, \eqref{bhtilda boundedness ENS} coupled with the Ritz-Stokes estimates \eqref{eq: Error Bounds Ritz-Stokes ENS} yields our desired result.

For the $\underline{k_\lambda = k_u}$ case, i.e. \eqref{eq: time der Ritz-Stokes inside 4 ENS},  the estimate for $\mathrm{\mathbf{R}}_2(\{\qh,\xih\})$ is evident if one considers the bound for $\bhtilda$ \eqref{bhtilda boundedness ENS} coupled with the Ritz-Stokes estimate \eqref{eq: Error Bounds Ritz-Stokes improved ENS} instead. 
The $\mathrm{\mathbf{R}}_1(\vh)$ estimate is a bit more involved, since we now want to utilize $L^2\times H^{-1}_h$ bounds for the pressure \eqref{eq: Error Bounds Ritz-Stokes improved ENS}. As before,  using  \eqref{g boundedness ENS}, \eqref{bfdh boundedness ENS} coupled with the Ritz-Stokes estimate \eqref{eq: Error Bounds Ritz-Stokes improved ENS} we obtain
\begin{equation}
    \begin{aligned}\label{eq: time der Ritz-Stokes inside 3 ENS}
| \mathrm{\mathbf{R}}_1(\vh)| \leq ch^{\widehat{r}_u}\norm{\bfu}_{H^{k_u+1}(\Ga)}\norm{\vh}_{\ah} + |\bhtilda(\TrVel;\vh,\{\mathcal{P}_h(\bfu),\mathcal{L}_h(\bfu)\})|.
    \end{aligned}
\end{equation}
We need to bound the last term appropriately. Applying \eqref{bhtilda boundedness ENS} yields

\begin{equation*}
|\bhtilda(\TrVel;\vh,\{\mathcal{P}_h(\bfu),\mathcal{L}_h(\bfu)\})|  \leq c\norm{\vh}_{\ah}\norm{\mathcal{P}_h(\bfu)}_{L^2(\Gaht)} +  |m_h(\mathcal{L}_h(\bfu),\vh\cdot\matd\nh) | + |g_h(\TrVel;\mathcal{L}_h(\bfu),\vh\cdot \nh)|,
\end{equation*}
where we see, recalling the definition of the $H_h^{-1}$ dual norm \eqref{eq: H^-1h definition ENS}, that 
\begin{equation*}
    \begin{aligned}
        &|m_h(\mathcal{L}_h(\bfu),\vh\cdot\matd\nh)| = m_h(\mathcal{L}_h(\bfu),\vh\cdot(\matd\nh-\matd\bfn)) + |m_h(\mathcal{L}_h(\bfu),\Ihz(\vh\cdot\matd\bfn))|\\
        &\qquad\qquad\qquad\qquad\qquad + |m_h(\mathcal{L}_h(\bfu),\Ihz(\vh\cdot\matd\bfn)-\vh\cdot\matd\bfn)|\\
        &\leq ch\norm{\vh}_{\ah}\norm{\mathcal{L}_h(\bfu)}_{L^2(\Gaht)} + c\norm{\vh}_{\ah}\norm{\mathcal{L}_h(\bfu)}_{H^{-1}_h(\Gaht)}\\
        &\leq c\norm{\vh}_{\ah}\norm{\mathcal{L}_h(\bfu)}_{H^{-1}_h(\Gaht)}, \quad\qquad\qquad  \text{ since $\norm{\xi_h}_{L^2(\Gaht)} \leq ch^{-1}\norm{\xi_h}_{H_h^{-1}(\Gaht)}$},
    \end{aligned}
\end{equation*}
where we also used the geometric errors \eqref{eq: geometric errors 1 ENS}, the Scott-Zhang super-approximation \eqref{eq: super-approximation estimate 2 ENS} and the fact that $\Ihz(\cdot)\in V_h$ coupled with the definition of the $H_h^{-1}$ dual norm \eqref{eq: H^-1h definition ENS}. Similarly one bounds $|g_h(\TrVel;\mathcal{L}_h(\bfu),\vh\cdot \nh)| \leq c\norm{\vh}_{\ah}\norm{\mathcal{L}_h(\bfu)}_{H^{-1}_h(\Gah)}$. Therefore, by \eqref{eq: Error Bounds Ritz-Stokes improved ENS} the above results yield
\begin{equation*}
    \begin{aligned}\label{eq: time der Ritz-Stokes inside 4 ENS}
        |\bhtilda(\TrVel;\vh,\{\mathcal{P}_h(\bfu),\mathcal{L}_h(\bfu)\})|   \leq ch^{\widehat{r}_u}\norm{\bfu}_{H^{k_u+1}(\Gat)}\norm{\vh}_{\ah},
    \end{aligned}
\end{equation*}
and so, going back to \eqref{eq: time der Ritz-Stokes inside 3 ENS}, finally gives our desired estimate.
\end{proof}

Now, before moving onto the interpolation error of the material derivative of the Ritz-Stokes projection, let us first consider the following auxiliary Stokes problem (similar to \cite[Lemma 5.10]{elliott2025unsteady}):
\begin{equation}
    \begin{aligned}\label{eq: new approx cont ENS}
        a(\matn\bfu,\bfv) \ + b^L(\bfv,\{\matn p,\matn\lambda\}) &= \mb(\bm{\mathcal{F}},\bfv) \ \ \ \ \ \quad\ \text{for all } \bfv\in \bfH^1(\Ga),\\
        b^L(\matn\bfu,\{q,\xi\})&=m(\bm{\ell},\{q,\xi\})) \ \ \  \text{ for all } \{q,\xi\}\in (L^2_0(\Ga)\times L^2(\Ga)),
    \end{aligned}
\end{equation}
where now 
\begin{equation*}
    \begin{aligned}
        &\mb(\bm{\mathcal{F}},\bfv) = \mb(\matn(\bff - \matn\bfu - \bfu\cdot\nbg^{(\cdot)}\bfu- \nbg\matn p ),\bfv) - \bfd(\FlVel; \bfu, \bfv) - \bfg(\FlVel; \bfu, \bfv) - \btil(\FlVel; \bfv,\{p,\lambda\}), \\
        &m(\bm{\ell},\{q,\xi\}) = \btil(\FlVel; \bfu,\{q,\xi\}).
    \end{aligned}
\end{equation*}
Then, if we let $(\bfu,\{p,\lambda\})$ and $(\matn\bfu,\{\matn p,\matn \lambda\})$ be the solutions of 
\eqref{weak lagrange hom NV dir ENS} (or \eqref{weak lagrange hom NV cov ENS}) and its time derivative, respectively,  by the transport formulae \cref{lemma: Transport formulae cont ENS}, it is  clear that $(\matn \bfu,\{0,0\})$ is the solution of \eqref{eq: new approx cont ENS}. We may then consider the standard finite element approximation of \eqref{eq: new approx cont ENS}, $(\matd\uh,\{\matd p_h,\matd\lambda_h\})$ which satisfies the following
\begin{equation}
    \begin{aligned}\label{eq: new approx ENS}
        a_h(\matd\bfu_h,\bfv_h) \ + b_h^L(\bfv_h,\{\matd p_h,\matd\lambda_h\}) &= \mb_h(\bm{\mathcal{F}}_h,\vh) \ \ \ \ \ \text{for all } \bfv_h\in \bfV_h,\\
        b_h^L(\matd\bfu_h,\{\qh,\xih\})&=m_h(\bm{\ell}_h,\{q_h,\xi_h\}) \ \ \  \text{ for all } \{q,\xi\}\in (Q_h\times \Lambda_h),
    \end{aligned}
\end{equation}
where now $\bm{\mathcal{F}}_h := \bm{\mathcal{F}}^{-\ell}$ and $m_h(\bm{\ell}_h,\{q_h,\xi_h\}):= -\bhtilda(\TrVel;\mathcal{R}_h(\bfu),\{q_h,\xi_h\})$. We select $\bm{\ell}_h$ as such,  so that  $\matd\uh - \matd\mathcal{R}_h(\bfu) \in \bfV_h^{div}$; see  \eqref{eq: time der Ritz Stokes eq ENS}. Refer also to \cref{remark: different lh in approx uh ENS} for other choices. Finally, set
\begin{equation*}
\bfV_h^{\bhtilda}\t := \{\vh \in \bfV_h\t: \, \bhtil(\bfv_h,\{\qh,\xi_h\}) =-\bhtilda(\TrVel;\mathcal{R}_h(\bfu),\{q_h,\xi_h\}), ~~\forall \, \{\qh,\xi_h\} \in Q_h\t\times\Lambda_h\t\}.  %
\end{equation*} 

\begin{lemma}\label{lemma: Error Bounds approx ENS}
Let  $\matn\bfu\in \bfH^1(\Gat)$ satisfy \eqref{eq: new approx cont ENS}, with $(\bfu,\{ p, \lambda\})$ as in \cref{lemma: Error Bounds Ritz-Stokes ENS}, and $(\matd\uh,\{\matd \ph,$ $\matd\lh \})\in \bfV_h\t\times (Q_h\t\times\Lambda_h\t)$ satisfy \eqref{eq: new approx ENS}. Then, for  $t\in [0,T]$ and $h \leq h_0$ with sufficiently small $h_0$, we have the following bounds,\vspace{-2mm}
\begin{align}
     \label{eq: Error Bounds new approx ENS}
         \norm{\bfP(\matd\bfu - \matd\uh)}_{L^2(\Gaht)} + h\norm{\matd\bfu - \matd\uh}_{\ah}\leq ch^{r_u+1}\sum_{j=0}^{\ell}\norm{(\matn)^j\bfu}_{H^{k_u+1}(\Gat)}.
 \end{align}\vspace{-3mm}
 
\noindent If furthermore,  we assume that $\underline{k_{\lambda} = k_u}$ (that is $V_h= \Lambda_h$) we obtain the following extra estimates 
\vspace{-2mm}
\begin{equation}
    \begin{aligned}\label{eq: Error Bounds new approx improved ENS}
                 \norm{(\matd\bfu - \matd\uh)\cdot\bfn}_{L^2(\Gaht)} \leq ch^{\ru}\sum_{j=0}^{\ell}\norm{(\matn)^j\bfu}_{H^{k_u+1}(\Gat)}, \\[-4pt]
         \norm{\matd\bfu - \matd\uh}_{H^{1}(\Gaht)}\leq ch^{r_u}\sum_{j=0}^{\ell}\norm{(\matn)^j\bfu}_{H^{k_u+1}(\Gat)},
    \end{aligned}
\end{equation}\vspace{-3mm}

\noindent where $r_u = min\{k_u,k_g-1\}$ and $\ru = min\{k_u+1/2,k_g\}$.
\end{lemma}
\begin{proof}
We focus on the $\underline{k_{\lambda} = k_u}$ case, since  $\underline{k_{\lambda} = k_u-1}$ follows in a similar way, where one uses the appropriate estimates relevant to this case.
The following calculations are similar to \cite[Lemma 6.12, Theorem 6.13, Corollary 6.1]{elliott2024sfem}.

Considering \eqref{eq: new approx cont ENS} and \eqref{eq: new approx ENS}, we apply the results in \cite[Theorem 6.11,Corollary 6.1]{elliott2024sfem}, and with our choice of $\bm{\ell}_h$ and standard perturbation estimates, we readily see that
\begin{equation}\label{eq: Error Bounds new approx inside 1 ENS}
     \norm{\matd\bfu - \matd\uh}_{\ah}  + \norm{\matd\bfu - \matd\uh}_{H^{1}(\Gaht)} \leq ch^{k_g}(\norm{\bfu}_{H^2(\Gat)}+\norm{\matn\bfu}_{H^{1}(\Gat)}) + \!\!\!\!\inf_{\vh\in\bfV_h^{\bhtilda}}\norm{\matd\bfu - \vh}_{\ah}.
\end{equation}

\noindent Therefore, it remains to bound the last term in \eqref{eq: Error Bounds new approx inside 1 ENS}. Using the  $L^2\times L^2$ \textsc{inf-sup} \eqref{eq: discrete inf-sup condition Gah Lagrange ENS} and recalling \cite[Lemma 6.12]{elliott2024sfem} or \cite[Lemma 5.5]{Ferroni2016},  we see that for $\wh \in \bfV_h$ there exists a unique $\zh \in (\bfV_h^{div})^\perp$ such that \vspace{-3mm}
\begin{align}\label{eq: Error Bounds new approx inside 2 ENS}
 \norm{\zh}_{\ah} &\leq \sup_{\{\qh,\xi_h\}\in Q_h\times\Lambda_h}\frac{ b_h^{L,*}(\zh,\{\qh,\xi_h\})}{\norm{\{\qh,\xi_h\}}_{L^2(\Gah)}},\\
 \label{eq: Error Bounds new approx inside 3 ENS}
    b_h^{L,*}(\zh,\{\qh,\xi_h\}) &=  b_h^{L,*}(\matd\bfu -\wh,\{\qh,\xi_h\}) - b_h^{L,*}(\matd\bfu,\{\qh,\xi_h\}) - \bhtilda(\TrVel;\mathcal{R}_h(\bfu),\{q_h,\xi_h\}),
    \end{align}
where $b_h^{L,*}$ the adjoint of $\bhtil$; see \cite[Lemma 6.12]{elliott2024sfem}. Furthermore, using the fact that  $\frac{d}{dt}\bLb(\bfu,\{q,\xi\})=0$ for all $\{q,\xi\}\in L^2_0(\Gat)\times L^2(\Gat)$, the transport formula \eqref{eq: Transport formulae 3 lift ENS} and setting a new pressure variable $\qh' = \qhl + m_q\in L^2_0(\Gat)$ for $m_q = - \frac{1}{|\Ga|}\int_{\Gat}\qhl\,\ds$ (notice that $\nbg\qh' = \nbg\qhl$) we get
    \begin{align}
        &\big|-\bhtil(\matd\bfu,\{\qh,\xi_h\})- \bhtilda(\TrVel;\mathcal{R}_h(\bfu),\{q_h,\xi_h\})\big| \leq \big|-\bhtil(\matd\bfu,\{\qh,\xi_h\})+ \bLb(\matdl\bfu,\{\qh',\xi_h\})\big|\nonumber\\
        &\qquad\qquad\qquad\qquad\qquad\qquad\qquad\qquad\qquad\qquad\quad +\big|   -\bLb(\matdl\bfu,\{\qh',\xi_h\})- \bhtilda(\TrVel;\mathcal{R}_h(\bfu),\{q_h,\xi_h\}) \big|\nonumber \\
        &\leq \big|-\bhtil(\matd\bfu,\{\qh,\xi_h\}) + \bLb(\matdl\bfu,\{\qh,\xi_h\})\big| + \big|  \btil(\TrVelL;\bfu,\{\qhl,\xi_h^{\ell}\})  - \bhtilda(\TrVel;\mathcal{R}_h(\bfu),\{q_h,\xi_h\}) \big|\nonumber\\
        &\leq ch^{k_g-1}\norm{\matdl\bfu}_{L^2(\Gat)}(\norm{\qh}_{L^2(\Gaht)} + h\norm{\xi_h}_{L^2(\Gaht)}) + |\mathrm{\mathbf{R}_2}(\{\qh,\xih\})| +|\mathrm{G}_{\btil}(\bfu,\{\qh,\xih\})|\nonumber\\
        \label{eq: Error Bounds new approx inside est for Ib3 ENS}
        &\leq ch^{r_u}(\norm{\matn\bfu}_{L^2(\Gat)}+ \norm{\bfu}_{H^{k_u+1}(\Gat)})(\norm{\qh}_{L^2(\Gaht)} + h\norm{\xi_h}_{L^2(\Gaht)}),
    \end{align}
where for the second inequality we used the second line of the perturbation estimate \eqref{eq: Geometric perturbations btilde ENS}, while in the third inequality we utilized the bound of $\mathrm{\mathbf{R}_2}$ in \eqref{eq: time der Ritz-Stokes inside 4 ENS} and the perturbation estimate \eqref{eq: Geometric perturbations btilda ENS}.

Then using the fact that $b_h^{L,*}$ is the adjoint of $\bhtil$, by integration by parts \eqref{eq: integration by parts ENS}, going back to \eqref{eq: Error Bounds new approx inside 2 ENS}, we may derive the following, after taking the infimum over all $\wh \in \bfV_h$, 
\begin{equation}\label{eq: Error Bounds new approx inside 4 ENS}
    \norm{\zh}_{\ah} \leq \inf_{\wh\in\bfV_h}\norm{\matd\bfu - \wh}_{\ah}+ch^{k_g-1}(\norm{\matn\bfu}_{L^2(\Gat)}+\norm{\bfu}_{H^{k_u+1}(\Ga)}).
\end{equation}

\noindent Similarly to the above calculations, if we use the $L^2\times H_h^{-1}$\textsc{inf-sup}  \eqref{eq: L^2 H^{-1} discrete inf-sup condition Gah Lagrange ENS} instead we obtain
\begin{equation}\label{eq: Error Bounds new approx inside 5 ENS}
     \norm{\zh}_{H^{1}(\Gaht)} \leq \inf_{\wh\in\bfV_h}\norm{\matd\bfu - \wh}_{H^{1}(\Gaht)}+ch^{k_g-1}(\norm{\matn\bfu}_{L^2(\Gat)}+\norm{\bfu}_{H^{k_u+1}(\Gat)}).
\end{equation}
\noindent Setting now  $ \vh = \zh + \wh \in \bfV_h^{\bhtilda}$ by a simple triangles inequality we then obtain
\begin{equation*}
    \begin{aligned}
       \norm{\matd\bfu - \vh}_{\ah} &\leq \norm{\matd\bfu - \wh}_{\ah} + \norm{\zh}_{\ah}, \\
       \norm{\matd\bfu - \vh}_{H^1({\Gaht})} &\leq \norm{\matd\bfu - \wh}_{H^1({\Gaht})} + \norm{\zh}_{H^1({\Gaht})},
    \end{aligned}
\end{equation*}
from which, taking the infimum over all $\wh \in \bfV_h$ and $ \vh\in \bfV_h^{\bhtilda}$, considering \eqref{eq: Error Bounds new approx inside 4 ENS}, \eqref{eq: Error Bounds new approx inside 5 ENS} and inserting it into \eqref{eq: Error Bounds new approx inside 1 ENS}, the bound in \eqref{eq: Error Bounds new approx ENS}, \eqref{eq: Error Bounds new approx improved ENS} for the $a_h$-, $H^1$-norms follow (e.g. $\wh = \Ih(\matd\bfu)$ the standard Lagrange interpolant). Having these results, the perturbation estimates \cref{lemma: Geometric perturbations lagrange ENS,lemma: Geometric perturbations lagrange II ENS}, and Ritz-Stokes estimate $\norm{\bfu-\mathcal{R}_h\bfu}_{L^2(\Gaht)}\leq ch^{\widehat{r}_u+1/2}$ \eqref{eq: Error Bounds Ritz-Stokes improved H1 ENS}, \eqref{eq: Error Bounds Ritz-Stokes L2 improved ENS}, the rest of the estimates, involving the tangential and normal $L^2$ norm, are obtained almost identically as in \cite[Theorem 6.14, 6.16]{elliott2024sfem}.
\end{proof}

\begin{remark}\label{remark: different lh in approx uh ENS}
Notice that for $\underline{k_\lambda=k_u}$, the convergence is worse than that proved in \cite{elliott2024sfem}. This is due to the fact that $\matn\bfu \cdot\bfng \neq 0$; see also \cref{remark: non tangential mat der ENS} for further details.   For the approximation $\matd\uh$ in \eqref{eq: new approx cont ENS} we could have chosen $\bm{\ell}_h:=\bhtilda(\TrVel;\bfu^{-\ell},\{\qh,\xih\})$ instead. Then, the calculations would change, since now $\matd\uh - \matd\mathcal{R}_h(\bfu) \notin \bfV_h^{div}$ (see also \eqref{eq: time der Ritz-Stokes inside 5 ENS} below where we use this fact), but ultimately the convergence would still not improve.
\end{remark}

Now we are ready to prove the error bounds for the time derivative of the Ritz-Stokes projection.

\begin{lemma}[Error Bounds in Material Derivative modified Ritz-Stokes map]\label{lemma: Error Bounds der Ritz-Stokes ENS}
 Let  $(\bfu,\{p,\lambda\}) \in \bfH^1(\Gat)\times(H^1(\Gat)\times L^2(\Gat))$ the solution of  \eqref{weak lagrange hom NV dir ENS} or \eqref{weak lagrange hom NV cov ENS} and $t \in [0,T]$. Then, for $\underline{k_{\lambda} = k_u}$ or $\underline{k_{\lambda} = k_u-1}$ and $k_g \geq 2$, the following error bounds for the normal discrete time derivative of the  Ritz-Stokes projection hold,\vspace{-3mm}
 \begin{align}
     \label{eq: Error Bounds der Ritz-Stokes ENS}
         \!\!\!\!\!\!\!\!\norm{\matd(\bfu - \mathcal{R}_h(\bfu))}_{\ah}  + \norm{\{\matd\mathcal{P}_h(\bfu),\matd\mathcal{L}_h(\bfu)\}}_{L^2(\Gaht)\times L^2(\Gaht)}&\leq ch^{r_u}\sum_{j=0}^{1}\norm{(\matn)^j\bfu}_{H^{k_u+1}(\Gat)}.
 \end{align}
 \vspace{-4mm}
 
\noindent  Furthermore if specifically $\underline{k_{\lambda} = k_u}$, the subsequent bounds also hold
\vspace{-2mm}
 \begin{align}
      \label{eq: Error Bounds der Ritz-Stokes H1 ENS}
         \norm{\matd(\bfu - \mathcal{R}_h(\bfu))}_{H^1(\Gaht)} &\leq ch^{r_u}\sum_{j=0}^{1}\norm{(\matn)^j\bfu}_{H^{k_u+1}(\Gat)},\\[-4pt]
         \label{eq: Error Bounds der Ritz-Stokes L2 ENS}
         \norm{\bfPg\matdl(\bfu - \mathcal{R}_h^{\ell}(\bfu))}_{L^2(\Gat)} + \norm{\matdl(\bfu - \mathcal{R}_h^{\ell}(\bfu))\cdot\bfng}_{L^2(\Gat)} &\leq ch^{\widehat{r}_u}\sum_{j=0}^{1}\norm{(\matn)^j\bfu}_{H^{k_u+1}(\Gat)},
 \end{align}
with $r_u = min\{k_u,k_g-1\}$, $\widehat{r}_u = min\{k_u,k_g\}$ and $c>0$ is independent of $h$, $t$.
\end{lemma}
\begin{proof}
Once again, we focus on the $\underline{k_{\lambda} = k_u}$ case. The other case can be derived analogously using the same ideas and applying the appropriate estimate concerning the $\underline{k_{\lambda} = k_u-1}.$

\underline{For the energy estimate}, adding and subtracting the Galerkin approximation $\matd\uh \in \bfV_h^{\bhtilda}\t$ in \eqref{eq: new approx ENS} we see, by the bound \eqref{eq: Error Bounds new approx ENS}, that
\begin{equation}
    \begin{aligned}\label{eq: time der Ritz-Stokes inside 111 ENS}
        \!\!\!\norm{\matd\bfu - \matd\mathcal{R}_h\bfu}_{\ah}^2 &= \ah(\matd\bfu - \matd\mathcal{R}_h(\bfu),\matd\bfu - \matd\uh) + \ah(\matd\bfu - \matd\mathcal{R}_h(\bfu),\matd\uh-\matd\mathcal{R}_h(\bfu))\\
        &\leq  Ch^{r_u}\norm{\matd\bfu - \matd\mathcal{R}_h(\bfu)}_{\ah} + \ah(\matd\bfu - \matd\mathcal{R}_h(\bfu),\matd\uh-\matd\mathcal{R}_h(\bfu)),
    \end{aligned}
\end{equation}
where the constant $C = \sum_{j=0}^{1}\norm{(\matn)^j\bfu}_{H^{k_u+1}(\Gat)}$. We are left to bound the last term. For this, recall equation \eqref{eq: time der Ritz Stokes eq ENS} and observe that for a test function $\vh \in \bfV_h^{div}$ the following holds true
\begin{equation}
\begin{aligned}\label{eq: time der Ritz-Stokes inside 5 ENS}
    a_h(\matd\bfu-\matd\mathcal{R}_h(\bfu),\vh) &= \mathrm{\bfG}_1(\vh) + \mathrm{\mathbf{R}}_1(\vh)\\
    &\leq  ch^{\widehat{r}_u}\norm{\bfu}_{H^{k_u+1}(\Gat)}\norm{\vh}_{\ah},
    \end{aligned}
\end{equation}
where we used the proven bounds \eqref{eq: time der Ritz-Stokes inside 2 ENS} and \eqref{eq: time der Ritz-Stokes inside 4 ENS}, coupled with the $H^1$ coercivity estimate \eqref{eq: improved h1-ah bound ENS} $(\norm{\vh}_{H^1(\Gaht)}\leq c\norm{\vh}_{\ah})$. Furthermore, by construction of $\matd\uh$ in \eqref{eq: new approx ENS} we have that $\matd\uh-\matd\mathcal{R}_h(\bfu) \in \bfV_h^{div}\t$. From these two fact we readily see that
\begin{equation}
    \begin{aligned}\label{eq: time der Ritz-Stokes inside 112 ENS}
        \ah(\matd\bfu - \matd\mathcal{R}_h(\bfu),\matd\uh-\matd\mathcal{R}_h(\bfu)) \leq Ch^{\widehat{r}_u}\big(\norm{\matd\uh - \matd\bfu}_{\ah} + \norm{\matd\bfu - \matd\mathcal{R}_h(\bfu)}_{\ah}\big).
    \end{aligned}
\end{equation}
So going back to \eqref{eq: time der Ritz-Stokes inside 111 ENS}, substituting \eqref{eq: time der Ritz-Stokes inside 112 ENS} and using Young's inequality and a simple kickback argument  yields our desired result for the energy estimate in \eqref{eq: Error Bounds der Ritz-Stokes ENS}.

\underline{For the $H^1$-estimate}, using the bound \eqref{eq: Error Bounds new approx improved ENS}  for $\matd\uh$ and the $H^1$ coercivity estimate \eqref{eq: improved h1-ah bound ENS}, since by construction $\matd\bfu_h-\matd\mathcal{R}_h\bfu \in \bfV_h^{div}$, we obtain
\begin{equation}
\begin{aligned}
    \norm{\matd\bfu- \matd\mathcal{R}_h(\bfu)}_{H^1(\Gaht)} &\leq \norm{\matd\bfu -\matd\bfu_h}_{H^1(\Gaht)} + \norm{\matd\bfu_h-\matd\mathcal{R}_h\bfu}_{H^1(\Gaht)}\\
    &\leq \norm{\matd\bfu -\matd\bfu_h}_{H^1(\Gaht)} + \norm{\matd\bfu_h-\matd\mathcal{R}_h\bfu}_{\ah},
    \end{aligned}
\end{equation}
from which \eqref{eq: Error Bounds der Ritz-Stokes H1 ENS} follows, since $\norm{\matd\bfu_h-\matd\mathcal{R}_h\bfu}_{\ah}\leq Ch^{r
_u}$ as proved before.

\underline{For the pressure estimates} we just go back to \eqref{eq: time der Ritz Stokes eq ENS}, solve for $\bhtil$ and use the $L^2\times H_h^{-1}$ discrete Lagrange \textsc{inf-sup} condition in \cref{lemma: L^2 H^{-1} discrete inf-sup condition Gah Lagrange ENS}. So, using  \eqref{eq: time der Ritz-Stokes inside 2 ENS}, \eqref{eq: time der Ritz-Stokes inside 4 ENS} and the energy estimate in \eqref{eq: Error Bounds der Ritz-Stokes ENS}, which we proved beforehand, implies that \vspace{-2mm}
\begin{equation*}
    \begin{aligned}
        \norm{\{\matd\mathcal{P}_h(\bfu),\matd\mathcal{L}_h(\bfu)\}}_{L^2(\Gaht)\times H_h^{-1}(\Gaht)} &\leq  \sup_{\vh\in\bfV_h} \frac{\ah(\matd\bfu-\matd\mathcal{R}_h(\bfu),\vh) +\mathrm{\bfG_1}(\vh)+ \mathrm{\mathbf{R}_1}(\vh)}{\norm{\vh}_{H^1(\Gaht)}}\\
        &\leq \sup_{\vh\in\bfV_h} \frac{\norm{\matd\bfu-\matd\mathcal{R}_h\bfu}_{\ah}\norm{\vh}_{\ah}+Ch^{\widehat{r}_u}\norm{\vh}_{H^{1}(\Gaht)}}{\norm{\vh}_{H^1(\Gaht)}} \\
        &\leq ch^{r_u}\sum_{j=0}^{1}\norm{(\matn)^j\bfu}_{H^{k_u+1}(\Gat)}.
    \end{aligned}
\end{equation*}
Using the $L^2\times L^2$ \textsc{inf-sup} in \cref{Lemma: Discrete inf-sup condition Gah Lagrange ENS}, instead, and the worse $H^1$ coercivity estimate \eqref{eq: coercivity and Korn's inequality Lagrange ENS}, also proves the estimate \eqref{eq: Error Bounds der Ritz-Stokes ENS}:\vspace{-3mm}
\begin{equation}\label{eq: time der Ritz-Stokes inside mat press ENS}
   \norm{\{\matd\mathcal{P}_h(\bfu),\matd\mathcal{L}_h(\bfu)\}}_{L^2(\Gaht)\times L^2(\Gaht)} \leq ch^{r_u}\sum_{j=0}^{1}\norm{(\matn)^j\bfu}_{H^{k_u+1}(\Gat)}.
\end{equation}\vspace{-3mm}


Now we are left to find \underline{the tangential and normal $L^2$-norm error bounds} of the material derivative of the velocity, i.e. \eqref{eq: Error Bounds der Ritz-Stokes L2 ENS}. For the tangential part, as in \cite[Lemma 5.11]{elliott2025unsteady}, we use a duality argument. So, let us consider 
$(\bfw, \pi,\mu) \in \bfH^1(\Gat) \times L^2_0(\Gat)\times L^2(\Gat)$ that satisfy \vspace{-1mm}
\begin{align}
\begin{cases}\label{eq: dual Ritz inside begin ENS}
        a(\bfw,\bfv) \ + \!\!\!\!& b^L(\bfv,\{\pi,\mu\}) = (\bfPg(\matdl\bfu - \matdl\mathcal{R}_h^{\ell}\bfu),\bfv) \ \ \ \ \ \text{for all } \bfv\in \bfH^1(\Gat),\\
        &b^L(\bfw,\{\sigma,\xi\})=0 \ \ \  \text{ for all } \{\sigma,\xi\} \in L_0^2(\Gat)\times L^2(\Gat),
    \end{cases}
\end{align}
for $t \in [0,T]$. Now, the surface $\Ga$ is sufficiently smooth and thus due to \cite[Lemma 2.1]{olshanskii2021inf} the solution also satisfies the following higher regularity estimate

\begin{equation}\label{eq: regularity estimate lagrange ENS}
    \norm{\bfw}_{H^2(\Ga)} + \norm{\{\pi,\mu\}}_{H^1(\Gat)} \leq \norm{\bfPg(\matdl\bfu - \matdl\mathcal{R}_h^{\ell}\bfu)}_{L^2(\Gat)}.
\end{equation} 
Set $\bfe_h = \matd\bfu - \matd\mathcal{R}_h\bfu$. Then, testing with $\bfe_h^{\ell} = \matdl\bfu - \matdl\mathcal{R}_h^{\ell}\bfu\in\bfH^1(\Gat)$, we clearly see that 
\begin{equation}
    \begin{aligned}\label{eq: dual der Ritz-Stokes inside main ENS}
        \boxed{\norm{\bfPg\bfe_h^{\ell}}_{L^2(\Gat)}^2 =  a(\bfe_h^{\ell},\bfw) + b^L(\bfe_h^{\ell},\{\pi,\mu\}).}
    \end{aligned}
\end{equation}
We need to bound the two terms on the right-hand side of \eqref{eq: dual der Ritz-Stokes inside main ENS} appropriately. This follows \cite[Lemma 5.11]{elliott2025unsteady} but now, since the non-tangentiality of $\matn\bfu$ limits us to geometric approximation error $\bigo(h^{k_g-1})$ in the energy ($\ah$) velocity and $L^2\times H_h^{-1}$ pressure norm \eqref{eq: Error Bounds der Ritz-Stokes ENS}, we should expect to obtain an improved bound for the tangential $L^2$-norm with geometric approximation error $\bigo(h^{k_g})$ (instead of the $\bigo(h^{k_g+1})$ in the stationary case \cite[Lemma 5.11]{elliott2025unsteady}); see \cref{remark: non tangential mat der ENS} for more details. Moreover, due to the transport formulae used to derive \eqref{eq: time der Ritz Stokes eq ENS}, there are also terms that prevent us from improving the degree of velocity approximation, yielding sub-optimal $\bigo(h^{k_u})$ order of convergence. Such term is $\bm{\mathcal{I}}_3^{a,9}$ in \eqref{eq: dual der Ritz-Stokes inside I9 ENS}. For more details see \cref{remark: issues mat der ENS}.
So, moving on to the calculations, for the two bilinear terms on the right-hand side of \eqref{eq: dual der Ritz-Stokes inside main ENS} we write

\begin{align}
    \label{eq: dual der Ritz-Stokes inside a(.,.) ENS}
        a(\bfe_h^{\ell},\bfw) &= \underbrace{a(\bfe_h^{\ell},\bfw - \Ihl(\bfw))}_{\bm{\mathcal{I}}
        _1^a} + \underbrace{a(\bfe_h^{\ell},\Ihl(\bfw)) - a_h(\bfe_h,\Ih(\bfw))}_{\bm{\mathcal{I}}
        _2^a} + \underbrace{a_h(\bfe_h,\Ih(\bfw))}_{\bm{\mathcal{I}}
        _3^a} \\
        b^L(\bfe_h^{\ell},\{\pi,\mu\}) &= \underbrace{b^L(\bfe_h^{\ell},\{\pi-\Ihl(\pi),\mu-\Ihl(\mu)\})}_{\bm{\mathcal{I}}
        _1^b}   + \underbrace{b^L(\bfe_h^{\ell},\{\Ihl(\pi),\Ihl(\mu)\}) -\bhtil(\bfe_h,\{\Ih(\pi),\Ih(\mu)\})}_{\bm{\mathcal{I}}
        _2^b} \nonumber  \\
        \label{eq: dual der Ritz-Stokes inside b(.,.) ENS}
        &\  + \underbrace{\bhtil(\bfe_h,\{\Ih(\pi),\Ih(\mu)\})}_{\bm{\mathcal{I}}
        _3^b},
\end{align}

\noindent where $\Ih(\cdot)$ the standard Lagrange interpolant. Let us now bound each term separately. Recall that for $\underline{k_{\lambda}=k_u}$ we have shown $\norm{\bfe_h}_{\ah} \leq \norm{\bfe_h}_{H^1(\Gaht)} \leq ch^{r_u}\sum_{j=0}^{1}\norm{(\matn)^j\bfu}_{H^{k_u+1}(\Gat)}$:
\begin{align}\label{eq: dual der Ritz-Stokes inside 2 ENS}
    \hspace{-50mm}\bullet \quad  \bm{\mathcal{I}}_1^a  = a(\bfe_h^{\ell},\bfw - \Ihl(\bfw)) &\leq c\norm{\bfw - \Ihl(\bfw)}_{H^1(\Gat)}\norm{\bfe_h}_{H^{1}(\Gaht)} \nonumber \\
        &\leq ch^{r_u+1}\norm{\bfw}_{H^2(\Gat)}\sum_{j=0}^{1}\norm{(\matn)^j\bfu}_{H^{k_u+1}(\Gat)},
\end{align}\vspace{-4mm}

\noindent where we used the fact that $\norm{\cdot}_{a}\leq c\norm{\cdot}_{H^1(\Gat)} \leq c\norm{\cdot}_{H^1(\Gaht)}$, see \eqref{eq: norm equivalence ENS} (one can also use the fact that $\norm{\cdot}_{a}\leq\norm{\cdot}_{\ah}$, see \eqref{eq: Geometric perturbations a 1 ENS}, but holds only for $k_g \geq 2$), and standard nodal interpolation estimates.
\begin{align}\label{eq: dual der Ritz-Stokes inside 3 ENS}
  \bullet \quad \bm{\mathcal{I}}_2^a &= a(\bfe_h^{\ell},\Ihl(\bfw)) - a_h(\bfe_h,\Ih(\bfw)) \leq ch^{k_g}\norm{\Ih(\bfw)}_{H^1(\Gah)}\norm{\bfe_h}_{H^1(\Gah)} \nonumber\\
        &\leq ch^{k_g + r_u}\norm{\bfw}_{H^2(\Ga)}\sum_{j=0}^{1}\norm{(\matn)^j\bfu}_{H^{k_u+1}(\Ga)} \qquad (\text{where we used } \eqref{eq: Geometric perturbations a ENS}).
\end{align}
\vspace{-3mm}

Now recalling the definition \eqref{eq: time der Ritz Stokes eq ENS},  where now $\vh =  \Ih(\bfw) \neq \bfV_h^{div}$, we reformulate it as followed \vspace{-1mm}
    \begin{align}
        \ \bullet \quad \bm{\mathcal{I}}_3^a &= \Grm_a(\matd\bfu,\Ih(\bfw)) + \Grm_{\gb}(\bfu,\Ih(\bfw)) + \gh(\TrVel;\bfu-\mathcal{R}_h(\bfu),\Ih(\bfw)- \bfw) +\gh(\TrVel;\bfu-\mathcal{R}_h(\bfu), \bfw) \nonumber\\
        &+\Grm_{\ddb}(\mathcal{R}_h(\bfu),\Ih(\bfw)) + \ddb(\TrVelL-\FlVel;\bfu-\mathcal{R}^{\ell}_h(\bfu),\Ih(\bfw)) + \ddb(\FlVel;\bfu-\mathcal{R}^{\ell}_h(\bfu),\Ihl(\bfw)-\bfw)\nonumber\\
        &+ \ddb(\FlVel;\bfu-\mathcal{R}^{\ell}_h(\bfu),\bfw) -  \bhtilda(\TrVel;\Ihl(\bfw),\{\mathcal{P}_h(\bfu),\mathcal{L}_h(\bfu)\}) - \bhtil (\Ih(\bfw),\{\matd\mathcal{P}_h(\bfu),\matd\mathcal{L}_h(\bfu)\})\nonumber\\
        \label{eq: dual der Ritz-Stokes inside I3a ENS}
        &:= \bm{\mathcal{I}}_3^{a,i} \qquad\qquad\qquad \text{for }i=1,...,10.
    \end{align}\vspace{-6mm}

\noindent From geometric perturbations \cref{lemma: Geometric perturbations lagrange ENS,lemma: Geometric perturbations lagrange II ENS} (Notice $\matn\bfu\cdot\bfng\neq0$), the bounds on the bilinear forms \cref{Lemma: discrete bounds and coercivity results ENS}, the Ritz-Stokes interpolation estimates in \eqref{eq: Error Bounds Ritz-Stokes improved ENS}, \eqref{eq: Error Bounds Ritz-Stokes improved H1 ENS} (notice that $\norm{\bfP(\bfu-\mathcal{R}_h(\bfu))}_{L^2(\Gah)}\leq ch^{\widehat{r}_u+1}$ and $\bfPg\bfw = \bfw$) and the estimate \eqref{eq: difference of velocities ENS}, we see that \vspace{-3mm}
\begin{equation}
    \begin{aligned}\label{eq: dual der Ritz-Stokes inside I3ai ENS}
        |\bm{\mathcal{I}}_3^{a,i}| \leq c(h^{k_g}+ ch^{\widehat{r}_u+1})\sum_{j=0}^{1}\norm{(\matn)^j\bfu}_{H^{k_u+1}(\Gat)}\norm{\bfw}_{H^2(\Gat)}, \text{ for } i=1,...,7,
    \end{aligned}
\end{equation}
\vspace{-4mm}

\noindent where we also use the standard Lagrange interpolation estimates. We are left to bound the last line of \eqref{eq: dual der Ritz-Stokes inside I3a ENS}, i.e. $\bm{\mathcal{I}}_3^{a,i}$, $i=8,9,10$. These require a bit more calculations. 

As done before, cf. \cref{remark: assertion bdfh ENS}, instead of bounding $\ddb(\bullet;\bullet,\bullet)$, let us just consider $\db(\bullet;\bullet,\bullet)$ instead for clarity reasons. Then the bound for $\ddb(\bullet;\bullet,\bullet)$ will follow. Similar to \cite[Theorem 6.2]{DziukElliott_L2} we need to perform integration by parts. Recall from \cref{def: bilinear forms ENS}:\vspace{-2mm}
\begin{equation*}
\begin{aligned}
d(\FlVel;\bfu-\mathcal{R}_h^{\ell}(\bfu),\bfw) &:= \int_{\Gat} \Tilde{\mathcal{B}}_{\Ga}(\FlVel,\bfPg)\nbg (\bfu-\mathcal{R}_h^{\ell}(\bfu)) : \nbgcov \bfw \,\ds\\
&+ \int_{\Gat} \Tilde{\mathcal{B}}_{\Ga}(\FlVel,\bfPg)\nbgcov (\bfu-\mathcal{R}_h^{\ell}(\bfu)) : \nbg \bfw \,\ds.
\end{aligned}
\end{equation*}\vspace{-2mm}

Using Einstein summations, see \cref{appendix: differential operators ENS} and \eqref{eq: surface operators abbreviations}, we define $\{B\}_{ik} = \Tilde{\mathcal{B}}_{\Ga}(\FlVel,\bfPg)$ and $ \{z\}_{k}= \{u-\mathcal{R}_h^{\ell}(u)\}_{k}$, then, using the integration by parts formula \eqref{eq: integration by parts cont ENS}, the first integral can be written as 
    \begin{align}
        &\int_{\Gat} B_{ik}(\nbg)_j(z_k)P_{mi}(\nbg)_j(w_m) = \int_{\Gat} (\nbg)_j \big(B_{ik}z_kP_{mi}(\nbg)_j(w_m)\big) - \int_{\Gat} z_k (\nbg)_j\big(B_{ik}(\nbg)_j(w_m)P_{mi}\big)\nonumber\\
        \label{eq: dual der Ritz-Stokes inside d ENS}
        &= -\underbrace{\int_{\Gat} \kappa n_j\big(B_{ik}z_kP_{mi}(\nbg)_j(w_m)\big)}_{=0}- \int_{\Gat} z_k \divg(\Tilde{(\mathcal{B}}_{\Ga}(\FlVel,\bfPg)\nbgcov\bfw)\bfe_k)\\
        &= \int_{\Gat} (\bfu - \mathcal{R}_h^{\ell}(\bfu)) \cdot \mathrm{\mathbf{div}_{\Ga}}((\Tilde{\mathcal{B}}_{\Ga}(\FlVel,\bfPg)\nbgcov\bfw)^t) \leq ch^{\widehat{r}_u+1/2}\sum_{j=0}^{1}\norm{(\matn)^j\bfu}_{H^{k_u+1}(\Gat)}\norm{\bfw}_{H^2(\Gat)}, \nonumber
    \end{align}
where $\bfe_k$ denotes the standard unit vector, and where we used the fact that $\norm{\bfu-\mathcal{R}^{\ell}_h(\bfu)}_{L^2(\Gat)}\leq ch^{\widehat{r}_u+1/2}$ by \eqref{eq: Error Bounds Ritz-Stokes L2 improved ENS}, \eqref{eq: Error Bounds Ritz-Stokes improved H1 ENS}. One can bound the second integral similarly. Therefore, we can readily see that $d(\FlVel;\bfu-\mathcal{R}_h^{\ell}(\bfu),\bfw) \leq ch^{\widehat{r}_u+1/2}\sum_{j=0}^{1}\norm{(\matn)^j\bfu}_{H^{k_u+1}(\Gat)}\norm{\bfw}_{H^2(\Gat)}$, and with a simple extrapolation argument we derive \vspace{-2mm} 
\begin{equation}
    \begin{aligned}\label{eq: dual der Ritz-Stokes inside I8 ENS}
        |\bm{\mathcal{I}}_3^{a,8}| \leq  ch^{\widehat{r}_u+1/2}\sum_{j=0}^{1}\norm{(\matn)^j\bfu}_{H^{k_u+1}(\Gat)}\norm{\bfw}_{H^2(\Gat)}.
    \end{aligned}
\end{equation}
Noticing that $b^L(\bfw,\{\matdl\mathcal{P}^{\ell}_h(\bfu),\matdl\mathcal{L}^{\ell}_h(\bfu)\}) =0$  and the already proven estimate \eqref{eq: time der Ritz-Stokes inside mat press ENS}, we readily see,
\begin{align*}
        |\bm{\mathcal{I}}_3^{a,10}| &\leq \big|\bhtil (\Ih(\bfw)-\bfw,\{\matd\mathcal{P}_h(\bfu),\matd\mathcal{L}_h(\bfu)\})\big| + \big|\bfG_{b}(\bfw,\{\matdl\mathcal{P}^{\ell}_h(\bfu),\matdl\mathcal{L}^{\ell}_h(\bfu)\}) \big|\nonumber\\
        &\leq c(h^{r_u+1} + ch^{k_g+r_u})\sum_{j=0}^{1}\norm{(\matn)^j\bfu}_{H^{k_u+1}(\Gat)}\norm{\bfw}_{H^2(\Gat)}.
\end{align*}
Finally, using the bound \eqref{eq: Error Bounds Ritz-Stokes improved ENS} and similar calculations to \eqref{eq: time der Ritz-Stokes inside 3 ENS}-\eqref{eq: time der Ritz-Stokes inside 4 ENS} for $\bhtilda(\TrVel;\Ih(\bfw),\{\mathcal{P}_h(\bfu),\mathcal{L}_h(\bfu)\})$ we are only able to derive the following \vspace{-3mm} 
\begin{equation}
    \begin{aligned}\label{eq: dual der Ritz-Stokes inside I9 ENS}
        |\bm{\mathcal{I}}_3^{a,9}| \leq c\norm{\Ih(\bfw)}_{\ah}\norm{\{\mathcal{P}_h(\bfu),\mathcal{L}_h(\bfu)\}}_{L^2(\Gaht)\times H_h^{-1}(\Gaht)}\leq ch^{\widehat{r}_u}\sum_{j=0}^{1}\norm{(\matn)^j\bfu}_{H^{k_u+1}(\Gat)}\norm{\bfw}_{H^2(\Gaht)},
    \end{aligned}
\end{equation}
where notice the low convergence order $\bigo(h^{\widehat{r}_u})$, and note how it restricts the improvement in the velocity approximation; see \cref{remark: issues mat der ENS} for more details.
Therefore, \underline{combining all $ |\mathcal{I}_3^{a,i}|$ $i=1,...,10$} above into \eqref{eq: dual der Ritz-Stokes inside I3a ENS} we obtain \vspace{-2mm}
\begin{equation}
    \begin{aligned}
       \boxed{|\bm{\mathcal{I}}_3^{a}| \leq  ch^{\widehat{r}_u}\sum_{j=0}^{1}\norm{(\matn)^j\bfu}_{H^{k_u+1}(\Gat)}\norm{\bfw}_{H^2(\Gaht)}.}
    \end{aligned}
\end{equation}

Now moving onto the $\bm{\mathcal{I}}_{j}^b$ $j=1,...,3$ bounds, by the standard interpolation estimates we see that
\begin{equation*}
        \begin{aligned}
            \ \bullet \quad\bm{\mathcal{I}}_{1}^b=b^L(\bfe_h^{\ell},\{\pi-\Ihl(\pi),\mu-\Ihl(\mu)\}) &\leq c\norm{\bfe_h}_{H^1(\Gah)}\norm{\pi-\Ihl(\pi)}_{L^2(\Ga)} + c\norm{\bfe_h}_{L^2(\Gah)}\norm{\mu-\Ihl(\mu)}_{L^2(\Ga)}\\
            &\leq ch^{r_u+1}\sum_{j=0}^{1}\norm{(\matn)^j\bfu}_{H^{k_u+1}(\Ga)}(\norm{\pi}_{H^1(\Ga)} + \norm{\mu}_{H^1(\Ga)}).
        \end{aligned}
    \end{equation*}\vspace{-2mm}
\begin{equation*}
        \begin{aligned}
            \hspace{-6mm} \bullet \quad   \bm{\mathcal{I}}_{2}^b &= b^L(\bfe_h^{\ell},\{\Ihl(\pi),\Ihl(\mu)\}) -\bhtil(\bfe_h,\{\Ih(\pi),\Ih(\mu)\}) \leq ch^{k_g}\norm{\bfe_h}_{L^2(\Gah)}(\norm{\pi}_{H^1(\Ga)} + \norm{\mu}_{H^1(\Ga)}) \\
            &\leq ch^{k_g + r_u}\norm{\bfu}_{H^{k_u+1}(\Ga)}(\norm{\pi}_{H^1(\Ga)} + \norm{\mu}_{H^1(\Ga)}) \qquad (\text{where we used the perturbation } \eqref{eq: Geometric perturbations btilde ENS}),
        \end{aligned}
    \end{equation*}  
where we also used the norm equivalence \eqref{eq: norm equivalence ENS} where appropriate. Finally, recalling the definition of $\matd\mathcal{R}_h(\bfu)$ (see second line of \eqref{eq: time der Ritz Stokes eq ENS}) and following the calculations in  \eqref{eq: Error Bounds new approx inside est for Ib3 ENS}, we obtain 
\begin{equation*}
    \begin{aligned}
        \hspace{-11mm} \bullet \quad \bm{\mathcal{I}}_{3}^b &= \bhtil(\matd\bfu,\{\Ih(\pi),\Ih(\mu)\}) - \bhtil(\matd\mathcal{R}_h(\bfu),\{\Ih(\pi),\Ih(\mu)\})\\
        &= \bhtil(\matd\bfu,\{\Ih(\pi),\Ih(\mu)\}) + \bhtilda(\TrVel;\mathcal{R}_h(\bfu),\{\Ih(\pi),\Ih(\mu)\})\\
        &\leq |\mathrm{G}_{b^L}(\matd\bfu,\{\Ih(\pi),\Ih(\mu)\})| +|\mathrm{G}_{\btil}(\bfu,\{\Ih(\pi),\Ih(\mu)\})| + |\mathrm{\mathbf{R}_2}(\{\Ih(\pi),\Ih(\mu)\})| \\
        &\leq ch^{k_g}\sum_{j=0}^{1}\norm{(\matn)^j\bfu}_{H^{k_u+1}(\Ga)}\norm{\{\pi,\mu\}}_{H^1(\Gaht)} + c\norm{\bfu-\mathcal{R}_h(\bfu)}_{L^2(\Gaht)}\norm{\{\pi,\mu\}}_{H^1(\Gaht)},
    \end{aligned}
\end{equation*}
where we also used the perturbation bounds \eqref{eq: Geometric perturbations btilda ENS}, \eqref{eq: Geometric perturbations btilde ENS} and the bound $\mathrm{\mathbf{R}_2}(\{\Ih(\pi),\Ih(\mu)\}) = \bhtilda(\TrVel;\bfu-\mathcal{R}_h(\bfu),\{\Ih(\pi),\Ih(\mu)\}) \leq  c\norm{\bfu-\mathcal{R}_h(\bfu)}_{L^2(\Gaht)}\norm{\{\pi,\mu\}}_{H^1(\Gaht)}$. Then, using once again the fact that $\norm{\bfu-\mathcal{R}_h(\bfu)}_{L^2(\Gah)}\leq ch^{\widehat{r}_u+1/2}$ (see \eqref{eq: Error Bounds Ritz-Stokes L2 improved ENS}, \eqref{eq: Error Bounds Ritz-Stokes improved H1 ENS}) we derive: \vspace{-1mm}
\begin{equation*}
     \boxed{|\bm{\mathcal{I}}_{3}^b| \leq c(h^{k_g}+h^{\widehat{r}_u+1/2})\sum_{j=0}^{1}\norm{(\matn)^j\bfu}_{H^{k_u+1}(\Ga)}\norm{\{\pi,\mu\}}_{H^1(\Gaht)}.}
\end{equation*}\vspace{-1mm}

\underline{By substituting $\bm{\mathcal{I}}_{j}^a,\, \bm{\mathcal{I}}_{j}^b$, $j=1,2,3$} into \eqref{eq: dual der Ritz-Stokes inside a(.,.) ENS} and \eqref{eq: dual der Ritz-Stokes inside b(.,.) ENS} 
and subsequently to \eqref{eq: dual der Ritz-Stokes inside main ENS}, we obtain \vspace{-1mm}
\begin{equation}
    \begin{aligned}\label{eq: dual der Ritz-Stokes inside Final ENS}
        \norm{\bfPg\bfe_h^{\ell}}_{L^2(\Ga)}^2 \leq c(h^{k_g}+h^{\widehat{r}_u+1/2} + h^{\widehat{r}_u})\sum_{j=0}^{1}\norm{(\matn)^j\bfu}_{H^{k_u+1}(\Ga)}(\norm{\bfw}_{H^2(\Ga)} + \norm{\{\pi,\mu\}}_{H^1(\Ga)}),
    \end{aligned}
\end{equation} \vspace{-2mm}

\noindent from which, recalling the regularity estimate \eqref{eq: regularity estimate lagrange ENS}, yields the tangent part our desired estimate \eqref{eq: Error Bounds der Ritz-Stokes L2 ENS}.

The normal part of  \eqref{eq: Error Bounds der Ritz-Stokes L2 ENS} can easily be derived by using \eqref{eq: Error Bounds new approx improved ENS} and \eqref{eq: divfree L2 norm kl=ku ENS}. Indeed, we just rewrite $\matdl(\bfu - \mathcal{R}_h^{\ell}(\bfu))\cdot\bfng = \matdl(\bfu -\uhl)\cdot\bfng + \matdl(\uhl- \mathcal{R}_h^{\ell}(\bfu))\cdot\bfng $ and use the fact that $\matd\uh - \matd\mathcal{R}_h(\bfu) \in \bfV_h^{div}$.

Now, let us briefly discuss the $\underline{k_\lambda=k_u-1}$ case. As mentioned in the beginning of this proof, the $(a_h)$-velocity and $L^2$-pressure bounds follow similarly, where one uses the respective estimates, e.g. the worse $H^1$ coercivity estimate \eqref{eq: coercivity and Korn's inequality Lagrange ENS} instead of \eqref{eq: improved h1-ah bound ENS} and the bounds \eqref{eq: time der Ritz-Stokes inside 444 ENS}. Unfortunately, for the tangential $L^2$-norm of the velocity, we are unable to find an improved bound for two main reasons: First, as discussed, the term $\bm{\mathcal{I}}_{3}^{a,9}$ \eqref{eq: dual der Ritz-Stokes inside I3a ENS} will now only give a bound of $\bigo(h^{r_u})$ and second, the term  $\bm{\mathcal{I}}_{3}^{a,8}$, and specifically  \eqref{eq: dual der Ritz-Stokes inside d ENS}, which will also give a bound of order $\bigo(h^{r_u})$, due to the suboptimal $L^2$-norm estimate $\norm{\bfu-\mathcal{R}^{\ell}_h(\bfu)}_{L^2(\Gat)}\leq ch^{r_u}$  in \eqref{eq: Error Bounds der Ritz-Stokes ENS}. For more details, see \cref{remark: issues mat der ENS}.
\end{proof}

\begin{remark}[Geometric error sub-optimality of $\matd\mathcal{R}_h(\bfu)$]\label{remark: non tangential mat der ENS}
In \cref{lemma: Error Bounds der Ritz-Stokes ENS} we observe suboptimal convergences results w.r.t. the geometric error, which stems from the non-tangentiality of the material derivative $\matn\bfu$. Indeed, in the stationary case \cite[Lemma 5.10]{elliott2025unsteady}, for $k_\lambda=k_u$  we were able to prove $\bigo(h^{k_g})$ and $\bigo(h^{k_g+1})$ geometric error convergence, for the time-derivative, since $\partial_t\bfu\cdot\bfng=0$ when $\bfu\cdot\bfng=0$. However, in an evolving setting, this is no longer true, as it is obvious that if $\bfu\cdot\bfng=0$ then $\matn\bfu\cdot\bfng\neq0$. That means, theorems as in \cite[Lemma 6.12]{elliott2024sfem} (see proof of \cref{lemma: Error Bounds approx ENS} in our case) and perturbation results, cf. \eqref{eq: Geometric perturbations btilde tangent extra regularity ENS}, \eqref{eq: Geometric perturbations a tangent extra regularity ENS} and  \cref{remark: Better Approximations ENS}, that depend on the tangentiality of the arguments, no longer hold for $\matn\bfu$. Therefore, in general, we lose one geometric order of convergence.
\end{remark}

\begin{remark}[Issues regarding tangential $L^2$-norm of the Material Derivative of Ritz-Stokes maps]\label{remark: issues mat der ENS}
    Let us talk more extensively about the issues regarding the sub-optimality of the tangential $L^2$-norm convergence of  $\matd\mathcal{R}_h(\bfu)$, and why we are unable to prove better results.  As mentioned in \cref{remark: non tangential mat der ENS} the reason for the loss of geometric error is the non-tangentiality of the material derivative, i.e. $\matn\bfu\cdot\bfng\neq0$. However, we also notice no improvement in the velocity approximation order $k_u$.

\noindent \ \ \ $\bullet$ $\boxed{k_\lambda=k_u}$ : Notice that we only improve the geometric approximation order up to $\bigo(h^{k_g})$. This is apparent from the terms $\bm{\mathcal{I}}_3^{a,1}$, $\bm{\mathcal{I}}_3^{a,5}$,  in \eqref{eq: dual der Ritz-Stokes inside I3a ENS} and $\bm{\mathcal{I}}_3^{b}$, where we can only show $\bigo(h^{k_g})$ convergence due to the restrictive bounds \eqref{eq: Geometric perturbations a ENS}, \eqref{eq: Geometric perturbations bfd ENS} and \eqref{eq: Geometric perturbations btilde ENS} respectively. We believe that $\bm{\mathcal{I}}_3^{a,1}$, $\bm{\mathcal{I}}_3^{a,5}$ can be improved by \cref{remark: Better Approximations ENS}, since we can rewrite $\mathrm{G}_{(\cdot)}(\cdot,\Ih(\bfw)) = \mathrm{G}_{(\cdot)}(\cdot,\Ih(\bfw)-\bfw) + \mathrm{G}_{(\cdot)}(\cdot,\bfw)$, with $\bfw \in \bfH^1_T\cap\bfH^2$. Finally, we notice that only the term $\bm{\mathcal{I}}_3^{a,9}$ in \eqref{eq: dual der Ritz-Stokes inside I3a ENS} is responsible for the lack of improvement in the velocity approximation order $k_u$. This term appeared only due to the transport formulae used to derive \eqref{eq: time der Ritz Stokes eq ENS}. We believe that this can be improved, and so all in all obtain an an improved bound $\bigo(h^{k_u+1/2})$ for the tangential Ritz-Stokes projection.

\noindent \ \ \  $\bullet$  $\boxed{k_\lambda=k_u-1}$ : In this case, we actually expect $\bigo(h^{k_g})$ geometric error (see \cite[Lemma 5.10, eq. (5.42)]{elliott2025unsteady} or \cref{lemma: Error Bounds Ritz-Stokes ENS}), thus the previous issues in $\bm{\mathcal{I}}_3^{a,1}$, $\bm{\mathcal{I}}_3^{a,5}$ and $\bm{\mathcal{I}}_3^{b}$ do not concern us. Despite that a new \emph{issue} arises in term $\bm{\mathcal{I}}_3^{a,8}$, since now by \cref{eq: dual der Ritz-Stokes inside d ENS} we can only obtain a bound w.r.t. the full $L^2$-norm of $\bfu-\mathcal{R}_h(\bfu)$, which in this case by \eqref{eq: Error Bounds der Ritz-Stokes ENS}, gives the suboptimal  estimate $\norm{\bfu-\mathcal{R}^{\ell}_h(\bfu)}_{L^2(\Gat)}\leq ch^{r_u}$. Therefore, this term limits us both in the geometric $k_g-1$ and velocity approximation order $k_u$, unless it is possible to perform calculations such that $\norm{\bfPg(\bfu-\mathcal{R}_h^{\ell}(\bfu))}_{L^2(\Ga)}$ appears instead. The issue regarding the term $\bm{\mathcal{I}}_3^{a,9}$ still remains.

\noindent Despite the above, in the numerical experiments (\cref{Sec: moving sphere ENS}) we observe higher convergence ($\bigo(h^{\widehat{r}_u+1/2})$, ($\bigo(h^{r_u+1/2})$) for the tangential $L^{\infty}_{L^2}$(or $L^2_{L^2}$)-norm of the velocity in both cases, which indicates that our estimates for the material derivative of the velocity might not be sharp; as we also suspected.
\end{remark}

Now for the the material derivative of the  standard Ritz-Stokes projection $\mathcal{R}_h^b(\bfu)$ in \cref{def: surface Ritz-Stokes projection std ENS}, it is not difficult to see that the calculations follow exactly as in the modified Ritz-Stokes case, where now we take the time derivative of \eqref{eq: surface Ritz-Stokes projection std ENS} instead. The extra term on the right-hand side does not cause any problems and can be treated easily with already established estimates. So, for the sake of brevity we omit the proof and provide the following Ritz-Stokes error bounds.
\begin{lemma}[Error Bounds for Material Derivative Ritz-Stokes]\label{lemma: Error Bounds der Ritz-Stokes std ENS}
 Let  $(\bfu,\{p,\lambda\}) \in \bfH^1(\Gat)\times(H^1(\Gat)\times L^2(\Gat))$ the solution of \eqref{weak lagrange hom NV dir ENS} or \eqref{weak lagrange hom NV cov ENS} and $t \in [0,T]$. Then, we have the following error bounds for the material derivative of the  Ritz-Stokes projection, \vspace{-2mm}
 \begin{align}
         &\norm{\matd(\bfu - \mathcal{R}_h^b(\bfu))}_{\ah}  + \norm{\{\matd (p-\mathcal{P}_h^b(\bfu)),\matd(\lambda-\mathcal{L}_h^b(\bfu))\}}_{L^2(\Gaht)}\leq  c h^m \sum_{j=0}^{1}\big( \norm{(\matn)^j\bfu}_{H^{k_u+1}(\Gat)}\nonumber\\
         \label{eq: Error Bounds der Ritz-Stokes std ENS}
         &\qquad\qquad\qquad\qquad\qquad\qquad\qquad\qquad\qquad\quad\ \ + \norm{(\matn)^j p}_{H^{k_{pr}+1}(\Gat)} + \norm{(\matn)^j \lambda}_{H^{k_{\lambda}+1}(\Gat)}\big),
 \end{align}
with $m=min\{r_u,k_{pr}+1,k_{\lambda}+1\}$, $r_u = min\{k_u,k_g-1\}$ and $k_g \geq 2$, for $h \leq h_0$ with sufficiently small $h_0$ and with $c$ independent of $h$ and $t$. 
\end{lemma}

\section{The Fully Discrete Scheme}\label{sec: The Fully Discrete Scheme Main ENS}
As mentioned in \cref{sec: Finite Element Spaces and Discrete Material Derivative} 
we consider the \emph{Taylor-Hood} surface finite elements $\mathrm{\mathbf{P}}_{k_u}$-- $\mathrm{P}_{k_{pr}}$-- $\mathrm{P}_{k_{\lambda}}$
 for a $k_g$-order approximation of the surface, and  set $$k_u\geq 2,~k_{pr}=k_u-1~~ \mbox{and} ~~k_{\lambda}\geq 1.$$ We maintain the notation  $k_g,k_u,k_{pr}$ and $k_\lambda$ in order to track how   error dependencies arise and define the following finite element spaces
\begin{equation}\label{eq: finite element spaces approximation ENS}
     V_h\t = S_{h,k_g}^{k_u}\t, ~~ \bfV_h\t = (S_{h,k_g}^{k_u}\t)^3, ~~
        Q_h = S_{h,k_g}^{k_{pr}}\t \cap L^2_0(\Gah\t), ~~
        \Lambda_h\t = S_{h,k_g}^{k_{\lambda}}\t.
   \end{equation}

In this section, we consider two different numerical schemes that are based on the two weak formulations \eqref{weak lagrange hom NV dir ENS}, \eqref{weak lagrange hom NV cov ENS} discussed in \cref{sec: Variational formulation ENS}. For the discrete scheme based on \eqref{weak lagrange hom NV dir ENS} we are able to show stability and optimal convergence results only when  $\underline{k_\lambda = k_u}$ using \emph{iso-parametric surface finite elements}. On the other hand, for the second scheme, we can establish optimal results for both  $\underline{k_\lambda = k_u}$ and  $\underline{k_\lambda = k_u-1}$ using \emph{iso-parametric} and \emph{super-parametric discretizations}, respectively (compare \eqref{eq: improved h1-ah bound ENS} and \eqref{eq: coercivity and Korn's inequality Lagrange ENS} from which the geometric order loss emerges). Moreover, the latter choice of $k_\lambda$ also requires further regularity assumptions for optimal pressure results; see \cref{assumption: Regularity assumptions for velocity estimate 2 ENS}.

Let us also recall the space of \emph{discrete weakly tangential divergence-free functions}
\begin{equation}\label{eq: discrete weakly divfree space ENS}
    \bfV_h^{div}\t := \{\wh \in \bfV_h\t : \bhtil(\wh,\{\qh,\xi_h\}) =0, ~~\forall \, \{\qh,\xi_h\} \in Q_h\t\times\Lambda_h\t\}.
\end{equation}
and the dual space of the finite element space $\Lambda_h\t$, denoted as
$H_h^{-1}(\Gah\t)$ in \eqref{eq: H^-1h definition ENS}.

\subsection{Further Notation: Time -- Discretization}
In what follows, we introduce some further notation needed to formulate and analyze the fully discrete scheme.  We split the time interval $I=[0,T]$ uniformly into $N>0$ time intervals $I_n=[t_n,t_{n-1}]$ with nodes $t_n=n\Delta t$, where $\Delta t $ is the uniform timestep $\Delta t = T/N$. We also use the shorthand $(\cdot)^n$ for any quantity calculated at timestep $t_n$, e.g. $\Gah^n:=\Gah(t_n)$, $\bfV_h^n:=\bfV_h(t_n)$, etc. For a discrete-time sequence $f^n,\, n=0,1,...,N$
we also use the shorthand notation \vspace{-1mm}
\begin{equation*}
    \begin{aligned}
        \partial_{\tau} f^n = \frac{1}{\Delta t } \big(f^n - f^{n-1}\big).
    \end{aligned}
\end{equation*}

Following the notation in \cref{sec: Finite Element Spaces and Discrete Material Derivative}, let
$\{\chi_{j,k_g}^{k,n}\}_{j=1}^J =\{\chi_{j,k_g}^{k}(\cdot,t_n)\}_{j=1}^J $ and $\{(\chi_{j,k_g}^{k,n})^{\ell}\}_{j=1}^J =\{\chi_{j,k_g}^{k,\ell}(\cdot,t_n)\}_{j=1}^J $ be the nodal basis functions of $S_{h,k_g}^{k,n}=S_{h,k_g}^{k}(t_n)$ and $(S_{h,k_g}^{k,n})^{\ell}=S_{h,k_g}^{k,{\ell}}(t_n)$ respectively. Then, as in \cite{dziuk2012fully,elliott2015error} we employ the following notation for $z_h^n \in S_{h,k_g}^{k,n}$ and $z_h^{n,\ell} \in (S_{h,k_g}^{k,n})^{\ell}$: 
\begin{align}\label{eq: fully discr notation fef ENS}
   z_h^n = \sum_{j=1}^J z_j^n\chi_{j,k_g}^{k,n}, \qquad  z_h^{n,\ell} = \sum_{j=1}^J z_j^n(\chi_{j,k_g}^{k,n})^{\ell}.
\end{align}
It is also convenient to define functions whose values are in a previous timestep, e.g. $t_{n-1}$, but belong to the next surface, e.g. $\Gah^n$. Hence, we may define for $\alpha = -1,0$ and $t \in [t_{n-1},t_n]$ the  \underline{\emph{time-lift}}: \vspace{-1mm}
\begin{equation}
    \begin{aligned}\label{eq: discrete time-lift ENS}
        \underline{z_h^{\!n+\alpha}}(\cdot ,t) &= \sum_{j=1}^J z_j^{\!n+\alpha}\chi_{j,k_g}^{k}(\cdot,t) \in S_{h,k_g}^{k}\t, \quad 
        \underline{z_h^{n+\alpha,\ell}}(\cdot ,t) = \sum_{j=1}^J z_j^{\!n+\alpha}\chi_{j,k_g}^{k,\ell}(\cdot,t)\in S_{h,k_g}^{k,\ell}(t).
    \end{aligned}
\end{equation}
So, for example, $\underline{z_h^n}(\cdot,t_{n-1}) \in S_{h,k_g}^{k}(t_{n-1})$ defined  on $\GahnN$ and $\underline{z_h^{n-1}}(\cdot,t_{n}) \in S_{h,k_g}^{k}(t_{n})$ defined  on $\Gahn$. We now introduce the time-discrete material derivative, where given $z_h^n \in S_{h,k_g}^{k,n}$ we have
\begin{equation}\label{eq: discrete time derivative ENS}
    \matdt z_h^n = \sum_{j=1}^J \partial_{\tau} z_j^n\chi_{j,k_g}^{k,n} \in S_{h,k_g}^k(t_n), \quad \matdtl  z_h^{n,\ell} = \sum_{j=1}^J \partial_{\tau} z_j^n(\chi_{j,k_g}^{k,n})^{\ell} \in S_{h,k_g}^{k,\ell}(t_n).
\end{equation}
As in the continuous case the \emph{transport property} \eqref{eq: transport property ENS} \eqref{eq: lifted transport property ENS}  (cf. \cite{dziuk2012fully}) still  holds for the fully discrete material derivative of the basis functions, therefore for $n=0,1,...,N$ we have 
\begin{equation}\label{eq: fully discrete transport property ENS}
    \matdt \chi_{j,k_g}^{k,n}=0, \quad \matdtl (\chi_{j,k_g}^{k,n})^{\ell}=0.
\end{equation} 
From \eqref{eq: discrete time-lift ENS}, it is obvious that on $[t_{n-1},t_n]$
\begin{equation}\label{eq: discrete time derivative bar 0 ENS}
 \matd   \underline{z_h}^{\!n+\alpha}(\cdot ,t) = 0, \quad \matdl \underline{z_h}^{\!n+\alpha,\ell}(\cdot ,t) = 0,
\end{equation}
which, implies that 
\begin{equation}\label{eq: n to n-1 on surface n ENS}
    z_h^n = \underline{z_h^{n-1}}(\cdot ,t_n) + \Delta t \, \matdt z_h^n. 
\end{equation}
All the above naturally extend to vector-valued functions, e.g. $\bm{z}_h^n \in (S_{h,k_g}^{k}(t_n))^3$. 

\subsection{The Fully Discrete Scheme}\label{sec: The Fully Discrete Scheme ENS}
 We consider two numerical schemes which are based on the two different weak formulations \eqref{eq: NS Lagrange ENS} and \eqref{eq: NS Lagrange new ENS} examined in \cref{sec: Variational formulation ENS}, where we assume that 
 \vspace{-1mm}
\begin{equation*}
    \boxed{\bm{\eta}=0.}
\end{equation*}\vspace{-3.5mm}

\noindent Refer to \cref{remark: about f source ENS} and the weak formulations \eqref{weak lagrange hom NV dir ENS}, \eqref{weak lagrange hom NV cov ENS}. The reasoning for considering this condition is explained in further detail in \cref{remark: about f source discrete ENS}.

So now, with the above assumptions, the discrete bilinear forms in \cref{def: bilinear forms discrete ENS} (see also \cref{remark: skew-symmetric Ga ENS,remark: skew-symmetric Gah ENS} regarding the skew-symmetrization of $c_h$ and $c_h^{cov}$), the transport formula \eqref{eq: Transport formulae 1 discrete ENS} and \eqref{eq: n to n-1 on surface n ENS} in mind, we proceed to define our fully discrete approximation based on \eqref{weak lagrange hom NV dir ENS}: 

The linearized \emph{directional evolving fully discrete Lagrangian surface Navier-Stokes} finite element approximation is the following.

\noindent {\textbf{(eNSW{\small d}{\tiny h}): }}
Given $\uh^0 \in \bfV_h^{0,div}$ and appropriate approximation for the data $\bff_h$, for $n=0,1,...,N$ find $\uhn \in \bfV_h^n$, $\{\phn,\lambda_h^n\}\in Q_h^n\times \Lambda_h^n$ such that 
\begin{align}
\begin{cases}
\tag{eNSW{\small d}{\tiny h}}
        \label{eq: fully discrete fin elem approx ENS}
\frac{1}{\Delta t}\Big(\mh(\uhn,\vhn) - \mh(\uhnN,\underline{\vhn}(\cdot,t_{n-1}))\Big) - \gh(\TrVel;\uhn,\vhn) + \ahhat(\uhn,\vhn)\\
\qquad\qquad\qquad\quad  \, + c_h(\underline{\uhnN}(\cdot,t_n),\uhn,\vhn) 
 +b_h^L(\vhn,\{\phn,\lhn\}) = \mh(\bff_h^n,\vhn), \\
 \qquad\qquad\qquad\qquad\qquad\qquad\qquad\qquad\qquad  \ \ \  \! b_h^L(\uhn,\{\qhn,\xihn\}) = 0. 
\end{cases}
 \end{align}
for all $\vhn \in \bfV_h^n$ and $\vhnN \in \bfV_h^{n-1}$ and $\{\qhn,\xihn\}\in Q_h^n\times\Lambda_h^n$. Here, in our trilinear form $c_h(\bullet;\bullet,\bullet)$ we used the \emph{time-lift} \eqref{eq: discrete time-lift ENS} $\underline{\uhnN}(\cdot,t_n) \in \bfV_h^n$ in the first argument
to linearize our formulation. We may consider this \emph{time-lift} as an interpolation of the previous solution to the next surface. 

We utilize the same idea to define the next fully discrete approximation scheme based on  the weak formulation \eqref{weak lagrange hom NV cov ENS}: The linearized \emph{covariant evolving fully discrete Lagrangian surface Navier-Stokes} finite element approximation is the following:

\noindent {\textbf{(eNSWc{\tiny h}): }}
Given $\uh^0 \in \bfV_h^{0,div}$ and appropriate approximation for the data $\bff_h$, for $n=0,1,...,N$ find $\uhn \in \bfV_h^n$, $\{\phn,\lambda_h^n\}\in Q_h^n\times \Lambda_h^n$ such that 
\begin{align}
\begin{cases}
\tag{eNSWc{\tiny h}}
        \label{eq: fully discrete fin elem approx cov ENS}
\frac{1}{\Delta t}\Big(\mh(\uhn,\vhn) - \mh(\uhnN,\underline{\vhn}(\cdot,t_{n-1}))\Big) - \gh(\TrVel;\uhn,\vhn) + \ahhat(\uhn,\vhn)\\
\qquad\qquad\qquad\quad  \, + c_h^{cov}(\underline{\uhnN}(\cdot,t_n),\uhn,\vhn) 
 +b_h^L(\vhn,\{\phn,\lambda_h^{n}\}) = \mh(\bff_h^n,\vhn), \\
 \qquad\qquad\qquad\qquad\qquad\qquad\qquad\qquad\qquad  \ \ \  \! b_h^L(\uhn,\{\qhn,\xihn\}) = 0.
\end{cases}
 \end{align}
for all $\vhn \in \bfV_h^n$ and $\vhnN \in \bfV_h^{n-1}$ and $\{\qhn,\xihn\}\in Q_h^n\times\Lambda_h^n$.  Again, the finite element scheme is linearised with the help of the time-lift $\underline{\uhnN}(\cdot,t_n)$.

In both schemes, we have chosen $\uh^0 \in \bfV_h^{0,div}$ for the initial data so that our schemes remain unconditionally stable (no inverse CFL condition is required). Some examples include $\uh^0 := \mathcal{R}_h\bfu^0$ or $\uh^0 := \mathcal{R}_h^b\bfu^0 = \mathcal{R}_h(\bfu^0, \{p^0,\lambda^0\})$ for $\bfu^0 \in \bfH^1(\Ga)$; see \Cref{remark: about initial conditions ENS} for further details. In practice \eqref{eq: fully discrete fin elem approx ENS} may be  more natural and convenient to implement compared to \eqref{eq: fully discrete fin elem approx cov ENS}, since less geometric information is needed for the approximation of the convective term, cf. \cref{Sec: set-up ENS}.

Note that in both schemes \eqref{eq: fully discrete fin elem approx ENS} and \eqref{eq: fully discrete fin elem approx cov ENS}  the same notation $\lh$ is used  for convenience, although  they approximate  different functions (depending on the choice of the convective term $c_h,$ or $c_h^{cov}$); see also \cref{sec: problem reform ENS,sec: Variational formulation ENS}.
Note that, see \cref{remark: diff formulations reason ENS},
\begin{itemize}
\item
For \eqref{eq: fully discrete fin elem approx ENS}, we are only able to provide stability and convergence results when $\underline{k_\lambda =k_u}$.
\item For \eqref{eq: fully discrete fin elem approx cov ENS} we are able to perform a complete analysis for both $\underline{k_\lambda =k_u}$ and $\underline{k_\lambda =k_u-1}$/\end{itemize} 
 For that reason, from now on,
when studying \eqref{eq: fully discrete fin elem approx ENS} we only consider $\underline{k_\lambda =k_u}$, whilst
when analyzing  \eqref{eq: fully discrete fin elem approx cov ENS} we only consider $\underline{k_\lambda =k_u-1}$.

\begin{remark}[About the non-homogeneous source term \eqref{eq: function f ENS}]\label{remark: about f source discrete ENS}
As discussed in \cref{remark: about f source ENS}, in the  analysis we set \eqref{eq: function f ENS} $\bfeta  = (\eta_1,\eta_2) = 0$. This reduces the problem to one with divergence-free and tangential constraints on the velocity $\bfu$. A similar condition was imposed in \cite{olshanskii2024eulerian}, where, the authors required a homogeneous divergence condition, since it was not clear how to found an appropriate discrete bound for the source term. In our case, we set
 $\bm{\eta}=0$, first for the sake of concision and brevity, since in  proving Ritz-Stokes estimates \cref{lemma: Error Bounds der Ritz-Stokes ENS} the computations are already quite heavy. Second, our equations now satisfy the condition $\bfu\cdot\bfng =0$, and therefore we can take advantage of the said tangentiality in our analysis; see for example \cref{lemma: Geometric perturbations lagrange ENS}, \cref{lemma: Error Bounds Ritz-Stokes ENS} and the results in \cite[Lemma 6.1, Lemma 6.12]{elliott2024sfem}. 
In practice, we implement a discretization of the standard surface Navier--Stokes equations \eqref{weak lagrange hom NV ENS}; see \cref{Sec: set-up ENS} and \eqref{eq: L.M. ENS} where optimal convergence are observed.
\end{remark}
\begin{remark}\label{remark: diff formulations reason ENS}
For the first discrete scheme, we can carry out our analysis (mainly regarding the pressure) only when $\Lambda_h$ is sufficiently rich, i.e., $\underline{k_\lambda=k_u}$. In this case, the improved $H^1$-coercivity estimate \eqref{eq: improved h1-ah bound ENS} now holds yielding  the following equivalence relation between the directional and covariant derivative:
\begin{align*}
    \norm{\wh}_{H^1(\Gaht)}\sim \norm{\wh}_{\ah} \qquad \text{ for }\wh \in \bfV_h^{div}\t.
\end{align*}
This implies the crucial observation  that the convective term $c_h$ is bounded w.r.t. the $\ah$-norm, cf. \eqref{ch boundedness ENS}.  In the $\underline{k_\lambda=k_u-1}$ case, we have only the worse  $H^1$-coercivity estimate \eqref{eq: coercivity and Korn's inequality Lagrange ENS}  and, in our analysis, this  results in a loss of an order of convergence.

On the other hand, we can prove stability and optimal convergence results  for both $\underline{k_\lambda=k_u}$ and $\underline{k_\lambda=k_u-1}$, in the case of the \eqref{eq: fully discrete fin elem approx cov ENS} scheme, since the convective term $c_h^{cov}$ can be bounded appropriately; see the new bound \eqref{ch cov boundedness ENS} which emerges due to the term $\nbgcovh\uh$. For that reason, in our analysis, cf. \cref{sec: Pressure a-priori estimates for kl= ku-1 ENS}, when considering $\underline{k_\lambda=k_u-1}$ we focus on the \eqref{eq: fully discrete fin elem approx cov ENS} scheme. 
\end{remark}

\subsection{Discrete properties of bilinear forms}\label{sec: Discrete properties of bilinear forms ENS}
In this section we recall and extend results found in \cite{elliott2015error,dziuk2012fully} regarding some discrete properties of bilinear forms \cref{def: bilinear forms discrete ENS}.
To start with, we extend the results in \cite[Lemma 3.6]{dziuk2012fully} for vector-valued functions, for time $t\in[t_{n-1},t_{n+1}]$ as in \cite{elliott2015error}, with the help of the transport formulae \cref{lemma: Transport formulae discrete ENS}, and the bounds \Cref{Lemma: discrete bounds and coercivity results ENS}. 

\begin{lemma}\label{Lemma: time differences properties bilinear ENS}
For $t\in[t_{n-1},t_{n+1}]$  and $\wh\t \in \bfV_h\t$ we have the estimates
\begin{align}\label{eq: time differences properties bilinear m ENS}
    &|\mh(\whn,\whn) -\mh(\underline{\whn}(\cdot,t),\underline{\whn}(\cdot,t))| \leq c\Delta t \, \mh(\whn,\whn), \\
    \label{eq: time differences properties bilinear ahat ENS}
        &|\ahhat(\whn,\whn) - \ahhat(\underline{\whn}(\cdot,t),\underline{\whn}(\cdot,t))| \leq c\Delta t \, \ahhat(\whn,\whn) \leq c\Delta t \, \ah(\whn,\whn),\\
        \label{eq: time differences properties bilinear a ENS}
        &|\ah(\whn,\whn) - \ah(\underline{\whn}(\cdot,t),\underline{\whn}(\cdot,t))| \leq c\Delta t \, \ah(\whn,\whn),\\
        \label{eq: time differences properties bilinear b ENS}
        &|\bhtil(\whn,\{\phn,\lhn\}) - \bhtil(\underline{\whn}(\cdot,t),\{\underline{\phn}(\cdot,t),\underline{\lhn}(\cdot,t)\})| \leq \Delta t  \, \norm{\whn}_{\ah} \norm{\{\phn,\lhn\}}_{L^2(\Gahn)},
\end{align}
and so for any $t\in[t_{n-1},t_{n+1}]$ and $\Delta t$ sufficiently small, it holds
\begin{align}\label{eq: time equivalence norms ENS}
    \norm{\underline{\uhn}(\cdot,t)}_{L^2(\Gaht)} &\leq c\norm{\uhn}_{L^2(\Gahn)}, \\
    \norm{\underline{\uhn}(\cdot,t)}_{\ahhat}&\leq c\norm{\uhn}_{\ahhat},\\
    \label{eq: time equivalence norms a ENS}
    \norm{\underline{\uhn}(\cdot,t)}_{\ah}&\leq c\norm{\uhn}_{\ah}.
\end{align}
\end{lemma}
\begin{proof}
Eq.\,\eqref{eq: time differences properties bilinear m ENS}-\eqref{eq: time differences properties bilinear a ENS} and \eqref{eq: time equivalence norms ENS}-\eqref{eq: time equivalence norms a ENS} are natural vector-valued extensions of \cite[Lemma 3.6]{dziuk2012fully}, where one uses the new transport formulae \cref{lemma: Transport formulae discrete ENS} and bounds \cref{Lemma: discrete bounds and coercivity results ENS}. For \eqref{eq: time differences properties bilinear b ENS}, using calculations similar to the previously proven estimates, i.e. the transport formula \eqref{eq: Transport formulae 3 discrete ENS}, the bound \eqref{bhtilda boundedness ENS}, and the already proven time inequalities \eqref{eq: time equivalence norms ENS}, \eqref{eq: time equivalence norms a ENS}, give the final result.
\end{proof}

From \eqref{eq: fully discrete fin elem approx ENS}, \eqref{eq: fully discrete fin elem approx cov ENS} we see that the solution $\uhn \in \bfV_h^{n,div}$. An essential issue is that, since the computational surface is time-varying, we have to deal with the fact that we need divergence conformity of velocity solutions at two different time-steps (surfaces). But, for $\uhnN\in\bfV_h^{n-1,div}$, using the \emph{time-lift} \eqref{eq: discrete time-lift ENS} we see that $\bhtil(\underline{\uhnN}(\cdot,t_n);\{\qhn,\xihn\})\neq0$ for $\{\qhn,\xihn\}\in Q_h^n\times \Lambda_h^n$, and so it is not \emph{discretely weakly tangential divergence-free} with respect to $\Gahn$, i.e.  $\underline{\uhnN}(\cdot,t_n)\notin \bfV_h^{n,div}$.
However, since we use \emph{evolving surface finite elements} (ESFEM) we have, by the time-lift \eqref{eq: discrete time-lift ENS} again, that
\begin{equation}\label{eq: blb important time property ENS}
    \boxed{\bhtil(\uhnN,\{\underline{\qhn}(\cdot,t_{n-1}),\underline{\xihn}(\cdot,t_{n-1})\} )=0 \quad \text{ for all } \{\underline{\qhn}(\cdot,t_{n-1}),\underline{\xihn}(\cdot,t_{n-1})\}\in Q_h^{n-1}\times\Lambda_h^{n-1}.}
\end{equation}
This important \emph{property} will be used throughout the text to help tackle the aforementioned issue.

In the upcoming \cref{Sec: asssumptions about discrete scheme ENS} in order to find pressure stability it is important that the previous time-step is in $\bfV_h^{n,div}$. For that reason, we define a \emph{time discrete version of the surface Leray projection}. 
\begin{definition}[Leray time-projection]\label{def: discrete time Leray proj ENS}
For $\wh^{n-1} \in \bfV_h^{n-1,div}$, we define $\hat{\bfw}_h^{n-1} \in \bfV_h^{n,div},$ $\{\hat{p}_h^{n},\hat{\lambda}_h^n\}\in Q_h^n\times \Lambda_h^n$ such that it is the unique solution of 
    \begin{equation}
    \begin{aligned}\label{eq: discrete time Leray proj ENS}
        \mh(\hat{\bfw}_h^{n-1},\bfv_h^n) + \bhtil(\bfv_h^n,\{\hat{p}_h^{n},\hat{\lambda}_h^n\}) &= \mh(\wh^{n-1},\underline{\bfv_h^n}(\cdot,t_{n-1}))\\
        \bhtil(\hat{\bfw}_h^{n-1},\{\qh^n,\xi_h^n\}) &= 0,
    \end{aligned}
\end{equation}
for all $\bfv_h^n \in \bfV_h^n,$ $\{q_h^n,\xi_h^n\}\in Q_h^n\times \Lambda_h^n$.
\end{definition}
This time-projection is clearly well-posed. Furthermore, by \eqref{eq: blb important time property ENS} it has the following properties.
\begin{lemma}\label{lemma: hatw to n-1w}
 For $\wh^{n-1} \in \bfV_h^{n-1,div}$ and $\hat{\bfw}_h^{n-1} \in \bfV_h^{n,div} $ and $\Delta t $ sufficiently small we have
 \begin{align}
     \label{eq: hatw to n-1w stab ENS}
     \norm{\hat{\bfw}_h^{n-1}}_{L^2(\Gah^n)} &\leq c\norm{\wh^{n-1}}_{L^2(\Gah^{n-1})},\\
     \label{eq: hatw to n-1w infsup ENS}
     \norm{\{\hat{p}_h^{n},\hat{\lambda}_h^n\}}_{L^2(\Gah^n)} &\leq \norm{\underline{\bfw_h^{n-1}}(\cdot,t_{n})-\hat{\bfw}_h^{n-1}}_{L^2(\Gah^n)} + c\Delta t \norm{\bfw_h^{n-1}}_{L^2(\Gah^{n-1})}, \\
     \label{eq: hatw to n-1w dt ENS}
     \norm{\underline{\bfw_h^{n-1}}(\cdot,t_{n})-\hat{\bfw}_h^{n-1}}_{L^2(\Gah^n)}&\leq c\Delta t \norm{\wh^{n-1}}_{\ah}.
 \end{align}
\end{lemma}
\begin{proof}
 For the first inequality we just test  \eqref{eq: discrete time Leray proj ENS} with $\hat{\bfw}_h^{n-1}$, therefore we see that
 \begin{equation*}
     \beta\norm{\hat{\bfw}_h^{n-1}}_{L^2(\Gah^n)}^2 \leq \norm{\underline{\hat{\bfw}_h^{n-1}}(\cdot,t_{n-1})}_{L^2(\Gah^{n-1})}\norm{\wh^{n-1}}_{L^2(\Gah^{n-1})} \underset{\eqref{eq: time equivalence norms ENS}}{\leq} \norm{\hat{\bfw}_h^{n-1}}_{L^2(\Gah^n)}\norm{\wh^{n-1}}_{L^2(\Gah^{n-1})},
 \end{equation*}
 which concludes our assertion. The inequality \eqref{eq: hatw to n-1w infsup ENS} then holds if we simply use the discrete \textsc{inf-sup} \eqref{eq: discrete inf-sup condition Gah Lagrange ENS},  add and subtract suitable terms and apply \eqref{eq: time differences properties bilinear m ENS}, to see that
 \begin{equation*}
     \begin{aligned}
         \norm{\{\hat{p}_h^{n},\hat{\lambda}_h^n\}}_{L^2(\Gah^n)} &\leq \sup_{\vh^n \in \bfV_h^n}\frac{\mh(\hat{\bfw}_h^{n-1}-\underline{\bfw_h^{n-1}}(\cdot,t_{n}),\vh^n)}{\norm{\vh^n}_{\ah}} + \frac{\mh(\underline{\bfw_h^{n-1}}(\cdot,t_{n}),\vh^n)-\mh(\wh^{n-1},\underline{\bfv_h^n}(\cdot,t_{n-1})) }{\norm{\vh^n}_{\ah}}\\
         &\leq \norm{\underline{\bfw_h^{n-1}}(\cdot,t_{n})-\hat{\bfw}_h^{n-1}}_{L^2(\Gah^n)} + c\Delta t \norm{\bfw_h^{n-1}}_{L^2(\Gah^{n-1})}.
     \end{aligned}
 \end{equation*}
 
 Now to show \eqref{eq: hatw to n-1w dt ENS} we notice from \eqref{eq: discrete time Leray proj ENS} that
 \begin{equation}\label{eq: hatw to n-1w dt inside 1 ENS}
     \mh(\underline{\bfw_h^{n-1}}(\cdot,t_{n})-\hat{\bfw}_h^{n-1},\vh^n) = \mh(\underline{\bfw_h^{n-1}}(\cdot,t_{n}),\vhn) - \mh(\wh^{n-1},\underline{\bfv_h^n}(\cdot,t_{n-1})) - \bhtil(\bfv_h^n,\{\hat{p}_h^{n},\hat{\lambda}_h^n\}).
 \end{equation}
Taking $\vh^n = \underline{\bfw_h^{n-1}}(\cdot,t_{n})-\hat{\bfw}_h^{n-1} \in \bfV_h^n$ in \eqref{eq: hatw to n-1w dt inside 1 ENS}, observing from \eqref{eq: blb important time property ENS}, \eqref{eq: time differences properties bilinear b ENS}, \eqref{eq: discrete time Leray proj ENS} that
\begin{equation*}
\begin{aligned}
      \bhtil(\underline{\bfw_h^{n-1}}(\cdot,t_{n})-\hat{\bfw}_h^{n-1},\{\hat{p}_h^{n},\hat{\lambda}_h^n\})  &= \bhtil(\underline{\bfw_h^{n-1}}(\cdot,t_{n}),\{\hat{p}_h^{n},\hat{\lambda}_h^n\}) - \underbrace{\bhtil(\bfw_h^{n-1},\{\underline{\hat{p}_h^{n}}(\cdot,t_{n-1}),\underline{\hat{\lambda}_h^{n}}(\cdot,t_{n-1})\})}_{:=0 \text{ since }  \wh^{n-1}\in \bfV_h^{n-1,div} \eqref{eq: blb important time property ENS}}\\
      &\leq c\Delta t \norm{\bfw_h^{n-1}}_{\ah}\norm{\{\hat{p}_h^{n},\hat{\lambda}_h^n\}}_{L^2(\Gah^n)}, 
      \end{aligned}
\end{equation*}
and using the prior proven estimate \eqref{eq: hatw to n-1w infsup ENS} and \eqref{eq: time differences properties bilinear m ENS} once again, we obtain the following
 \begin{equation*}
     \begin{aligned}
         &\mh(\underline{\bfw_h^{n-1}}(\cdot,t_{n})-\hat{\bfw}_h^{n-1},\underline{\bfw_h^{n-1}}(\cdot,t_{n})-\hat{\bfw}_h^{n-1}) \leq  c\Delta t \norm{\bfw_h^{n-1}}_{L^2(\Gah^{n-1})}\norm{\underline{\bfw_h^{n-1}}(\cdot,t_{n})-\hat{\bfw}_h^{n-1}}_{L^2(\Gah^n)}\\
         &\quad + \Delta t \norm{\bfw_h^{n-1}}_{\ah}\norm{\underline{\bfw_h^{n-1}}(\cdot,t_{n})-\hat{\bfw}_h^{n-1}}_{L^2(\Gah^n)} + c(\Delta t)^2 \norm{\bfw_h^{n-1}}_{\ah} \norm{\bfw_h^{n-1}}_{L^2(\Gah^{n-1})}.
     \end{aligned}
 \end{equation*}
Then, from the fact that $\norm{\bfw_h^{n-1}}_{L^2(\Gah^{n-1})} \leq \norm{\bfw_h^{n-1}}_{\ah}$, an application of Young's inequality and a kickback argument gives our desired result.
\end{proof}

\subsection{Well-posedness}\label{Sec: asssumptions about discrete scheme ENS}
We now prove stability results for our discrete \emph{velocities} $\uh^n$ and the two \emph{pressures} $\{\ph^n,\lh^n\}$ $n=0,...,N$. For the pressures, we mainly study the \eqref{eq: fully discrete fin elem approx ENS} scheme and so consider $\underline{k_\lambda = k_u}$, as mentioned in \cref{remark: diff formulations reason ENS}.
Nevertheless, in \cref{remark: About the pressure stability for arbitrary ENS} we also show how one can establish pressure stability when $\underline{k_\lambda = k_u-1}$ for the scheme \eqref{eq: fully discrete fin elem approx cov ENS} by imposing additional conditions. 
\subsubsection{Well-posedness of the velocity}
We start with the velocities, where we make use of the following matrix formulation \eqref{eq: matrix-vector form ENS} of \eqref{eq: fully discrete fin elem approx ENS} and \eqref{eq: fully discrete fin elem approx cov ENS}:
Consider $\{\bfchi_j^n\}_{j=1}^J=\{\{\chi_j^n\bfe_i\}_{i=1}^3\}_{j=1}^J$, $\{\psi_j^n\}_{j=1}^J$ and $\{\varphi_j^n\}_{j=1}^J$ to be the nodal basis of $\bfV_h^n$, $Q_h^n$ and $\Lambda_h^n$ respectively. Then, we can write the following \emph{matrix-vector formulation} of the finite element approximation \eqref{eq: fully discrete fin elem approx ENS}:Determine the vector-matrix $\bfU^n\in \mathbb{R}^{J\times3}$ and coefficient vectors $Q^n,\, \Lambda^n \in \mathbb{R}^J$, given $\bfU^{n-1}$ such that they solve
\begin{equation}
    \begin{aligned}\label{eq: matrix-vector form ENS}
        \frac{1}{\Delta t} (\bfM^n\bfU^n- \bfM^{n-1}\bfU^{n-1}) - \bfG^n\bfU^n+  \hat{\bfA}^n\bfU^n + \mathbf{C}^n\bfU^n +\bfB^n\{Q^n,\Lambda^n\}
       &= \bfF^n\\
       B^n\bfU^n&=0,
    \end{aligned}
\end{equation}
where $\uhn = \sum_{j=1}^J\bfU^n\bfchi_j^n$ and $\qhn= \sum_{j=1}^JQ^n\psi_j^n$, $\lhn = \sum_{j=1}^J\Lambda^n\varphi_j^n$ and where the time-depending matrices are defined as
\begin{equation}
    \begin{aligned}\label{eq: matrix-vector form matrices ENS}
        (\bfM^n)_{ij} &:= \mh(t_n;\bfchi_i^n,\bfchi_j^n), \quad \  (\hat{\bfA}^n)_{ij} := \ahhat(t_n;\bfchi_i^n,\bfchi_j^n), \qquad  \quad \ \ (\bfA^n)_{ij} := (\hat{\bfA}^n)_{ij}+ (\bfM^n)_{ij}\\
        (\bfG^n)_{ij} &:= \gh(t_n;\bfchi_i^n,\bfchi_j^n),   \qquad  \!  (\mathbf{C}^n)_{ij} := c_h(t_n;\underline{\uhnN};\bfchi_i^n,\bfchi_j^n), \qquad (\bfF^n)_{ij} := \mh(t_n;\bff_{h,i}^n,\bfchi_j^n) \\
        (\bfB^n)_{ij} &:= \bhtil(t_n;\bfchi_i^n,\{\varphi_j^n,\psi_j^n\}).
    \end{aligned}
\end{equation}
One can define the \emph{matrix-vector formulation} of \eqref{eq: fully discrete fin elem approx cov ENS} analogously.

\begin{lemma}[Velocity Estimate]\label{Lemma: velocity stab estimate ENS}
    For the velocity solution $\uh^k \in \bfV_h^k$, $k=0,1,...,n$ of \eqref{eq: fully discrete fin elem approx ENS} or  \eqref{eq: fully discrete fin elem approx cov ENS} we have the following energy estimates, which holds for $n \geq 1$
    \begin{equation}
        \begin{aligned}\label{eq: velocity stab estimate ENS}
            \norm{\uh^n}_{L^2(\Gahn)}^2  + \Delta t \sum_{k=1}^{n} \norm{\uh^k}_{\ah}^2 
             \leq exp(ct_n)\big(\norm{\uh^0}_{L^2(\Gah^0)}^2 + \Delta t \sum_{k=0}^{n}\norm{\bff_h^{k}}_{L^2(\Gah^k)}^2 \big),
        \end{aligned}
    \end{equation}
    where $\uh^0 \in \bfV_h^{div,0}$ an appropriate approximation $\bfu^0 \in \bfH^1(\Ga)$, e.g. $\uh^0 := \mathcal{R}_h(\bfu^0)$, and where the constant $c>0$ is independent of $h,\,t$.
\end{lemma}
\begin{proof}
Let us consider the \emph{matrix-vector formulation} \eqref{eq: matrix-vector form ENS} and the matrices \eqref{eq: matrix-vector form matrices ENS}, where $\uhn = \sum_{j=1}^J\bfU^n\bfchi_j^n$ and $\qhn= \sum_{j=1}^JQ^n\psi_j^n$, $\lhn = \sum_{j=1}^J\Lambda^n\varphi_j^n$. Testing this equation with $\bfU^n$ we obtain
\begin{equation}
    \begin{aligned}\label{eq: velocity stab estimate inside main ENS}
        \frac{1}{\Delta t} (\bfM^n\bfU^n- \bfM^{n-1}\bfU^{n-1})\cdot\bfU^n +  \bfA^n\bfU^n\cdot\bfU^n
       = \bfF^n\cdot \bfU^n + \bfG^n\bfU^n\cdot\bfU^n + \bfM^n\bfU^n\cdot\bfU^n,
    \end{aligned}
\end{equation}
where we used the fact that $c_h(\bullet;\bullet,\bullet)$ is skew-symmetric and $\uhn \in \bfV_h^{div}$, therefore $\mathbf{C}^n\bfU^n\cdot\bfU^n = B^n\bfU^n\cdot\{Q^n,\Lambda^n\}=0$, and that $\ah = \ahhat + \mh$. Using the classic identity $(\bfa-\bfb)\cdot \bfa = \frac{1}{2}(\bfa^2 - \bfb^2) + \frac{1}{2}(\bfa-\bfb)^2$, the first two term can be rewritten as  
    \begin{align}\label{eq: velocity stab estimate inside 1 ENS}
        \bfM^n\bfU^n\cdot \bfU^n - \bfM^{n-1}\bfU^{n-1}\cdot \bfU^n &= \bfM^{n-1}\bfU^n\cdot \bfU^n - \bfM^{n-1}\bfU^{n-1}\cdot \bfU^n + (\bfM^n-\bfM^{n-1})\bfU^{n}\cdot \bfU^n \nonumber\\
        &= \frac{1}{2}\big(\bfM^{n-1}\bfU^n\cdot \bfU^n -  \bfM^{n-1}\bfU^{n-1}\cdot \bfU^{n-1}\big)\\
        & \  +\frac{1}{2}\bfM^{n-1}(\bfU^n-\bfU^{n-1})\cdot(\bfU^n-\bfU^{n-1}) + (\bfM^n-\bfM^{n-1})\bfU^{n}\cdot \bfU^n\nonumber.
    \end{align}
Now the last term above is bounded using \eqref{eq: time differences properties bilinear m ENS} by
\begin{equation}\label{eq: velocity stab estimate inside 2 ENS}
    (\bfM^n-\bfM^{n-1})\bfU^{n}\cdot U^n\geq -c\Delta t \norm{\uhn}_{L^2(\Gahn)}^2,
\end{equation}
while by the fully discrete time derivative \eqref{eq: n to n-1 on surface n ENS} and the time-lift \eqref{eq: discrete time-lift ENS} we see that the first term in the last line of  \eqref{eq: velocity stab estimate inside 1 ENS} can be  written as 
\begin{equation}
    \begin{aligned}\label{eq: velocity stab estimate inside 3 ENS}
        \frac{1}{2}m_h(\underline{\uhn}(\cdot,t_{n-1}) - \uhnN,\underline{\uhn}(\cdot,t_{n-1}) - \uhnN) = \frac{(\Delta t )^2}{2}m_h(\underline{\matdt\uhn}(\cdot,t_{n-1}),\underline{\matdt\uhn}(\cdot,t_{n-1})).
    \end{aligned}
\end{equation}
So applying \eqref{eq: velocity stab estimate inside 1 ENS}, \eqref{eq: velocity stab estimate inside 2 ENS} and \eqref{eq: velocity stab estimate inside 3 ENS} to the main equation \eqref{eq: velocity stab estimate inside main ENS}, along with the bound \eqref{g boundedness ENS} for $\bfg(\bullet,\bullet)$ and the $\epsilon$-weight Young's inequality,  we can rewrite \eqref{eq: velocity stab estimate inside main ENS}  as 
\begin{equation}
    \begin{aligned}
        \frac{1}{2}\big(\norm{\uhn}_{L^2(\Gahn)}^2 - \norm{\uhnN}_{L^2(\GahnN)}^2\big) + \frac{(\Delta t)^2 }{2}m_h(\underline{\matdt\uhn}(\cdot,t_{n-1}),\underline{\matdt\uhn}(\cdot,t_{n-1})) + \Delta t\norm{\uhn}_{\ah}\\
        \leq c(\epsilon)\Delta t\norm{\uhn}_{L^2(\Gahn)}^2  + \frac{\Delta t}{\epsilon}\norm{\bff_h^n}_{L^2(\Gahn)}^2.
    \end{aligned}
\end{equation}
Finally summing the inequalities for $k=1,2,...,n$, and applying a discrete Gr\"onwall argument, for small enough time step $\Delta t$, we obtain our desired result. A similar stability bound is obtained for  \eqref{eq: fully discrete fin elem approx cov ENS}, since the only change is the convective term $c^{cov}(\bullet;,\bullet,\bullet)$ which is still skew-symmetric.
\end{proof}

\subsubsection{Well-posedness of the pressure}
Before  proving pressure stability bounds we highlight two points.
\begin{itemize}
\item
First, further assumptions are needed regarding the velocity approximation. As in the proof for the  stationary surface case \cite{elliott2025unsteady}, we require that $\underline{\uhnN}(\cdot,t_n)$ in $c_h^{(\cdot)}(\bullet;\bullet,\bullet)$ is uniformly bounded in the energy $(\ah)$-norm \eqref{eq: energy norm ENS} for each for each $n=1,...,N$, that is noting \eqref{eq: time equivalence norms a ENS} 
\begin{equation}\label{eq: uniform bound b stability ENS}
 \sup_{n} \norm{\underline{\uhnN}(\cdot,t_n)}_{\ah} \leq \sup_{n} \norm{\uh^{n-1}}_{\ah} \leq \sup_{n} \norm{\uh^{n}}_{\ah} \leq C_a,    
\end{equation}
for some constant $C_a>0$.
Assuming this, allows us to  to prove stability results for the two pressures (in reality when $k_\lambda=k_u$; for $k_\lambda=k_u-1$ see \cref{remark: About the pressure stability for arbitrary ENS} where further assumptions are needed), which otherwise would not be possible, due to the skew-symmetric form of $c_h^{(\cdot)}(\bullet;\bullet,\bullet)$ and the  $H^1-$conforming Taylor-Hood finite elements.

However, 
we are able to show that solutions of the initial approximation of the surface Navier-Stokes problem \eqref{eq: fully discrete fin elem approx ENS}, i.e. $\uhn$, actually satisfy \eqref{eq: uniform bound b stability ENS}, under standard regularity conditions on the true velocity solution $\bfu$. This becomes clear upon proving velocity error estimates in \cref{sec: velocity a-priori estimates ENS},  with the help of the Ritz-Stokes projection defined in equation \eqref{eq: surface Ritz-Stokes projection ENS}; see \Cref{remark: About pressure stability ENS} for further details.
\item
Second, there is  the complication that $\uhnN \notin\bfV_h^{n,div}$, cf. \cref{sec: Discrete properties of bilinear forms ENS}, since the computational surface is time-varying. This issue has been raised in the literature \cite{von2022unfitted,neilan2024eulerian,burman2022eulerian} in the case of unfitted and Eulerian time-stepping schemes on evolving domains and in \cite{olshanskii2024eulerian} in the case of evolving surfaces, where it was, therefore, not possible to compute standard stability or convergence estimates for the $L^2_{L^2}$-norm of the pressure. To circumvent it, the authors instead use a weaker discrete $L^1_{H^1}$ semi-norm coupled with an added inverse CFL condition, where they managed to establish stability -- error bounds solely in that weaker norm.
However, in our case, since we are using \emph{ESFEM}, we are able to find an $L^2_{L^2}$ estimate for pressure $\phn$ with the help of the \emph{Leray time-projection} \eqref{eq: discrete time Leray proj ENS} and its properties \cref{lemma: hatw to n-1w} (see also \eqref{eq: blb important time property ENS}).
\end{itemize}

In the stationary case, both for surfaces \cite{elliott2025unsteady} and domains \cite{FrutosGradDivOseen2016,JohnBook2016}, the authors  find an auxiliary bound for the time-derivative using the inverse Stokes operator. Now, due to the evolving setting, we utilize both the \emph{Leray time-projection} \eqref{eq: discrete time Leray proj ENS} and a \emph{discrete inverse Stokes operator} 
$\mathcal{A}_h^n: \, \bfV_h^{n,div} \to \bfV_h^{n,div}$, defined by the following unique solution 
 to the discrete Stokes problem: Given $\whn \in \bfV_h^{n,div}$, consider
\begin{equation}
    \begin{aligned}\label{eq: Discrete inverse Stokes ENS}
        \ah(\mathcal{A}_h^n \wh^n, \vh^n)  &= (\wh^n,\vh^n)_{L^2(\Gah^n)}\\
        \bhtil(\mathcal{A}_h \wh^n,\{\qh^n,\xih^n\}) &= 0,
    \end{aligned}
\end{equation}
for every $\vh^n \in \bfV_h^{n,div}$ and $\{\qh^n,\xih^n\} \in Q_h^n\times\Lambda_h^n$. It is clear that the operator is coercive and also bounded, thus, it is well-defined. We can also see by testing appropriately that 
\begin{equation}
\begin{aligned}
   \norm{\wh^n}_{\mathcal{A}_h}^2 &=\norm{\mathcal{A}_h^n \wh^n}_{\ah}^2  = (\mathcal{A}_h^n\wh^n,\wh^n)_{L^2(\Gah^n)},  \qquad 
   \norm{\wh^n}_{\mathcal{A}_h}  =  \norm{\mathcal{A}_h^n \wh^n}_{\ah} \leq \norm{\wh^n}_{L^2(\Gah^n)}.
   \end{aligned}
\end{equation}
\noindent This allows us to establish the following \emph{time-projected} discrete time-derivative auxiliary result.

\begin{lemma}[Auxiliary bound]\label{Lemma: auxilary stab bounds ENS}
Let assumption \eqref{eq: uniform bound b stability ENS} hold. Then for the solution $\uh^k \in \bfV_h^k$ of \eqref{eq: fully discrete fin elem approx ENS} the following estimate holds for $n \geq 1$
    \begin{equation}\label{eq: auxilary stab bounds ENS}
        \sum_{k=1}^n \Delta t \norm{\frac{\mathcal{A}_h^k(\uh^k - \hat{\bfu}_h^{k-1})}{\Delta t}}_{\ah}^2 \leq  C(C_a)\bfA(\uh^0,\bff_h),
    \end{equation}
where $\uh^0 \in \bfV_h^{div,0}$ an appropriate approximation, e.g. $\uh^0 := \mathcal{R}_h(\bfu^0)$, $\bfA = \bfA(\uh^0,\bff_h)$ the bound in \eqref{eq: velocity stab estimate ENS} and $C(C_a)$ the constant depending on the assumption \eqref{eq: uniform bound b stability ENS}, independent of $h$, $t$.
\end{lemma}

\begin{proof}
Testing the fully discrete scheme \eqref{eq: fully discrete fin elem approx ENS} with $\mathcal{A}_h^n(\uh^n-\hat{\bfu}_h^{n-1})\in \bfV_h^{n,div}$, which by \eqref{eq: Discrete inverse Stokes ENS} is well-defined, since $\hat{\bfu}_h^{n-1} \in \bfV_h^{n,div},$ and recalling the Leray time-projection \eqref{eq: discrete time Leray proj ENS}, yields
for $1 \leq n \leq N$
\begin{align}
    \label{eq: Auxiliary stab bound inside 1 ENS}
    &\mh(\frac{\uh^n-\hat{\bfu}_h^{n-1}}{\Delta t},\mathcal{A}_h^n(\uh^n - \hat{\bfu}_h^{n-1}))
         -\gh(\TrVel;\uh^n,\mathcal{A}_h^n(\uh^n - \hat{\bfu}_h^{n-1})) + \ahhat(\uh^n,\mathcal{A}_h^n(\uh^n - \hat{\bfu}_h^{n-1})) \nonumber \\
        & +c_h(\underline{\uhnN}(\cdot,t_n);\uh^n,\mathcal{A}_h^n(\uh^n - \hat{\bfu}_h^{n-1})) = (\bff_h^n,\mathcal{A}_h^n(\uh^n - \hat{\bfu}_h^{n-1}))_{L^2(\Gah^n)} 
\end{align}
Now, Let us bound each term appropriately:

$\bullet \  $ By the definition of $\mathcal{A}_h$ \eqref{eq: Discrete inverse Stokes ENS} we obtain 
\begin{align*}
    \mh(\frac{\uh^n-\hat{\bfu}_h^{n-1}}{\Delta t},\mathcal{A}_h^n(\uh^n - \hat{\bfu}_h^{n-1})) =\frac{1}{\Delta t }\norm{\mathcal{A}_h^n(\uh^n - \hat{\bfu}_h^{n-1})}^2_{\ah}.
\end{align*}

$\bullet \  $ By the bounds \eqref{g boundedness ENS}, \eqref{ah boundedness ENS}, $\norm{\cdot}_{\ahhat} \leq \norm{\cdot}_{\ah}$ and Young's inequality we obtain 
\begin{align*}
    \gh(\TrVel;\uh^n,\mathcal{A}_h^n(\uh^n - \hat{\bfu}_h^{n-1})) &\leq c\norm{\uhn}_{L^2(\Gah^n)}\norm{\mathcal{A}_h^n(\uh^n - \hat{\bfu}_h^{n-1})}_{L^2(\Gah^n)}\\
    &\leq c\Delta t \norm{\uh^n}_{L^2(\Gah^n)}^2 + \frac{1}{5\Delta t }\norm{\mathcal{A}_h^n(\uh^n - \hat{\bfu}_h^{n-1})}_{L^2(\Gah^n)}^2,\\
    \ahhat(\uh^n,\mathcal{A}_h^n(\uh^n - \hat{\bfu}_h^{n-1}))&\leq c\norm{\uhn}_{\ahhat}\norm{\mathcal{A}_h^n(\uh^n - \hat{\bfu}_h^{n-1})}_{\ahhat}\\
    &\leq  c\Delta t \norm{\uh^n}_{\ah}^2 + \frac{1}{5\Delta t }\norm{\mathcal{A}_h^n(\uh^n - \hat{\bfu}_h^{n-1})}_{\ah}^2.
\end{align*}

$\bullet \  $ Recalling the assumption \eqref{eq: uniform bound b stability ENS}, the bound \eqref{ch boundedness ENS}, and using Young's inequality we 
see that
\begin{align*}
   c_h(\underline{\uhnN};\uh^n,\mathcal{A}_h^n(\uh^n - \hat{\bfu}_h^{n-1})) &\leq \norm{\underline{\uhnN}}_{\ah}\norm{\uh^n}_{\ah}\norm{\mathcal{A}_h^n(\uh^n - \hat{\bfu}_h^{n-1})}_{\ah}\\
   &\leq C_a\Delta t \norm{\uh^n}_{\ah}^2 +  \frac{1}{5\Delta t }\norm{\mathcal{A}_h(\uh^n - \uh^{n-1})}_{\ah}^2.
\end{align*}

$\bullet \  $ Similar calculations give the following
\begin{align*}
    (\bff_h^n,\mathcal{A}_h^n(\uh^n - \hat{\bfu}_h^{n-1}))_{L^2(\Gah^n)} \leq c\Delta t \norm{\bff_h^n}_{L^2(\Gah^n)}^2 +\frac{1}{5\Delta t }\norm{\mathcal{A}_h(\uh^n -  \hat{\bfu}_h^{n-1})}_{\ah}^2.
\end{align*}

\noindent The combination of the above results in \eqref{eq: Auxiliary stab bound inside 1 ENS} yields
\begin{equation*}
    \frac{1}{\Delta t }\norm{\mathcal{A}_h^n(\uh^n - \hat{\bfu}_h^{n-1})}^2_{\ah} \leq  C_a\Delta t \norm{\uh^n}_{\ah}^2 +  c\Delta t \norm{\bff_h^n}_{L^2(\Gah^n)}^2,
\end{equation*}
where summing this inequality for $n=1,2,...,\, k$ and recalling the velocity bounds in \Cref{Lemma: velocity stab estimate ENS} concludes our assertion.
\end{proof}

Now, we need to bound the fully discrete time-derivative $\frac{1}{\Delta t}\Big(\mh(\whn,\vhn) - \mh(\whnN,\underline{\vhn}(\cdot,t_{n-1}))\Big)$ in a negative norm (dual $H^1$-norm), with respect to already known bounds, e.g.  the auxiliary bound for $\norm{\mathcal{A}_h^n(\wh^n-\hat{\bfw}_h^{n-1})}_{\ah}$ \eqref{eq: auxilary stab bounds ENS}, and \eqref{eq: velocity stab estimate ENS}.
Similar estimates have been considered in \cite{FrutosGradDivOseen2016,FrutosGradDivNS2016,AyusoPostNS2005,JohnBook2016} for the Oseen and Navier--Stokes problems, where the authors prove a similar $H^{-1}$-norm bound on fixed domains without variational crimes involving domain approximation. These estimates were expanded in \cite{elliott2025unsteady} for $k_g$-order approximated surfaces $\Gah$. Now, we further extend this result, with the help of the \emph{Leray time-projections} \eqref{eq: discrete time Leray proj ENS} for evolving triangulated surfaces $\Gaht$ of order $k_g$. The following proof utilizes the result in \cite[Lemma 6.5]{elliott2025unsteady}.

\begin{lemma}\label{lemma: dual estimate ENS}
Assume that $\underline{k_{\lambda} = k_u}$, then for a vector-valued functions $\wh^n \in \bfV_h^{n,div}$ the following estimate holds for constant $c>0$ independent of $t,\,h$:
    \begin{align}\label{eq: dual estimate ENS}
        \sup_{\vh^n \in \bfV_h^n}\frac{\mh(\wh^n,\vh^n) - \mh(\wh^{n-1},\underline{\vh^n}(\cdot,t^{n-1}))}{\norm{\vhn}_{H^1(\Gahn)}} &\leq ch\norm{\wh^n-\hat{\bfw}_h^{n-1}}_{L^2(\Gah^n)} + c\norm{\mathcal{A}_h^n(\wh^n-\hat{\bfw}_h^{n-1})}_{\ah}\nonumber \\
        &\quad + c\Delta t \norm{\wh^{n-1}}_{\ah}.
    \end{align}
\end{lemma}
\begin{proof}
We start by rearranging, adding and subtracting appropriate $L^2$-terms on different surfaces $\Gah^n,\,\Gah^{n-1}$, so that with the help of the Leray time-projection \eqref{eq: discrete time Leray proj ENS}, we can define \emph{weakly tangential divergence-free functions} on the same time-surface, and hence make use of \cite[Lemma 6.5]{elliott2025unsteady}. So, for $\whn \in \bfV_h^{n,div}$ and $\hat{\bfw}_h^{n-1}\in \bfV_h^{n,div}$, cf. \eqref{eq: discrete time Leray proj ENS}, we derive the following 
\begin{equation}
    \begin{aligned}\label{eq: dual estimate inside main ENS}
       \mh(\wh^n,\vh^n) - \mh(\wh^{n-1},\underline{\vh^n}(\cdot,t_{n-1})) &= \mh(\wh^n- \hat{\bfw}_h^{n-1},\vh^n) + \mh(\hat{\bfw}_h^{n-1}-\underline{\bfw_h^{n-1}}(\cdot,t_{n}),\vh^n)\\
       &\ +\mh(\underline{\bfw_h^{n-1}}(\cdot,t_{n}),\vh^n)- \mh(\wh^{n-1},\underline{\vh^n}(\cdot,t_{n-1})). 
    \end{aligned}
\end{equation}
The second term on the right-hand side of \eqref{eq: dual estimate inside main ENS} can be bounded using \eqref{eq: hatw to n-1w dt ENS} by
\begin{align}\label{eq: dual estimate inside 1 ENS}
    \mh(\hat{\bfw}_h^{n-1}-\underline{\bfw_h^{n-1}}(\cdot,t_{n}),\vh^n) \leq c\Delta t \norm{\whnN}_{\ah}\norm{\vhn}_{L^2(\Gahn)},
\end{align}
while using the time-difference bound \eqref{eq: time differences properties bilinear m ENS} the last two terms are bounded by
\begin{align}\label{eq: dual estimate inside 2 ENS}
    \mh(\underline{\bfw_h^{n-1}}(\cdot,t_{n}),\vh^n)- \mh(\wh^{n-1},\underline{\vh^n}(\cdot,t_{n-1})) \leq c\Delta t \norm{\whnN}_{L^2(\GahnN)}\norm{\vhn}_{L^2(\Gahn)}.
\end{align}
Now, as we mentioned, since $\wh^n- \hat{\bfw}_h^{n-1}\in\bfV_h^{n,div}$, by \cite[Lemma 6.5]{elliott2025unsteady} we can readily see that the first term can be bounded as
\begin{equation}\label{eq: dual estimate inside 3 ENS}
    \sup_{\vh^n \in \bfV_h^n} \frac{\mh(\wh^n- \hat{\bfw}_h^{n-1},\vh^n)}{\norm{\vhn}_{H^1(\Gahn)}} \leq ch\norm{\wh^n- \hat{\bfw}_h^{n-1}}_{L^2(\Gahn)} + c\norm{\mathcal{A}_h^n(\wh^n- \hat{\bfw}_h^{n-1})}_{\ah}.
\end{equation}
Now combining \eqref{eq: dual estimate inside 1 ENS}-\eqref{eq: dual estimate inside 3 ENS} and \eqref{eq: dual estimate inside main ENS}, and taking the supremum over $\vhn \in \bfV_h^n$ gives our result.
\end{proof}

Now we move on to prove stability estimates for the two pressures $\{\phn,\lhn\}$.
First, we present the two discrete \textsc{inf-sup} conditions, which play an important role. These were proven in  \cite[Lemma 5.6, Lemma 5.8]{elliott2024sfem} for stationary surfaces, but can be easily be extended to the evolving setting.
\begin{lemma}[$L^2\times L^2$  Discrete Lagrange \textsc{inf-sup} condition]
\label{Lemma: Discrete inf-sup condition Gah Lagrange ENS}
  Assume a quasi-uniform triangulation of $\Gah$, then there exists a constant $\beta >0$ independent of the mesh parameter $h$ and the time $t$, but dependent on $\mathcal{G}_T$, such that 
    \begin{equation}\label{eq: discrete inf-sup condition Gah Lagrange ENS}
    \beta\norm{\{\qh,\xih\}}_{L^2(\Gaht)\times L^2(\Gaht)}\leq \sup_{\vh \in \bfV_h\t} \frac{\bhtil(\vh,\{\qh,\xih\})}{\norm{\vh}_{\ah}} \ \ \ \forall \{\qh,\xih\} \in Q_h\t \times \Lambda_h\t.
    \end{equation}
\end{lemma}

\begin{lemma}[$L^2\times H_h^{-1}$ Discrete Lagrange \textsc{inf-sup} condition]\label{lemma: L^2 H^{-1} discrete inf-sup condition Gah Lagrange ENS}
Assume a quasi-uniform triangulation of $\Gah$, then there exists a constant $\beta >0$ independent of the mesh parameter $h$ and the time $t$, but dependent on $\mathcal{G}_T$, such that 
\begin{equation}\label{eq: L^2 H^{-1} discrete inf-sup condition Gah Lagrange ENS}
    \beta\norm{\{\qh,\xih\}}_{L^2(\Gaht)\times H_h^{-1}(\Gaht)}\leq \sup_{\vh \in \bfV_h\t} \frac{\bhtil(\vh,\{\qh,\xih\})}{\norm{\vh}_{H^1(\Gaht)}} \ \ \ \forall \{\qh,\xih\} \in Q_h\t \times \Lambda_h\t.
\end{equation}
\end{lemma}

\begin{lemma}[$L^2_{L^2}\times L^2_{H_h^{-1}}$ Pressure Bounds]\label{Lemma: Pressure stab Estimate ENS}
      Let \Cref{lemma: L^2 H^{-1} discrete inf-sup condition Gah Lagrange ENS} and the assumptions in \Cref{Lemma: auxilary stab bounds ENS} and \Cref{lemma: dual estimate ENS} hold.  Then, for the pressure solutions $\{\ph^k,\lh^k\} \in Q_h^k\times L_h^k$, $k=1,...,n$ of \eqref{eq: fully discrete fin elem approx ENS} the following stability estimate holds true
     \begin{equation}\label{eq: Pressure stab Estimate UNS}
            \Delta t \sum_{k=1}^{n} \norm{\{\ph^k,\lh^k\}}_{L^2(\Gah^k)\times H_h^{-1}(\Gah^k)}^2 \leq C(C_a)\bfA(\uh^0,\bff_h),
     \end{equation}
where $\bfA = \bfA(\uh^0,\bff_h)$ the bound in \eqref{eq: velocity stab estimate ENS} and $C(C_a)$ a constant depending on the assumption \eqref{eq: uniform bound b stability ENS}.
 \end{lemma}
 \begin{proof}
Recall the discrete \textsc{inf-sup} condition in \Cref{lemma: L^2 H^{-1} discrete inf-sup condition Gah Lagrange ENS}, then we have 
    \begin{equation}
        \begin{aligned}\label{eq: Pressure stab Estimate UNS inside 2 inf-sup ENS}
            \norm{\{\ph^n,\lh^n\}}_{L^2(\Gah^n) \times H_h^{-1}(\Gah^n)} \leq \sup_{\vh^n \in \bfV_h^n} \frac{\bhtil(\vh^n,\{\ph^n,\lh^n\})}{\norm{\vh^n}_{H^1(\Gahn)}}.
        \end{aligned}
    \end{equation}    
Using the weak formulation \eqref{eq: fully discrete fin elem approx ENS} along with the bounds in \cref{Lemma: discrete bounds and coercivity results ENS} and the $H^1$-coercivity bound \eqref{eq: improved h1-ah bound ENS}, where appropriate ($\uhn\in \bfV_h^{n,div}$), we see that
\begin{equation*}
    \begin{aligned}
        \bhtil(\vh^n,\{\ph^n,\lh^n\}) \leq  -\dfrac{1}{\Delta t}\Big(\mh(\uhn,\vhn) - \mh(\uhnN,\underline{\vhn}(\cdot,t_{n-1}))\Big) + c\norm{\uh^n}_{L^2(\Gah^n)}\norm{\vhn}_{L^2(\Gah^n)}\\
        + c\norm{\uh^n}_{\ah}\norm{\vhn}_{\ah} + \norm{\underline{\uhnN}(\cdot,t_n)}_{\ah}\norm{\uh^n}_{\ah}\norm{\vhn}_{H^1(\Gahn)} + \norm{\bff_h^{n}}_{L^2(\Gah^n)}\norm{\vh^n}_{L^2(\Gah^n)}.
    \end{aligned}
\end{equation*}
Returning to \eqref{eq: Pressure stab Estimate UNS inside 2 inf-sup ENS}, using the dual inequality  \eqref{eq: dual estimate ENS} and recalling the assumption \eqref{eq: uniform bound b stability ENS} we obtain
\begin{equation}\label{eq: Pressure stab Estimate UNS inside inf-sup ENS}
    \begin{aligned}
    \norm{\{\ph^n,\lh^n\}}_{L^2(\Gah^n) \times H_h^{-1}(\Gah^n)} &\leq   ch\norm{\frac{\uh^n-\hat{\bfu}_h^{n-1}}{\Delta t }}_{L^2(\Gah^n)} + c\norm{\frac{\mathcal{A}_h^n(\uh^n-\hat{\bfu}_h^{n-1})}{\Delta t}}_{\ah}  +  c\norm{\uh^{n-1}}_{\ah} \\
    &+  c(1 + C_a)\norm{\uh^n}_{\ah} + \norm{\bff_h^{n}}_{L^2(\Gah^n)} \\
    &\leq  c\norm{\frac{\mathcal{A}_h^n(\uh^n-\hat{\bfu}_h^{n-1})}{\Delta t}}_{\ah}  +  c\norm{\uh^{n-1}}_{\ah} +  c(1 + C_a)\norm{\uh^n}_{\ah} + \norm{\bff_h^{n}}_{L^2(\Gah^n)},
    \end{aligned}
\end{equation}
where in the last inequality we used the definition of the discrete inverse Stokes operator $\mathcal{A}_h^n$ \eqref{eq: Discrete inverse Stokes ENS}, the fact that $\uh^n-\hat{\bfu}_h^{n-1} \in \bfV_h^{n,div}$ and the inverse inequality to show that
\begin{equation*}
    \begin{aligned}
        \norm{\frac{\uh^n-\hat{\bfu}_h^{n-1}}{\Delta t }}_{L^2(\Gah^n)}^2 \leq \norm{\frac{\mathcal{A}_h^n(\uh^n-\hat{\bfu}_h^{n-1})}{\Delta t}}_{\ah}\norm{\frac{\uh^n-\hat{\bfu}_h^{n-1}}{\Delta t}}_{\ah} \leq ch^{-1}\norm{\frac{\mathcal{A}_h^n(\uh^n-\hat{\bfu}_h^{n-1})}{\Delta t}}_{\ah}\norm{\frac{\uh^n-\hat{\bfu}_h^{n-1}}{\Delta t}}_{L^2(\Gah^n)},
    \end{aligned}
\end{equation*}
 hence $h\norm{\frac{\uh^n-\hat{\bfu}_h^{n-1}}{\Delta t }}_{L^2(\Gah^n)} \leq c \norm{\frac{\mathcal{A}_h^n(\uh^n-\hat{\bfu}_h^{n-1})}{\Delta t}}_{\ah}.$ Now squaring and then multiplying \eqref{eq: Pressure stab Estimate UNS inside inf-sup ENS} by $\Delta t$, simple calculations yield
\begin{equation*}
    \begin{aligned}
         \Delta t\norm{\{\ph^n,\lh^n\}}_{L^2(\Gah^n) \times H_h^{-1}(\Gah^n)}^2 \leq c\Delta t \norm{\frac{\mathcal{A}_h^n(\uh^n-\hat{\bfu}_h^{n-1})}{\Delta t}}_{\ah}^2 + C_a\Delta t ( \norm{\uhnN}_{\ah}^2 + \norm{\uhn}_{\ah}^2) + \norm{\bff_h^n}_{L^2(\Gah^n)}. 
    \end{aligned}
\end{equation*}
Finally, summing for $k=1,...,n$ and using the velocity and auxiliary bounds \eqref{eq: velocity stab estimate ENS} and \eqref{eq: auxilary stab bounds ENS} respectively, we obtain our final result.
\end{proof}

\begin{remark}
Notice that in \cref{Lemma: Pressure stab Estimate ENS} using the \textsc{inf-sup} condition \eqref{eq: L^2 H^{-1} discrete inf-sup condition Gah Lagrange ENS}, we were only able to show $L^2_{L^2}\times L^2_{H_h^{-1}}$ stability bounds for the pressures. This is because in our numerical scheme \eqref{eq: fully discrete fin elem approx ENS} the bound of the trilinear form $c_h(\bullet;\bullet,\bullet)$ is given by \eqref{ch boundedness ENS} and therefore the test function $\vhn$ is bounded w.r.t. the $H^1$-norm. If instead the \eqref{eq: fully discrete fin elem approx cov ENS} scheme was used (the one we use for $k_\lambda = k_u-1$), then the bound of $c_h^{cov}(\bullet;\bullet,\bullet)$ is given by \eqref{ch cov boundedness ENS} and therefore, since now the test function is bounded by the energy $(\ah)$-norm instead, we would have been able to use the $L^2\times L^2$ discrete \textsc{inf-sup} condition \eqref{eq: discrete inf-sup condition Gah Lagrange ENS}, which in turn would provide us with $L^2_{L^2}\times L^2_{L^2}$ pressure bounds. This follows similarly to \cite{elliott2025unsteady}.
\end{remark}

\begin{remark}
In the case that the right-hand side of \eqref{eq: fully discrete fin elem approx ENS} was not zero, but some approximation function $\bfeta_h$ of $\bfeta$ in \eqref{weak lagrange hom NV ENS}, then in the stability results the following extra terms would arise \vspace{-1mm}
\begin{equation}\label{eq: constant aux stab ENS}
    \begin{aligned}
        \sum_{k=1}^n\big( \int_{t_{n-1}}^{t_n}\norm{\matd \bfeta_h(\cdot,t)}_{L^2(\Gaht)}^2 + \norm{\bfeta_h(\cdot,t)}_{L^2(\Gaht)}^2 \, dt \big).
    \end{aligned}
\end{equation}
\end{remark}\vspace{-3mm}

\begin{remark}[About $L^2_{L^2}\times L^2_{L^2}$ pressure stability for $k_\lambda=k_u-1$]\label{remark: About the pressure stability for arbitrary ENS}
Now we briefly wish to discuss the $L^2_{L^2}\times L^2_{L^2}$ pressure stability for when $k_\lambda=k_u-1$. First, as previously mentioned in \cref{remark: diff formulations reason ENS} we use the finite element formulation \eqref{eq: fully discrete fin elem approx cov ENS}. Second, the dual estimate \eqref{eq: dual estimate ENS} only holds for $k_\lambda = k_u$ and therefore we need an alternative approach. Basically, instead of the auxiliary result \eqref{eq: auxilary stab bounds ENS} we need to find a stability estimate for the stronger expression \vspace{-1mm}
\begin{equation}\label{eq: fully discrete L2 deriv bound ENS}
    \Delta t \sum_{k=1}^n\mh(\matdt\uhn,\matdt\uhn)= \Delta t \sum_{k=1}^n \frac{\mh(\uh^n-\underline{\uhnN}(\cdot,t_n),\uh^n-\underline{\uhnN}(\cdot,t_n))}{(\Delta t)^2}.
\end{equation}\vspace{-2mm}

\noindent This is almost identical to \cite[Remark 6.1]{elliott2025unsteady} for stationary surfaces where one requires a bound for $\Delta t\sum_{n=1}^n\norm{\frac{\uh^n-\uh^{n-1}}{\Delta t}}_{L^2(\Gah)}^2$.
However, to establish this result, we first need stronger assumptions regarding $\underline{\uhnN}(\cdot,t_n)$, as opposed to \eqref{eq: uniform bound b stability ENS}. Mainly that
\begin{equation}\label{eq: b extra stability ENS}
\sup_{n} \norm{\underline{\uhnN}(\cdot,t_n)}_{L^{\infty}(\Gah)} \leq \sup_{n} \norm{\uhn}_{L^{\infty}(\Gah)} \leq C_a^{\infty}.
\end{equation}\vspace{-3mm}

\noindent We will see that this new assumption holds as long as \emph{further regularity assumptions} are imposed; see \cref{remark: About the pressure stability for arbitrary convergence ENS}. So assuming and using \eqref{eq: b extra stability ENS}, we can then just test  \eqref{eq: fully discrete fin elem approx cov ENS} with $\vhn = \uh^n-\underline{\uhnN}(\cdot,t_n) \in \bfV_h^n$ and after simple calculations, including the fact that by the property \eqref{eq: blb important time property ENS}:
\begin{equation*}
\begin{aligned}
        \bhtil(\uh^n-\underline{\uhnN}(\cdot,t_n),\{\phn,\lhn\}) &= \bhtil(\uhnN,\{\underline{\phn}(\cdot,t_{n-1}),\underline{\lhn}(\cdot,t_{n-1})\} ) - \bhtil(\underline{\uhnN}(\cdot,t_n),\{\phn,\lhn\}) \\
        &\leq c\Delta t \norm{\uhnN}_{\ah}+ \epsilon\Delta t\norm{\{\phn,\lhn\}}_{L^2(\Gahn)} \qquad \qquad \qquad \text{ by \eqref{eq: time differences properties bilinear b ENS}},
        \end{aligned}
\end{equation*}
we can find a stability bound for \eqref{eq: fully discrete L2 deriv bound ENS}, and subsequently for the two pressures in the $L^2_{L^2}\times L^2_{L^2}$-norm, using  the discrete \textsc{inf-sup} condition \cref{Lemma: Discrete inf-sup condition Gah Lagrange ENS}.
We omit further details and point to \cref{sec: Pressure a-priori estimates for kl= ku-1 ENS}, where such analysis is carried out for proving the pressure error estimates instead.    
\end{remark}

\section{Error Analysis}\label{sec: error analysis ENS}
In this section, we derive a-priori error bounds for the two linearized  \emph{evolving fully discrete Lagrangian surface Navier-Stokes} finite element schemes \eqref{eq: fully discrete fin elem approx ENS}, \eqref{eq: fully discrete fin elem approx cov ENS} with initial condition $\uh^0\in \bfV_h^{0,div}$.
The error analysis follows typical arguments. We split the errors by decomposing them into an approximation  error and a discrete remainder error in the appropriate finite element spaces. Initially, this is carried out with the help of the modified Ritz-Stokes projection $\mathcal{R}_h\bfu^n$ defined in \eqref{eq: surface Ritz-Stokes projection ENS} for the velocity $\bfu$ and the standard Lagrange interpolant for the two pressures $\{p, \lambda\}$ ( $(\bfu,\{p,\lambda\})$ solutions of either \eqref{weak lagrange hom NV dir ENS} or \eqref{weak lagrange hom NV cov ENS}). In \cref{sec: Pressure a-priori estimates for kl= ku-1 ENS} we consider a different decomposition, using the standard Ritz-Stokes projection \eqref{eq: surface Ritz-Stokes projection std ENS} instead. Nontheless, we have the following: \vspace{-1mm}
\begin{align}\label{eq: decomposition error 1 Ritz ENS}
    \eu^{n} = \bfu^n - \uh^n &= \underbrace{(\bfu^n- \mathcal{R}_h\bfu^n) }_{\text{Approximation error}} + \underbrace{(\mathcal{R}_h\bfu^n - \uh^n)}_{\text{discrete remainder}} := \rho_{\bfu}^n + \theta_{\bfu}^n , \\
        \label{eq: decomposition error 2 lagrange ENS} 
         \ep^n = p^n - \ph &= \underbrace{(p^n- \Ih p^n)}_{\text{Approximation error}} + \underbrace{(\Ih p^n - \ph)}_{\text{discrete remainder}}:= \rho_{p}^n + \theta_{p}^n,
        \\
        \label{eq: decomposition error 3 lagrange ENS} \el^n =  \lambda^n - \lh &= \underbrace{(\lambda^n- \Ih \lambda^n )}_{\text{Approximation error}} + \underbrace{(\Ih \lambda^n - \lh)}_{\text{discrete remainder}}:= \rho_{\lambda}^n + \theta_{\lambda}^n.      
\end{align}
The bounds of the modified Ritz-Stokes approximation error have already been proved in \cref{lemma: Error Bounds Ritz-Stokes ENS,lemma: Error Bounds der Ritz-Stokes ENS}. The interpolation errors for the pressures are well-known
\begin{equation}
    \begin{aligned}\label{eq: interpolation errors pressures ENS}
        \norm{p^n - \Ih p^n}_{L^2(\Gah)}  &\leq h^{k_{pr}+1}\norm{p^n}_{H^{k_{pr}+1}},\\
        \norm{\lambda^n - \Ih \lambda^n}_{L^2(\Gah)}  &\leq h^{k_{\lambda}+1}\norm{\lambda^n}_{H^{k_{\lambda}+1}}.
    \end{aligned}
\end{equation}
Notice that for convenience, we once again drop the inverse lift extension $(\cdot)^{-\ell}$ notation, since it should be clear from the context whether a quantity is defined on $\Ga$ or $\Gah$, e.g. for functions $\bfu,\, p,\, \lambda$.

It is clear by \eqref{eq: fully discrete fin elem approx ENS} and \eqref{eq: fully discrete fin elem approx cov ENS} and the modified Ritz-Stokes projection  \eqref{eq: surface Ritz-Stokes projection ENS} that $\thbfu^n\in \bfV_h^{n,div}$ for all $n=1,...,N$. To ensure that $\thbfu^n \in \bfV_h^{n,div}$ for $n=0$, the initialization of the  time-stepping procedure is chosen to satisfy $\uh^0 \in \bfV_h^{0,div}$ (see also \cref{sec: The Fully Discrete Scheme ENS}). Although the  choice of  initial condition the  modified Ritz-Stokes projection $\uh^0 = \mathcal{R}_h\bfu^0$ is suitable, other  approximations are possible.

\begin{remark}\label{remark: about initial conditions ENS}
Other choices of initial condition, see also \cref{sec: Pressure a-priori estimates for kl= ku-1 ENS}, can be:
    \begin{itemize}
        \item[-] $\uh^0 = \Ih(\bfu^0)$ : In this case $\uh^0$ is not weakly divergence-free so that we can only have conditional stability, and an inverse type CFL condition is necessary, see also \cite{burman2009galerkin}.
        \item[-] $\uh^0 = \mathcal{R}_h(\bfu^0,\{p^0,\lambda^0\})$ : The main issue is that the initial values $\{p^0,\lambda^0\}$ are not prescribed in the initial value problem.
    \end{itemize}
\end{remark}
In what follows, note that we abuse the notation and define $\underline{\eu^{n-1}}(t_n) :=(\underline{\bfu}^{n-1}(\cdot,t_n))^{-\ell} - \underline{\uh^{n-1}}(\cdot,t_n)$, where $(\underline{\bfu}^{n-1}(\cdot,t_n))^{-\ell} := \bfu(\cdot,t_{n-1})^{-\ell}\circ\Phi^h_{t_{n-1}}\circ\Phi^h_{-t_{n}}$ and $\underline{\bfu}^{n-1} := \bfu(t_{n-1})\circ\Phi^{\ell}_{t_{n-1}}\circ\Phi^{\ell}_{-t_{n}}$. Then, by a simple modification of the time-inequalities \eqref{eq: time equivalence norms ENS} and the properties of the map $\Phi^h_{t}$ \eqref{eq: discrete stong mat deriv ENS}  and $\Phi^{\ell}_{t}$  \eqref{eq: lifted discrete stong mat deriv ENS} (notice that $\matd\underline{\bfu}^{n-1}(\cdot,s)^{-\ell}= \Big(\frac{d}{ds}\bfu^{n-1}\circ\Phi^h(\cdot,t_{n-1})\Big)\circ\Phi^h(\cdot,s)=0$ for $s\in[t_{n-1},t_n]$) , it is not difficult to see that
\begin{equation}\label{eq: time ineq cont u ENS}
\begin{aligned}
    \norm{(\underline{\bfu}^{n-1})^{-\ell}(\cdot,t_n)}_{L^2(\Gah^n)} &\leq c \norm{(\bfu^{n-1})^{-\ell}}_{L^2(\Gah^{n-1})}, \qquad \norm{(\underline{\bfu}^{n-1})^{-\ell}(\cdot,t_n)}_{\ah} \leq c \norm{(\bfu^{n-1})^{-\ell}}_{\ah}, \\
    \norm{\underline{\bfu}^{n-1}(\cdot,t_n)}_{L^2(\Ga^n)} &\leq c \norm{\bfu^{n-1}}_{L^2(\Ga^{n-1})}, \qquad \qquad \qquad \norm{\underline{\bfu}^{n-1}(\cdot,t_n)}_{a} \leq c \norm{\bfu^{n-1}}_{a},
    \end{aligned}
\end{equation}
therefore it also holds that
\begin{equation}\label{eq: time ineq eu ENS}
    \norm{\underline{\eu^{n-1}}(\cdot,t_n)}_{L^2(\Gah^n)} \leq  \norm{\eu^{n-1}}_{L^2(\Gah^{n-1})}, \qquad  \norm{\underline{\eu^{n-1}}(\cdot,t_n)}_{\ah} \leq  \norm{\eu^{n-1}}_{\ah}.
\end{equation}
From now on, we also omit the space argument in the \emph{time-lift} \eqref{eq: discrete time-lift ENS}.

We will make use of the following  regularity assumption.
\begin{assumption}[Regularity I]\label{assumption: Regularity assumptions for velocity estimate ENS}
We assume that the solution $(\bfu,\{p,\lambda\})$ has the following regularity   
\begin{equation*}
    \begin{aligned}
   \int_{0}^{T}\norm{\matn{\bfu}}_{H^{k_u+1}(\Gat)}^2 \, dt + \sup_{t\in[0,T]}\big\{\norm{\bfu}_{H^{k_{u}+1}(\Gat)}^2 &+ \norm{p}_{H^{k_{pr}+1}(\Gat)}^2 + \norm{\lambda}_{H^{k_{\lambda}+1}(\Gat)}^2 \big\}\\
   &+ \int_{0}^{T} \norm{\matn\matn\bfu}_{L^2(\Gat)}^2 +\norm{\matn\bfu}_{H^1(\Gat)}^2 \, dt \leq C.
    \end{aligned}
\end{equation*}
\end{assumption}
\subsection{A-priori Velocity Error Bounds}\label{sec: velocity a-priori estimates ENS}


Recall that by $c^{(\cdot)}_h(\bullet;\bullet,\bullet)$ we denote either the convective term $c_h(\bullet;\bullet,\bullet)$  or $c^{cov}_h(\bullet;\bullet,\bullet)$, and thus consider either the numerical scheme \eqref{eq: fully discrete fin elem approx ENS} or \eqref{eq: fully discrete fin elem approx cov ENS}, respectively.
Using the Ritz-Stokes projection \eqref{eq: surface Ritz-Stokes projection ENS}  $\thbfu^n\in \bfV_h^{n,div}$ for $n=0,1,...,N$ ,
the discrete pressure error decouples from the velocity error. This  is crucial to find error bounds for $\thbfu^n$ independently of $\theta_{p}^n,\, \theta_{\lambda}^n$. The following error equations hold.


\begin{lemma}\label{lemma: error equation ENS}
  For $\thbfu^n,\, \theta_{p}^n,\theta_{\lambda}^n$ as in  \eqref{eq: decomposition error 1 Ritz ENS}, \eqref{eq: decomposition error 2 lagrange ENS}, \eqref{eq: decomposition error 3 lagrange ENS} the following error equations holds true for $n \geq 1$:
   \begin{align}
\boxed{\begin{cases}\label{eq: error equation ENS}\tag{EQ}
      &{\dfrac{1}{\Delta t}\Big(\bfm_h(\theta_{\bfu}^n,\vhn) - \bfm_h(\theta_{\bfu}^{n-1},\underline{\vhn}(\cdot,t^{n-1}))\Big) } + \ahhat(\theta_{\bfu}^n,\vhn) +b_h^L(\vhn,\{\theta_{p}^n,\theta_{\lambda}^n\})\\
       &\qquad\qquad\qquad\quad\ ~~ = \underbrace{\sum_{i=1}^5 \textsc{Err}_{i}^{C}(\vhn)}_{\textsc{Cons. Errors}} +  \underbrace{\sum_{i=1}^2 \textsc{Err}_{i}^{I}(\vhn)}_{\textsc{Inter. Errors}} + \!\!\!\!\! \underbrace{\mathcal{C}^{(\cdot)}(\vhn)}_{\textsc{Inert. Errors}},\\
         &b_h^L(\theta_{\bfu}^n,\{\qhn,\xihn\})= 0, 
    \end{cases}}
\end{align}
for all $\vhn \in \bfV_h^n$ and $\{\qhn,\xi_h^n\}\in Q_h^n\times \Lambda_h^n$,
with \emph{consistency errors}:
\begin{itemize}
    \item [\textbullet\hspace{9mm}] \hspace{-10mm}$\textsc{Err}_1^{C}(\vhn) := \dfrac{1}{\Delta t}\Big(\mh(\bfu^n,\vhn) - \mh(\bfu^{n-1},\underline{\vhn}(t_{n-1}))\Big) - \bfm(\matn\bfu^n,(\vhn)^{\ell}) - \gh(\TrVel;\uhn,\vhn)$
\end{itemize}
\begin{minipage}[t]{9cm}
    \begin{itemize}
        \item[\textbullet\hspace{9mm}] \hspace{-10mm}$ \textsc{Err}_2^{C}(\vhn) := \bhtil(\vhn,\{p^n,\lambda^n\})- b^L((\vhn)^{\ell},\{p^n,\lambda^n\})$\vspace{1mm}
        \item[\textbullet\hspace{9mm}] \hspace{-10mm}$ \textsc{Err}_3^{C,(\cdot)}(\vhn) := c_h^{(\cdot)}(\underline{\bfu}^{n-1}(t_n);\bfu^n,\vhn)  - c^{(\cdot)}(\bfu^n;\bfu^n,(\vhn)^{\ell})$
    \end{itemize}
\end{minipage}
\begin{minipage}[t]{8cm}
    \begin{itemize}
        \item[\textbullet\hspace{6mm}] \hspace{-8mm} $ \textsc{Err}_4^{C}(\vhn) :=  \bfm(\bff^n,(\vhn)^{\ell})- \mh(\bff_h^n,\vhn)$\vspace{1mm}
        \item[\textbullet\hspace{6mm}] \hspace{-8mm} $ \textsc{Err}_5^{C}(\vhn) := \ahhat(\mathcal{R}_h\bfu^n,\vhn) - \ahat(\bfu^n,(\vhn)^{\ell})$
    \end{itemize}
\end{minipage}

\vspace{3mm}

\noindent and \emph{interpolations errors}: \vspace{2mm}

\begin{minipage}[t]{9cm}
    \begin{itemize}
        \item[\textbullet\hspace{13mm}] \hspace{-12mm}$ \textsc{Err}_1^{I}(\vhn) :=-\dfrac{1}{\Delta t}\Big(\mh(\rho_{\bfu}^n,\vhn) - \mh(\rho_{\bfu}^{n-1},\underline{\vhn}(t_{n-1}))\Big)$\vspace{1mm}
        \item[\textbullet\hspace{13mm}] \hspace{-12mm}$ \textsc{Err}_2^{I}(\vhn) := -\bhtil(\vhn,\{\rho_{p}^n,\rho_{\lambda}^n\})$
    \end{itemize}
\end{minipage}
\begin{minipage}[t]{8cm}
    \begin{itemize}
    \item[\phantom{\textbullet}] \vspace{1mm}
        \item[\textbullet\hspace{31mm}] \hspace{-33mm} $ \mathcal{C}^{(\cdot)}(\vhn) :=  c_h^{(\cdot)}(\underline{\eu^{n-1}}(t_n);\bfu^n;\vhn) - c_h^{(\cdot)}(\underline{\uhnN}(t_n);\eu^n;\vhn)$
    \end{itemize}
\end{minipage}
\end{lemma}
\begin{proof}
Considering the decompositions \eqref{eq: decomposition error 1 Ritz ENS}, \eqref{eq: decomposition error 2 lagrange ENS}, \eqref{eq: decomposition error 3 lagrange ENS} and the modified Ritz-Stokes map \eqref{eq: surface Ritz-Stokes projection ENS}, and using the finite element approximations \eqref{eq: fully discrete fin elem approx ENS} and \eqref{eq: fully discrete fin elem approx cov ENS}, we derive the following 
\begin{equation*}
    \begin{aligned}
         &(\mathrm{EQ}):=-\dfrac{1}{\Delta t}\Big(\mh(\thbfu^n,\vhn) - \mh(\thbfu^{n-1},\underline{\vhn}(t_{n-1}))\Big) + \ahhat(\thbfu^n,\vhn) + b_h^L(\vhn,\{\thp^n,\thl^n\}) = - \mh(\bff_h^n,\vhn)\\ 
         &\quad - \gh(\TrVel;\uhn,\vhn)
         +  \dfrac{1}{\Delta t}\Big(\mh(\mathcal{R}_h\bfu^n,\vhn) - \mh(\mathcal{R}_h\bfu^{n-1},\underline{\vhn}(t_{n-1}))\Big) + \ahhat(\mathcal{R}_h\bfu^n,\vhn)\\
         &\quad + b_h^L(\vhn,\{\Ih p^n,\Ih \lambda^n\}) + c_h^{(\cdot)}(\underline{\uhnN}(t_n);\uhn;\vhn).
    \end{aligned}
\end{equation*}
Now recalling continuous weak formulations \eqref{weak lagrange hom NV dir ENS} and \eqref{weak lagrange hom NV cov ENS} and adding and subtracting the term $c_h^{(\cdot)}(\underline{\bfu}^{n-1}(t_n);\bfu^n;\vh)$,  careful calculations yield 

\begin{equation*}
    \begin{aligned}
      &(\mathrm{EQ}) = \dfrac{1}{\Delta t}\Big(\mh(\rho_{\bfu}^n,\vhn) - \mh(\rho_{\bfu}^{n-1},\underline{\vhn}(t_{n-1}))\Big) -  \bhtil(\vhn,\{\rho_{p}^n,\rho_{\lambda}^n\}) \\
      &+ \dfrac{1}{\Delta t}\Big(\mh(\bfu^n,\vhn) - \mh(\bfu^{n-1},\underline{\vhn}(\cdot,t_{n-1}))\Big) - \bfm(\matn\bfu^n,(\vhn)^{\ell}) - \gh(\TrVel;\uhn,\vhn)\\
      & +b_h^L(\vhn,\{p^n,\lambda^n\})- b^L((\vhn)^{\ell},\{p^n,\lambda^n\}) + c_h^{(\cdot)}(\underline{\bfu}^{n-1};\bfu^n,\vh)  - c^{(\cdot)}(\bfu^n;\bfu^n,\vhl)  + \bfm(\bff^n,(\vhn)^{\ell})- \mh(\bff_h^n,\vhn)\\
      & + \ahhat(\mathcal{R}_h\bfu^n,\vhn) - \ahat(\bfu^n,(\vhn)^{\ell}) - c_h^{(\cdot)}(\underline{\eu^{n-1}}(t_n);\bfu^n;\vhn) - c_h^{(\cdot)}(\underline{\uhnN}(t_n);\eu^n;\vhn)\\
      & =: \text{Err}_1^{I}(\vhn) + \text{Err}_2^{I}(\vhn) +  \text{Err}_1^{C}(\vhn) +  \text{Err}_2^{C}(\vhn) +  \text{Err}_3^{C,(\cdot)}(\vhn) + \text{Err}_4^{C}(\vhn) + \text{Err}_5^{C}(\vhn) + \mathcal{C}^{(\cdot)}(\vhn),
    \end{aligned}
\end{equation*}
where we have  used 
\begin{equation*}\label{eq: discrete inertia error ENS}
    c_h^{(\cdot)}(\underline{\uh^{n-1}}(t_n);\uh^n;\vh) - c_h^{(\cdot)}(\underline{\bfu}^{n-1}(t_n);\bfu^n;\vh) =  - c_h^{(\cdot)}(\underline{\eu^{n-1}}(t_n);\bfu^n;\vhn) - c_h^{(\cdot)}(\underline{\uhnN}(t_n);\eu^n;\vhn).
\end{equation*}
The second equation in \eqref{eq: error equation ENS} follows since $\thbfu^n \in \bfV_h^{n,div}$ for all $n\geq 0$; $\thbfu^0 = \mathcal{R}_h\bfu^0$ see \cref{remark: about initial conditions ENS}.
\end{proof}

\begin{lemma}[Interpolation errors]\label{lemma: Interpolation errors ENS}
The following interpolation error holds for all $\vhn \in \bfV_h^n$ \vspace{-2mm}
    \begin{align}\label{eq: Interpolation errors ENS}
        |\textsc{Err}^{I}_1(\vhn)| &\leq  \dfrac{c}{\sqrt{\Delta t}}\Big(\int_{t_{n-1}}^{t^n}\sum_{j=0}^{1}\norm{(\matd)^j\rho_{\bfu}\t}_{L^2(\Gat)}^2 \, dt\Big)^{1/2}\norm{\vhn}_{L^2(\Gah^n)}, \\
        \label{eq: Interpolation error 2 ENS}
        |\textsc{Err}^{I}_2(\vhn)|  &\leq c(h^{k_{pr}+1} + h^{k_{\lambda}+1})( \norm{p^n}_{H^{k_{pr}+1}(\Ga^n)} + \norm{\lambda^n}_{H^{k_{\lambda}+1}(\Ga^n)})\norm{\vhn}_{\ah},
\end{align}
If furthermore $\underline{k_\lambda = k_u}$ and $\vhn \in \bfV_h^{n,div}$ the following holds\vspace{-2mm}
\begin{align}
    \label{eq: Interpolation errors improved ENS}
        |\textsc{Err}^{I}_1(\vhn)| \leq  c(\dfrac{h^{\widehat{r}_u} + h^{\widehat{r}_u+1}}{\sqrt{\Delta t}}+ h^{\widehat{r}_u}\sqrt{\Delta t})\Big(\int_{t_{n-1}}^{t^n}\sum_{j=0}^{1}\norm{(\matn)^j\bfu}_{H^{k_u+1}(\Gat)}^2 \, dt\Big)^{1/2}\norm{\vhn}_{\ah},
\end{align}
where $\widehat{r}_u = min\{k_u,k_g\}$ and the constants $c>0$ are independent of $h$ and $t$.
\end{lemma}

\begin{proof}
Rewriting $\textsc{Err}^{I}_1(\cdot)$ in integral form and using the transport property \eqref{eq: Transport formulae 2 discrete ENS}, the time-inequality \eqref{eq: time equivalence norms ENS}, the Ritz-Stokes projection estimates \eqref{eq: Error Bounds Ritz-Stokes ENS}, \eqref{eq: Error Bounds der Ritz-Stokes L2 ENS} and a simple Cauchy-Schwartz inequality, leads to
    \begin{equation}
        \begin{aligned}\label{eq: Interpolation errors inside 1 ENS}
            &\dfrac{1}{\Delta t}\Big(\mh(\rho_{\bfu}^n,\vhn) - \mh(\rho_{\bfu}^{n-1},\underline{\vhn}(\cdot,t_{n-1}))\Big) = \dfrac{1}{\Delta t} \int_{t_{n-1}}^{t^n} \frac{d}{dt} \Big(\mh(\rho_{\bfu}(t),\underline{\vhn}(t))\Big)\, dt\\
            &\qquad =\dfrac{1}{\Delta t} \int_{t_{n-1}}^{t^n} \mh(\matd\rho_{\bfu}(t),\underline{\vhn}(t)) + \gh(\TrVel;\rho_{\bfu}(t), \underline{\vhn}(t))\, dt\\
            &\qquad \leq \dfrac{c}{\Delta t}\int_{t_{n-1}}^{t^n} (\norm{\matd\rho_{\bfu}}_{L^2(\Gaht)}+\norm{\rho_{\bfu}}_{L^2(\Gaht)})\norm{\vhn}_{L^2(\Gah^n)}\, dt \\
            &\qquad\leq \dfrac{c}{\sqrt{\Delta t}}\Big(\int_{t_{n-1}}^{t^n}(\norm{\matd\rho_{\bfu}}_{L^2(\Gaht)}^2+ \norm{\rho_{\bfu}}_{L^2(\Gaht)}^2) \, dt\Big)^{1/2}\norm{\vhn}_{L^2(\Gah^n)}
        \end{aligned}
    \end{equation}
    yielding the  desired result.  Next, \eqref{eq: Interpolation error 2 ENS} can be dealt with through standard techniques using the $\bhtil$ bound \eqref{bhtilde boundedness ENS} and the interpolation errors \eqref{eq: interpolation errors pressures ENS}.

Now regarding the  estimate \eqref{eq: Interpolation errors improved ENS}, assuming $\vhn \in \bfV_h^{n,div}$, going back to the second line of  \eqref{eq: Interpolation errors inside 1 ENS} and splitting the first term into a tangent and a normal direction, along with the Leray time-projection \eqref{eq: discrete time Leray proj ENS} from time $t^n$ to time $t$ for some $t\in[t_{n-1},t_n]$, i.e. $\hat{\bfv}_h^n(t) \in \bfV_h^{t,div}$, we see that
        \begin{align}\label{eq: Interpolation errors inside 2 ENS}
            &\dfrac{1}{\Delta t} \int_{t_{n-1}}^{t^n} \mh(\matd\rho_{\bfu}(t),\underline{\vhn}(t)) = \dfrac{1}{\Delta t} \int_{t_{n-1}}^{t^n} \mh(\bfP\matd\rho_{\bfu}(t),\underline{\vhn}(t)) +  \mh(\matd\rho_{\bfu}(t)\cdot\bfn(t),\underline{\vhn}(t)\cdot\bfn(t)) \nonumber\\
        &=  \dfrac{1}{\Delta t} \int_{t_{n-1}}^{t^n} \mh(\bfP\matd\rho_{\bfu}(t),\underline{\vhn}(t)) +  \mh\big(\matd\rho_{\bfu}(t)\cdot\bfn(t),(\underline{\vhn}(t)-\hat{\bfv}_h^n(t))\cdot\bfn(t) + \hat{\bfv}_h^n(t)\cdot\bfn(t)\big)\nonumber\\
        &\leq  \dfrac{1}{\sqrt{\Delta t}}\Big(\int_{t_{n-1}}^{t^n}\norm{\bfP\matd\rho_{\bfu}}_{L^2(\Gaht)}^2 \, dt\Big)^{1/2}\norm{\vhn}_{L^2(\Gah^n)}\nonumber\\
        & \quad +  \dfrac{1}{{\Delta t}}\int_{t_{n-1}}^{t^n}\norm{\matd\rho_{\bfu}\cdot\bfn}_{L^2(\Gaht)}^2 \big(c\norm{(\underline{\vhn}(t)-\hat{\bfv}_h^n(t))}_{L^2(\Gaht)}+\norm{\hat{\bfv}_h^n(t)\cdot\bfn\t}_{L^2(\Gaht)} \big) \,dt.
        \end{align}
Using the Leray time-projection estimates  \eqref{eq: hatw to n-1w dt ENS} and \eqref{eq: hatw to n-1w stab ENS} (they still hold for $\hat{\bfv}_h^n(t) \in \bfV_h^{t,div}$) along with the normal control \eqref{eq: divfree L2 norm kl=ku ENS}  and geometric estimate \eqref{eq: geometric errors 1 ENS} we have

\begin{align}
\label{eq: Interpolation errors inside 3 ENS}
    \norm{(\underline{\vhn}(t)-\hat{\bfv}_h^n(t))}_{L^2(\Gah^n)} &\leq \Delta t \norm{\vhn}_{\ah},\\
    \label{eq: Interpolation errors inside 4 ENS}
    \norm{\hat{\bfv}_h^n(t) \cdot\bfn(t)}_{L^2(\Gaht)} &\leq \norm{\hat{\bfv}_h^n(t) \cdot(\bfn(t)-\nh\t)}_{L^2(\Gaht)} + \norm{\hat{\bfv}_h^n(t)\cdot\nh\t)}_{L^2(\Gaht)}\nonumber \\ 
    &\leq ch\norm{\hat{\bfv}_h^n(t)}_{L^2(\Gaht)} \leq ch \norm{\vhn}_{L^2(\Gah^n)}.
\end{align}
Therefore, plugging \eqref{eq: Interpolation errors inside 3 ENS} and \eqref{eq: Interpolation errors inside 4 ENS} into \eqref{eq: Interpolation errors inside 2 ENS} and using the Ritz-Stokes estimate \eqref{eq: Error Bounds der Ritz-Stokes L2 ENS} we find
\begin{equation*}
    \begin{aligned}\label{eq: Interpolation errors inside 5 ENS}
        &\dfrac{1}{\Delta t} \int_{t_{n-1}}^{t^n} \mh(\matd\rho_{\bfu}(t),\underline{\vhn}(t)) \leq  \dfrac{1}{\sqrt{\Delta t}}\Big(\int_{t_{n-1}}^{t^n}\norm{\bfP\matd\rho_{\bfu}}_{L^2(\Gaht)}^2+ h\norm{\matd\rho_{\bfu}\cdot\bfn}_{L^2(\Gaht)} \, dt\Big)^{1/2}\norm{\vhn}_{L^2(\Gah^n)}\\
        &\qquad\qquad\qquad\qquad\qquad\qquad\ \ + \sqrt{\Delta t}\Big(\int_{t_{n-1}}^{t^n}\norm{\matd\rho_{\bfu}\cdot\bfn}_{L^2(\Gaht)} \, dt\Big)^{1/2}\norm{\vhn}_{\ah}\\
        &\leq \dfrac{h^{\widehat{r}_u}+h^{\widehat{r}_u+1}}{\sqrt{\Delta t}}\Big(\int_{t_{n-1}}^{t^n}\sum_{j=0}^{1}\norm{(\matn)^j\bfu}_{H^{k_u+1}(\Gat)}^2 \, dt\Big)^{1/2}\norm{\vhn}_{L^2(\Gah^n)}\\
        &\qquad\qquad\qquad\qquad\qquad\qquad\ \  + h^{\widehat{r}_u}\sqrt{\Delta t}\Big(\int_{t_{n-1}}^{t^n}\sum_{j=0}^{1}\norm{(\matn)^j\bfu}_{H^{k_u+1}(\Gat)}^2 \, dt\Big)^{1/2}\norm{\vhn}_{\ah},
    \end{aligned}
\end{equation*}
where we also applied the norm equivalence \eqref{eq: norm equivalence ENS} and  \eqref{eq: relate matd to matdl ENS} as needed. This gives our assertion.
\end{proof}

\begin{lemma}[Consistency errors]\label{lemma: Consistency errors ENS}
There exists a constant $c>0$ independent of $h$ and $t$ such that for all $\vhn \in \bfV_h^n$ the following error bounds hold
     \begin{align}\label{eq: Consistency error 1 ENS}
        &|\textsc{Err}_1^{C}(\vhn)| \leq  c\sqrt{\Delta t}\Big(\int_{t_{n-1}}^{t^n} \norm{\matn\matn\bfu}_{L^2(\Gat)}^2 +\norm{\matn\bfu}_{L^2(\Gat)}^2 \, dt\Big)^{1/2}\norm{\vhn}_{L^2(\Gah^n)} \nonumber\\
        &\qquad\qquad\quad\,+ \frac{ch^{\widehat{r}_u+1}}{\sqrt{\Delta t }}\Big(\int_{t_{n-1}}^{t^n}+\norm{\matn\bfu}_{L^2(\Gat)}^2 + \norm{\bfu}_{H^1(\Gat)}^2 \, dt\Big)^{1/2}\norm{\vhn}_{L^2(\Gah^n)}\nonumber
        \\
        &\qquad\qquad\quad\, +  ch^{\widehat{r}_u+1}\!\!\!\!\!\!\! \sup_{t\in[t_{n-1},t_n]}\norm{\bfu}_{L^2(\Gat)}\norm{\vhn}_{H^1(\Gah^n)} + \big(\norm{\rho_{\bfu}^n}_{L^2(\Gah^n)}+\norm{\theta_{\bfu}^n}_{L^2(\Gah^n)}\big)\norm{\vhn}_{L^2(\Gah^n)},\\
        \label{eq: Consistency error 2 ENS}
        &|\textsc{Err}_2^{C}(\vhn)| \leq ch^{k_g}(\norm{p^n}_{H^1(\Ga^n)} + \norm{\lambda^n}_{L^2(\Ga^n)})\norm{\vhn}_{L^2(\Gahn)} ,\\
        \label{eq: Consistency error 3 ENS}
        &|\textsc{Err}_3^{C}(\vhn)| \leq c\big(\sqrt{\Delta t}\Big(\int_{t_{n-1}}^{t^n} \norm{\matn\bfu}_{H^1(\Gat)}^2\, dt\Big)^{1/2} \!\!\!\!\!+ h^{k_g}\norm{\bfu^{n-1}}_{H^1(\Ga^{n-1})}\norm{\bfu^{n}}_{H^1(\Ga^{n})}\big) \norm{\vhn}_{H^1(\Gah^n)},\\
        &| \textsc{Err}_3^{C,cov}(\vhn)| \leq c\big(\sqrt{\Delta t}\Big(\int_{t_{n-1}}^{t^n} \norm{\matn\bfu}_{H^1(\Gat)}^2\, dt\Big)^{1/2}\norm{\vhn}_{\ah}\nonumber\\
        \label{eq: Consistency error 3 cov ENS}
        &\qquad\qquad\qquad\qquad\qquad\qquad\qquad\qquad + h^{k_g}\norm{\bfu^{n-1}}_{H^1(\Ga^{n-1})}\norm{\bfu^{n}}_{H^1(\Ga^{n})} \big)\norm{\vhn}_{H^1(\Gah^n)},\\
        \label{eq: Consistency error 4 ENS}
        &| \textsc{Err}_4^{C}(\vhn)| \leq ch^{k_g+1}\norm{\bff^n}_{L^2(\Ga^n)}\norm{\vhn}_{L^2(\Gah^n)}.
\end{align}
\end{lemma}
\begin{proof}
Let us estimate each consistency error separately.

\noindent\underline{$\textsc{Err}_1^{C} :$} We rewrite our expression in integral form, apply the transport formula \eqref{eq: Transport formulae 2 discrete ENS} to find that
        \begin{align}\label{eq: consist error inside err1 1 ENS}
            &\!\!\!\!\!\!\!\textsc{Err}_1^{C}(\vh)= \dfrac{1}{\Delta t} \int_{t_{n-1}}^{t^n} \mh(\matd\bfu,\underline{\vhn}(t)) + \gh(\TrVel;\bfu,\underline{\vhn}(t)) \, dt - \bfm(\matn\bfu^n,(\vhn)^{\ell})  - \gh(\TrVel;\uhn,\vhn) \nonumber \\
            &\!\!\!\!\!= \underbrace{\dfrac{1}{\Delta t} \int_{t_{n-1}}^{t^n} \big[ \mh(\matd\bfu,\underline{\vhn}(t)) - \bfm(\matdl\bfu,(\underline{\vhn}(t))^{\ell}) +\gh(\TrVel;\bfu,\underline{\vhn}(t)) -\bfg(\TrVelL;\bfu,(\underline{\vhn}(t))^{\ell} )\big] \, dt}_{=:\mathcal{E}_1}\\
            &\!\!+ \underbrace{\dfrac{1}{\Delta t} \int_{t_{n-1}}^{t^n} \big[ \bfm(\matdl\bfu,(\underline{\vhn}(t))^{\ell}) + \bfg(\TrVelL;\bfu,(\underline{\vhn}(t))^{\ell} )\big] \,dt  - \bfm(\matn\bfu^n,(\vhn)^{\ell})  - \gh(\TrVel;\uhn,\vhn),}_{=:\mathcal{E}_2} \nonumber
        \end{align}
where using the geometric perturbations \eqref{eq: errors of domain of integration data ENS}, \eqref{eq: Geometric perturbations g ENS}, the time inequality \eqref{eq: time equivalence norms ENS}, \eqref{eq: relate matd to matdl ENS}, the norm equivalence \eqref{eq: norm equivalence ENS} and a Cauchy-Schwartz inequality as in \eqref{eq: Interpolation errors inside 1 ENS} we obtain
\begin{equation*}
    \begin{aligned}
        |\mathcal{E}_1| \leq  \frac{ch^{\widehat{r}_u+1}}{\sqrt{\Delta t }} \Big(\int_{t_{n-1}}^{t^n}\norm{\matn\bfu}_{L^2(\Gat)}^2+ \norm{\bfu}_{H^1(\Gat)}^2 \, dt\Big)^{1/2}\norm{\vhn}_{L^2(\Gah^n)}.
    \end{aligned}
\end{equation*}
Now for $\mathcal{E}_2$, first recall the transport formulae \eqref{eq: Transport formulae 2 ENS} and \eqref{eq: Transport formulae 2 lift ENS} to see that 
\begin{equation}\label{eq: consist errors inside 1 ENS}
      \bfm(\matdl\bfu,\vhl) + \bfg(\TrVelL;\bfu,\vhl) = \bfm(\matn\bfu,\vhl) + \bfg(\FlVel;\bfu,\vhl) - \bfm(\bfu,(\matdl - \matn)\vhl),
\end{equation}
therefore we readily see that
\begin{equation*}
    \begin{aligned}
        \mathcal{E}_2 &= \dfrac{1}{\Delta t} \int_{t_{n-1}}^{t^n} \Big[ \bfm(\matn\bfu,(\underline{\vhn}(t))^{\ell}) - \bfm(\matn\bfu^n,(\vhn)^{\ell})+ \bfg(\FlVel;\bfu,(\underline{\vhn}(t))^{\ell} )- \bfg(\TrVel;\bfu^n,(\vhn)^{\ell})\\
        & \quad + \bfm(\bfu,(\matdl - \matn)(\underline{\vhn}(t))^{\ell})\Big] \,dt + \Big(\bfg(\FlVel;\bfu^n,(\vhn)^{\ell}) - \gh(\TrVel;\uhn,\vhn) \Big)=: \mathcal{E}_2^1 + \mathcal{E}_2^2.
    \end{aligned}
\end{equation*}
Once again rewriting the expressions in integral form, along with time inequality \eqref{eq: time equivalence norms ENS}, the bound \eqref{eq: relate matd to matdl ENS}, transport formulae in \cref{lemma: Transport formulae cont ENS} and the velocity $\FlVel$ assumptions \eqref{eq: normal vel flow phd}, we obtain

\begin{equation*}
    \begin{aligned}
        &|\mathcal{E}_2^1| \leq \Big|\dfrac{1}{\Delta t} \int_{t_{n-1}}^{t^n}\int_{t}^{t^n} \frac{d}{ds} \big(\bfm(\matn\bfu,(\underline{\vhn}(s))^{\ell}) + \bfg(\FlVel;\bfu,(\underline{\vhn}(s))^{\ell} ) \big) \,ds\,dt + \dfrac{1}{\Delta t} \int_{t_{n-1}}^{t^n}\bfm(\bfu,(\matdl - \matn)(\underline{\vhn}(t))^{\ell})\,dt\Big|\\
        &\leq  \dfrac{1}{\Delta t} \int_{t_{n-1}}^{t^n} \int_{t}^{t^n}\bfm(\matn\matn\bfu,(\underline{\vhn}(s))^{\ell}) + 2\bfg(\FlVel;\matn\bfu,(\underline{\vhn}(s))^{\ell} ) + \bfm(\matn\divg(\FlVel)\,\matn\bfu,(\underline{\vhn}(s))^{\ell})\,ds\,dt \\
        &\qquad + ch^{\widehat{r}_u+1} \!\!\!\!\!\! \sup_{t\in[t_{n-1},t_n]}\big(\norm{\bfu}_{L^2(\Gat)} \big)\, \norm{\vhn}_{H^1(\Gah^n)}\\
        &\leq c\sqrt{\Delta t}\Big(\int_{t_{n-1}}^{t^n} \norm{\matn\matn\bfu}_{L^2(\Gat)}^2 +\norm{\matn\bfu}_{L^2(\Gat)}^2 \, dt\Big)^{1/2}\norm{\vhn}_{L^2(\Gah^n)}  + ch^{\widehat{r}_u+1} \!\!\!\!\!\!\! \sup_{t\in[t_{n-1},t_n]}\big(\norm{\bfu}_{L^2(\Gat)} \big)\norm{\vhn}_{H^1(\Gah^n)}.
    \end{aligned}
\end{equation*}
Adding and subtract suitable terms we also see from \eqref{eq: difference of velocities ENS}, \eqref{eq: Geometric perturbations g ENS}, and the bound \eqref{g boundedness ENS} that
\begin{equation*}
    \begin{aligned}
       \mathcal{E}_2^2 &= \bfg(\FlVel-\TrVelL;\bfu^n,(\vhn)^{\ell}) - \Grm_{\gb}(\bfu^n,\vhn) - \gh(\TrVel;\rho_{\bfu}^n,\vhn) - \gh(\TrVel;\theta_{\bfu}^n,\vhn)\\
       &\leq ch^{\widehat{r}_u+1}\norm{\bfu^n}_{L^2(\Gah^n)}\norm{\vhn}_{L^2(\Gah^n)} + \big(\norm{\rho_{\bfu}^n}_{L^2(\Gah^n)}+\norm{\theta_{\bfu}^n}_{L^2(\Gah^n)}\big)\norm{\vhn}_{L^2(\Gah^n)}.
    \end{aligned}
\end{equation*}
Combining the bounds of $\mathcal{E}_1$ and $\mathcal{E}_2$ into \eqref{eq: consist error inside err1 1 ENS} we obtain \eqref{eq: Consistency error 1 ENS}.

\noindent$\underline{\textsc{Err}^C_2},\underline{\textsc{Err}^C_2}:$ These estimates are straightforward if we simply apply \eqref{eq: Geometric perturbations btilde ENS} and \eqref{eq: errors of domain of integration data ENS} respectively.

\noindent $\underline{\textsc{Err}^C_3}:$ Adding and subtracting appropriate terms we see that 
\begin{equation}\label{eq: Consistency errors inside err3 1 ENS}
   \textsc{Err}^C_3 = c_h(\underline{\bfu}^{n-1}(t_n);\bfu^n,\vh) - c(\underline{\bfu}^{n-1}(t_n);\bfu^n,\vhl) + c(\underline{\bfu}^{n-1}(t_n);\bfu^n,\vhl) - c(\bfu^n;\bfu^n,\vhl)=: \mathrm{C}_1 + \mathrm{C}_2.
\end{equation}
First recall $c(\bullet;\bullet,\bullet)$ defined in \eqref{eq: c equiv expression ENS} (with $\eta_1,\eta_2=0$). Then, we see that using the perturbations in \Cref{lemma: errors of geometric pert ENS}, the geometric estimates $\norm{\bfP- \bfPh}_{L^\infty(\Gaht)} \leq ch^{k_g}$, $\norm{1-\muh}_{L^{\infty}(\Gaht)} \leq ch^{k_g+1}$ as seen in \cref{sec: Lifts to Exact Surface}, \eqref{eq: geometric errors 2 ENS}, and the time inequalities \eqref{eq: time equivalence norms ENS}, \eqref{eq: time ineq cont u ENS} we obtain
\begin{equation}
    \begin{aligned}\label{eq: Consistency errors inside err3 2 ENS}
        |\mathrm{C}_1| &= \Big|\frac{1}{2}\Big(\int_{\Gah^n}((\underline{\bfu}^{n-1}(t_n)\cdot\nbgh)\bfu^n)\cdot\vhn - \int_{\Ga^n}((\underline{\bfu}^{n-1}(t_n)\cdot\nbg)\bfu^n)\cdot(\vhn)^{\ell}\Big)\Big|\\
         &\qquad \qquad \qquad+\Big|\frac{1}{2}\Big(\int_{\Gah^n}((\underline{\bfu}^{n-1}(t_n)\cdot\nbgh)\vhn)\cdot\bfu^n -\int_{\Ga^n}((\underline{\bfu}^{n-1}(t_n)\cdot\nbg)(\vhn)^{\ell})\cdot\bfu^n\Big)\Big|\\
         &\leq  ch^{k_g}\norm{\bfu^{n-1}}_{H^1(\Ga^{n-1})}\norm{\bfu^n}_{H^1(\Ga^n)}\norm{\vhn}_{H^1(\Gah^n)}.
    \end{aligned}
\end{equation}
For the second estimate in \eqref{eq: Consistency errors inside err3 1 ENS} we have that 
\begin{equation*}
    \begin{aligned}
        |\mathrm{C}_2| =  \frac{1}{2}\Big|\int_{\Ga^n}\big( (\bfu^n-\underline{\bfu}^{n-1}(t_n))\cdot\nbg\bfu^n \big)\cdot(\vhn)^{\ell}\Big| +  \frac{1}{2}\Big|\int_{\Ga^n}\big((\bfu^n-\underline{\bfu}^{n-1}(t_n))\cdot\nbg\big)(\vhn)^{\ell})\cdot\bfu^n \Big|.
    \end{aligned}
\end{equation*}
From the fact that $\bfu^n = \underline{\bfu}^n(t_n) = \bfu(t_n)\circ\Phi^{\ell}_{t_n}\circ\Phi^{\ell}_{-t_n}$, and $\dfrac{d}{dt} \bfu(t)\circ\Phi^{\ell}_{t} = \big(\dfrac{d}{dt}\bfu(t) + (\nb\bfu)\TrVelL\big)\Phi^{\ell}_{t} = \matdl\bfu(t)\circ \Phi^{\ell}_{t}$ (see \eqref{eq: lifted discrete stong mat deriv ENS}) it follows that 
\begin{equation}
    \begin{aligned}\label{eq: un-un-1 formula ENS}
        \bfu^n-\underline{\bfu}^{n-1}(t_n) &= \bfu(t_n)\circ\Phi^{\ell}_{t_n}\circ\Phi^{\ell}_{-t_n}- \bfu(t_{n-1})\circ\Phi^{\ell}_{t_{n-1}}\circ\Phi^{\ell}_{-t_n}\\
        &=\int_{t_{n-1}}^{t_n} \frac{d}{dt} \bfu(s)\circ\Phi^{\ell}_{s} \,ds \circ \Phi^{\ell}_{-t_n} = \int_{t_{n-1}}^{t_n} \matdl\bfu\circ \Phi^{\ell}_{s}\,ds \circ \Phi^{\ell}_{-t_n},
    \end{aligned}
\end{equation}
and therefore a simple use of H\"older's inequality, the Sobolev embedding $L^4 \hookrightarrow H^1$, \eqref{eq: relate matd to matdl ENS} and, by this point, standard use of the Cauchy-Schwarz inequality, shows that 
\begin{equation}
    \begin{aligned}\label{eq: Consistency errors inside err3 3 ENS}
        |\mathrm{C}_2| \leq c \sqrt{\Delta t} \Big(\int_{t_{n-1}}^{t_n} \norm{\matn\bfu}_{H^1(\Gat)}^2 + h^{2\widehat{r}_u}\norm{\bfu}_{H^1(\Gat)}^2\Big)^{1/2}\norm{\bfu^n}_{H^1(\Ga^n)}\norm{\vhn}_{H^1(\Gah^n)},
    \end{aligned}
\end{equation}
in which we also used the norm equivalence \eqref{eq: norm equivalence ENS} where appropriate. Therefore inserting \eqref{eq: Consistency errors inside err3 2 ENS}, and \eqref{eq: Consistency errors inside err3 3 ENS} into \eqref{eq: Consistency errors inside err3 1 ENS} gives our assertion \eqref{eq: Consistency error 3 ENS}.

\noindent $\underline{\textsc{Err}^{C,cov}_3}:$ As above, adding and subtracting appropriate terms we see that 
\begin{align}\label{eq: Consistency errors inside err3 1 cov ENS}
    \textsc{Err}^{C,cov}_3 &= c_h^{cov}(\underline{\bfu}^{n-1}(t_n);\bfu^n,\vh) - c^{cov}(\underline{\bfu}^{n-1}(t_n);\bfu^n,\vhl) + c^{cov}(\underline{\bfu}^{n-1}(t_n);\bfu^n,\vhl) - c^{cov}(\bfu^n;\bfu^n,\vhl)\nonumber\\
    &=: \mathrm{C}_1^{cov} + \mathrm{C}_2^{cov},
    \end{align}
where we recall $c_h^{cov}$, $c^{cov}$ as defined in \cref{def: bilinear forms discrete ENS} and \eqref{eq: c equiv expression ENS} (with $\eta_1,\eta_2=0$) respectively. Now following \eqref{eq: Consistency errors inside err3 1 ENS} and using the fact that $\nbgcov \vhl = \nbgcov(\bfPg\vhl) + (\vhl\cdot\bfn) \bfH$ \eqref{eq: split cov ENS}, we clearly see that
    \begin{align}\label{eq: Consistency errors inside err3 2 cov ENS}
        &|\mathrm{C}_1^{cov}| \leq ch^{k_g}\norm{\bfu^{n-1}}_{H^1(\Ga^{n-1})}\norm{\bfu^n}_{H^1(\Ga^n)}\norm{\vhn}_{H^1(\Gah^n)} + \Big| -  \frac{1}{2}\int_{\Gahn} \bfu^n\cdot\nh \underline{\bfu}^{n-1}(t_n)\cdot\bfH_h\vhn \Big| \nonumber\\
        &\!\!\!\!+ \Big| \frac{1}{2}\int_{\Gahn} \vhn\cdot(\nh-\bfn) \underline{\bfu}^{n-1}(t_n)\cdot\bfH_h\bfu^n \Big| + \Big| -  \frac{1}{2}\int_{\Gahn} (\vhn\cdot\bfn) \underline{\bfu}^{n-1}(t_n)\cdot\bfH_h\bfu^n + \frac{1}{2}\int_{\Ga^n} ((\vhn)^{\ell}\cdot\bfng) \underline{\bfu}^{n-1}(t_n)\cdot\bfH\bfu^n \Big| \nonumber\\
        &\leq ch^{k_g}\norm{\bfu^{n-1}}_{H^1(\Ga^{n-1})}\norm{\bfu^n}_{H^1(\Ga^n)}\norm{\vhn}_{H^1(\Gah^n)},
    \end{align}
where we used the fact that $\bfu^n$ solves \eqref{weak lagrange hom NV cov ENS} and therefore $\bfu^n\cdot\nh = \bfu^n\cdot(\nh-\bfn)$ coupled with the geometric estimates $\norm{\bfn- \nh}_{L^\infty(\Gah)}\leq ch^{k_g}$ \eqref{eq: geometric errors 1 ENS}, the time-inequalities \eqref{eq: time equivalence norms ENS}, and \eqref{eq: time ineq cont u ENS}, the bound $\norm{\bfH_h}_{L^{\infty}(\Gah)}\leq c$ (see \eqref{eq: geometric errors 1 ENS}), and, lastly, for the last difference in \eqref{eq: Consistency errors inside err3 2 cov ENS} we use  \cref{lemma: weingarten map improved ENS} with $\beta = (\vh^n)^{\ell}\cdot\bfng$, and $\vhn = \underline{\bfu}^{n-1}(t_n)^{-\ell}\in \bfH^1(\Gahn)$.

Regarding $\mathrm{C}_2^{cov}$ we readily see that
\begin{equation*}
    \begin{aligned}
        |\mathrm{C}_2^{cov}| =  \frac{1}{2}\Big|\int_{\Ga^n}\big( (\bfu^n-\underline{\bfu}^{n-1}(t_n))\cdot\nbgcov\bfu^n \big)\cdot(\vhn)^{\ell}\Big| +  \frac{1}{2}\Big|\int_{\Ga^n}\big((\bfu^n-\underline{\bfu}^{n-1}(t_n))\cdot\nbgcov\big)\bfPg(\vhn)^{\ell})\cdot\bfu^n \Big|.
    \end{aligned}
\end{equation*}
Then, following exact calculations as for $\mathrm{C}_2$ beforehand, where we have replaced $\nbg$ with $\nbgcov$, along with the fact that $\norm{\nbgcov\bfPg\vhl}_{L^2(\Ga)} \leq c\norm{\bfPg\vhl}_{a} \leq c\norm{\vh}_{\ah}$ by the perturbation bound \eqref{eq: Geometric perturbations a ENS} and the $H^1$ coercivity bound \eqref{eq: korn inequality Ph ENS}, the following holds
\begin{equation}
    \begin{aligned}\label{eq: Consistency errors inside err3 3 cov ENS}
        |\mathrm{C}_2| \leq c \sqrt{\Delta t} \Big(\int_{t_{n-1}}^{t_n} \norm{\matn\bfu}_{H^1(\Gat)}^2 + h^{2(\widehat{r}_u)}\norm{\bfu}_{H^1(\Gat)}^2\Big)^{1/2}\norm{\bfu^n}_{H^1(\Ga^n)}\norm{\vhn}_{\ah}.
    \end{aligned}
\end{equation}
See \cite[Lemma 7.4]{elliott2025unsteady} where this exact calculations are performed in more detail for the stationary case. Combining then, \eqref{eq: Consistency errors inside err3 2 cov ENS}, \eqref{eq: Consistency errors inside err3 3 cov ENS} into \eqref{eq: Consistency errors inside err3 1 cov ENS} gives the final result.
\end{proof}

\begin{lemma}[Inertia errors]\label{lemma: Trilinear Errors ENS}
    The following bounds are satisfied, for all $\vh \in \bfV_h$ 
    \begin{align}\label{eq: Trilinear Errors general 1 ENS}
        |\mathcal{C}(\vhn)| &\leq c\big(\norm{\eu^{n-1}}_{L^2(\GahnN)}^{1/2}\norm{\eu^{n-1}}_{\ah}^{1/2} +  \norm{\uh^{n-1}}_{\ah}^{1/2}\norm{\eu^n}_{H^1(\Gah^n)}\big) \norm{\vhn}_{H^1(\Gah^n)},\\
        \label{eq: Trilinear Errors general 1 cov ENS}
        |\mathcal{C}^{cov}(\vhn)| &\leq c\big(\norm{\eu^{n-1}}_{L^2(\GahnN)}^{1/2}\norm{\eu^{n-1}}_{\ah}^{1/2} +  \norm{\uh^{n-1}}_{\ah}^{1/2}\norm{\eu^n}_{\ah} + \norm{\uh^{n-1}}_{\ah}^{1/2}\norm{\eu^n}_{L^2(\Gahn)}^{1/2}\norm{\eu^n}_{\ah}^{1/2}\big) \norm{\vhn}_{\ah}.
    \end{align}
In case where $\vhn = \thbfu^n$, considering the skew-symmetric nature of the inertia term, we have
    \begin{align}\label{eq: Trilinear Errors general ENS}
     &|\mathcal{C}(\thbfu^n)| \leq c( \norm{\rho_{\bfu}^{n-1}}_{L^2(\Gah^{n-1})}^{1/2}\norm{\rho_{\bfu}^{n-1}}_{\ah}^{1/2} + \norm{\theta_{\bfu}^{n-1}}_{L^2(\Gah^{n-1})}^{1/2}\norm{\theta_{\bfu}^{n-1}}_{\ah}^{1/2} + \norm{\uh^{n-1}}_{\ah}^{1/2}\norm{\rho_{\bfu}^n}_{H^{1}(\Gah^n)})\norm{\thbfu^n}_{H^1(\Gah^n)},\\
     \label{eq: Trilinear Errors general cov ENS}
     &|\mathcal{C}^{cov}(\thbfu^n)| \leq c( \norm{\rho_{\bfu}^{n-1}}_{L^2(\Gah^{n-1})}^{1/2}\norm{\rho_{\bfu}^{n-1}}_{\ah}^{1/2} + \norm{\theta_{\bfu}^{n-1}}_{L^2(\Gah^{n-1})}^{1/2}\norm{\theta_{\bfu}^{n-1}}_{\ah}^{1/2} + \norm{\uh^{n-1}}_{\ah}^{1/2}\norm{\rho_{\bfu}^n}_{\ah})\norm{\thbfu^n}_{\ah}. 
    \end{align}
\end{lemma}
\begin{proof}
Let us start with \eqref{eq: Trilinear Errors general 1 ENS}, then due to the skew-symmetric nature of the inertia term, as defined in \cref{def: bilinear forms discrete ENS},  \eqref{eq: Trilinear Errors general ENS} will be apparent. Recall that $\mathcal{C}(\vh) = - c_h(\underline{\eu^{n-1}}(\cdot,t_n);\bfu^n;\vhn) - c_h(\underline{\uhnN}(\cdot,t_n);\eu^n;\vhn),$ then by the continuity bound \eqref{ch boundedness ENS} and time inequalities \eqref{eq: time ineq eu ENS} and \eqref{eq: time equivalence norms ENS} we see that
\begin{equation*}
    \begin{aligned}
        c_h(\underline{\eu^{n-1}}(\cdot,t_n);\bfu^n;\vhn) \leq \norm{\eu^{n-1}}_{L^2(\GahnN)}^{1/2}\norm{\eu^{n-1}}_{\ah}^{1/2}\norm{\bfu^n}_{H^1(\Ga^n)}\norm{\vhn}_{H^{1}(\Gah^n)}
    \end{aligned}
\end{equation*}
and similarly, using the stability estimates for the discrete velocity \eqref{eq: velocity stab estimate ENS} as well, we obtain
    \begin{equation*}
        \begin{aligned}
             c_h(\underline{\uhnN}(\cdot,t_n);\eu^n;\vhn) \leq c\norm{\uh^{n-1}}_{\ah}^{1/2}\norm{\eu^n}_{H^{1}(\Gah^n)}\norm{\vhn}_{H^{1}(\Gah^n)},
        \end{aligned}
    \end{equation*}
thus proving the initial bound  \eqref{eq: Trilinear Errors general 1 ENS}. 

Now choosing $\vhn = \theta_{\bfu}^n \in \bfV_h^{n,div}$, recalling the decomposition \eqref{eq: decomposition error 1 Ritz ENS}, the above estimate and the skew-symmetric form of $c_h(\bullet;\bullet,\bullet)$ (see \cref{def: bilinear forms discrete ENS}), the final result  \eqref{eq: Trilinear Errors general ENS} indeed holds true: \vspace{-1mm}
\begin{equation*}
    \begin{aligned}
        |\mathcal{C}(\theta_{\bfu}^n)| &= |c_h(\underline{\eu^{n-1}}(\cdot,t_n);\bfu^n;\vhn) - c_h(\underline{\uhnN}(\cdot,t_n);\rho_{\bfu}^n;\theta_{\bfu}^n) -\overbrace{c_h(\underline{\uhnN}(\cdot,t_n);\theta_{\bfu}^n;\theta_{\bfu}^n)}^{=0}| \\
        &\leq \big(\norm{\rho_{\bfu}^{n-1}}_{L^2(\Gah^{n-1})}^{1/2}\norm{\rho_{\bfu}^{n-1}}_{\ah}^{1/2} + \norm{\theta_{\bfu}^{n-1}}_{L^2(\Gah^{n-1})}^{1/2}\norm{\theta_{\bfu}^{n-1}}_{\ah}^{1/2} \big)\norm{\bfu^n}_{H^1(\Ga^n)}\norm{\theta_{\bfu}^n}_{H^{1}(\Gah^n)}\\
        &\qquad + c\norm{\uh^{n-1}}_{\ah}^{1/2}\norm{\rho_{\bfu}^n}_{H^{1}(\Gah^n)}\norm{\thbfu^n}_{H^{1}(\Gah^n)}.
    \end{aligned}
\end{equation*}

Regarding $\mathcal{C}^{cov}(\cdot)$, it is clear that if we use the bound \eqref{ch cov boundedness ENS} instead, the analogous results follow.
\end{proof}

Now we are ready to show error estimates involving the velocity for either the \eqref{eq: fully discrete fin elem approx ENS} or the \eqref{eq: fully discrete fin elem approx cov ENS} scheme when $\underline{k_\lambda = k_u}$. 
\begin{lemma}\label{lemma: discrete remainder velocity estimates ENS}
Assume $\underline{k_\lambda = k_u}$, and that the regularity \cref{assumption: Regularity assumptions for velocity estimate ENS} hold. 
Let $(\bfu,\{p,\lambda\})$ be the solution of \eqref{weak lagrange hom NV dir ENS} or \eqref{weak lagrange hom NV cov ENS} and $(\uh^k,\{\ph^k,\lh^k\})$, $k=1,...,n$ be the discrete solutions of \eqref{eq: fully discrete fin elem approx ENS} or \eqref{eq: fully discrete fin elem approx cov ENS} respectively, with initial condition $\uh^0=\mathcal{R}_h\bfu^0$. Then, for sufficiently small time-step $\Delta t$, the following estimate holds for $1 \leq n \leq N$:
\begin{equation}
    \begin{aligned}\label{eq: discrete remainder velocity estimates ENS}
        \norm{\thbfu^n}_{L^2(\Gah^n)}^2 + \Delta t \sum_{k=1}^n \norm{\thbfu^k}_{\ah}^2 \leq C_{const} \big(\Delta t \big)^2 + C_{int}\big(h^{2\widehat{r}_u} + h^{2k_{pr}+2} +h^{2k_{\lambda}+2}\big),
    \end{aligned}
\end{equation}
\vspace{-4mm}
\begin{align}\label{eq: constant Cconst ENS}
   \text{with }\qquad  C_{cons} = c&\int_{0}^{T} \norm{\matn\matn\bfu}_{L^2(\Gat)}^2 +\norm{\matn\bfu}_{H^1(\Gat)}^2 \, dt,
    \\[-4pt]
    C_{int} = C_{stab} +  c\int_{0}^{T}&\norm{\matn{\bfu}}_{H^{k_u+1}(\Gat)}^2 \, dt + c\sup_{t\in[0,T]}\big\{\norm{\bfu}_{H^{k_{u}+1}(\Gat)}^2+ \norm{p}_{H^{k_{pr}+1}(\Gat)}^2 + \norm{\lambda}_{H^{k_{\lambda}+1}(\Gat)}^2 \big\} \nonumber,
\end{align}
and $\widehat{r}_u = min\{k_{u},k_g\}$, where $C_{stab}$ the constant in \eqref{eq: velocity stab estimate ENS}.\vspace{-1mm}
\end{lemma}
\begin{proof}
Testing the error equation \eqref{eq: error equation ENS} with $\vhn = \thbfu^n \in \bfV_h^{n,div}$ and recalling \eqref{eq: n to n-1 on surface n ENS} (see also calculations in \cref{Lemma: velocity stab estimate ENS}) we see that straightforward calculations and re-arrangements give
\begin{equation}
    \begin{aligned}\label{eq: m bilinear manipulation kl=ku ENS}
        &{\dfrac{1}{\Delta t}\Big(\bfm_h(\theta_{\bfu}^n,\theta_{\bfu}^n) - \bfm_h(\theta_{\bfu}^{n-1},\underline{\theta_{\bfu}^n}(t^{n-1}))\Big)} = {\dfrac{1}{2\Delta t}\Big(\bfm_h(\theta_{\bfu}^n,\theta_{\bfu}^n) - \bfm_h(\theta_{\bfu}^{n-1},\theta_{\bfu}^{n-1} )\Big)} \\
        &+ \underbrace{\frac{\Delta t}{2}\bfm_h(\underline{\matdt\theta_{\bfu}^n}(t^{n-1}),\underline{\matdt\theta_{\bfu}^n}(t^{n-1}))}_{\geq 0} + \frac{1}{2\Delta t } \underbrace{\Big(\bfm_h(\theta_{\bfu}^n,\theta_{\bfu}^n) - \bfm_h(\underline{\theta_{\bfu}^n}(t^{n-1}),\underline{\theta_{\bfu}^n}(t^{n-1}))\Big)}_{\leq \Delta t \norm{\theta_{\bfu}^n}_{L^2(\Gah^n)}^2 \text{ by } \eqref{eq: time differences properties bilinear m ENS}},
    \end{aligned}
\end{equation}\vspace{-3mm}

\noindent and therefore we derive
\begin{equation}
    \begin{aligned}\label{eq: discrete remainder velocity estimates inside ENS}
        \bfm_h(\theta_{\bfu}^n,\theta_{\bfu}^n) - \bfm_h(\theta_{\bfu}^{n-1},\theta_{\bfu}^{n-1} ) &+ 2\Delta t \, \ahhat(\theta_{\bfu}^n,\theta_{\bfu}^n)\\
        \qquad\qquad\qquad\leq 2\Delta t \norm{\theta_{\bfu}^n}_{L^2(\Gah^n)}^2 &+
        2\Delta t \Big(\sum_{i=1}^5 \textsc{Err}_{i}^{C}(\theta_{\bfu}^n) +  \sum_{i=1}^2 \textsc{Err}_{i}^{I}(\theta_{\bfu}^n) +  \mathcal{C}(\theta_{\bfu}^n)\Big),
    \end{aligned}
\end{equation}\vspace{-1mm}

\noindent where, since $\thbfu^n\in \bfV_h^{n,div}$, by the Ritz-Stokes projection defined in \eqref{eq: surface Ritz-Stokes projection ENS}, $\textsc{Err}^5_C(\theta_{\bfu}^n)=\mb(\bfu^n,(\theta_{\bfu}^n)^{\ell}) - \mh(\mathcal{R}_h\bfu^n,\theta_{\bfu}^n) = G_{\mb}(\bfu^n,\theta_{\bfu}^n) + \mh(\bfu^n-\mathcal{R}_h\bfu^n,\theta_{\bfu}^n) \leq ch^{\widehat{r}_u+1/2}\norm{\bfu}_{H^{k_u+1}(\Gahn)}\norm{\theta_{\bfu}^n}_{L^2(\Gahn)}$, due to the bounds \eqref{eq: errors of domain of integration data ENS} and \eqref{eq: Error Bounds Ritz-Stokes improved H1 ENS}, \eqref{eq: Error Bounds Ritz-Stokes L2 improved ENS}. Considering, now, the consistency bounds in \cref{lemma: Consistency errors ENS}, the fact that $\norm{\thbfu^n}_{H^1(\Gah)} \leq \norm{\thbfu^n}_{\ah}$, cf. \eqref{eq: improved h1-ah bound ENS}, and the Ritz-Stokes bound  $\norm{\rho_{\bfu}^n}_{L^2(\Gahn)}\leq ch^{\widehat{r}_u+1/2}\norm{\bfu}_{H^{k_u+1}(\Gahn)}$, cf. \eqref{eq: Error Bounds Ritz-Stokes improved H1 ENS}, \eqref{eq: Error Bounds Ritz-Stokes L2 improved ENS}, using Young's inequality appropriately we obtain
\begin{align}\label{eq: discrete remainder velocity estimates inside 1 ENS}
    |\sum_{i=1}^5 \textsc{Err}_{i}^{C}(\theta_{\bfu}^n)| &\leq c\Delta t\Big(\int_{t_{n-1}}^{t^n} \norm{\matn\matn\bfu}_{L^2(\Gat)}^2 +\norm{\matn\bfu}_{H^1(\Gat)}^2 \, dt\Big) + ch^{2k_g}\big(\norm{\bfu^{n-1}}_{H^1(\Ga^{n-1})}^2+ \norm{p^n}_{H^1(\Ga^n)}^2\nonumber\\
    & \!\!\!\!\!\!+ \norm{\lambda^n}_{L^2(\Ga^n)}^2 + \norm{\bff^n}_{L^2(\Ga^n)}^2\big)    +  \frac{ch^{2\widehat{r}_u+2}}{\Delta t} \Big(\int_{t_{n-1}}^{t^n} \norm{\matn\bfu}_{L^2(\Gat)}^2 + \norm{\bfu}_{H^1(\Gat)}^2 \, dt \Big) \\
    &\!\!\!\!\!\!+ ch^{2\widehat{r}_u+2}\!\!\!\!\!\!\! \sup_{t\in[t_{n-1},t_n]}\norm{\bfu}_{L^2(\Gat)}^2 +ch^{2\widehat{r}_u+1}\norm{\bfu^n}_{H^{k_u+1}(\Ga^n)}^2 + c\norm{\theta_{\bfu}^n}_{L^2(\Gah^n)}^2 + \frac{1}{5}\norm{\theta_{\bfu}^n}_{\ah}^2.\nonumber
\end{align}
Similarly, using instead the interpolation errors in \eqref{eq: Interpolation errors improved ENS} and \eqref{eq: Interpolation errors inside 2 ENS} we  find
\begin{align}\label{eq: discrete remainder velocity estimates inside 2 ENS}
    |\sum_{i=1}^2 \textsc{Err}_{i}^{I}(\theta_{\bfu}^n)| &\leq \big(\frac{h^{2\widehat{r}_u}+h^{2\widehat{r}_u+2}}{\Delta t}+ \Delta t h^{2\widehat{r}_u}\big)\int_{t_{n-1}}^{t^n}\sum_{j=0}^{1}\norm{(\matn)^j\bfu}_{H^{k_u+1}(\Gat)}^2\, dt  \nonumber \\
    &+ c(h^{2k_{pr}+2} + h^{2k_{\lambda}+2})\big( \norm{p^n}_{H^{k_{pr}+1}(\Ga^n)}^2  
        + \norm{\lambda^n}_{H^{k_{\lambda}+1}(\Ga^n)}^2\big)+ \frac{1}{5}\norm{\theta_{\bfu}^n}_{L^2(\Gah^n)}^2.  
\end{align}
Lastly, considering \eqref{eq: Trilinear Errors general ENS}, again with the help of the Ritz-Stokes bounds \eqref{eq: Error Bounds Ritz-Stokes improved ENS}, \eqref{eq: Error Bounds Ritz-Stokes improved H1 ENS} and simple application of Young's inequality we find
\begin{align}\label{eq: discrete remainder velocity estimates inside 3 ENS}
    |\mathcal{C}(\thbfu^n)| \leq c(1+\norm{\uh^{n-1}}_{\ah})h^{2\widehat{r}_u} + c\norm{\thbfu^{n-1}}_{L^2(\Gah^{n-1})}^2  + \frac{1}{5}\norm{\thbfu^{n-1}}_{\ah}^2 + \frac{1}{5}\norm{\thbfu^{n}}_{\ah}^2.
\end{align}
 Now, plugging \eqref{eq: discrete remainder velocity estimates inside 1 ENS}-\eqref{eq: discrete remainder velocity estimates inside 3 ENS} into \eqref{eq: discrete remainder velocity estimates inside ENS}, considering $\Delta t $ sufficiently small, applying kickback arguments where appropriate, and taking into account appropriate inclusions, e.g. $\int_{0}^{T} \norm{{\bfu}}_{H^{k_u+1}(\Gat)}^2\, dt \leq c(T)\sup_{t\in[0,T]} \{\norm{\bfu}_{H^{k_u+1}(\Gat)}\}$, we obtain upon summing for $k=1,...,n$ that
\begin{align*}
    \frac{1}{2}\norm{\thbfu^n}_{L^2(\Gah^n)}^2 + \frac{1}{5}\Delta t \sum_{k=1}^n \norm{\thbfu^k}_{\ah}^2 &\leq  c\Delta t  \sum_{k=1}^n \norm{\thbfu^{k}}_{L^2(\Gah^{k})}^2 + C_{cons} (\Delta t)^2 \\
    &+ C_{int}(h^{2\widehat{r}_u}+ h^{2k_{pr}+2} +h^{2k_{\lambda}+2}) + \big(\Delta t  \sum_{k=1}^n\norm{\uh^{k-1}}_{\ah}\big)\,h^{2\widehat{r}_u}, 
\end{align*}
with $ C_{cons}$ and $C_{int}$ as in \eqref{eq: constant Cconst ENS} and where we also used the fact that $\theta_{\bfu}^0=0$, by choice of initial condition $\uh^0 = \mathcal{R}_h\bfu^0$. Applying a discrete Gr\"onwall argument (assuming $\Delta t$ small enough), using the stability estimate $\Delta t \sum_{k=1}^n \norm{\uh^{n-1}}_{\ah} \leq C_{stab}$, where  $C_{stab}$ the constant in \eqref{eq: velocity stab estimate ENS},  and simple simplifications concerning the previously mentioned constants, yield our desired estimate.
\end{proof}

\begin{lemma}\label{lemma: discrete remainder velocity estimates kl=ku-1 ENS}
Assume $\underline{k_\lambda = k_u-1}$, $k_g \geq 2$ and that the regularity \cref{assumption: Regularity assumptions for velocity estimate ENS} hold. 
Let $(\bfu,\{p,\lambda\})$ be the solution of \eqref{weak lagrange hom NV cov ENS} and $(\uh^k,\{\ph^k,\lh^k\})$, $k=1,...,n$ be the discrete solutions of \eqref{eq: fully discrete fin elem approx cov ENS}, with initial condition $\uh^0=\mathcal{R}_h\bfu^0$. Then, for sufficiently small time-step $\Delta t$, the following estimate holds for $1 \leq n \leq N$:
\begin{equation}
    \begin{aligned}\label{eq: discrete remainder velocity estimates kl=ku-1 ENS}
        \norm{\thbfu^n}_{L^2(\Gah^n)}^2 + \Delta t \sum_{k=1}^n \norm{\thbfu^k}_{\ah}^2 \leq C_{const} \big(\Delta t \big)^2 + C_{int}\big(h^{2r_u} + h^{2k_{pr}+2} +h^{2k_{\lambda}+2}\big).
    \end{aligned}
\end{equation}
where $r_u = min\{k_{u},k_g-1\}$, with constants $C_{conc}$, $C_{int}$ as in \cref{lemma: discrete remainder velocity estimates ENS}. 
\end{lemma}
\begin{proof}
The proof follows almost identically as in \cref{lemma: discrete remainder velocity estimates ENS}, where one now makes use of the bounds \eqref{eq: Interpolation errors ENS} for $\textsc{Err}_1^{I}$ and \eqref{eq: Consistency error 3 cov ENS} for $ \textsc{Err}_3^{C,cov}$ instead of \eqref{eq: Interpolation errors improved ENS} and \eqref{eq: Consistency error 3 ENS} respectively, the worse Ritz-Stokes estimates \eqref{eq: Error Bounds Ritz-Stokes ENS},\eqref{eq: Error Bounds der Ritz-Stokes ENS} and lastly the inverse inequality  $\norm{\thbfu^n}_{H^1(\Gahn)} \leq h^{-1}\norm{\thbfu^n}_{L^2(\Gahn)}$ instead of the improved $H^1$ coercivity bound \eqref{lemma: improved h1-ah bound ENS}, where appropriate.
\end{proof}
\begin{remark}[About pressure stability]\label{remark: About pressure stability ENS}
Using the modified Ritz-Stokes projection \eqref{eq: surface Ritz-Stokes projection ENS}, we see with the help of the inverse inequality that
\begin{equation}\label{eq: inside stab uha ENS}
    \norm{\uh^n}_{\ah} \leq \norm{\mathcal{R}_h \bfu^n}_{\ah} + \norm{\thbfu^n}_{\ah}\leq  \norm{\mathcal{R}_h \bfu^n}_{\ah} + h^{-1}\norm{\thbfu^n}_{L^2(\Gah^n)}.
\end{equation}
Bearing in mind the \cref{assumption: Regularity assumptions for velocity estimate ENS}, the previously proven estimate \eqref{eq: discrete remainder velocity estimates ENS}, and the stability of the Ritz-Stokes map \eqref{eq: Stab estimates for Ritz-Stokes projection ENS}, we see that the following holds true:
\begin{align}
\label{eq: uniform bound uha stability ENS}
    \sup_n \norm{\uh^n}_{\ah} \leq C_a, \qquad \text{ if } \Delta t \leq c h.
\end{align}\vspace{-4mm}

\noindent This is exactly the condition \eqref{eq: uniform bound b stability ENS}, which we needed to prove the $L^2_{L^2}\times L^2_{H_h^{-1}}$ pressure stability in \cref{Lemma: Pressure stab Estimate ENS}. Notice, that the time-step restriction is reasonable since it will be automatically satisfied if we try to balance spatial and temporal discretization errors, as this leads to $\Delta t \sim h^{k_u}$.
\end{remark}

\begin{remark}[About $L^2_{L^2}\times L^2_{L^2}$ pressure stability for $k_\lambda-1$]\label{remark: About the pressure stability for arbitrary convergence ENS}
 As mentioned before in \cref{remark: About the pressure stability for arbitrary ENS}, to prove $L^2_{L^2}\times L^2_{L^2}$ pressure stability for $k_\lambda-1$, we need an $L^{\infty}$-uniform bound instead, cf. \eqref{eq: b extra stability ENS}. For this to hold, however, we need further regularity assumption for the velocity, namely $\bfu \in L^{\infty}_{W^{2,\infty}(\Gat)}$. To see this, by similar calculations as in  \eqref{eq: inside stab uha ENS}, where instead we use the $L^{\infty}$-bound of the  Ritz-Stokes projection  \eqref{eq: W1infty estimate Ritz-Stokes ENS}, we find that (assuming that $\Delta t \leq ch$ and $k_g \geq 2$)
 \begin{equation}
    \begin{aligned}\label{eq: Linfty uh bound ENS}
        \norm{\uh^n}_{L^{\infty}(\Gah^n)} \leq \norm{\mathcal{R}_h \bfu^n}_{L^{\infty}(\Gah^n)} + \norm{\thbfu^n}_{L^{\infty}(\Ga^n)}\leq  \norm{\bfu^n}_{W^{2,\infty}(\Ga^n)} + h^{-1}\norm{\thbfu^n}_{L^2(\Gah^n)},
    \end{aligned}
\end{equation}
and therefore $\sup_{n} \norm{\bfu_h^n}_{L^{\infty}(\Gahn)} \leq \norm{\bfu}_{ L^{\infty}_{W^{2,\infty}(\Gat)}} \!\!\!\!+ C_{int} + C_{cons} \leq C_a^{\infty}$, i.e. \eqref{eq: b extra stability ENS}. 
\end{remark}

Now, factoring in the above discrete remainder bounds \eqref{eq: discrete remainder velocity estimates ENS} and \eqref{eq: discrete remainder velocity estimates kl=ku-1 ENS}, the decomposition \eqref{eq: decomposition error 1 Ritz ENS}, the triangle's inequality and the Ritz-Stokes estimates \eqref{eq: Error Bounds Ritz-Stokes improved ENS}, the following \emph{main result} readily holds.
\begin{theorem}[Velocity error bounds]\label{theorem: Velocity Error Estimates ENS}
Under the Assumption  \ref{assumption: Regularity assumptions for velocity estimate ENS} and \Cref{lemma: discrete remainder velocity estimates ENS} the following velocity error estimates hold for $1 \leq n \leq N$
\begin{equation}\label{eq: Velocity Error Estimates ENS}
    \begin{aligned}
       \norm{\eu^n}_{L^2(\Gah^n)}^2 + \Delta t \sum_{k=1}^n \norm{\eu^k}_{\ah}^2 \leq C_{const} \big(\Delta t \big)^2 + C_{int}\big(h^{2\widehat{r}_u} + h^{2k_{pr}+2} +h^{2k_{\lambda}+2}\big), 
    \end{aligned}
\end{equation}
while under the assumptions of \Cref{lemma: discrete remainder velocity estimates kl=ku-1 ENS}, instead, we have the following for $1 \leq n \leq N$
\begin{equation}\label{eq: Velocity Error Estimates ENS}
    \begin{aligned}
       \norm{\eu^n}_{L^2(\Gah^n)}^2 + \Delta t \sum_{k=1}^n \norm{\eu^k}_{\ah}^2 \leq C_{const} \big(\Delta t \big)^2 + C_{int}\big(h^{2r_u} + h^{2k_{pr}+2} +h^{2k_{\lambda}+2}\big), 
    \end{aligned}
\end{equation}
where $\widehat{r}_u = min\{k_{u},k_g\}$, $r_u = min\{k_{u},k_g-1\}$, with $C_{const}, \, C_{int}$ the constants in \Cref{lemma: discrete remainder velocity estimates ENS}.
\end{theorem}

\subsection{Pressure a-priori error bounds for $\underline{k_\lambda = k_u}$}\label{sec: pressure a-priori kl=ku ENS}
We now analyze pressure results for $\underline{k_\lambda = k_u}$ and the fully discrete scheme \eqref{eq: fully discrete fin elem approx ENS} (similar results would hold for the alternative scheme \eqref{eq: fully discrete fin elem approx cov ENS}).
That means in our error equation \eqref{eq: error equation ENS} we use the estimates for $\textsc{Err}_3^{C}(\vhn)$ and $\mathcal{C}(\vhn)$.

To establish a-priori estimates for the two pressures, we shall make use of the ideas introduced in the stability analysis \cref{Sec: asssumptions about discrete scheme ENS}. That is, utilize the \emph{discrete Leray time-projection} \eqref{eq: discrete time Leray proj ENS} in combination with the \emph{discrete inverse Stokes operator} to $\mathcal{A}_h^n:\bfV_h^{n,div} \to \bfV_h^{n,div}$ \eqref{eq: Discrete inverse Stokes ENS}, to first establish an auxiliary bound on the \emph{time-projected} fully discrete time derivative in a negative energy norm, and then by the duality estimate in \cref{lemma: dual estimate ENS} and the discrete $L^2\times H_h^{-1}$ \textsc{inf-sup} condition \eqref{eq: L^2 H^{-1} discrete inf-sup condition Gah Lagrange ENS} provide optimal error bounds for the two pressures (albeit in a weaker norm for $\lh$). 

\begin{lemma}[Auxiliary result]\label{lemma: Auxiliary convergence result ENS}
Assume $\underline{k_\lambda = k_u}$, and let the regularity \cref{assumption: Regularity assumptions for velocity estimate ENS} hold with time-step $\Delta t \leq ch$. Let $(\bfu,\{p,\lambda\})$ be the solution of \eqref{weak lagrange hom NV dir ENS}  and $(\uh^k,\{\ph^k,\lh^k\})$, $k=1,...,n$ be the discrete solutions of \eqref{eq: fully discrete fin elem approx ENS}, with initial condition $\uh^0=\mathcal{R}_h\bfu^0$. Then, recalling the bound \eqref{eq: uniform bound uha stability ENS}, the following estimate holds for $1 \leq n \leq N$
    \begin{equation}\label{eq: Auxiliary convergence result imrpoved ENS}
        \sum_{k=1}^n \Delta t \norm{\frac{\mathcal{A}_h^k(\thbfu^k - \hat{\theta}_{\bfu}^{k-1})}{\Delta t}}_{\ah}^2 \leq C_{const} \big(\Delta t \big)^2 + C_{int}\big(h^{2\widehat{r}_u} + h^{2k_{pr}+2} +h^{2k_{\lambda}+2}\big) ,
    \end{equation}
where $\widehat{r}_u = min\{k_{u},k_g\}$, with $C_{const}, \, C_{int}$ the constants in \Cref{lemma: discrete remainder velocity estimates ENS}.
\end{lemma}

\begin{proof}
First, recall the \emph{discrete Leray time-projection} defined in \eqref{eq: discrete time Leray proj ENS}, then by \eqref{eq: Discrete inverse Stokes ENS} it is clear that $\mathcal{A}_h^n(\thbfu^n - \hat{\theta}_{\bfu}^{n-1})\in \bfV_h^{n,div}$ is well-posed (since $\thbfu^n - \hat{\theta}_{\bfu}^{n-1}\in \bfV_h^{n,div}$ for all $n\geq 0$ by \eqref{eq: discrete time Leray proj ENS}). So, testing \eqref{eq: error equation ENS} with $\vhn = \mathcal{A}_h^n(\thbfu^n - \hat{\theta}_{\bfu}^{n-1})\in \bfV_h^{n,div}$, and applying  \eqref{eq: discrete time Leray proj ENS}, yields for $1 \leq n \leq N$
\begin{align}
    \label{eq: Auxiliary convergence result inside ENS}
         \!\!\!\big(\frac{\thbfu^n - \hat{\theta}_{\bfu}^{n-1}}{\Delta t},\mathcal{A}_h^n(\thbfu^n - \hat{\theta}_{\bfu}^{n-1})\big)_{L^2(\Gah^n)} + \ah(\uh^n,\mathcal{A}_h^n(\thbfu^n - \hat{\theta}_{\bfu}^{n-1})) \leq 2\Big(|\sum_{i=1}^5 \textsc{Err}_{i}^{C}(\mathcal{A}_h^n(\thbfu^n - \hat{\theta}_{\bfu}^{n-1}))| \nonumber \\[-5pt]
         \!\!\! +  |\sum_{i=1}^2 \textsc{Err}_{i}^{I}(\mathcal{A}_h^n(\thbfu^n - \hat{\theta}_{\bfu}^{n-1}))| +  |\mathcal{C}(\mathcal{A}_h^n(\thbfu^n - \hat{\theta}_{\bfu}^{n-1}))|\Big), 
\end{align}
where we again we use that by the Ritz-Stokes projection defined in \eqref{eq: surface Ritz-Stokes projection ENS} $|\textsc{Err}_{5}^{C}(\mathcal{A}_{h}(\thbfu^n - \hat{\theta}_{\bfu}^{n-1}))|\leq h^{\widehat{r}_u+1/2}\norm{\bfu^n}_{H^{k_u+1}(\Ga^n)}\norm{\mathcal{A}_{h}(\thbfu^n - \hat{\theta}_{\bfu}^{n-1})}_{L^2(\Gahn)}$ (see below \eqref{eq: discrete remainder velocity estimates inside ENS} in \cref{lemma: discrete remainder velocity estimates ENS}). Now, we bound each term appropriately as in \cref{Lemma: auxilary stab bounds ENS} with the help \emph{of Young's inequality} in the following manner: \vspace{2mm}

$\bullet \  $ By the definition of $\mathcal{A}_h^n(\cdot)$ \eqref{eq: Discrete inverse Stokes ENS} we obtain 
\begin{align*}
    \ \big(\frac{\thbfu^n - \hat{\theta}_{\bfu}^{n-1}}{\Delta t},\mathcal{A}_h^n(\thbfu^n - \hat{\theta}_{\bfu}^{n-1})\big)_{L^2(\Gah^n)} =\frac{1}{\Delta t }\norm{\mathcal{A}_h^n(\thbfu^n - \hat{\theta}_{\bfu}^{n-1})}^2_{\ah}.
\end{align*}

$\bullet \  $ The second term in \eqref{eq: Auxiliary convergence result inside ENS} can be bounded as
\begin{align*}
    \ahhat(\thbfu^n,\mathcal{A}_h^n(\thbfu^n - \hat{\theta}_{\bfu}^{n-1}))\leq  c\Delta t \norm{\thbfu^n}_{\ah}^2 + \frac{1}{5\Delta t }\norm{\mathcal{A}_h^n(\thbfu^n - \hat{\theta}_{\bfu}^{n-1})}_{\ah}^2.
\end{align*}

$\bullet \  $ Using \Cref{lemma: Interpolation errors ENS} (similar to \eqref{eq: discrete remainder velocity estimates inside 2 ENS}) we obtain the following bound for the interpolation error 
\begin{align*}
    &|\sum_{i=1}^2 \textsc{Err}_{i}^{I}(\mathcal{A}_h^n(\thbfu^n - \hat{\theta}_{\bfu}^{n-1}))| \leq c\Delta t \big(\frac{h^{2\widehat{r}_u}+h^{2\widehat{r}_u+2}}{\Delta t}+ \Delta t h^{2\widehat{r}_u}\big)\int_{t_{n-1}}^{t^n}\norm{\matn{\bfu}}_{H^{k_u+1}(\Gat)}^2 \, dt \\
    &+ c\Delta t (h^{2k_{pr}+2} + h^{2k_{\lambda}+2}) \big(\norm{\bfu}_{H^{k_{u}+1}(\Ga^n)}^2+ \norm{p}_{H^{k_{pr}+1}(\Ga^n)}^2 + \norm{\lambda}_{H^{k_{\lambda}+1}(\Ga^n)}^2 \big) + \frac{1}{5\Delta t }\norm{\mathcal{A}_h^n(\thbfu^n - \hat{\theta}_{\bfu}^{n-1})}^2_{\ah}.
\end{align*}

$\bullet \  $ For the consistency error by \Cref{lemma: Consistency errors ENS} (similar to \eqref{eq: discrete remainder velocity estimates inside 1 ENS}) and the  $H^1$-coercivity \eqref{eq: improved h1-ah bound ENS} we obtain 
\begin{align*}
    &|\sum_{i=1}^5 \textsc{Err}_{i}^{C}(\mathcal{A}_h^n(\thbfu^n - \hat{\theta}_{\bfu}^{n-1}))| \leq c(\Delta t)^2\int_{t_{n-1}}^{t^n} \norm{\matn\matn\bfu}_{L^2(\Gat)}^2 +\norm{\matn\bfu}_{H^1(\Gat)}^2  \, dt +  c\Delta th^{2\widehat{r}_u+1}\norm{\bfu^n}_{H^{k_u+1}(\Ga^n)}  \\
     & \qquad\qquad +c\Delta t \norm{\theta_{\bfu}^n}_{L^2(\Gah^n)}^2+c\Delta th^{2k_g}\big(\norm{\bff^n}_{L^2(\Ga^n)}^2+ \norm{p^n}_{H^{1}(\Ga^n)}^2 + \norm{\lambda^n}_{L^2(\Ga^n)}^2\big) + \frac{1}{5\Delta t }\norm{\mathcal{A}_h^n(\thbfu^n - \hat{\theta}_{\bfu}^{n-1})}^2_{\ah}.
\end{align*}

$\bullet \  $ Finally, recalling the uniform bound \eqref{eq: inside stab uha ENS}, we can calculate the inertia error term using \eqref{eq: Trilinear Errors general 1 ENS} $\phantom{aaaa}$ in  \Cref{lemma: Trilinear Errors ENS}, and the $H^1$-coercivity \eqref{eq: improved h1-ah bound ENS}, to find 
\begin{align*}
    |\mathcal{C}(\mathcal{A}_h^n(\thbfu^n - \hat{\theta}_{\bfu}^{n-1}))| \leq  \Delta t \norm{\eu^{n-1}}_{L^2(\Gah^{n-1})}\norm{\eu^{n-1}}_{\ah} + C_a\Delta t \norm{\eu^n}_{H^1(\Gahn)}^2 + \frac{1}{5\Delta t }\norm{\mathcal{A}_h(\thbfu^n - \hat{\theta}_{\bfu}^{n-1})}_{\ah}^2,
\end{align*}
$\phantom{aaaa}$ where $C_a$ the constant in \eqref{eq: uniform bound uha stability ENS}. Notice, also  that by the Ritz-Stokes estimate \eqref{eq: Error Bounds Ritz-Stokes improved H1 ENS} and \eqref{eq: improved h1-ah bound ENS}
$\phantom{aaaa}\norm{\eu^n}_{H^1(\Gahn)} \leq  
 \norm{\rho_{\bfu}^n}_{H^1(\Gahn)} + \norm{\thbfu^n}_{H^1(\Gahn)} \leq $  $ch^{\widehat{r}_u}\norm{\bfu^n}_{H^{k_u+1}(\Ga^n)} + \norm{\thbfu^n}_{\ah}$.\\

\noindent So, combining the above into \eqref{eq: Auxiliary convergence result inside ENS} and applying a kickback argument against $\frac{1}{\Delta t }\norm{\mathcal{A}_h^n(\thbfu^n - \hat{\theta}_{\bfu}^{n-1})}^2_{\ah}$ upon summing over $k=1,...,n$ and using the velocity estimates  \eqref{eq: discrete remainder velocity estimates ENS}, \eqref{eq: Velocity Error Estimates ENS}, completes the proof.
\end{proof}

\begin{lemma}\label{lemma: discrete remainder pressure estimates improved ENS}
     Assume $\underline{k_\lambda = k_u}$, and let the regularity \cref{assumption: Regularity assumptions for velocity estimate ENS} hold with time-step $\Delta t \leq ch$. Let $(\bfu,\{p,\lambda\})$ be the solution of \eqref{weak lagrange hom NV dir ENS}  and $(\uh^k,\{\ph^k,\lh^k\})$, $k=1,...,n$ be the discrete solutions of \eqref{eq: fully discrete fin elem approx ENS}, with initial condition $\uh^0=\mathcal{R}_h\bfu^0$. Then, recalling \eqref{eq: uniform bound uha stability ENS}, the following holds for $1 \leq n \leq N$\vspace{-1mm}
     \begin{equation}\label{eq: Pressure stab Estimate ENS}
            \Delta t \sum_{k=1}^{n} \norm{\{\theta_p^k,\theta_{\lambda}^k\}}_{L^2(\Gah^k)\times H_h^{-1}(\Gah^k)}^2 \leq  C_{const} \big(\Delta t \big)^2 + C_{int}\big(h^{2\widehat{r}_u} + h^{2k_{pr}+2} +h^{2k_{\lambda}+2}\big) ,
    \end{equation}\vspace{-2mm}
    
\noindent where $\widehat{r}_u = min\{k_{u},k_g\}$, with $C_{const}, \, C_{int}$ the constants in \Cref{lemma: discrete remainder velocity estimates ENS}.
 \end{lemma}

\begin{proof}
First, let us recall the discrete $L^2\times H_h^{-1}$ \textsc{inf-sup} condition \eqref{eq: L^2 H^{-1} discrete inf-sup condition Gah Lagrange ENS} to see that
   \begin{equation}
        \begin{aligned}\label{eq: inside conv press inf-sup ENS}
            \norm{\{\theta_p^n,\theta_{\lambda}^n\}}_{L^2(\Gah^n)\times H_h^{-1}(\Gah^n)} \leq \sup_{\vh \in \bfV_h^n} \frac{\bhtil(\vhn,\{\theta_p^n,\theta_{\lambda}^n\})}{\norm{\vhn}_{H^{1}(\Gah)}}.
        \end{aligned}
    \end{equation} 
Solving \eqref{eq: error equation ENS} for $\bhtil(\vh,\{\theta_p^n,\theta_{\lambda}^n\})$ with $\vhn \in \bfV_h^n$, and using the Ritz-Stokes projection defined in \eqref{eq: surface Ritz-Stokes projection ENS} to obtain $|\textsc{Err}_{5}^{C}(\vhn)| \leq |b_h^{L}(\vhn,\mathcal{P}_h(\bfu),\mathcal{L}_h(\bfu))| + ch^{\widehat{r}_u+1/2}\norm{\bfu^n}_{H^{k_u+1}(\Ga^n)}\norm{\vhn}_{L^2(\Gah^n)}$ (see calculations below \eqref{eq: discrete remainder velocity estimates inside ENS} in \cref{lemma: discrete remainder velocity estimates ENS} again), we find that
    \begin{align}\label{eq: inside conv press ENS}
        &\bhtil(\vhn,\{\theta_p^n,\theta_{\lambda}^n\}) \leq \Big|{\dfrac{1}{\Delta t}\bfm_h(\theta_{\bfu}^n,\vhn) - \bfm_h(\theta_{\bfu}^{n-1},\underline{\vhn}(\cdot,t_{n-1}))}\big| + \ahhat(\theta_{\bfu}^n,\vhn) +|b_h^{L}(\vhn,\mathcal{P}_h(\bfu),\mathcal{L}_h(\bfu))| \nonumber\\
        &+ ch^{\widehat{r}_u+1/2}\norm{\bfu^n}_{H^{k_u+1}(\Ga^n)}\norm{\vhn}_{L^2(\Gah^n)} + \Big(\sum_{i=1}^4 |\textsc{Err}_{i}^{C}(\vhn)| +  \sum_{i=1}^2 |\textsc{Err}_{i}^{I}(\vhn)| +  |\mathcal{C}(\vhn)|\Big). 
    \end{align}
We focus on the first and third term of the above eq. \eqref{eq: inside conv press ENS}, since the rest of the terms have already been handled before (see \cref{lemma: Interpolation errors ENS,lemma: Consistency errors ENS,lemma: Trilinear Errors ENS} and \cref{theorem: Velocity Error Estimates ENS}). As in the stability results \cref{Lemma: Pressure stab Estimate ENS} using the dual estimate \eqref{eq: dual estimate ENS} we arrive at 
\begin{align}\label{eq: inside conv press 2 ENS}
    \sup_{\vh^n\in\bfV_h^n} \dfrac{\frac{1}{\Delta t }\Big(\bfm_h(\theta_{\bfu}^n,\vhn) - \bfm_h(\theta_{\bfu}^{n-1},\underline{\vhn}(\cdot,t_{n-1}))\Big)}{\norm{\vhn}_{H^1(\Gah^n)}} \leq c\norm{\frac{\mathcal{A}_h^n(\thbfu^n - \hat{\theta}_{\bfu}^{n-1})}{\Delta t}}_{\ah} + c\norm{\thbfu^{n-1}}_{\ah}.
\end{align}
On the other hand, using the Ritz-Stokes bound \eqref{eq: Error Bounds Ritz-Stokes improved ENS}, the super-approximation property \eqref{eq: super-approximation estimate 2 ENS} and the interpolation estimates \eqref{eq: bounds of Scott-Zhang interpolant ENS} for $\Ihz$ and the $H_h^{-1}$-norm definition \eqref{eq: H^-1h definition ENS}, yields
\begin{equation}
    \begin{aligned}\label{eq: inside conv press 3 ENS}
        &|b_h^{L}(\vhn,\mathcal{P}_h(\bfu),\mathcal{L}_h(\bfu))| \leq \norm{\vhn}_{\ah}\norm{\mathcal{P}_h(\bfu^n)}_{L^2(\Gah^n)} + \big|m_h(\vhn\cdot(\nh-\bfn),\mathcal{L}_h(\bfu^n))\big| \\
    &\qquad\qquad\qquad\qquad\qquad + \big|m_h(\Ihz(\vhn\cdot\bfn),\mathcal{L}_h(\bfu^n)) \big|+ \big|m_h(\vhn\cdot\bfn - \Ihz(\vhn\cdot\bfn),\mathcal{L}_h(\bfu^n))\big|\\
    &\leq \norm{\vhn}_{\ah}\norm{\mathcal{P}_h(\bfu^n)}_{L^2(\Gah^n)} + ch\norm{\vhn}_{L^2(\Gah^n)}\norm{\mathcal{L}_h(\bfu^n)}_{L^2(\Gahn)} + \norm{\vhn}_{H^1(\Gah^n)}\norm{\mathcal{L}_h(\bfu^n)}_{H_h^{-1}(\Gah^n)} \\
    &\leq ch^{\widehat{r}_u}\norm{\bfu^n}_{H^{k_u+1}(\Ga^n)}\norm{\vhn}_{H^1(\Gahn)}.
    \end{aligned}
\end{equation}

So, plugging first \eqref{eq: inside conv press ENS} into the \textsc{inf-sup} condition \eqref{eq: inside conv press inf-sup ENS}, using \eqref{eq: inside conv press 2 ENS}, \eqref{eq: inside conv press 3 ENS} and the already proven bounds in \cref{lemma: Auxiliary convergence result ENS}, \cref{lemma: Interpolation errors ENS}, \cref{lemma: Consistency errors ENS}, \cref{lemma: Trilinear Errors ENS} $\big($specifically for \eqref{eq: Trilinear Errors general 1 ENS}, by the Ritz-Stokes estimate \eqref{eq: Error Bounds Ritz-Stokes improved H1 ENS} we have $\norm{\eu^n}_{H^1(\Gahn)} \leq
 \norm{\rho_{\bfu}^n}_{H^1(\Gahn)} + \norm{\thbfu^n}_{H^1(\Gahn)} \leq $  $h^{\widehat{r}_u}\norm{\bfu^n}_{H^{k_u+1}(\Ga^n)} + \norm{\thbfu^n}_{\ah}\!\big)$, applying the Cauchy-Schwarz inequality and factor out $\norm{\vhn}_{H^1(\Gah^n)}$ (since $\norm{\vhn}_{L^2(\Gah^n)} \leq \norm{\vhn}_{\ah} \leq \norm{\vhn}_{H^1(\Gah^n)}$), after some tedious (but already carried out earlier) calculations we derive
\begin{equation*}
    \begin{aligned}
        &\norm{\{\theta_p^n,\theta_{\lambda}^n\}}_{L^2(\Gah^n)\times H_h^{-1}(\Gah^n)} \leq \norm{\frac{\mathcal{A}_h^n(\thbfu^n - \hat{\theta}_{\bfu}^{n-1})}{\Delta t}}_{\ah} + c\norm{\thbfu^{n-1}}_{\ah} + c\norm{\thbfu^{n}}_{\ah} + ch^{\widehat{r}_u+1/2}\norm{\bfu^n}_{H^{k_u+1}(\Ga^n)} \\
        &\quad + c\sqrt{\Delta t}\Big(\int_{t_{n-1}}^{t^n} \norm{\matn\matn\bfu}_{L^2(\Gat)}^2 +\norm{\matn\bfu}_{H^1(\Gat)}^2  \, dt\Big)^{1/2} + h^{k_g}\big(\norm{\bff^n}_{L^2(\Ga^n)}+ \norm{\bfu^n}_{H^{1}(\Ga^n)}  + \norm{p^n}_{H^{1}(\Ga^n)}\\
        &\quad + \norm{\lambda^n}_{L^2(\Ga^n)}\big) + \norm{\theta_{\bfu}^n}_{L^2(\Gah^n)} + \dfrac{c}{\sqrt{\Delta t}}\Big(\int_{t_{n-1}}^{t^n}\sum_{j=0}^{1}\norm{(\matd)^j\rho_{\bfu}\t}_{H^{k_u+1}(\Gat)}^2 \, dt\Big)^{1/2} + c\big(h^{k_{pr}+1} \\
        &\quad + h^{k_{\lambda}+1}\big)\big( \norm{p^n}_{H^{k_{pr}+1}(\Ga^n)} + \norm{\lambda^n}_{H^{k_{\lambda}+1}(\Ga^n)}\big) + c\norm{\eu^{n-1}}_{L^2(\GahnN)} + c\norm{\eu^{n-1}}_{\ah} +  C_a h^{\widehat{r}_u}\norm{\bfu^n}_{H^{k_u+1}(\Ga^n)}.
    \end{aligned}
\end{equation*}
Then, squaring and applying the operator $\Delta t \sum_{k=1}^n$ and using the Ritz-Stokes estimates \eqref{eq: Error Bounds Ritz-Stokes L2 improved ENS}, \eqref{eq: Error Bounds Ritz-Stokes improved H1 ENS}, \eqref{eq: Error Bounds der Ritz-Stokes L2 ENS}, along with the auxiliary result \eqref{eq: Auxiliary convergence result imrpoved ENS} in \cref{lemma: Auxiliary convergence result ENS} and the velocity estimates in \cref{lemma: discrete remainder velocity estimates ENS} and \cref{theorem: Velocity Error Estimates ENS} we obtain our desired result.
\end{proof}

Now, by the decompositions \eqref{eq: decomposition error 2 lagrange ENS}, \eqref{eq: decomposition error 3 lagrange ENS}, the interpolation bounds \eqref{eq: interpolation errors pressures ENS}, and a simple application of the triangle inequality, one may then derive the following \emph{main result} regarding the pressure errors.
\begin{theorem}[Pressure error bounds] \label{theorem: Pressures Error Estimate ENS}
Under the regularity assumption \cref{assumption: Regularity assumptions for velocity estimate ENS} and \Cref{lemma: discrete remainder pressure estimates improved ENS} the following pressure error estimates holds for $1 \leq n \leq N$
     \begin{equation}\label{eq: Pressure stab Estimate ENS}
            \Delta t \sum_{k=1}^{n} \norm{\{\ep^k,\el^k\}}_{L^2(\Gah^k)\times H_h^{-1}(\Gah^k)}^2 \leq  C_{const} \big(\Delta t \big)^2 + C_{int}\big(h^{2\widehat{r}_u} + h^{2k_{pr}+2} +h^{2k_{\lambda}+2}\big) ,
    \end{equation}
where $\widehat{r}_u = min\{k_{u},k_g\}$, with $C_{const}, \, C_{int}$ the constants in \Cref{lemma: discrete remainder velocity estimates ENS}.
 \end{theorem}

\subsection{Pressure a-priori error bounds for $\underline{k_\lambda = k_u-1}$}\label{sec: Pressure a-priori estimates for kl= ku-1 ENS}

Now we seek to prove a-priori pressure estimates when $k_\lambda=k_u-1$. As mentioned in \cref{sec: The Fully Discrete Scheme ENS} and more specifically in \cref{remark: diff formulations reason ENS} we now consider the continuous weak formulation \eqref{weak lagrange hom NV cov ENS} and its fully discrete finite element approximation \eqref{eq: fully discrete fin elem approx cov ENS}.  

There is one main difference between this case and \cref{sec: pressure a-priori kl=ku ENS}, namely the coercivity estimate in  \cref{lemma: improved h1-ah bound ENS} and the dual estimate in \cref{lemma: dual estimate ENS} do not hold for $k_\lambda=k_u-1$; see \cite{elliott2025unsteady} for the stationary case. As a result, it is not possible to follow the steps in \cref{sec: pressure a-priori kl=ku ENS}.
So, as mentioned in \cref{remark: About the pressure stability for arbitrary ENS}, we need to find an error bound for the $L^2_{L^2}$-norm of the fully discrete time derivative, that is, the expression
\begin{equation}\label{eq: fully discrete L2 deriv bound error ENS}
    \Delta t \sum_{k=1}^n\mh(\matdt\thbfu^n,\matdt\thbfu^n)= \Delta t \sum_{k=1}^n \frac{\mh(\thbfu^n-\underline{\thbfu^{n-1}}(t_n),\thbfu^n-\underline{\thbfu^{n-1}}(t_n))}{(\Delta t)^2},
\end{equation}
where we recall $\thbfu^n = \underline{\thbfu^{n-1}}(\cdot ,t_n) + \Delta t \, \matdt \thbfu^n$ by \eqref{eq: n to n-1 on surface n ENS}. To achieve this, one could ``just'' test the error eq. \eqref{eq: error equation ENS} with $\vhn = \thbfu^n- \underline{\thbfu^{n-1}}(t_n) \in \bfV_h^n$, since we can rewrite the time-derivative approximation as 
\begin{equation}\label{eq: m bilinear manipulation kl=ku-1 ENS}
  \mh(\thbfu^n,\vhn) - \bfm_h(\theta_{\bfu}^{n-1},\underline{\vhn}(\cdot,t^{n-1})) = \mh(\thbfu^n-\underline{\theta_{\bfu}^{n-1}}(t_n),\vhn) + \big[\mh(\underline{\theta_{\bfu}^{n-1}}(t_n),\vhn) - \bfm_h(\theta_{\bfu}^{n-1},\underline{\vhn}(\cdot,t^{n-1}))\big],
\end{equation}
but it would not be possible to find optimal convergence rates. Indeed, the test functions $\vhn$ in the error bounds 
 \eqref{eq: Interpolation error 2 ENS}, \eqref{eq: Consistency error 3 cov ENS}, \eqref{eq: Trilinear Errors general 1 cov ENS} are with respect to the energy $(a_h)$-norm \eqref{eq: energy norm ENS}, which cannot be controlled (e.g. using Young's inequality and a kickback argument) by \eqref{eq: fully discrete L2 deriv bound error ENS}, unless an inverse inequality is applied, hence losing an order of convergence. To alleviate this issue, we must utilize test functions w.r.t. the $L^2$-norm.
 This is accomplished through a combination of two factors:

 i) We use the standard Ritz-Stokes projection $(\mathcal{R}_h^b(\bfu^n),\{\mathcal{P}_h^b(\bfu^n),\mathcal{L}_h^b(\bfu^n)\})$ defined in \eqref{eq: surface Ritz-Stokes projection std ENS} instead. With this adjustment, we manage to remove one of the aforementioned problematic bounds \eqref{eq: Interpolation error 2 ENS} involving the bilinear form $\bhtil$. To make things simpler, we will also assume \emph{new initial conditions} for the discrete scheme \eqref{eq: fully discrete fin elem approx cov ENS}, mainly that $\uh^0 = \mathcal{R}_h^b\bfu^0= \mathcal{R}_h(\bfu^0,\{p^0,\lambda^0\})$. This is done so that at time-step  $t_0=0$ the \emph{discrete remainder} defined below in \eqref{eq: decomposition error 1 Ritz 2 ENS} remains zero, 
  $\sigma_{\bfu}^0=0$, and so the calculations follow identical to \cref{sec: velocity a-priori estimates ENS} for $k_\lambda=k_u-1$, see also \cref{remark: init cond and stab ENS}  below.

 ii) We use the integration by part formula \eqref{eq: integration by parts ENS} for the inertia terms \eqref{eq: Consistency error 3 cov ENS}, \eqref{eq: Trilinear Errors general 1 cov ENS}, so that test functions are bounded w.r.t. the $L^2$-norm instead.

 All of the above changes come at the expense of slightly higher regularity assumptions,  compared to \cref{sec: pressure a-priori kl=ku ENS}, which are the following:

\begin{assumption}[Regularity II]\label{assumption: Regularity assumptions for velocity estimate 2 ENS}
We assume that the solution $(\bfu,\{p,\lambda\})$ has the following regularity   
\begin{equation*}
    \begin{aligned}
   \sup_{t\in[0,T]}\big\{ \sum_{j=0}^1\norm{(\matn)^{j}\bfu}_{H^{k_{u}+1}(\Gat)}^2 &+ \norm{(\matn)^{j} p}_{H^{k_{pr}+1}(\Gat)}^2 + \norm{(\matn)^{j} \lambda}_{H^{k_{\lambda}+1}(\Gat)}^2 \big\}\\[-5pt]
   &+\sup_{t\in[0,T]}\norm{\bfu}_{W^{2,\infty}(\Ga\t)}^2+ \int_{0}^{T} \norm{\matn\matn\bfu}_{L^2(\Gat)}^2 \, dt \leq C.
    \end{aligned}
\end{equation*}
\end{assumption}

 Considering the standard Ritz-Stokes projection \eqref{eq: Error Bounds Ritz-Stokes std ENS} we use the following new decompositions:
\begin{align}\label{eq: decomposition error 1 Ritz 2 ENS}
    \eu^{n} = \bfu^n - \uh^n &= \underbrace{(\bfu^n- \mathcal{R}_h^b(\bfu^n)) }_{\text{Interpolation error}} + \underbrace{(\mathcal{R}_h^b(\bfu^n) - \uh^n)}_{\text{discrete remainder}} := \eta_{\bfu}^n + \sigma_{\bfu}^n , \\
    \label{eq: decomposition error 2 Ritz 2 ENS}
     \{\ep^n, \el^n\} = \{p^n,\lambda^n\} - \{\ph^n,\lambda_h^n\} &= \underbrace{(\{p^n,\lambda^n\}-  \{\mathcal{P}_h^b(\bfu^n),\mathcal{L}_h^b(\bfu^n)\}}_{\text{Interpolation error}} + \underbrace{( \{\mathcal{P}_h^b(\bfu^n),\mathcal{L}_h^b(\bfu^n)\} - \{\ph,\lambda_h\})}_{\text{discrete remainder}}\nonumber \\
     &:= \eta_{\{p,\lambda\}}^n + \sigma_{\{p,\lambda\}}^n.
\end{align}
The bounds on the Ritz-Stokes interpolation error have already been proved in \cref{lemma: Error Bounds Ritz-Stokes std ENS}, \cref{lemma: Error Bounds der Ritz-Stokes std ENS}. Following, now, the calculation as in \cref{lemma: error equation ENS}, due to the new pressure decompositions \eqref{eq: decomposition error 2 Ritz 2 ENS}, the new initial condition $\uh^0 = \mathcal{R}_h^b\bfu^0$, and utilizing the fact that $\ah = \ahhat + \mh$, cf. \cref{def: bilinear forms discrete ENS}, we derive the new error equation, which is presented below.

\begin{lemma}\label{lemma: error equation 2 ENS}
  For  $\sigma_{\bfu}^n,\, \sigma_{\{p,\lambda\}}^n$ as in  \eqref{eq: decomposition error 1 Ritz 2 ENS}, \eqref{eq: decomposition error 2 Ritz 2 ENS},  the following error equations holds true for $n \geq 1$:
   \begin{align}
\boxed{\begin{cases}\label{eq: error equation 2 ENS}\tag{EQ{\tiny c}}
      &{\dfrac{1}{\Delta t}\Big(\bfm_h(\sigma_{\bfu}^n,\vhn) - \bfm_h(\sigma_{\bfu}^{n-1},\underline{\vhn}(\cdot,t^{n-1}))\Big) } + \ah(\sigma_{\bfu}^n,\vhn) +b_h^L(\vhn,\sigma_{\{p,\lambda\}}^n)\\
       &\qquad\qquad\qquad ~~ = \underbrace{\sum_{i=1}^3 \textsc{Err}_{i}^{C}(\vhn)}_{\textsc{Cons. Errors}} +  \underbrace{\textsc{Err}_{1}^{I}(\vhn)}_{\textsc{Inter. Errors}} + \!\!\!\!\! \underbrace{\mathcal{C}^{cov}(\vhn)}_{\textsc{Inert. Errors}}  \!\!\!+  \, \mh(\sigma_{\bfu}^n,\vhn),\\
         &b_h^L(\sigma_{\bfu}^n,\{\qhn,\xihn\})= 0, 
    \end{cases}}
\end{align}
for all $\vhn \in \bfV_h^n$ and $\{\qhn,\xi_h^n\}\in Q_h^n\times \Lambda_h^n$,
with \emph{consistency errors}:
\begin{itemize}
    \item [\textbullet\hspace{9mm}] \hspace{-9mm}$\textsc{Err}_1^{C}(\vhn) := \dfrac{1}{\Delta t}\Big(\mh(\bfu^n,\vhn) - \mh(\bfu^{n-1},\underline{\vhn}(t_{n-1}))\Big) - \bfm(\matn\bfu^n,(\vhn)^{\ell}) - \gh(\TrVel;\uhn,\vhn)$
\end{itemize}
\begin{minipage}[t]{13cm}
    \begin{itemize}
        \item[\textbullet\hspace{9mm}] \hspace{-9mm}$ \textsc{Err}_2^{C,cov}(\vhn) := c_h^{cov}(\underline{\bfu}^{n-1}(t_n);\bfu^n,\vhn)  - c^{cov}(\bfu^n;\bfu^n,(\vhn)^{\ell})$\vspace{1mm}
         \item[\textbullet\hspace{9mm}] \hspace{-10mm} $ \textsc{Err}_3^{C}(\vhn) :=  \bfm(\bff^n,(\vhn)^{\ell})- \mh(\bff_h^n,\vhn)$\vspace{1mm}
    \end{itemize}
\end{minipage}

\vspace{2mm}
\noindent and \emph{interpolations errors}: \vspace{2mm}

\begin{minipage}[t]{10cm}
    \begin{itemize}
        \item[\textbullet\hspace{13mm}] \hspace{-12mm}$ \textsc{Err}_1^{I}(\vhn) :=\dfrac{1}{\Delta t}\Big(\mh(\eta_{\bfu}^n,\vhn) - \mh(\eta_{\bfu}^{n-1},\underline{\vhn}(t_{n-1}))\Big)$\vspace{1mm}
        \item[\textbullet\hspace{13mm}] \hspace{-12mm}$  \mathcal{C}^{cov}(\vhn) :=  c_h^{cov}(\underline{\eu^{n-1}}(t_n);\bfu^n;\vhn) - c_h^{cov}(\underline{\uhnN}(t_n);\eu^n;\vhn)$
    \end{itemize}
\end{minipage}
\begin{minipage}[t]{8cm}
    \begin{itemize}
    \item[\phantom{\textbullet}] \vspace{1mm}
    \end{itemize}
\end{minipage}
\end{lemma}

Now, due to the choice of initial condition, we see from \cref{sec: velocity a-priori estimates ENS},by following almost identical calculations to \cref{lemma: discrete remainder velocity estimates kl=ku-1 ENS} and \cref{theorem: Velocity Error Estimates ENS} that the following velocity error estimates hold:
\begin{align}
\label{eq: discrete remainder velocity estimates kl=ku-1 HR ENS}
\norm{\sigma_{\bfu}^n}_{L^2(\Gahn)}^2 + \Delta t \sum_{k=1}^n \norm{\sigma_{\bfu}^k}_{\ah}^2 &\leq C \Big(\big(\Delta t \big)^2 +h^{2r_u} + h^{2k_{pr}+2} +h^{2k_{\lambda}+2}\Big),\\
\label{eq: Velocity Error Estimates kl=ku-1 HR ENS}
     \norm{\eu^n}_{L^2(\Gahn)}^2 + \Delta t \sum_{k=1}^n \norm{\eu^k}_{\ah}^2  &\leq  C \Big(\big(\Delta t \big)^2 +h^{2r_u} + h^{2k_{pr}+2} +h^{2k_{\lambda}+2}\Big),
\end{align}
where $r_u = min\{k_{u},k_g-1\}$, with positive constant $C$ depending on the regularity \cref{assumption: Regularity assumptions for velocity estimate 2 ENS}, and where for the last estimate we used the standard Ritz-Stokes interpolation estimates \eqref{eq: Error Bounds Ritz-Stokes std ENS}, \eqref{eq: Error Bounds der Ritz-Stokes std ENS}.

\begin{remark}\label{remark: init cond and stab ENS}
Let us discuss, once more, the initial condition  and the condition \eqref{eq: Linfty uh bound ENS}:
\begin{enumerate}
    \item[i)] As mentioned, we choose $\uh^0 = \mathcal{R}_h^b\bfu^0$ so that $\sigma_{\bfu}^0=0$. Therefore 
    the computations follow almost identical to \cref{sec: velocity a-priori estimates ENS}, giving the estimates \eqref{eq: discrete remainder velocity estimates kl=ku-1 HR ENS}, \eqref{eq: Velocity Error Estimates kl=ku-1 HR ENS}. However, in practice, see \cref{remark: about initial conditions ENS}, this is not convenient and  therefore a more standard choice would have been $\uh^0 = \mathcal{R}_h\bfu^0$. In that case, the term $\norm{\sigma_{\bfu}^0}_{\ah} = \norm{\mathcal{R}_h^b\bfu^0 - \mathcal{R}_h\bfu^0}_{\ah}\neq0$  would emerge at $t=0$, but due to the linearity of the Ritz-Stokes projection \eqref{eq: surface Ritz-Stokes projection std ENS} we see that $\norm{\sigma_{\bfu}^0}_{\ah} =  \norm{\mathcal{R}_h(0,\{p^0,\lambda^0\})}_{\ah} \leq c(h^{k_{pr}+1}+h^{k_{\lambda}+1} + h^{k_g})(\norm{p}_{H^{k_{pr}+1}} + \norm{\lambda}_{H^{k_{\lambda}+1}})$ and hence the error bounds would still hold.
    \item[ii)] By the $L^{\infty}$ stability of the standard Ritz-Stokes map \eqref{eq: W1infty estimate Ritz-Stokes ENS} and \eqref{eq: discrete remainder velocity estimates kl=ku-1 HR ENS}, 
    \cref{remark: About the pressure stability for arbitrary convergence ENS} still stands.
\end{enumerate}
\end{remark}

Now, let us revisit the consistency, interpolation and inertia bounds. Following the similar calculation to \cref{lemma: Interpolation errors ENS}, where we use the Ritz-Stokes estimates \eqref{lemma: Error Bounds der Ritz-Stokes std ENS} we obtain the following.
\begin{lemma}[Interpolation errors II]\label{lemma: Interpolation errors HR ENS}
The following interpolation error hold for all $\vhn \in \bfV_h^n$ \vspace{-1mm}
    \begin{align}\label{eq: Interpolation errors HR ENS}
        |\textsc{Err}^{I}_1(\vhn)| &\leq ch^{r_u}\!\!\!\!\!\sup_{t\in[t_{n-1},t_n]}\sum_{j=0}^{1}\big( \norm{(\matn)^j\bfu}_{H^{k_u+1}(\Gat)}\big)\norm{\vhn}_{L^2(\Gah^n)}\\
        &+ c(h^{k_{pr}+1}+ h^{,k_{\lambda}+1})\!\!\!\!\!\sup_{t\in[t_{n-1},t_n]} \sum_{j=0}^{1}\big(\norm{(\matn)^j p}_{H^{k_{pr}+1}(\Gat)} + \norm{(\matn)^j \lambda}_{H^{k_{\lambda}+1}(\Gat)}\big)\norm{\vhn}_{L^2(\Gah^n)},\nonumber
\end{align}
with $r_u = min\{k_u,k_g-1\}$ and the constants $c>0$ are independent of $h$ and $t$.
\end{lemma}

\begin{lemma}[Consistency errors II]\label{lemma: Consistency errors HR ENS}
There exists  a constant $c>0$ independent of $h$ and $t$ such that the following bounds hold for all $\vhn \in \bfV_h^n$
     \begin{align}\label{eq: Consistency error 1 HR ENS}
        &\!\!\!\!|\textsc{Err}_1^{C}(\vhn)| \leq  c\sqrt{\Delta t}\Big(\int_{t_{n-1}}^{t^n} \norm{\matn\matn\bfu}_{L^2(\Gat)}^2 +\norm{\matn\bfu}_{L^2(\Gat)}^2 \, dt\Big)^{1/2}\norm{\vhn}_{L^2(\Gah^n)} \nonumber\\
        &\qquad\qquad\quad\, \!\!\!\!+ \frac{ch^{\widehat{r}_u+1}}{\sqrt{\Delta t }}\Big(\int_{t_{n-1}}^{t^n}+\norm{\matn\bfu}_{L^2(\Gat)}^2 + \norm{\bfu}_{H^1(\Gat)}^2 \, dt\Big)^{1/2}\norm{\vhn}_{L^2(\Gah^n)}
        \\
        &\qquad\qquad\quad\, \!\!\!\! +  ch^{\widehat{r}_u+1}\!\!\!\!\!\!\! \sup_{t\in[t_{n-1},t_n]}\big(\norm{\bfu}_{L^2(\Gat)}\big)\norm{\vhn}_{H^1(\Gah^n)} + \big(\norm{\eta_{\bfu}^n}_{L^2(\Gah^n)}+\norm{\sigma_{\bfu}^n}_{L^2(\Gah^n)}\big)\norm{\vhn}_{L^2(\Gah^n)},\nonumber\\
        &\!\!\!\!|\textsc{Err}_2^{C,cov}(\vhn)| \leq c(\Delta t+h^{\widehat{r}_u})  \sup_{t \in I_n}\Big\{\sum_{j=1}^0\norm{(\matn)^j\bfu}_{H^{1}(\Gat)}\Big\}\norm{\bfu^n}_{W^{1,\infty}(\Ga^n)}\norm{\vhn}_{L^2(\Gahn)}\nonumber\\
        \label{eq: Consistency error 3 cov HR ENS}
        &\qquad\qquad\qquad\qquad\qquad\qquad\qquad\qquad + ch^{k_g-1}\norm{\bfu^{n-1}}_{H^1(\Ga^{n-1})}\norm{\bfu^{n}}_{H^1(\Ga^{n})} \norm{\vhn}_{L^2(\Gah^n)},\\
        \label{eq: Consistency error 4 HR ENS}
        &\!\!\!\! |\textsc{Err}_3^{C}(\vhn)| \leq ch^{k_g+1}\norm{\bff^n}_{L^2(\Ga^n)}\norm{\vhn}_{L^2(\Gah^n)}.
\end{align}
\end{lemma}
\begin{proof}
    $\textsc{Err}_1^{C}(\vhn)$ and $\textsc{Err}_3^{C}(\vhn)$ are obtained identically to \cref{lemma: Consistency errors ENS}. Let us focus on $\textsc{Err}^{C,cov}_2$. As mentioned, now we need to apply the integration by parts formula so that the test function is bounded w.r.t the $L^2$-norm. As in  \cref{lemma: Consistency errors ENS} we write \eqref{eq: Consistency errors inside err3 1 cov ENS} again:
    \begin{align}\label{eq: Consistency errors inside err3 1 cov HR ENS}
    \textsc{Err}^{C,cov}_2(\vhn) &= c_h^{cov}(\underline{\bfu}^{n-1}(t_n);\bfu^n,\vhn) - c^{cov}(\underline{\bfu}^{n-1}(t_n);\bfu^n,(\vhn)^{\ell}) \\
    &\qquad+ c^{cov}(\underline{\bfu}^{n-1}(t_n);\bfu^n,(\vhn)^{\ell}) - c^{cov}(\bfu^n;\bfu^n,\vhl)\nonumber\\
    &=: \mathrm{C}_1^{cov} + \mathrm{C}_2^{cov},
    \end{align}
where we recall $c_h^{cov}$, $c^{cov}$ as defined in \cref{def: bilinear forms discrete ENS} and \eqref{eq: c equiv expression ENS} (with $\eta_1,\eta_2=0$) respectively. Following \eqref{eq: Consistency errors inside err3 2 ENS}, and also using the inverse inequality, yields
\begin{align}\label{eq: Consistency errors inside err3 2 cov HR ENS}
     |\mathrm{C}_1^{cov}|\leq  ch^{k_g-1}\norm{\bfu^{n-1}}_{H^1(\Ga^{n-1})}\norm{\bfu^n}_{H^1(\Ga^n)}\norm{\vhn}_{L^2(\Gah^n)}.
\end{align}
For the second estimate in \eqref{eq: Consistency errors inside err3 1 cov HR ENS} we proceed similarly to \eqref{eq: Consistency errors inside err3 3 ENS}:
    \begin{align}
        |\mathrm{C}_2^{cov}| &=  \frac{1}{2}\Big|\int_{\Ga^n}\big( (\bfu^n-\underline{\bfu}^{n-1}(t_n))\cdot\nbgcov\bfu^n \big)\cdot(\vhn)^{\ell}\Big| +  \frac{1}{2}\Big|\int_{\Ga^n}\big((\bfu^n-\underline{\bfu}^{n-1}(t_n))\cdot\nbgcov\big)\bfPg(\vhn)^{\ell})\cdot\bfu^n \Big| \nonumber\\
        \label{eq: Consistency errors inside err3 3 cov HR ENS}
        & = \textbf{I}c + \textbf{II}c,
    \end{align}
 where now using the further assumed regularity \cref{assumption: Regularity assumptions for velocity estimate 2 ENS} and \eqref{eq: relate matd to matdl ENS} we derive 
\begin{equation*}
    \begin{aligned}
       \textbf{I}c \leq  (\Delta t+h^{\widehat{r}_u+1})  \sup_{t \in I_n}\big\{\norm{\matn\bfu}_{L^{2}(\Gat)}+\norm{\bfu}_{H^{1}(\Gat)}\big\}\norm{\bfu^n}_{W^{1,\infty}(\Ga^n)}\norm{\vhn}_{L^2(\Gahn)}.
    \end{aligned}
\end{equation*}
Recall that $\bfu$ is the solution of \eqref{weak lagrange hom NV cov ENS}, then by the integration by parts formula \eqref{eq: integration by parts cont ENS} we see that
\begin{equation}
    \begin{aligned}\label{eq: consistency inside 2 HR ENS}
        \int_{\Ga^n} ((\bfu^{n} \cdot \nb^{cov}_{\Ga^n})\bfPg\vhl )\cdot \bfu^{n} \; \ds &= \int_{\Ga^n} \bfu^n \cdot (\nb_{\Ga^n}(\vhl \cdot \bfu^n))  \; \ds -\int_{\Ga^n} ((\bfu^{n} \cdot \nb_{\Ga^n})\bfu^{n} ) \cdot \bfPg \vhl  \; \ds\\
        & = - \int_{\Ga^n} \mathrm{div}_{\Ga^n}(\bfu^n) (\vhl \cdot \bfu^n)  \; \ds  -\int_{\Ga^n} ((\bfu^{n} \cdot \nb^{cov}_{\Ga^n})\bfu^{n} ) \cdot \vhl  \; \ds,
    \end{aligned}
\end{equation}
where by $\nb^{cov}_{\Ga^n},\, \mathrm{div}_{\Ga^n}$ we denote the surface operators defined at the surface $\Ga^n$.
So considering \eqref{eq: consistency inside 2 HR ENS} it is clear that \vspace{-3mm}
\begin{equation*}
    \begin{aligned}
       \textbf{II}c = \textbf{I}c +  \overbrace{\int_{\Ga^n} \mathrm{div}_{\Ga^n}(\bfu^n-\underline{\bfu}^{n-1}(t_n))) (\vhl \cdot \bfu^n)  \; \ds }^{\textbf{III}c}.
    \end{aligned}
\end{equation*}
Using then again \eqref{eq: un-un-1 formula ENS}, i.e.
\begin{equation*}
    \bfu^n-\underline{\bfu}^{n-1}(t_n) = \int_{t_{n-1}}^{t_n} \matdl\bfu\circ \Phi^{\ell}_{s}\,ds \circ \Phi^{\ell}_{-t_n}
\end{equation*}
along with the pull back formula \cite[Section 3]{church2020domain} $\mathrm{div}_{\Ga^n}(\bfu^n)\circ\Phi_{t^n} ^{\ell}= tr(\mathcal{U}_t^0\nb_{\Ga^0}\bfu^n\circ\Phi_{t^n}^{\ell})$, where $\mathcal{U}_t^0 = \nb_{\Ga^0}\Phi_{t^n}^{\ell}\mathrm{G}_0^t$ with $\mathrm{G}_0^t=(\mathrm{G}_t^0)^{-1} = ((\nb_{\Ga^0}\Phi_{t^n}^{\ell})^t\nb_{\Ga^0}\Phi_{t^n}^{\ell} + \bfn_{\Ga^0}\otimes\bfn_{\Ga^0})^{-1} =\phi_{-t}^{\ell}(\nb_{\Ga^n}\Phi_{-t^n}^{\ell})(\nb_{\Ga^n}\Phi_{-t^n}^{\ell})^{t} + \bfn_{\Ga^0}\otimes\bfn_{\Ga^0} $ and the bounds on the mapping $\Phi_{t^n}^{\ell}$, we readily see that 
\begin{equation*}
    \begin{aligned}
        \textbf{III}c \leq (\Delta t+h^{\widehat{r}_u})  \sup_{t \in I_n}\Big\{\sum_{j=1}^0\norm{(\matn)^j\bfu}_{H^{1}(\Gat)}\Big\}\norm{\bfu^n}_{L^{\infty}(\Ga^n)}\norm{\vhn}_{L^2(\Gahn)},
    \end{aligned}
\end{equation*}
where we also applied \eqref{eq: relate matd to matdl ENS}. Combining now $\textbf{I}c-\textbf{III}c$ into \eqref{eq: Consistency errors inside err3 3 cov HR ENS} coupled with \eqref{eq: Consistency errors inside err3 2 cov HR ENS}, yields \eqref{eq: Consistency error 3 cov HR ENS}.
\end{proof}

\begin{lemma}[Inertia error II]\label{lemma: Trilinear Errors HR ENS}
The following error bound is satisfied
\begin{equation}
    \begin{aligned}\label{eq: Trilinear Errors general HR 1 ENS}
        |\mathcal{C}^{cov}(\vh)| &\leq  c\norm{\bfu^n}_{W^{2,\infty}(\Ga^n)} \big(\norm{\eu^{n-1}}_{L^2(\GahnN)} + \norm{\eu^{n}}_{L^2(\Gahn)} + \norm{\eu^{n-1}}_{\ah}+  \norm{\eu^{n}}_{\ah}\big) \norm{\vhn}_{L^2(\Gahn)} \\
        &\ + c\norm{\sigma_{\bfu}^{n-1}}_{L^2(\GahnN)}^{1/2}\norm{\sigma_{\bfu}^{n-1}}_{\ah}^{1/2}\norm{\eu^{n}}_{\ah}\norm{\vhn}_{L^2(\Gahn)}^{1/2}\norm{\vhn}_{\ah}^{1/2}\\
    &\ +c\norm{\sigma_{\bfu}^{n-1}}_{L^2(\GahnN)}^{1/2}\norm{\sigma_{\bfu}^{n-1}}_{\ah}^{1/2}\norm{\eu^{n}}_{L^2(\Gah)}^{1/2}\norm{\eu^{n}}_{\ah}^{1/2}\norm{\vhn}_{\ah},
    \end{aligned}
\end{equation}
for all $\vh \in \bfV_h$, where $c>0$ constant independent of $h$ and $t$ and $k_g \geq 2$.
\end{lemma}
\begin{proof}
The proof follows identically to \cite[Lemma 7.14]{elliott2025unsteady},  where one uses integration by parts \eqref{eq: integration by parts ENS}, \eqref{eq: geometric errors 1b ENS}, and the $W^{1,\infty}$-Ritz-Stokes bounds \eqref{eq: W1infty estimate Ritz-Stokes ENS}, along with the fact that $k_g \geq 2$, to obtain the desired estimate. This time, we also use the time-inequality bounds \eqref{eq: time ineq eu ENS}, were appropriate. 
\end{proof}

Now, instead of \cref{lemma: Auxiliary convergence result ENS}, we prove the following auxiliary result.

\begin{lemma}[Auxiliary result II]\label{lemma: Auxiliary convergence result HR ENS}
Assume $\underline{k_\lambda = k_u-1}$ and $k_g \geq 2$, and let the regularity \cref{assumption: Regularity assumptions for velocity estimate 2 ENS} hold with time-step $\Delta t \leq ch$. Let $(\bfu,\{p,\lambda\})$ be the solution of \eqref{weak lagrange hom NV cov ENS}  and  let $(\uh^k,\{\ph^k,\lh^k\})$, $k=1,...,n$ be the discrete solutions of \eqref{eq: fully discrete fin elem approx cov ENS}, with initial condition $\uh^0=\mathcal{R}_h^b\bfu^0$. Then, the following estimate holds for $1 \leq n \leq N$
    \begin{equation}\label{eq: Auxiliary convergence result HR ENS}
    \begin{aligned}
        \!\!\!\Delta t\sum_{k=1}^n  \mh(\matdt\sigma_{\bfu}^k,\matdt\sigma_{\bfu}^k) + \norm{\sigma_{\bfu}^n}_{\ah}^2 &\leq C(1 + \frac{h^{2m}}{\Delta t })\Big( \big(\Delta t \big)^2 +h^{2m}\Big)  + \epsilon\Delta t\sum_{k=1}^n\norm{\sigma_{\{p,\lambda\}}^n}_{L^2(\Gahn)}^2,
            \end{aligned}
    \end{equation}
where $m=min\{r_u,k_{pr}+1,k_{\lambda}+1\}$, $r_u = min\{k_{u},k_g-1\}$, with positive constant $C=C(\epsilon)$ depending on arbitrary $\epsilon>0$ and the regularity \cref{assumption: Regularity assumptions for velocity estimate 2 ENS}.
\end{lemma}
\begin{proof}
To start with, recall \eqref{eq: m bilinear manipulation kl=ku-1 ENS} and the fact that   $\sigma_{\bfu}^n = \underline{\sigma_{\bfu}^{n-1}}(\cdot ,t_n) + \Delta t \, \matdt \sigma_{\bfu}^n$ by \eqref{eq: n to n-1 on surface n ENS} to see that 
\begin{equation}\label{eq: Auxiliary convergence result inside HR start 1 ENS}
\begin{aligned}
        {\dfrac{1}{\Delta t}\Big(\bfm_h(\sigma_{\bfu}^n,\vhn) - \bfm_h(\sigma_{\bfu}^{n-1},\underline{\vhn}(\cdot,t^{n-1}))\Big) } &= \mh(\matdt\sigma_{\bfu}^n,\vhn) \\
    &+ \frac{1}{\Delta t }\big[\mh(\underline{\sigma_{\bfu}^{n-1}}(t_n),\vhn) - \bfm_h(\sigma_{\bfu}^{n-1},\underline{\vhn}(\cdot,t^{n-1}))\big].
    \end{aligned}
\end{equation}
Then testing the equation error equations \eqref{eq: error equation 2 ENS} with $\vhn = \sigma_{\bfu}^n-\underline{\sigma_{\bfu}^{n-1}}(t_n) \in \bfV_h^n$ and $\{\qhn,\xihn\}=\sigma_{\{p,\lambda\}}^n$, and using simple calculations similar to \eqref{eq: m bilinear manipulation kl=ku ENS} to derive
\begin{equation}
    \begin{aligned}\label{eq: Auxiliary convergence result inside HR start 2 ENS}
        \ah(\sigma_{\bfu}^n,\sigma_{\bfu}^n-\underline{\sigma_{\bfu}^{n-1}}(t_n)) &= \frac{1}{2}\big(\norm{\sigma_{\bfu}^n}_{\ah}^2 - \norm{\sigma_{\bfu}^{n-1}}_{\ah}^2\big) + \frac{1}{2}\norm{\sigma_{\bfu}^n-\underline{\sigma_{\bfu}^{n-1}}(t_n)}_{\ah}^2\\
        &+ \frac{1}{2}\big(\ah(\sigma_{\bfu}^{n-1},\sigma_{\bfu}^{n-1}) - \ah(\underline{\sigma_{\bfu}^{n-1}}(t_n),\underline{\sigma_{\bfu}^{n-1}}(t_n))\big),
    \end{aligned}
\end{equation}
along with the fact that $\sigma_{\bfu}^n \in \bfV_h^{n,div}$, $n=0,...,N$, thus $\bhtil(\sigma_{\bfu}^{n},\{\qhn,\xihn\}) = \bhtil(\sigma_{\bfu}^{n-1},\{\underline{\qhn}(\cdot,t_{n-1}),\underline{\xihn}(\cdot,t_{n-1})\} )=0 $ $ \text{ for all } \{\qhn,\xihn\}\in Q_h^n\times \Lambda_h^n$, $\{\underline{\qhn}(\cdot,t_{n-1}),\underline{\xihn}(\cdot,t_{n-1})\}\in Q_h^{n-1}\times\Lambda_h^{n-1}$ by  \eqref{eq: blb important time property ENS} and  therefore
\begin{equation*}
    \begin{aligned}
        \bhtil(\sigma_{\bfu}^n-\underline{\sigma_{\bfu}^{n-1}}(t_n),\sigma_{\{p,\lambda\}}^n) &= \bhtil(\sigma_{\bfu}^{n-1},\underline{\sigma_{\{p,\lambda\}}^n}(t_{n-1}) ) - \bhtil(\underline{\sigma_{\bfu}^{n-1}}(t_n),\sigma_{\{p,\lambda\}}^n) \\
        &\leq c(\epsilon)\Delta t \norm{\sigma_{\bfu}^{n-1}}_{\ah}^2 + \epsilon\Delta t\norm{\sigma_{\{p,\lambda\}}^n}_{L^2(\Gahn)}^2 \qquad  \text{by \eqref{eq: time differences properties bilinear b ENS} and $\epsilon$-Weighted Young's Ineq.},
    \end{aligned}
\end{equation*}
we derive, once we also apply the time inequality \eqref{eq: time differences properties bilinear a ENS} to the last difference in \eqref{eq: Auxiliary convergence result inside HR start 2 ENS} and \eqref{eq: time differences properties bilinear m ENS} on the last difference in \eqref{eq: Auxiliary convergence result inside HR start 1 ENS} together with Young's inequality and the fact that $\norm{\cdot}_{L^2(\Gahn)}\leq \norm{\cdot}_{\ah}$, the following
\begin{equation}
    \begin{aligned}\label{eq: Auxiliary convergence result inside HR main ENS}
        &\Delta t \, \mh(\matdt\sigma_{\bfu}^n,\matdt\sigma_{\bfu}^n) + \frac{1}{2}\big(\norm{\sigma_{\bfu}^n}_{\ah}-\norm{\sigma_{\bfu}^{n-1}}_{\ah}^2\big) + \frac{1}{2}\norm{\sigma_{\bfu}^n-\underline{\sigma_{\bfu}^{n-1}}(t_n)}_{\ah}^2 \leq c(\epsilon)\Delta t \norm{\sigma_{\bfu}^{n-1}}_{\ah}^2 \\
        &+ \Big( |\sum_{i=1}^3 \textsc{Err}_{i}^{C}(\sigma_{\bfu}^n - \underline{\sigma_{\bfu}^{n-1}}(t_n))| +  | \textsc{Err}_{1}^{I}(\sigma_{\bfu}^n -\underline{\sigma_{\bfu}^{n-1}}(t_n))| + |\mathcal{C}(\sigma_{\bfu}^n -\underline{\sigma_{\bfu}^{n-1}}(t_n))|\\
        &+ |\mh(\sigma_{\bfu}^n,\sigma_{\bfu}^n -\underline{\sigma_{\bfu}^{n-1}}(t_n))|\Big) + \frac{1}{6\Delta t}\norm{\sigma_{\bfu}^n -\underline{\sigma_{\bfu}^{n-1}}(t_n)}_{L^2(\Gahn)}^2  + \epsilon\Delta t\norm{\sigma_{\{p,\lambda\}}^n}_{L^2(\Gahn)}^2  .
    \end{aligned}
\end{equation}
We now want to approximate each term on the right-hand side appropriately. This will be possible with the help of the newly established bounds in \cref{lemma: Interpolation errors HR ENS,lemma: Consistency errors HR ENS,lemma: Trilinear Errors HR ENS}. For simplicity and brevity (compared to \cref{sec: velocity a-priori estimates ENS,sec: pressure a-priori kl=ku ENS}) we assume that our constant $C$ depends on the on the regularity \cref{assumption: Regularity assumptions for velocity estimate 2 ENS}. Then, a simple use of Young's inequality (similarly to \cref{lemma: Auxiliary convergence result ENS}) gives us the following bounds :

$\bullet$ \ Using \cref{lemma: Consistency errors HR ENS} and Ritz-Stokes bounds \eqref{eq: Error Bounds Ritz-Stokes std ENS} we derive similarly to \cref{lemma: Auxiliary convergence result ENS}
\begin{equation*}
    \begin{aligned}
        |\sum_{i=1}^3 \textsc{Err}_{i}^{C}(\sigma_{\bfu}^n - \underline{\sigma_{\bfu}^{n-1}}(t_n))| \leq c\big((\Delta t)^2+h^{\widehat{r}_u+1}\big) \Big(\int_{t_{n-1}}^{t^n} \norm{\matn\matn\bfu}_{L^2(\Gat)}^2 +\norm{\matn\bfu}_{L^2(\Gat)}^2 \, dt\Big) \\ + C\Delta t\big(h^{2r_u}+h^{2k_{pr}+2} + h^{2k_{\lambda}+2}) + \frac{1}{6\Delta t }\norm{\sigma_{\bfu}^n -\underline{\sigma_{\bfu}^{n-1}}(t_n)}_{L^2(\Gahn)}^2.
    \end{aligned}
\end{equation*}

$\bullet$ \ For interpolation error, once again, we just recall \cref{lemma: Interpolation errors HR ENS} to see that
\begin{equation*}
    \begin{aligned}
        | \textsc{Err}_{1}^{I}(\sigma_{\bfu}^n - \underline{\sigma_{\bfu}^{n-1}}(t_n))| \leq C \Delta t (h^{2r_u} +h^{2k_{pr}+2} + h^{2k_{\lambda}+2}) + \frac{1}{6\Delta t}\norm{\sigma_{\bfu}^n - \underline{\sigma_{\bfu}^{n-1}}(t_n)}_{L^2(\Gahn)}^2.
    \end{aligned}
\end{equation*}

$\bullet$ \ For convective error we use the bounds in \cref{lemma: Trilinear Errors HR ENS}, where applying Young's inequality   consecutively and the inverse inequality appropriately, yields
\begin{equation*}
    \begin{aligned}
        &\mathcal{C}(\sigma_{\bfu}^n - \underline{\sigma_{\bfu}^{n-1}}(t_n)) \leq C\Delta t \big(\norm{\eu^{n-1}}_{L^2(\GahnN)}^2 + \norm{\eu^{n}}_{L^2(\Gahn)}^2 + \norm{\eu^{n-1}}_{\ah}^2 +\norm{\eu^{n}}_{\ah}^2\big) + \frac{1}{12\Delta t}\norm{\sigma_{\bfu}^n - \underline{\sigma_{\bfu}^{n-1}}(t_n)}_{L^2(\Gahn)}^2 \\ 
        &\qquad\qquad+ ch^{-2}\norm{\sigma_{\bfu}^{n-1}}_{L^2(\GahnN)}^2\Delta t \norm{\eu^n}_{\ah}^2+ \frac{1}{12\Delta t}\norm{\sigma_{\bfu}^n - \underline{\sigma_{\bfu}^{n-1}}(t_n)}_{L^2(\Gahn)}^2\\
        &\qquad\qquad  c\frac{1}{\Delta t}\norm{\sigma_{\bfu}^{n-1}}_{L^2(\GahnN)}\norm{\eu^{n}}_{L^2(\Gahn)}\Delta t \big(\norm{\sigma_{\bfu}^{n-1}}_{\ah}^2 + \norm{\eu^{n}}_{\ah}^2\big) + \frac{1}{2}\norm{\sigma_{\bfu}^n - \underline{\sigma_{\bfu}^{n-1}}(t_n)}_{\ah}^2.
    \end{aligned}
\end{equation*}

$\bullet$ \ Finally , regarding the second to last term in \eqref{eq: Auxiliary convergence result inside HR main ENS} a simple Young's inequality yields
\begin{equation*}
    \begin{aligned}
        |\mh(\sigma_{\bfu}^n,\sigma_{\bfu}^n -\underline{\sigma_{\bfu}^{n-1}}(t_n))\Big)| \leq \Delta t \norm{\sigma_{\bfu}^n}_{L^2(\Gahn)} +\frac{1}{6\Delta t} \norm{\sigma_{\bfu}^n -\underline{\sigma_{\bfu}^{n-1}}(t_n)}_{L^2(\Gahn)}.
    \end{aligned}
\end{equation*}

Now, combining all the bounds above into \eqref{eq: Auxiliary convergence result inside HR main ENS},  applying the applying the velocity error estimates in \eqref{eq: discrete remainder velocity estimates kl=ku-1 HR ENS}, \eqref{eq: Velocity Error Estimates kl=ku-1 HR ENS}, using again the fact that by \eqref{eq: n to n-1 on surface n ENS} $\frac{1}{\Delta t}\norm{\sigma_{\bfu}^n - \underline{\sigma_{\bfu}^{n-1}}(t_n)}_{L^2(\Gahn)}^2 = \Delta t \norm{\matdt\sigma_{\bfu}^n}_{L^2(\Gahn)}^2 = \Delta t\mh(\matdt\sigma_{\bfu}^n,\matdt\sigma_{\bfu}^n)$ and employing a kickback argument for this term and $\frac{1}{2}\norm{\sigma_{\bfu}^n - \underline{\sigma_{\bfu}^{n-1}}(t_n)}_{\ah}^2$, we obtain the following, after summing for $k=1,...,n$
\begin{equation*}
    \begin{aligned}
        &\Delta t\sum_{k=1}^n  \mh(\matdt\sigma_{\bfu}^k,\matdt\sigma_{\bfu}^k) + \norm{\sigma_{\bfu}^n}_{\ah}^2 \leq C_1(C,\epsilon)\Big( \big(\Delta t \big)^2 +h^{2r_u} + h^{2k_{pr}+2} +h^{2k_{\lambda}+2}\Big) +  \epsilon\Delta t\norm{\sigma_{\{p,\lambda\}}^n}_{L^2(\Gahn)}^2\\
        &+ C\big(\frac{1}{\Delta t } + h^{-2}\big)\big((\Delta t )^2+ h^{2r_u} + h^{2k_{pr}+2} +h^{2k_{\lambda}+2}\big)\Delta t\sum_{k=1}^n\big(\norm{\sigma_{\bfu}^{n-1}}_{\ah}^2 + \norm{\eu^{n}}_{\ah}^2\big),
    \end{aligned}
\end{equation*}
since $\sigma_{\bfu}^0 = 0$ due to the choice of the initial condition. Applying once again \eqref{eq: discrete remainder velocity estimates kl=ku-1 HR ENS}, \eqref{eq: Velocity Error Estimates kl=ku-1 HR ENS} and recalling the fact that $k_g \geq 2$ yields our assertion.
\end{proof}

\begin{lemma}\label{lemma: discrete remainder pressure estimates improved HR ENS}
     Assume $\underline{k_\lambda = k_u-1}$, $k_g \geq 2$, and let the regularity \cref{assumption: Regularity assumptions for velocity estimate 2 ENS} hold with time-step $\Delta t \leq ch$. Let $(\bfu,\{p,\lambda\})$ be the solution of \eqref{weak lagrange hom NV cov ENS} and let $(\uh^k,\{\ph^k,\lh^k\})$, $k=1,...,n$ be the discrete solutions of \eqref{eq: fully discrete fin elem approx cov ENS}, with initial condition $\uh^0=\mathcal{R}_h^b\bfu^0$. Then, the following holds for $1 \leq n \leq N$
     \begin{equation}\label{eq: Pressure stab Estimate ENS}
            \Delta t \sum_{k=1}^{n} \norm{\sigma_{\{p,\lambda\}}^k}_{L^2(\Gah^k)\times L^2(\Gah^k)}^2 \leq  C(1 + \frac{h^{2m}}{\Delta t })\Big( \big(\Delta t \big)^2 +h^{2m}\Big),
    \end{equation}
where $m=min\{r_u,k_{pr}+1,k_{\lambda}+1\}$, $r_u = min\{k_{u},k_g-1\}$, with constant $C>0$ depending on the regularity \cref{assumption: Regularity assumptions for velocity estimate 2 ENS}.
 \end{lemma}
\begin{proof}
 The proof follows almost identically to \Cref{lemma: discrete remainder pressure estimates improved ENS}, where now we use the auxiliary result \cref{lemma: Auxiliary convergence result HR ENS} and employ the $L^2\times L^2$ $\textsc{inf-sup}$ \eqref{eq: discrete inf-sup condition Gah Lagrange ENS} instead. So, recall that
 \begin{equation}
        \begin{aligned}\label{eq: Pressure conv Estimate UNS inside inf-sup HR ENS}
            \norm{\{\sigma_p^n,\sigma_{\lambda}^n\}}_{L^2(\Gahn)\times L^2(\Gahn)} \leq \sup_{\vh^n \in \bfV_h^n} \frac{\bhtil(\vh^n,\{\sigma_p^n,\sigma_{\lambda}^n\})}{\norm{\vh^n}_{\ah}}.
        \end{aligned}
    \end{equation}
Considering the error equation \eqref{eq: error equation 2 ENS}, the calculations in \eqref{eq: Auxiliary convergence result inside HR start 1 ENS}, \eqref{eq: n to n-1 on surface n ENS} and \eqref{eq: time differences properties bilinear m ENS},  and the estimates in \cref{lemma: Interpolation errors HR ENS,lemma: Consistency errors HR ENS,lemma: Trilinear Errors HR ENS}, then solving for  $\bhtil(\vh^n,\{\sigma_p^n,\sigma_{\lambda}^n\})$ yields
\begin{equation}
    \begin{aligned}\label{eq: Pressure conv Estimate UNS inside inf-sup HR 1 ENS}
       &\norm{\{\sigma_p^n,\sigma_{\lambda}^n\}}_{L^2(\Gahn)\times L^2(\Gahn)} \leq \norm{\matdt\sigma_{\bfu}^n}_{L^2(\Gahn)} + \norm{\sigma_{\bfu}^n}_{\ah} + c\norm{\sigma_{\bfu}^n}_{L^2(\Gahn)}\\
       &+  c(\Delta t+h^{m})\sup_{t \in I_n}\Big\{\sum_{j=1}^0\big(\norm{(\matn)^j\bfu}_{H^{1}(\Gat)}+\norm{(\matn)^j p}_{H^{k_{pr}+1}(\Gat)} + \norm{(\matn)^j \lambda}_{H^{k_{\lambda}+1}(\Gat)}\big)\Big\} \\
       &+c(\sqrt{\Delta t}+\frac{h^{\widehat{r}_u+1}}{\sqrt{\Delta t}})\Big(\int_{t_{n-1}}^{t^n} \norm{\matn\matn\bfu}_{L^2(\Gat)}^2\, dt\Big)^{1/2} +\norm{\sigma_{\bfu}^{n-1}}_{L^2(\GahnN)}^{1/2}\norm{\sigma_{\bfu}^{n-1}}_{\ah}^{1/2}\norm{\eu^{n}}_{\ah}\\
       &+ \norm{\bfu^n}_{W^{2,\infty}(\Ga^n)}\big(\norm{\eu^{n-1}}_{L^2(\GahnN)} + \norm{\eu^{n}}_{L^2(\Gahn)} +\norm{\eu^{n-1}}_{\ah}+  \norm{\eu^{n}}_{\ah}\big),
    \end{aligned}
\end{equation}
with $m=min\{r_u,k_{pr}+1,k_{\lambda}+1\}$, $r_u = min\{k_{u},k_g-1\}$, where we used the fact that $\norm{\cdot}_{L^2(\Gahn)}\leq \norm{\cdot}_{\ah}$ in \eqref{eq: Trilinear Errors general HR 1 ENS}  and $\norm{\cdot}_{H^1(\Gahn)}\leq ch^{-1}\norm{\cdot}_{\ah}$ \eqref{eq: coercivity and Korn's inequality Lagrange ENS} in \eqref{eq: Consistency error 1 HR ENS} for the test function.

Now, squaring \eqref{eq: Pressure conv Estimate UNS inside inf-sup HR 1 ENS} and applying the operator $\Delta t \sum_{k=1}^n$ (i.e. multiplying $\Delta t $ and summing over $k=1,...,n$), it is clear that if we apply the velocity error convergence results \eqref{eq: discrete remainder velocity estimates kl=ku-1 HR ENS}, \eqref{eq: Velocity Error Estimates kl=ku-1 HR ENS} and the afore-mentioned auxiliary estimate \eqref{eq: Auxiliary convergence result HR ENS} along with a kickback for the last pressure term in  \eqref{eq: Auxiliary convergence result HR ENS} (assumming $\epsilon>0$ is small enough, i.e. $\epsilon=1/2$) yields our assertion.
\end{proof}

Finally, by applying the triangle inequality on the decomposition \eqref{eq: decomposition error 2 Ritz 2 ENS}, the above \Cref{lemma: discrete remainder pressure estimates improved HR ENS} and the Ritz-Stokes interpolation error \eqref{eq: Error Bounds Ritz-Stokes std ENS} we obtain the final \emph{error pressure result}.
\begin{theorem}[Pressure error bounds II] \label{theorem: Pressures Error Estimate HR ENS}
Under the regularity assumption \cref{assumption: Regularity assumptions for velocity estimate 2 ENS} and \Cref{lemma: discrete remainder pressure estimates improved HR ENS} the following pressure error estimates holds for $1 \leq n \leq N$
     \begin{equation}\label{eq: Pressure stab Estimate HR ENS}
            \Delta t \sum_{k=1}^{n} \norm{\{\ep^k,\el^k\}}_{L^2(\Gah^k)\times L^2(\Gah^k)}^2 \leq  C(1 + \frac{h^{2m}}{\Delta t })\Big( \big(\Delta t \big)^2 +h^{2m}\Big) ,
    \end{equation}
where $m=min\{r_u,k_{pr}+1,k_{\lambda}+1\}$, $r_u = min\{k_{u},k_g-1\}$, with positive constant $C$ depending on the regularity \cref{assumption: Regularity assumptions for velocity estimate 2 ENS}.
 \end{theorem}

 \section{Numerical results}\label{sec: Numerical results ENS}
Now we present some numerical results to validate our theoretical conclusions. We start in \Cref{Sec: set-up ENS} with a general setup of our experiments, defining the parameters, the formulations used, and specifying the notation.
In \Cref{Sec: moving sphere ENS} we examine and test the choice of the finite element  space $\Lambda_h$, i.e. the choice of the two parameters $k_{\lambda} = k_u$ and $k_{\lambda} = k_u-1$.
Finally, in \Cref{Sec: num comparison ENS} we compare approximation schemes of \eqref{weak lagrange hom NV ENS}, see \eqref{eq: L.M. ENS} below, to a penalty-based formulation \eqref{eq: P.M. ENS} for the tangential Navier-Stokes on an evolving surface, with an appropriate penalty parameter; see \cite{reusken2024analysis,olshanskii2024eulerian,Krause2023ASurfaces}.

\subsection{Setup}\label{Sec: set-up ENS}

As mentioned in \cref{remark: about f source discrete ENS}, in practice we do not implement the numerical schemes \eqref{eq: fully discrete fin elem approx ENS}, \eqref{eq: fully discrete fin elem approx cov ENS} that are based on the simplified weak formulations \eqref{weak lagrange hom NV dir ENS}, \eqref{weak lagrange hom NV cov ENS} according to \cref{remark: about f source ENS}. Instead, we have $\bfeta \neq 0$; see \eqref{eq: function f ENS}, and therefore want to consider finite element approximations based on the weak formulations \eqref{weak lagrange hom NV ENS}; recall that \eqref{weak lagrange hom NV ENS} represents two problems depending on the choice of $c^{(\cdot)}(\bullet;\bullet,\bullet)$. So, we define our  numerical schemes \eqref{eq: L.M. ENS} (labeled as LMM: Lagrange Multiplier Method since we make use of a Lagrange multiplier for the normal velocity restriction) naturally as follows: 

\noindent {\textbf{(LMM): }}
Given $\uh^0 \in \bfV_h^{0,div}$ and appropriate approximation for the data $\bff_h,\,{\eta}_h^n$, for $n=0,1,...,N$ find $\uhn \in \bfV_h^n$, $\{\phn,\lambda_h^n\}\in Q_h^n\times \Lambda_h^n$ such that 
\begin{align}
\!\!\!\begin{cases}\label{eq: L.M. ENS}
\tag{LMM}
\frac{1}{\Delta t}\Big(\mh(\uhn,\vhn) - \mh(\uhnN,\vhnN)\Big) - \gh(\TrVel;\uhn,\vhn) + \ahhat(\uhn,\vhn)\\
 + c_h^{(\cdot)}(\underline{\uhnN}(\cdot,t_n),\uhn,\vhn) 
 +b_h^L(\vhn,\{\phn,\lambda_h^{n}\}) = \mh(\bff_h^n,\vhn) + \bfm_h(\uhnN,\underline{\matdt\vhn}(\cdot,t^{n-1})), \\
 \qquad\qquad\qquad\qquad\qquad\qquad \  \! b_h^L(\uhn,\{\qhn,\xihn\}) = m_h (\bm{\eta}_h^n,\{\qhn,\xihn\}), 
\end{cases}
 \end{align}
 for all $\vhn \in \bfV_h^n$ and $\vhnN \in \bfV_h^{n-1}$ and $\{\qhn,\xihn\}\in Q_h^n\times\Lambda_h^n$. 
 Here we  used \eqref{eq: n to n-1 on surface n ENS}, compared to  \eqref{eq: fully discrete fin elem approx ENS}, \eqref{eq: fully discrete fin elem approx cov ENS}, to define our schemes and furthermore, now,  $c_h^{(\cdot)}(\bullet,\bullet,\bullet)$ denotes the skew-symmetrized type discretizations of $c^{(\cdot)}$ as described in \eqref{eq: c equiv expression ENS} and  \eqref{eq: ccov equiv expression ENS} (recall also \cref{def: bilinear forms discrete ENS} used in the schemes \eqref{eq: fully discrete fin elem approx ENS}, \eqref{eq: fully discrete fin elem approx cov ENS} analyzed in \cref{sec: error analysis ENS}):
 \begin{align}\label{eq: implem c ENS}
      &c_h(\underline{\uhnN}(t_n); \uh,\vh) := \frac{1}{2} \Big(\int_{\Gaht } ((\underline{\uhnN}(t_n) \cdot \nbgh)\uh ) \vh \, \dsh - \int_{\Gaht } ((\underline{\uhnN}(t_n) \cdot \nbgh)\vh ) \uh \, \dsh \Big) \nonumber\\
     &\hspace{28mm}-\frac{1}{2}\int_{\Gaht } \eta_{h,2}\uhn\cdot\vhn \, \dsh,\\
    &c_h^{cov}(\underline{\uhnN}(t_n); \uh,\vh) :=\frac{1}{2} \Big(\int_{\Gah } ((\underline{\uhnN}(t_n) \cdot \nbgcovh)\uh ) \cdot\vh \; \dsh - \int_{\Gah } ((\underline{\uhnN}(t_n) \cdot \nbgcovh)\vh ) \cdot \uh \; \dsh \Big)\nonumber\\ 
      &\qquad\qquad\qquad\qquad - \frac{1}{2}\Big(\int_{\Gaht} (\uh\cdot \nh) \underline{\uhnN}(t_n) \cdot \bfH_h \vh \; \dsh - \int_{\Gaht} (\vh\cdot \nh) \underline{\uhnN}(t_n) \cdot \bfH_h \uh \; \dsh \Big)\nonumber\\
      \label{eq: implem c cov ENS}
    &\qquad\qquad\qquad\qquad +\frac{1}{2}\int_{\Gaht } \eta_{h,1}\uhn\cdot\bfH_h\vhn \, \dsh -\frac{1}{2}\int_{\Gaht } \eta_{h,2}\uhn\cdot\bfPh\vhn \, \dsh.
 \end{align}
 It is clear that for \eqref{eq: L.M. ENS}, we also require a sufficiently accurate approximation of the data $\bm{\eta}$ \eqref{eq: function f ENS} in the right-hand side of the second equation in \eqref{weak lagrange hom NV ENS} (i.e. $\bm{\eta}_h$), e.g.
\begin{equation}\label{eq: right-hand side approx ENS}
     \bm{\eta}_h = ({\eta}_{h,1},{\eta}_{h,2}) = (-\kappa_hV_{\Gah},V_{\Gah}),
 \end{equation}
where $V_{\Gah}$ is an interpolation of $V_{\Ga}$ similar to  $\TrVel$  in \eqref{eq: vel discrete surf ENS} and $\kappa_h = tr(\bfH_h)$; see also \cref{remark: formulation no eta_1 ENS}. Clearly, the expression \eqref{eq: implem c ENS} is easier to implement than  \eqref{eq: implem c cov ENS}, since it does  not require an approximation for either the mean curvature $\bfH_h$ or $\eta_{h,1}$. This formulation still yields the expected relative errors that we proved in \cref{sec: error analysis ENS} for the numerical schemes \eqref{eq: fully discrete fin elem approx ENS}, \eqref{eq: fully discrete fin elem approx cov ENS}, i.e. for $\bm{\eta}=0$.
Lastly, recall that for nodal test functions $\underline{\matdt\chi_j^n}(\cdot,t^{n-1}) =0$; see \eqref{eq: fully discrete transport property ENS}, and therefore for the implementation of \eqref{eq: L.M. ENS} the last term in the first equations is ignored.
  
 \begin{remark}\label{remark: formulation no eta_1 ENS}
 We can also consider a new scheme by rewriting the second equation of \eqref{weak lagrange hom NV ENS} with the help of the integration by parts formula \eqref{eq: integration by parts cont ENS} to obtain
 \begin{equation*}
     \int_{\Gat} \divg\bfu \,q  \, \ds + \int_{\Gat} \xi\, \bfu\cdot\bfng  \, \ds = \int_{\Gat} \xi\,V_\Ga  \, \ds,
 \end{equation*}
 and discretize the above expression instead. This means that we no longer need to approximate the inhomogeneous term $\eta_{1}$ ($\eta_{h,1} = -\kappa_hV_{\Gah}$; see \eqref{eq: right-hand side approx ENS}), so no approximation of a geometric quantity is needed anymore, unless we discretize the convective term \eqref{eq: ccov equiv expression ENS}; see \eqref{eq: implem c cov ENS} above. This also might be a reason to prefer to discretize \eqref{eq: L.M. ENS} with $c_h$. We note that we observed  similar convergence results.
 \end{remark}

The numerical results were implemented using the \emph{Firedrake} package \cite{FiredrakeUserManual}, where the linear systems were solved using GMRES with the help of built in multigrid preconditioners. The finite element approximations were carried out on sequences of regular bisection mesh-refinements. As noted in \cref{remark: diff formulations reason ENS}, 
  for \eqref{eq: fully discrete fin elem approx cov ENS} we consider  $\underline{k_\lambda=k_u-1}$, whilst for \eqref{eq: fully discrete fin elem approx ENS} we mainly set  $\underline{k_\lambda=k_u}$. We also observed that using the \eqref{eq: fully discrete fin elem approx ENS} scheme for $k_\lambda=k_u-1$, appeared to \emph{produce optimal convergence results}, although we were unable to find such theoretical results.

 We, finally, introduce the following notation for the error measure:
  \begin{equation*}
    \begin{aligned}
    \eu^{L^2(\ah)} &:= \Big(\Delta t \sum_{n=1}^N \norm{\nbgcovh(\bfu^n - \uh^n)}_{L^2(\Gahn)}^2 \Big)^{1/2}, \quad \bfe_{\divg \bfu}^{L^{\infty}({L^2})}:= \max_{0 \leq n \leq N}\norm{\divgh(\bfu^n-\uh^n)}_{L^2(\Gahn)},\\
        \eu^{L^2(H^1)} &:= \Big( \Delta t \sum_{n=1}^N \norm{\nbgh(\bfu^n - \uh^n)}_{L^2(\Gahn)}^2 \Big)^{1/2}, \quad \ep^{L^2({L^2})} := \Big(\Delta t \sum_{n=1}^N \norm{p^n-\ph^n}_{L^2(\Gahn)}^2\Big)^{1/2}, \\
       \eu^{L^{\infty}({L^2})} &:=  \max_{0 \leq n \leq N}\norm{\bfu^n-\uh^n}_{L^2(\Gahn)}, \qquad \qquad \ \   \bfe_{\bfPg \bfu}^{L^{\infty}({L^2})} := \max_{0 \leq n \leq N}\norm{\bfPh(\bfu^n-\uh^n)}_{L^2(\Gahn)},\\
        \bfe_{\bfng}^{L^{\infty}({L^2})} &:=  \max_{0 \leq n \leq N}\norm{\uh^n\cdot\nh-V_{\Gamma}}_{L^2(\Gahn)}.
    \end{aligned}
\end{equation*}
\subsection{Example 1: Moving sphere}\label{Sec: moving sphere ENS}
Similar to \cite{olshanskii2024eulerian}, this example concerns a sphere translating to the right. The evolving surface $\Gat$ is defined via a level set function as 
\begin{equation}
    D(\bfx,t) = (x_1-0.2t)^2 + (x_2)^2 + (x_3)^2 -1, \qquad \Gat = \{\bfx\in \mathbb{R}^3\ |\  D(\bfx,t) =0\}.
\end{equation}
with final time $T=2$. The normal velocity is given by $\FlVel = \frac{-\partial_tD}{\norm{\nb D}}\frac{\nb D}{\norm{\nb D}}$ = $V_{\Gamma}\bfng$ , cf. \cite{DziukElliott_acta}. Using the parameterization \begin{equation}\label{eq: projection time-sphere ENS}
    F(s,t) = (s_1+0.2t,s_2,s_3), \qquad p \in \mathcal{S},
\end{equation}
such that $\Gat = F( \mathcal{S},t)$, the evolving triangulated surface is constructed by creating an initial  mesh $\mathcal{S}$ of the sphere  and then evolving  the nodes via $F$. Hence the vertices evolve with a velocity $\bfw=\partial_tF(s,t) = (0.2t,0,0)^T$ and  we consider the tangential part of $\bfw$ as an arbitrary tangential velocity (ALE-scheme), cf. \cite{elliott2015error}. The viscosity parameter is $\mu =1/2$ and the density distributionis  $\rho=1$. Now, instead of using the example  in \cite{olshanskii2024eulerian}, our solution is defined as
\begin{equation}
    \bfu_T = \textbf{curl}_{\Ga} \psi \quad \text{with} \quad \psi = (x_1-0.2tx_3)x_2 - 2t, \qquad p=(x_1-0.2t)x_2+ x_3,
\end{equation}
where $\textbf{curl}_{\Ga} = \bfn \times\nbg$, which by \cite{reusken2018stream} we know  is tangential. Moreover $\divg\bfu_T =0$, while $p \in L^2_0(\Gat)$.  Despite not including the plot here, our experiments, would give the \emph{exact same convergence rates} for  the solution of \cite[Section 5]{olshanskii2024eulerian}. 
The data on the right-hand side $(\bff_h$, $\bm{\eta}_h)$, see \eqref{eq: L.M. ENS},
is an interpolation of the known smooth data.

We compare the primitive solutions with respect to  the parameter choice \circled{$k_\lambda$}.  First, we examine the case where $\underline{k_\lambda=k_u=2}$, for which the polynomial degree of the Lagrange multiplier  space $\Lambda_h$ is equal to the degree of the velocity  space $\bfV_h$. In this case, we use the formulation \eqref{eq: fully discrete fin elem approx ENS} (i.e. \eqref{eq: L.M. ENS} with $c_h(\bullet;\bullet,\bullet)$, see \cref{Sec: set-up ENS}) and employ  $\mathrm{\mathbf{P}}_{2}$ -- $\mathrm{P}_{1}$ --$\mathrm{P}_{2}$ \emph{Taylor-Hood iso-parametric finite elements}, so $k_g = k_u=2$. Second, we examine the $\underline{k_\lambda=k_u-1=1}$ case and  utilize the formulation \eqref{eq: fully discrete fin elem approx cov ENS} (i.e. \eqref{eq: L.M. ENS} with $c_h^{cov}(\bullet;\bullet,\bullet)$) according to the error analysis in \cref{sec: error analysis ENS}. So, we employ $\mathrm{\mathbf{P}}_{2}$ -- $\mathrm{P}_{1}$ -- $\mathrm{P}_{1}$ \emph{Taylor-Hood super-parametric surface finite elements}, i.e. with surface order approximation $k_g=3$. This reflects the limiting geometric factor $\bigo(h^{k_g-1})$ appearing in \cref{theorem: Velocity Error Estimates ENS,theorem: Pressures Error Estimate HR ENS}.
\begin{figure}
    \centering
    \includegraphics[width = \textwidth]{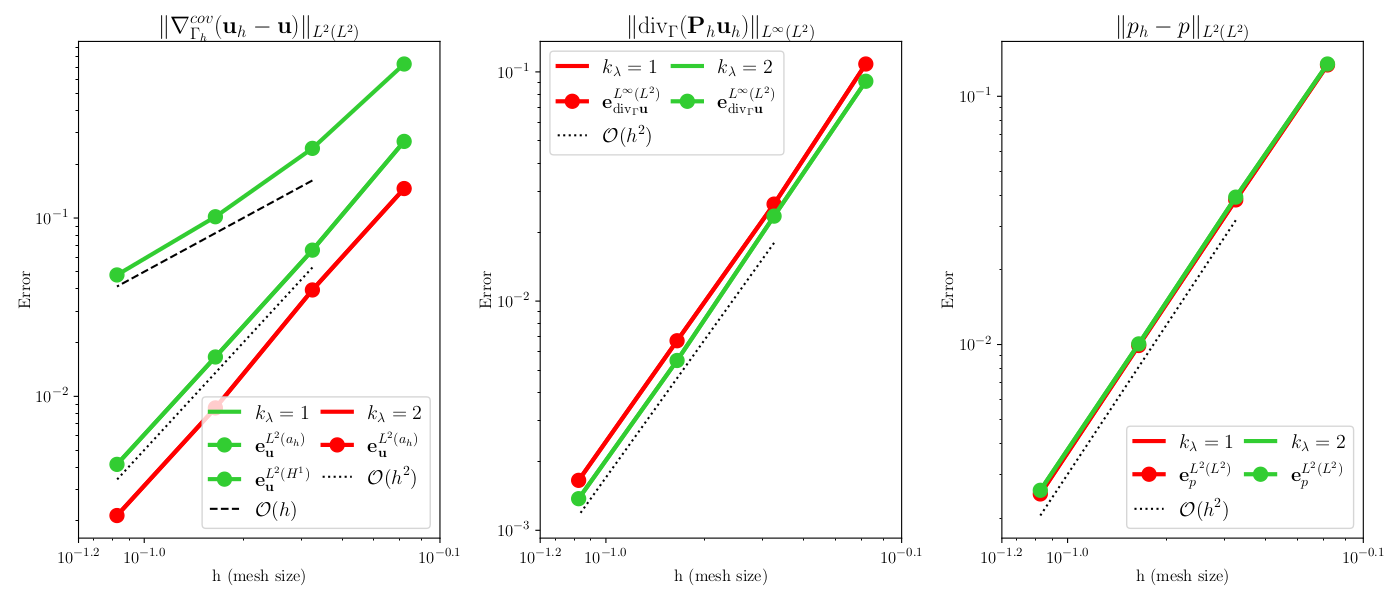}
    \caption{Moving-Right Sphere | Velocity - Pressure Errors $\bfe_{\bfu}^{L^{\infty}(\ah)}$, $\bfe_{\divg\bfu}^{L^{\infty}(L^2)}$,  $\bfe_{p}^{L^{2}(L^2)}$| Case 1: $k_{\lambda}=k_u$ corresponds to geometric approximation $k_g=k_u=2$ | Case 2: $k_{\lambda}=k_u-1$ correspond to geometric approximation $k_g=k_u+1=3$.}
    \label{fig: MovingSphere ah p}
\end{figure}

Starting  with an initial mesh size of $h_0=0.66$ and time-step $\Delta t_0 = 0.5$ we consider a series of regular bisection mesh-refinements, such that the spatial mesh size is halved while the temporal step is reduced by a factor of four, so that $\Delta t \sim h^2$. As mentioned, in this particular example we notice that the spatial error dominates the behavior of our numerical approximations.

In \cref{fig: MovingSphere ah p} we observe optimal rates for all our quantities. In particular for the energy and pressure estimates $\eu^{L^2(\ah)}$, $e_p^{L^2(L^2)}$ and we observe second order convergence, as expected from our theoretical results \cref{theorem: Velocity Error Estimates ENS,theorem: Pressures Error Estimate ENS,theorem: Pressures Error Estimate HR ENS}. Moreover, we notice reduced rate, $\bigo(h)$, for $\eu^{L^2(H^1)}$, in the $k_{\lambda}=k_u-1$ case, in line with the worse $H^1$-coercivity estimate \eqref{eq: coercivity and Korn's inequality Lagrange ENS}. This is not the case for $k_\lambda=k_u$ (as expected); see the improved $H^1$-coercivity estimate \eqref{eq: improved h1-ah bound ENS}.

The $L^{\infty}_{L^2}$ velocity errors are presented in \cref{fig: MovingSphere u Pu}. First, we report expected $\bigo(h^{2})$ convergence rates for the normal approximation $\bfe_{\bfng}^{L^{\infty}({L^2})}$ when $k_\lambda=k_u-1$ and better than expected $\bigo(h^{3})$ for $k_\lambda=k_u$; see \cite[Theorem 6.16]{elliott2024sfem} where $\bigo(h^{2.5})$ 
error bounds were established for the Stokes problem on stationary surfaces (or \cref{lemma: Error Bounds Ritz-Stokes ENS}) and \cite{elliott2025unsteady} where it was observed numerically for the unsteady Navier-Stokes equations on stationary surfaces.

\begin{figure}
    \centering
    \includegraphics[width = 0.9\textwidth]{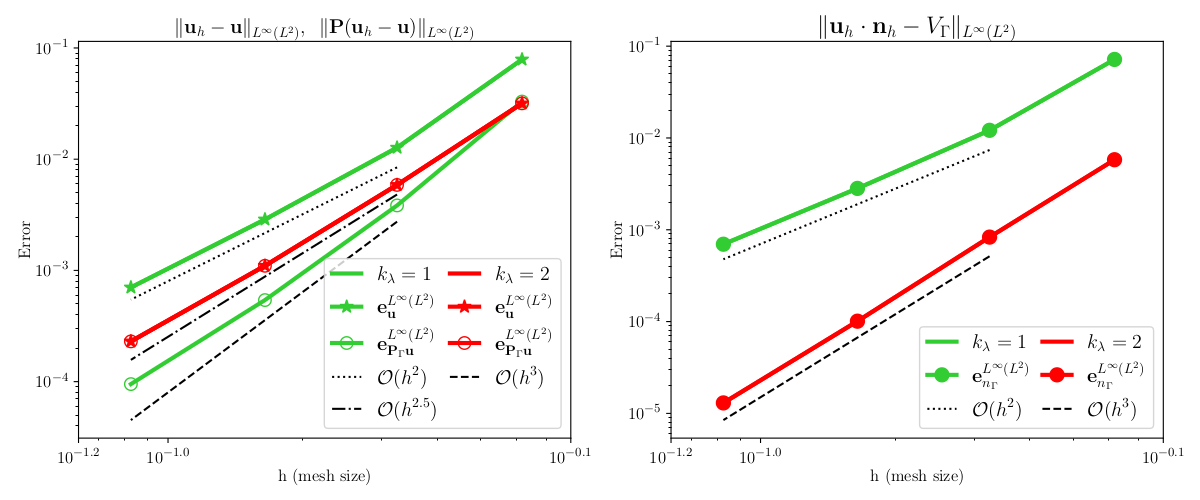}
    \caption{Moving-Right Sphere | Velocity Errors $\bfe_{\bfu}^{L^{\infty}(L^2)}$, $\bfe_{\bfPg\bfu}^{L^{\infty}(L^2)}$,  $\bfe_{\bfng}^{L^{\infty}(L^2)}$ | Case 1: $k_{\lambda}=k_u$ corresponds to geometric approximation $k_g=k_u=2$ | Case 2: $k_{\lambda}=k_u-1$ correspond to geometric approximation $k_g=k_u+1=3$.}
    \label{fig: MovingSphere u Pu}
\end{figure}

We also observe optimal convergence for the full $\eu^{L^{\infty}(L^2)}$ errors  in the $k_\lambda=k_u-1$ case; see \cref{theorem: Velocity Error Estimates ENS}, but notice higher than expected convergence rates $\bigo(h^{2.5})$ for the tangent part $\bfe_{\bfPg\bfu}^{L^{\infty}(L^2)}$. 
Similarly, we observe $\bigo(h^{2.5})$ convergence in both full and tangential $L^{\infty}_{L^2}$ velocity errors for the $k_\lambda=k_u$ case (it is expected that the full and tangential velocity errors would coincide in this case, since the normal velocity error is $\bigo(h^{3})$; see \cite[Theorem 6.14, Theorem 6.16]{elliott2024sfem}). Although this behavior is expected for the stationary surface Stokes (or unsteady Navier-Stokes) case, see \cite[Theorem 6.14, Theorem 6.16]{elliott2024sfem} or \cite{elliott2025unsteady} (in actuality in that case we expect an increase of one full order $\bigo(h^{3})$ and not half order $\bigo(h^{2.5})$),  it is not clear that this should be the case in the evolving setting. Indeed,  the Ritz-Stokes estimates for the material derivative in \cref{lemma: Error Bounds der Ritz-Stokes ENS,lemma: Error Bounds der Ritz-Stokes std ENS} do not indicate any increase in convergence. Therefore, we believe that this might be indicative of the material derivative Ritz-Stokes estimates not being sharp, e.g. \eqref{eq: Error Bounds der Ritz-Stokes L2 ENS} where we only prove an $\bigo(h^{\widehat{r}_u})=\bigo(h^{2})$ convergence. This is also speculated in \cref{remark: issues mat der ENS,remark: non tangential mat der ENS}, where we believe that it should be possible to show the  half-order increase, at least in the geometric part of the error.



Lastly, we also note a surprising super-convergence result for the $k_\lambda= k_u$ case, which was also observed in \cite[Section 7.3]{elliott2024sfem}, namely, that for affine finite element ($k_g=1$) we observe almost $\bigo(h^2)$ convergence rates across all of our quantities \cref{fig: MovingSphere kg1}.

\underline{About the choice of the approximation space $\Lambda_h$}, i.e. the choice of $k_\lambda$: Letting $k_\lambda= k_u$, while we do obtain optimal results even when using \emph{iso-parametric surface finite elements}, this comes at the expense of stability of the system, since for this choice the condition of the corresponding systems of equations scales much worse w.r.t. the mesh parameter $h$; for more details we refer to \cite{fries2018higher}.
\begin{figure}
    \centering
    \includegraphics[width = \textwidth]{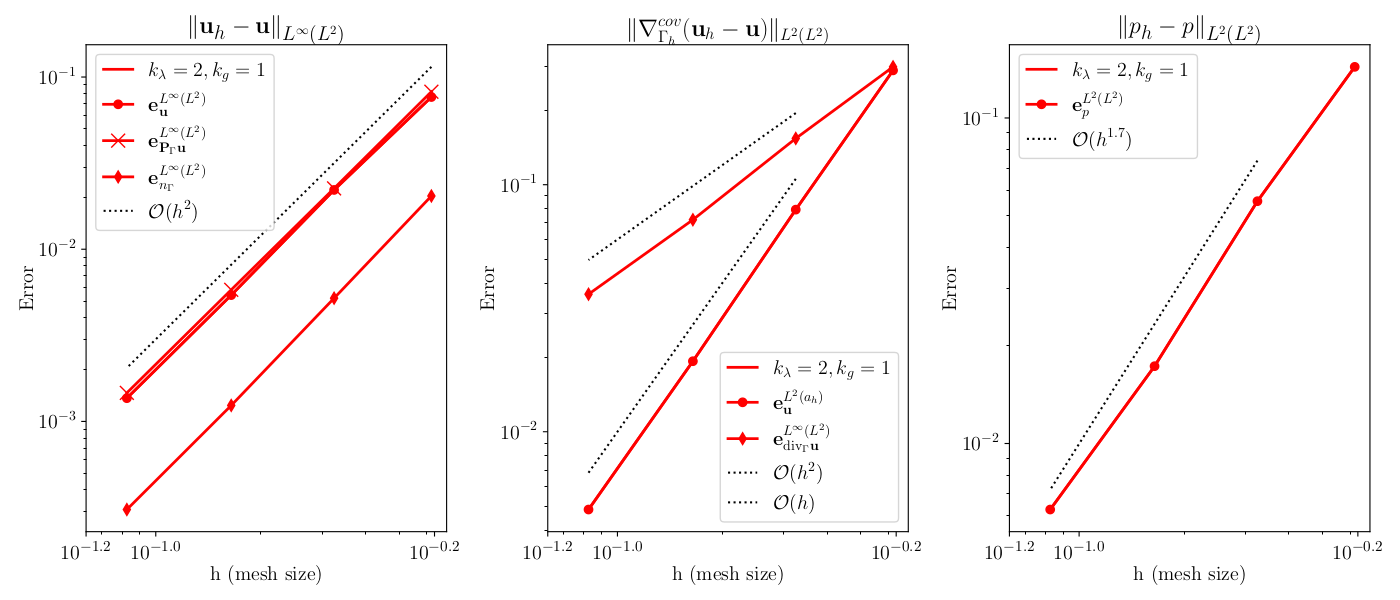}
    \caption{Moving-Right Sphere | Velocity - Pressure Errors | Flat Case: $k_{\lambda}=k_u=2$ corresponds to geometric approximation $k_g=1$.}
    \label{fig: MovingSphere kg1}
\end{figure}

\subsection{Example 2: Comparison Oscillating Sphere}\label{Sec: num comparison ENS}
In this example, we perform a simple comparison test for our numerical scheme \eqref{eq: L.M. ENS} based on the weak formulation \eqref{eq: NS Lagrange new ENS} (i.e.  $c^{(\cdot)} = c^{cov}$; see also \eqref{eq: fully discrete fin elem approx cov ENS} which we analyzed) when $k_\lambda=k_u-1$ and a penalty-based numerical scheme for the tangential Navier-Stokes on a passively evolving surfaces with prescribed velocity \cite{olshanskii2022tangential}; see \eqref{eq: P.M. ENS} below. The latter scheme was considered and studied extensively in \cite{olshanskii2024eulerian} where the authors used unfitted finite elements (TraceFem).
In our case, we present a naive numerical approximation \eqref{eq: P.M. ENS} (PM: Penalty Method) using fitted finite elements (ESFEM). Similar schemes were studied for the Stokes equations on stationary surfaces \cite{reusken2024analysis}, where a comparison to the Lagrange-based approach was carried out in \cite{elliott2024sfem,elliott2025unsteady}.
Furthermore, a scheme more similar to ours \eqref{eq: L.M. ENS} was also studied numerically in \cite{Krause2023ASurfaces}, where the authors, this time, used a penalty-based approach to constraint the normal velocity. However, since they   used a mesh regularization approach, and did not perform an error analysis, we decided against comparing it to our method. The penalty scheme we compare it with, instead, is the following:

\noindent {\textbf{(PM): }}
Given $\uh^0 \in \bfV_h^{0,div}$ and appropriate approximation for the data $\bff_h,\,{\eta}_h^n$, for $n=0,1,...,N$ find $\uhn \in \bfV_h^n$, $\{\phn,\lambda_h^n\}\in Q_h^n\times \Lambda_h^n$ such that 
\begin{align}
\!\!\!\!\!\!\!\!\!\!\begin{cases}
\tag{PM}
        \label{eq: P.M. ENS}
\frac{1}{\Delta t}\Big(\mh(\uhn,\bfP_{h}^{n}\vhn) - \mh(\uhnN,\bfP_{h}^{n-1}\vh^{n-1})\Big) - \gh(\TrVel;\uhn,\bfP_{h}^{n}\vhn) + \ahhat^T(\uhn,\vhn)\\
 \quad + c_h^{cov}(\underline{\uhnN}(\cdot,t_n),\uhn,\vhn) 
 +b_h(\vhn,\phn) = \mh(\bff_h^n,\bfP_{h}^{n}\vhn) + \mh(\uhnN,\bfP_{h}^{n-1}\underline{\matdt\vhn}(\cdot,t_{n-1})), \\
 \qquad\qquad\qquad\qquad\qquad\qquad\qquad\qquad\qquad  \ \ \  \! b_h(\uhn,\qhn) = m_h({\eta}_h^n,\qhn). 
\end{cases}
 \end{align}
for all $\vhn \in \bfV_h^n$ and $\vhnN \in \bfV_h^{n-1}$ and $\{\qhn,\xihn\}\in Q_h^n\times\Lambda_h^n$, where $\ahhat^T(\cdot,\cdot)$ is the  bilinear form $\hat{a}_n(\cdot,\cdot)$ in \cite[Eq. (4.1)]{olshanskii2024eulerian}, without the added stabilization term (in our setting $w_N = V_{\Gamma}$). Therefore, hidden also in this bilinear form is the penalty term $\tau\mh(\uh^n\cdot\nhtil^n,\vh\cdot\nhtil^n)$ which corresponds to the tangentiality condition $\bfu^n\cdot\bfn =0$. According to \cite{olshanskii2024eulerian}, to obtain optimal convergence we consider a higher-order approximation for the normal $\nhtil$, than the usual $\nh$, i.e. $\norm{\bfn-\nhtil}_{L^{\infty}(\Gah)}\leq ch^{k_g+1}$, and also choose the penalty term appropriately, e.g. $\tau = \frac{h^{-2}}{2}$. For the construction of such normal $\nhtil$ see \cite[Remark 3.3]{hansbo2020analysis} for further details. Finally, notice that the source terms $\bff_h^n$, ${\eta}_h^n$ are different from the source terms in \eqref{eq: L.M. ENS}. For instance, now, ${\eta}_h = \eta_{h,1}$ and so it does not include the term $\eta_{h,2}=V_{\Gah}$; see \eqref{eq: right-hand side approx ENS}, since we approximate the tangential Navier-Stokes on a passively evolving surface instead.

\begin{figure}
    \centering
    \includegraphics[width = \textwidth]{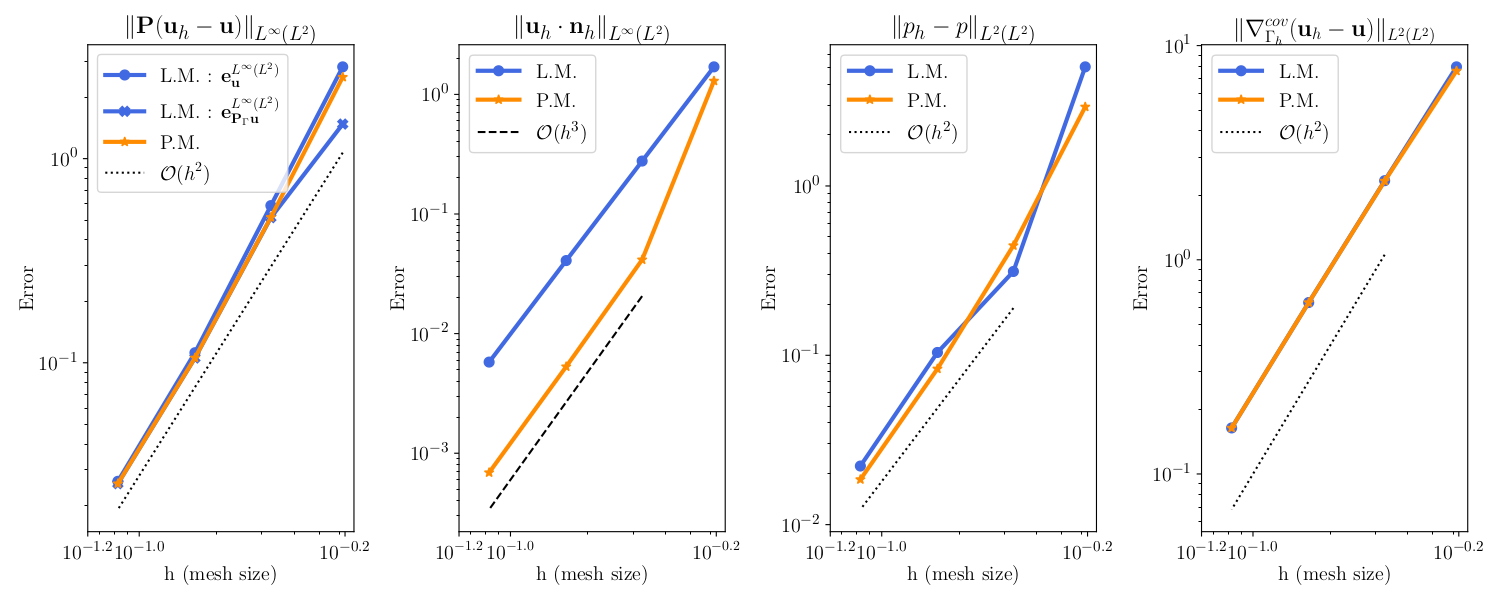}
    \caption{Oscillating Sphere | Errors | \eqref{eq: L.M. ENS}: \{$k_u=2, \, k_{pr}=k_{\lambda}=1,\, k_g=3$\},  \eqref{eq: P.M. ENS}: \{$k_u=2, \, k_{pr}=1,\, k_g=2,\, \tau \sim h^{-2}, \,  |\nhtil-\bfn| \sim h^{k_g+1}\sim h^{3}$\}}
    \label{fig: Osc Sphere All ENS}
\end{figure}

To compare the two methods, we consider an oscillating sphere $\mathcal{S}$ given by
\begin{equation}
    d(\bfx,t) = \norm{\bfx}_2^2 - r(t), \qquad \text{with } r(t)=1 + \frac{1}{4}sin(2\pi t),
\end{equation}
where $T=1$ the final time, and therefore the surface normal velocity $\FlVel$ is given by $\FlVel = V_{\Gamma}\bfng$, where $V_{\Gamma}\t = r'(t)=\frac{\pi}{2}cos(2\pi t)$. We also describe an exact tangential solution as
\begin{equation}
\begin{aligned}
    \bfu = \textbf{curl}_{\Ga} \psi, \ &\text{with} \ \psi =  (1-2t)\frac{1}{2\pi}\cos(2\pi x_1)\cos(2\pi x_2)\cos(2\pi x_3),\\
    &p = \sin(\pi x_1)\sin(2\pi x_2)\sin(2\pi x_3).
\end{aligned}
\end{equation}
where $\textbf{curl}_{\Ga} = \bfn \times\nbg$, which by \cite{reusken2018stream} we know it is tangential and moreover $p \in L^2_0(\Gat)$. The data then on the right-hand side of both \eqref{eq: L.M. ENS} and \eqref{eq: P.M. ENS}, can be calculated numerically with the help of the exact solution as interpolation of the smooth data.

For \eqref{eq: L.M. ENS}  we use $\mathrm{\mathbf{P}}_2$ -- $\mathrm{P}_1$ -- $\mathrm{P}_1$ \emph{Taylor-Hood} \emph{super-parametric finite elements}, with surface approximation of order $k_g=3$. While for \eqref{eq: P.M. ENS} we use $\mathrm{\mathbf{P}}_2$ -- $\mathrm{P}_1$ \emph{Taylor-Hood} iso-parametric surface finite elements with surface order approximation $k_g = 2$, while as noted before, for the penalty parameter we use $\tau = \frac{h^{-2}}{2}$ and the improved normal $\nhtil$.

We set the viscosity parameter $\mu =2\times10^{-2}$ and the density distribution $\rho=1$, for both cases, and use a sufficiently accurate quadrature rule, starting with an initial mesh size of $h_0=0.62$ and a time step $\Delta t_0 = 0.5$ and perform a series of regular bisection mesh-refinements, such that the spatial refinement is halved while the temporal refinement is reduced by a factor of four, ensuring that $\Delta t \sim h^2$. That, then, allows us to obtain optimal convergence; see \Cref{theorem: Velocity Error Estimates ENS,theorem: Pressures Error Estimate HR ENS}.

Now, \cref{fig: Osc Sphere All ENS} depicts the experimental order of convergence for both \eqref{eq: L.M. ENS} and \eqref{eq: P.M. ENS}. We observe the same optimal convergence for all our quantities. In particular, the only difference (in magnitude of error to precise) appears in the normal approximations, which was expected. Indeed, the normal convergence rate for \eqref{eq: P.M. ENS} aligns with the known literature \cite{reusken2024analysis,elliott2024sfem,elliott2025unsteady} on stationary surfaces, while we observe an improved rate $\bigo(h^{3})$ for \eqref{eq: L.M. ENS}  \cite{elliott2024sfem,elliott2025unsteady} (expected an $\bigo(h^{2})$ for $k_{\lambda}=k_g-1$). For the normal approximation, we want to note, once again, that \eqref{eq: L.M. ENS} compares the quantity $\max_{0 \leq n \leq N}\norm{\uh^n\cdot\nh-V_{\Gamma}}_{L^2(\Gahn)}$, while the penalty method \eqref{eq: P.M. ENS} the quantity $\max_{0 \leq n \leq N}\norm{\uh^n\cdot\nhtil^n-0}_{L^2(\Gahn)}$. 


\begin{figure}[h]
\centering
  \subfloat[Velocity $t=0$]{\includegraphics[width=0.45\textwidth]{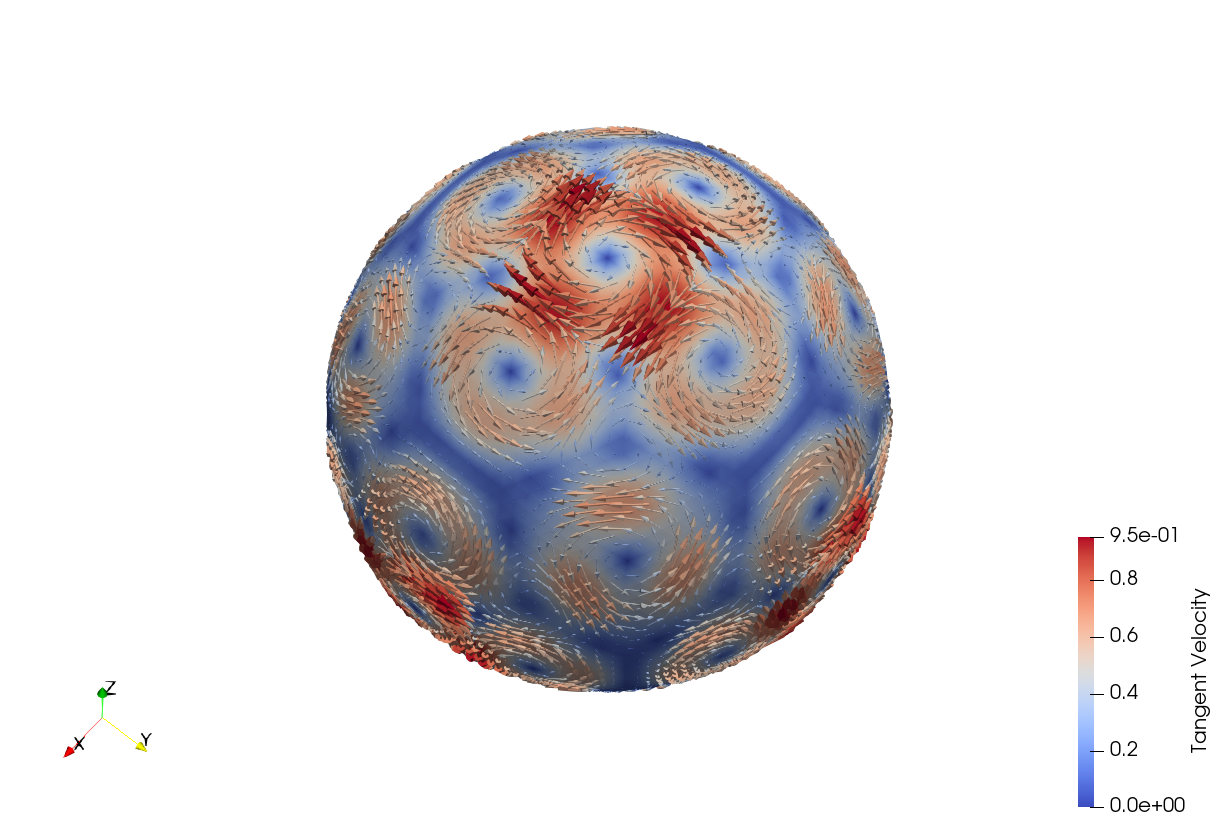}}
  \subfloat[Velocity $t=0.15$]{\includegraphics[width=0.45\textwidth]{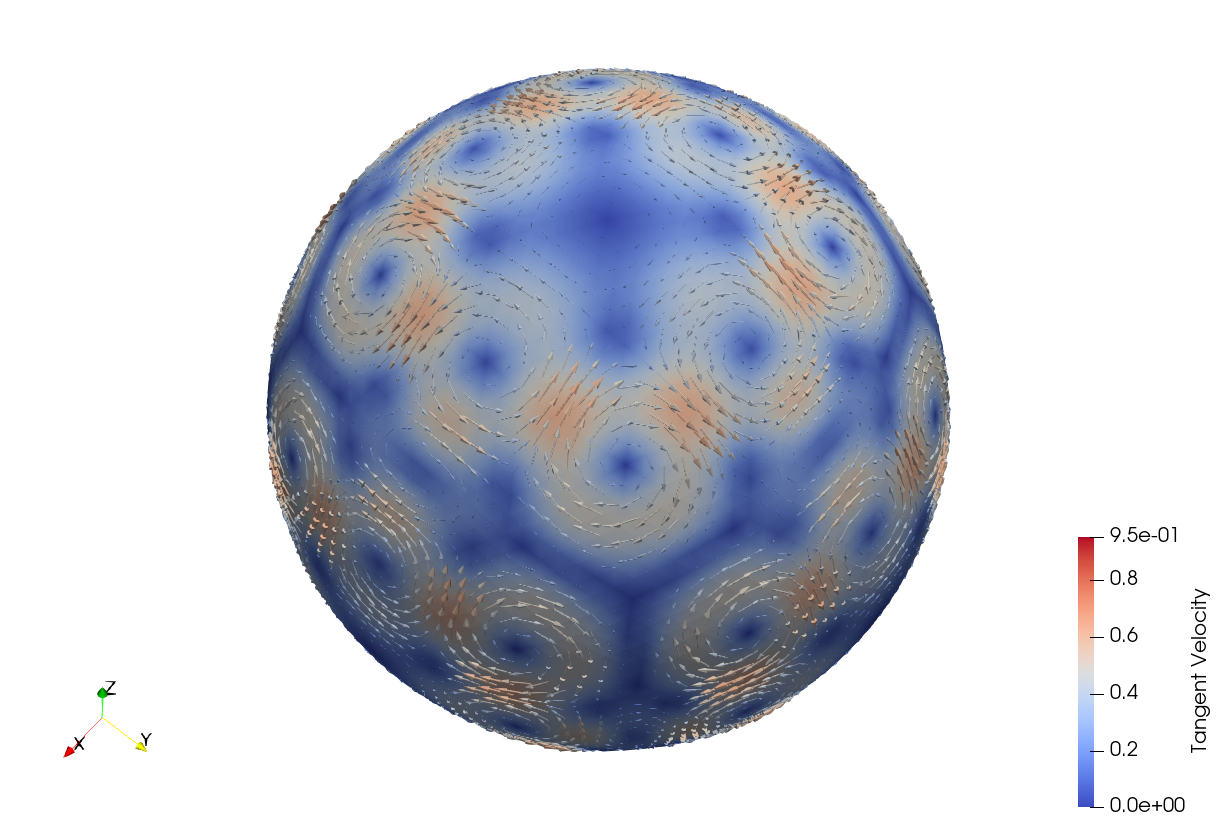}}
  \hfill
   \subfloat[Velocity $t=0.85$]{\includegraphics[width=0.45\textwidth]{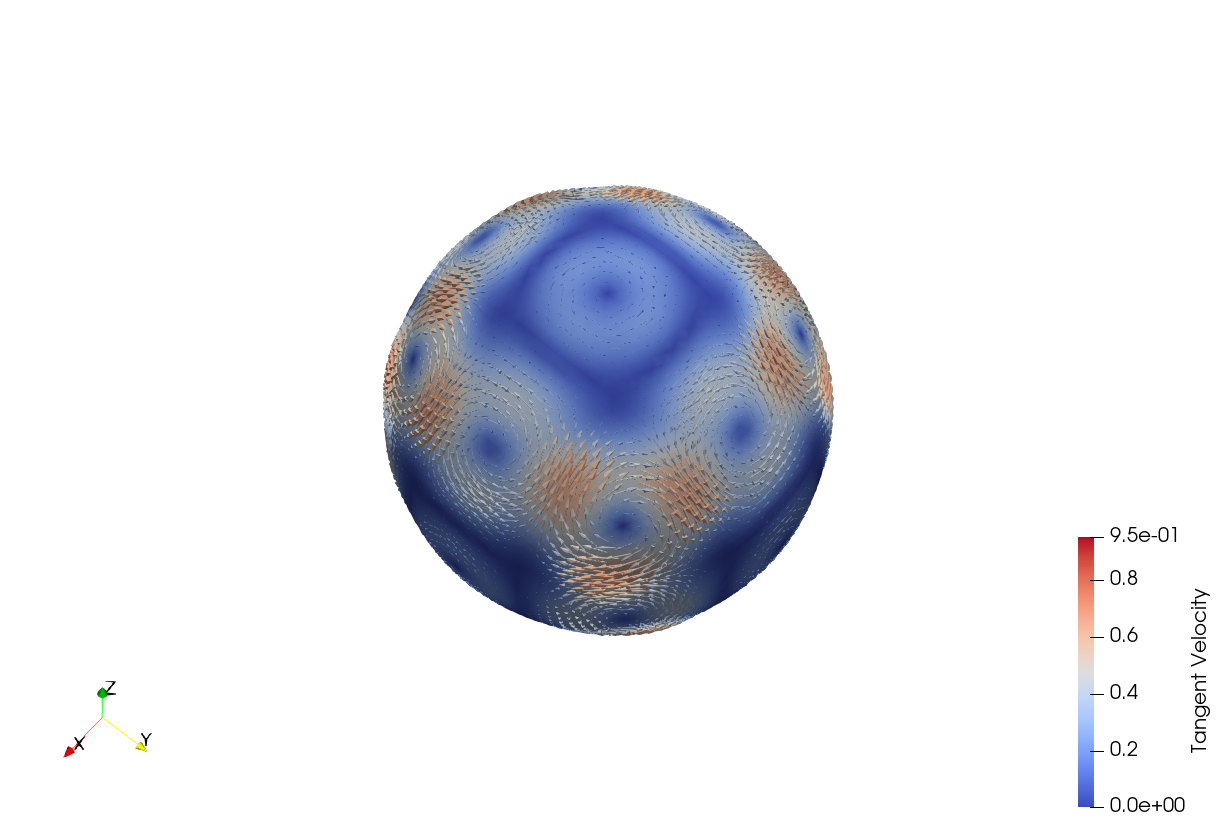}}
   \hfill
    \subfloat[Velocity $t=1$]{\includegraphics[width=0.45\textwidth]{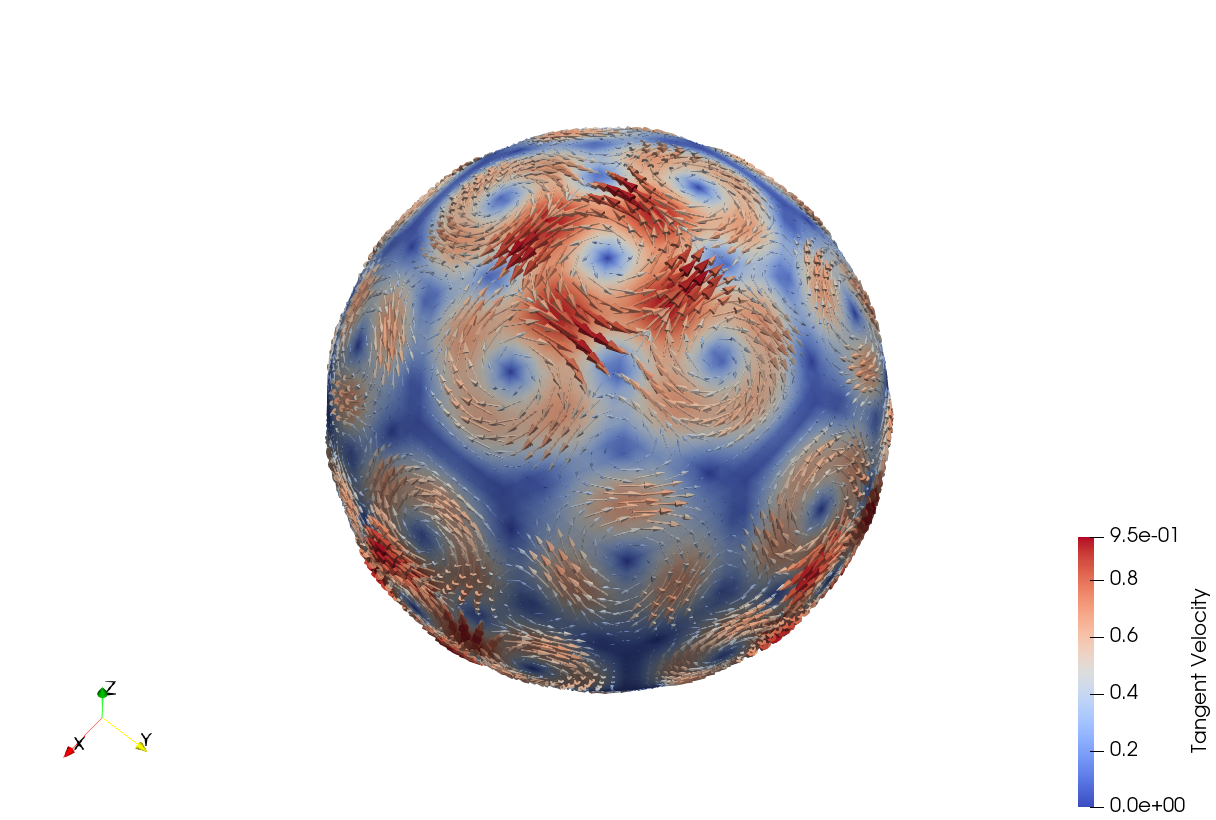}}
  \caption{Oscillating Sphere | Tangent Velocity $\bfPh\uh^n$ at different times-steps for $\mathbf{k_g=3}$, $\mathbf{k_u=2}$, $\mathbf{k_{pr}=k_{\lambda}=1}$, with mesh size $h=0.09$ and $\Delta t =2.5\times10^{-2}$. The length and direction of the arrows depict the strength and orientation of the current.}
  \label{fig: Dziuk Velocity}
\end{figure}

\section{Brief Summary and Outlook}

\quad \  \textbf{Brief Summary:}    
We addressed the numerical study of the evolving surface Navier--Stokes equations with prescribed normal velocity. Using Taylor--Hood surface finite elements we developed two fully time-discrete backwards Euler discretizations schemes, see \cref{sec: The Fully Discrete Scheme ENS}, depending on the choice of  the convective term $c^{(\cdot)}(\bullet;\bullet,\bullet)$ and approximation of the Lagrange multiplier involving the normal velocity constraint $\lh$. In order to simplify our calculations, we developed our numerical analysis
based on the discretizations of the solenoidal formulations introduced in \cref{sec: Variational formulation ENS}; see \cref{remark: about f source discrete ENS} also. For these simplified schemes in \cref{Sec: asssumptions about discrete scheme ENS,sec: error analysis ENS}, we established well-posedness as well as  optimal $L^2_{\ah}$ and $L^2_{L^2}\times L^2_{H^{-1}_h}$ (or $L^2_{L^2}\times L^2_{L_2}$) error bounds in \cref{sec: error analysis ENS}, for the velocity and pressures, respectively (again depending on the choice of the discretization analyzed and approximation space of $\lh$). To achieve this, we used two important ideas. First, we introduced a new surface Ritz-Stokes projection, presented in \cref{sec: Ritz-Stokes Projection ENS}, along with appropriate interpolation bounds. Second, we proposed a novel \emph{Leray time-projection} \eqref{eq: discrete time Leray proj ENS} to ensure weakly divergence conformity for our discrete velocity
solution at two different time-steps (surfaces) and therefore reduce, in a sense, our analysis to that of the  stationary case \cite{elliott2025unsteady}. That is, proving an auxiliary bound for the approximation of the time-derivative (material derivative) in a weaker dual norm, with the help of an inverse surface Stokes operator, to establish the pressure bounds. Lastly, we also provided numerical results to accompany and validate our theoretical results; see \cref{sec: Numerical results ENS}.

\vspace{5mm}

\textbf{Outlook:}  Let us briefly discuss some open problems and/or extensions, regarding the numerical study of the evolving surface Navier-Stokes equation \eqref{eq: NS Lagrange begin ENS}:
\begin{itemize}
    \item 
    In order to present a complete analysis, i.e. prove also an improved $L^2_{L^2}$(or $L^{\infty}_{L^2}$)-norm velocity error bounds, in accordance with our numerical investigation in \cref{Sec: moving sphere ENS}, it seems one would need to improve the material derivative estimates of the surface Ritz-Stokes projection in \cref{lemma: Error Bounds der Ritz-Stokes ENS}; see \cref{remark: non tangential mat der ENS,remark: issues mat der ENS} for more detailed discussion where we also proved some initial improvements. In general, though, the issue stems from the fact that the material derivative is no longer tangential $\matn\bfu\cdot\bfng\neq0$ even if the velocity is $\bfu\cdot\bfng=0$. 
    \item  In \cite{olshanskii2024eulerian} an Eulerian finite element method was studied for the tangential surface Navier-Stokes on a passively evolving surface (TSNSE) using Trace-FEM. Based on this, in the numerical experiment in \cref{Sec: num comparison ENS}, we introduced a new naive penalty approximation scheme (\cref{eq: P.M. ENS}) for (TSNSE) using evolving surface finite elements (ESFEM) instead. Following our analysis in this paper, we believe that it should be possible to prove stability and error bounds for \cref{eq: P.M. ENS} (or similar schemes), since then the analysis simply reduces to finding an appropriate Ritz-Stokes projection estimate \cref{lemma: Error Bounds Ritz-Stokes ENS,lemma: Error Bounds der Ritz-Stokes ENS}.
    \item The introduction of the \emph{Leray time-projection} \eqref{eq: discrete time Leray proj ENS} enables us, in a sense, to reduce the numerical analysis of the evolving surface Navier-Stokes equation  to the analysis of the stationary case one \cite{elliott2025unsteady}. Since the latter follows the arguments and ideas of standard NS equations on bulk domains (adapted to surfaces) e.g. \cite{JohnBook2016,ArchUnifDataAssimNS2020,FrutosGradDivNS2016}, we believe that results presented in the bulk domain case, e.g. grad-div stabilization \cite{FrutosGradDivNS2016,FrutosGradDivOseen2016} or other stabilization schemes \cite{fries2018higher,JohnBook2016}, should be easily accessible to evolving surfaces.
\end{itemize} 
\appendix
\section{Differential operators}\label{appendix: differential operators ENS}
In this appendix, we focus on the properties of surface differential operators. We assume that the surface $\Ga$ is smooth enough. All smooth extensions will be omitted for clarity reasons, i.e. $\partial_i f^e = \partial_i f$. Considering two vectors $\bfa, \, \bfb \, \in \mathbb{R}^3$ we use the Einstein summation convention, that is, $\bfa \cdot \bfb = \sum_{i=1}^3a_ib_i = a_ib_i$. We also introduce the following abbreviation for the i-th partial derivative, $\partial_i := \frac{\partial}{\partial x_i}$. Then 
$(\nb \bfu)_{ij} = \partial_j(u_i), \ 1\leq i,\,j \leq 3$.

Using these notation, we can express the surface differential operators for suitable scalar $f : \mathbb{R}^3 \to \mathbb{R}$
functions and vector-valued functions $\bfv  : \mathbb{R}^3 \to \mathbb{R}^3$ as
\begin{equation}\label{eq: surface operators abbreviations}
    \begin{aligned}
        &(\nbg f)_i = P_{ij}\partial_j f, \qquad \quad
        (\nbg \bfv)_{ij} =  \partial_k v_iP_{kj}, \qquad \quad (\nbgcov (\bfv))_{ij} = P_{jk}\partial_k v_lP_{li},\\
        &\qquad \qquad \qquad  \qquad\divg(\bfv) = (\nbgcov (\bfv))_{ii} = P_{ik}\partial_k v_lP_{li} = \partial_k v_iP_{ki},
    \end{aligned}
\end{equation}
where $(\bfPg)_{ij} = P_{ij} $  the orthogonal projection defined in \cref{Sec: The closed smooth surface ENS}. The following relations also hold:
\begin{equation}\label{eq: frequent relations}
    \begin{aligned}
        P_{ij} = \delta_{ij} - n_in_j,\quad P_{ij}n_j=0,\quad H_{ij} = \partial_i(n_j) = \partial_j(n_i)
        \\\partial_k{P_{ij}} = -n_iH_{kj} - n_jH_{ki}, \quad P_{ik}H_{kj} = H_{ij}.
    \end{aligned}
\end{equation}
Many of the results presented here can be found in \cite{jankuhn2018incompressible,Miura2020}.

\begin{lemma}[Identities]
For vector-valued functions $\bfw,\bfv \in C^1(\Ga)^3$, where $\bfv = \bfv_T + \bfv_n$ with $\bfv_n = v_N \bfng$, the following relationships hold: \vspace{-1mm}
\begin{align}
\label{eq: identity dot product nbgdir}
\nbg(\bfw\cdot\bfv) = \nbg\bfw^{t}\bfv + \nbg\bfv^{t}\bfw \\ 
        \nbg\bfv = \nbgcov \bfv +  \bfng\nbg(\bfv \cdot \bfng)^t - \bfng(\bfH \bfv)^t  \label{eq: identity dir to cov},\\
         \nbgcov \bfv = \nbgcov \bfv_T + \bfH v_N \label{eq: nbgcov split}.
    \end{align}
\end{lemma}
\begin{proof}
Eq. \eqref{eq: identity dot product nbgdir} is easy to see, since 
\begin{align*}
   \nbg(\bfw\cdot\bfv) = P_{ij} \partial_{j}(w_{k}v_{k}) = P_{ij}(\partial_{j}w_{k})v_{k} + (P_{ij}\partial_{j}v_{k})w_{k}.
\end{align*}
To prove \eqref{eq: identity dir to cov}   we notice that 
\begin{equation}\label{eq: nbgu to nbgcov}
   \nbg\bfv = \nbgcov\bfv + \bfng\bfng^t\nbg(\bfv_T) +  \bfng\bfng^t\nbg(\bfv_n). 
\end{equation}
We keep the first term unchanged and concentrate on the last two terms. So we have 
\begin{equation}
    \begin{aligned}\label{eq: nbgPu}
        \bfng\bfng^t\nbg\bfv_T &= (\bfng\bfng^t\nbg \bfPg\bfv)_{ij} =  
        n_i n_l\partial_k\big((\bfPg\bfv)_l\big)P_{kj} = n_i n_l \Big(\partial_k(P_{lm})v_m + P_{lm}\partial_k(v_m) \Big)P_{kj}\\
        &=  n_i n_l \big(-H_{km}n_lv_m - H_{kl}n_m v_m \big)P_{kj}  =  -H_{jm}v_m n_i = -\bfng(\bfH\bfu)^t,
    \end{aligned}
\end{equation}
where we have used the relations in \eqref{eq: frequent relations}. Now considering that $\bfv_n = (\bfu\cdot \bfng) \bfng =  v_N \bfng$, we may also obtain the following \vspace{-1mm}
\begin{equation}
    \begin{aligned}\label{eq: nbguN}
         \bfng\bfng^t\nbg\bfv_n &= n_i n_l \partial_k\big( v_N n_l  \big)P_{kj} = n_i n_l \partial_k(v_N)n_lP_{kj} + \underbrace{n_i n_l \partial_k(n_l)P_{kj}v_N}_{=0}  \\[-10pt]
        &= n_i\partial_k(v_N)P_{kj} =  \bfng\nbg^t v_N.
    \end{aligned}
\end{equation}\vspace{-3mm}

\noindent Thus, from \eqref{eq: nbgu to nbgcov}, \eqref{eq: nbgPu}, \eqref{eq: nbguN}, equation \eqref{eq: identity dir to cov} follows. Finally, for \eqref{eq: nbgcov split}, applying analogous calculations as in \eqref{eq: nbguN}, we obtain \vspace{-2mm}
\begin{equation*}
    \begin{aligned}
        \nbgcov\bfv_n &= P_{jk}\partial_k\big( v_N n_l  \big)P_{li} = \underbrace{P_{jk}\partial_k(v_N)n_lP_{li} }_{=0} + P_{jk}\partial_k(n_l)P_{li}v_N \\
        &= P_{jk}H_{lk}P_{li} v_N = H_{ij}v_N =  \bfH v_N.
    \end{aligned}
\end{equation*}\vspace{-5mm}

\end{proof}

\section{Transport formulae}\label{appendix: Transport formulae ENS}
\noindent The following identity, which can be found in \cite[Lemma 2.6]{dziukCons2013}, plays an important role for the transport formula:
\begin{equation}\label{eq: commutation derivative ENS}
    \matn \nb_{\Gat} g = \nb_{\Gat} \matn g - \mathcal{D}_{\Ga}(\FlVel)\nb_{\Gat}g,
\end{equation}
where (by our notation in \cref{sec: scalar-vector func ENS} and \eqref{eq: identity dot product nbgdir}) $\mathcal{D}_{\Ga}(\FlVel)$ the surface deformation tensor, written as
\begin{equation}\label{eq: def tensor transport 1 ENS}
\begin{aligned}
    \mathcal{D}_{\Ga}(\FlVel)_{ij} &= (\nb_{\Ga})_i(\FlVel)_j + n_i n_{k} (\nb_{\Ga})_j (\FlVel)_k=\nb_{\Ga}^t\FlVel - \bfng \bfng^t \nb_{\Ga}\FlVel\\
    &= \nb_{\Ga}^t\FlVel + \bfng (\nbg^tV_{\Ga} - (\bfH\bfv_n)^t) =  \nb_{\Ga}^t\FlVel + \bfng \nbg^tV_{\Ga},
    \end{aligned}
\end{equation}
where we omit the time argument (t), i.e. $\Gat = \Ga$, and where $\bfv\cdot \bfng = V_{\Gamma}$ by \eqref{eq: NS Lagrange ENS}.
\begin{lemma}[Transport formulae I]\label{lemma: Transport formulae app ENS}
    Let $\mathcal{M}\t$ be an evolving surface with normal velocity $\FlVel=V_{\Ga}\bfng$. Assume that the functions $f, \,g$ and the vector field $\bfw$ are sufficiently smooth, so that all the following quantities exist, then
    \begin{align}
    \label{eq: Transport formulae 1 app ENS}
        \frac{d}{dt} \int_{\mathcal{M}\t} g \,\ds &= \int_{\mathcal{M}\t} \matn g + g \,\divg(\bfv_n) \,\ds,\\
         \label{eq: Transport formulae 2 app ENS}
         \frac{d}{dt} \int_{\mathcal{M}\t} \bfw \cdot \nbg g \, \ds &=  \int_{\mathcal{M}\t} \matn \bfw \cdot \nbg g + \bfw \cdot \nbg \matn g   + \bfw \cdot \mathcal{B}_{\Ga}^{\diver}(\FlVel)\nbg g \, \ds, 
         \\
         &\hspace{-52.5mm} \text{ where the tensor }  \mathcal{B}_{\Ga}^{\diver}(\FlVel) = \mathrm{I}\,\divg(\FlVel) - \mathcal{D}_{\Ga}(\FlVel)\,(\mathrm{I}\text{ the } 3\times 3 \text{ Identity matrix}) \text{ and,}\nonumber
         \\
        \label{eq: Transport formulae 3 app ENS}
        \frac{d}{dt} \int_{\mathcal{M}\t} \mathcal{A}_{\Ga}\nbg f \cdot \nbg g \, \ds &=  \int_{\mathcal{M}\t} \mathcal{A}_{\Ga}\nbg \matn f \cdot \nbg g + \mathcal{A}_{\Ga}\nbg f \cdot \nbg \matn g \, \ds\\
        &\ + \int_{\mathcal{M}\t} \mathcal{B}_{\Ga}(\FlVel, \mathcal{A}_{\Ga}) \nbg f \cdot \nbg g \, \ds,\nonumber
    \end{align}
where the tensor $\mathcal{B}_{\Ga}(\FlVel, \mathcal{A}_{\Ga})$ is given by 
\begin{equation}\label{eq: mathcal b tensor}
 \mathcal{B}_{\Ga}(\FlVel, \mathcal{A}_{\Ga}) = \matn \mathcal{A}_{\Ga} + \divg(\FlVel) \mathcal{A}_{\Ga} - 2D(\FlVel, \mathcal{A}_{\Ga}),   
\end{equation} \vspace{-3mm}

\noindent with the deformation tensor defined by
$D(\FlVel, \mathcal{A}_{\Ga})_{ij} = \frac{1}{2} (\mathcal{A}_{\Ga} \nbg \FlVel + \nbg^t \FlVel \mathcal{A}_{\Ga}^t)$.
\end{lemma}
\begin{proof}
   The proof \eqref{eq: Transport formulae 1 app ENS}, \eqref{eq: Transport formulae 2 app ENS}, \eqref{eq: Transport formulae 3 app ENS} are given in \cite[Lemma A.1]{elliott2015error}, or \cite[Lemma 2.1, Lemma 4.2]{DziukElliott_L2}. 
\end{proof}

\begin{lemma}[Transport formulae II]\label{lemma: Transport formulae app II ENS}
Let $\mathcal{M}\t$ be an evolving surface with normal velocity $\FlVel=V_{\Ga}\bfng$. Assume that the the vector fields $\bfw, \, \bfv$ are sufficiently smooth, so that all the following quantities exist, then the following transport formulae hold
    \begin{align}\label{eq: Transport formulae 4 app ENS}
         &\!\!\!\!\!\!\!\!\!\!\!\!\!\!\!\frac{d}{dt} \int_{\mathcal{M}\t} \nbgcov \bfw : \nbgcov \bfv \,\ds =  \int_{\mathcal{M}\t} \nbgcov \matn \bfw : \nbgcov \bfv + \nbgcov \bfw : \nbgcov \matn \bfv \,\ds  +d(\FlVel;\bfw,\bfv), \\
         \label{eq: Transport formulae 55 app ENS}
         &\!\!\!\!\!\!\!\!\!\!\!\!\!\!\!\frac{d}{dt} \int_{\mathcal{M}\t} \nbg^{cov,t} \bfw : \nbg^{cov,t} \bfv \,\ds =  \int_{\mathcal{M}\t} \nbg^{cov,t} \matn \bfw : \nbg^{cov,t} \bfv + \nbg^{cov,t} \bfw :\nbg^{cov,t} \matn \bfv \,\ds  +d_1(\FlVel;\bfw,\bfv),\\
         \label{eq: Transport formulae 66 app ENS}
         &\!\!\!\!\!\!\!\!\!\!\!\!\!\!\!\frac{d}{dt} \int_{\mathcal{M}\t} \nbg^{cov,t} \bfw : \nbgcov \bfv \,\ds =  \int_{\mathcal{M}\t} \nbg^{cov,t} \matn \bfw : \nbgcov \bfv + \nbg^{cov,t} \bfw :\nbgcov \matn \bfv \,\ds  +d_2(\FlVel;\bfw,\bfv),\\
         \label{eq: Transport formulae 77 app ENS}
         &\!\!\!\!\!\!\!\!\!\!\!\!\!\!\!\frac{d}{dt} \int_{\mathcal{M}\t} \nbgcov \bfw : \nbg^{cov,t} \bfv \,\ds =  \int_{\mathcal{M}\t} \nbgcov \matn \bfw : \nbg^{cov,t} \bfv + \nbgcov \bfw :\nbg^{cov,t} \matn \bfv \,\ds  +d_3(\FlVel;\bfw,\bfv),
    \end{align}
where $d(\FlVel;\bfw,\bfv),\, d_1(\FlVel;\bfw,\bfv),\,d_2(\FlVel;\bfw,\bfv),\,d_3(\FlVel;\bfw,\bfv)$ are given by
    \begin{align}\label{eq: Transport formulae d app ENS}
       \!\!\!\!\!\!\!\!\!\!\!\!\!\!d(t;\FlVel;\bfw,\bfv) &:= \int_{\mathcal{M}\t} \nbg \bfw : \nbgcov \bfv \Tilde{\mathcal{B}}_{\Ga}(\FlVel,\bfPg) \,\ds + \int_{\mathcal{M}\t} \nbgcov \bfw : \nbg \bfv \Tilde{\mathcal{B}}_{\Ga}(\FlVel,\bfPg) \,\ds,\\
       \label{eq: Transport formulae d1 app ENS}
       \!\!\!\!\!\!\!\!\!\!\!\!\!\!d_1(t;\FlVel;\bfw,\bfv) &:= \int_{\mathcal{M}\t} \Tilde{\mathcal{B}}_{\Ga}(\FlVel,\bfPg) \nbg^{t} \bfw : \nbg^{cov,t} \bfv  \,\ds + \int_{\mathcal{M}\t} \Tilde{\mathcal{B}}_{\Ga}(\FlVel,\bfPg) \nbg^{cov,t} \bfw : \nbg^{t} \bfv  \,\ds,\\ \label{eq: Transport formulae d2 app ENS}
       \!\!\!\!\!\!\!\!\!\!\!\!\!\!d_2(t;\FlVel;\bfw,\bfv) &:= \int_{\mathcal{M}\t} \big(\matn \bfPg + \frac{1}{2}\divg(\FlVel)\bfPg +  \nbg(\FlVel) \big) \nbg \bfw : \nbg^{cov,t} \bfv  \,\ds \\ 
       &\quad + \int_{\mathcal{M}\t} \nbgcov \bfw :  \nbg^t \bfv \big(\matn \bfPg + \frac{1}{2}\divg(\FlVel)\bfPg +  \nbg^t(\FlVel) \big)  \,\ds, \nonumber \\
       \label{eq: Transport formulae d3 app ENS}
       \!\!\!\!\!\!\!\!\!\!\!\!\!\!d_3(t;\FlVel;\bfw,\bfv) &:= \int_{\mathcal{M}\t} \nbg^t \bfw \big(\matn \bfPg + \frac{1}{2}\divg(\FlVel)\bfPg +  \nbg^t(\FlVel) \big)  : \nbg^{cov} \bfv  \,\ds \\ 
       &\quad + \int_{\mathcal{M}\t} \nbg^{cov,t} \bfw : \big(\matn \bfPg + \frac{1}{2}\divg(\FlVel)\bfPg +  \nbg(\FlVel) \big) \nbg^t \bfv  \,\ds, \nonumber
    \end{align}

\noindent where the tensor $\Tilde{\mathcal{B}}_{\Ga}(\FlVel,\bfPg)$ is given by
\begin{equation}
     \Tilde{\mathcal{B}}_{\Ga}(\FlVel,\bfPg) :=  \matn \bfPg + \frac{1}{2}\divg(\FlVel) \bfPg - E(\FlVel),
\end{equation}
with $E(\cdot)$ the rate-of-strain tensor and $D(\FlVel,\bfPg)=E(\FlVel)$. Notice, also, that $\Tilde{\mathcal{B}}_{\Ga}(\FlVel,\bfPg)$ is symmetric.
\end{lemma}
\begin{proof}
Let us prove \eqref{eq: Transport formulae 4 app ENS} now, then the rest of the transport formulae is proven in an similar manner, using transpose properties. Considering the formula \eqref{eq: commutation derivative ENS} for vector-valued functions we see that 
\begin{equation*}
    \matn \nbg \bfw = \nbg \matn \bfw - \nbg\bfw\mathcal{D}_{\Ga}^t(\FlVel),
\end{equation*}
and therefore, with similar calculations to \eqref{eq: Transport formulae 3 app ENS} (see \cite[Lemma 2.1]{DziukElliott_L2}), we obtain
\begin{equation*}
    \begin{aligned}
       & \frac{d}{dt} \int_{\mathcal{M}\t} \nbgcov \bfw : \nbgcov \bfv \,\ds = \frac{d}{dt}  \int_{\mathcal{M}\t} \bfPg \nbg \bfw : \bfPg \nbg \bfv \, \ds = 
         \int_{\mathcal{M}\t} \matn (\bfPg) \nbg \bfw : \nbgcov \bfv \, \ds \\ 
         &\qquad\qquad\qquad\qquad+ \int_{\mathcal{M}\t} \big(\nbg \matn \bfw - \nbg\bfw\nbg\FlVel \big) : \nbgcov \bfv \, \ds
         + \int_{\mathcal{M}\t} \nbgcov \bfw : \matn(\bfPg) \nbg \bfv \, \ds\\
         &\qquad\qquad\qquad\qquad+  \int_{\mathcal{M}\t} \nbgcov \bfw : \big(\nbg \matn \bfv - \nbg\bfv\nbg\FlVel \big) \, \ds +    \int_{\mathcal{M}\t} \divg(\FlVel) \nbgcov \bfw : \nbgcov \bfw  \, \ds,
    \end{aligned}
\end{equation*}
from which our desired estimate can be derived, since $\int_{\mathcal{M}\t} \nbg \bfw : \nbgcov \bfv \, \ds = \int_{\mathcal{M}\t} \nbgcov \bfw : \nbgcov \bfv \, \ds$ and 
\begin{equation*}
    \int_{\mathcal{M}\t} \nbg\bfw\nbg\FlVel : \nbgcov \bfv \, \ds = \int_{\mathcal{M}\t} \nbgcov\bfw\nbg\FlVel : \nbgcov \bfv \, \ds = \int_{\mathcal{M}\t} \nbgcov\bfw : \nbgcov \bfv \nbg^t\FlVel \, \ds,
\end{equation*}
where we also make use of the fact that $D(\FlVel,\bfPg) = E(\FlVel)$, the rate-of-strain tensor, $\bfPg \nbg(\cdot) = \nbgcov(\cdot)$ and $\matn(\bfPg)$ is symmetric. This gives \eqref{eq: Transport formulae 4 app ENS}. The rest \eqref{eq: Transport formulae 55 app ENS} -- \eqref{eq: Transport formulae 77 app ENS} can be derived similarly.
\end{proof}

In our case, we want to find a transport theorem w.r.t. the rate-of-strain tensor $E(\cdot)$. This is much more complicated since $E(\bfw) = \frac{1}{2}(\nbgcov\bfw + \nb_{\Ga}^{cov,t}\bfw)$, but notice that we have calculated the  transport property of all the possible combinations in \eqref{eq: Transport formulae 4 app ENS} -- \eqref{eq: Transport formulae 77 app ENS}. Based on these results, we can easily derive the following explicit expression. We also,  require an appropriate uniform in time bounds.
\begin{lemma}[Transport formulae III]\label{lemma: Transport formulae strain tensor app ENS}
 Let $\mathcal{M}\t$ be an evolving surface with normal velocity $\bfv_n=V_{\Ga}\bfng$. Assume that the vector fields $\bfw, \, \bfv$ are sufficiently smooth, so that all the following quantities exist, then    
  \begin{align}\label{eq: Transport formulae 5 app ENS}
         \frac{d}{dt} \int_{\mathcal{M}\t} E( \bfw) :  E( \bfv) \,\ds &=  \int_{\mathcal{M}\t}  E(  \matn \bfw) : E(\bfv)+  E(\bfw): E(\matn \bfv)\,\ds \\
         & +\bfd(t;\FlVel;\bfw,\bfv), \nonumber
    \end{align}
where by $\bfd(t;\cdot;\cdot,\cdot)$ we denote the expression
\begin{equation}\label{eq: bfd formula ENS}
    \bfd(t;\FlVel;\bfw,\bfv) = \frac{1}{4}\big(d(t;\FlVel;\bfw,\bfv)+d_1(t;\FlVel;\bfw,\bfv)+d_2(t;\FlVel;\bfw,\bfv)+d_3(t;\FlVel;\bfw,\bfv)\big),
\end{equation}
which is bounded uniform in time as:
\begin{equation}\label{eq: Transport formulae bfd app ENS}
  \bfd(t;\FlVel;\bfw,\bfv) \leq c \norm{\bfw}_{\ahat}\norm{\bfv}_{\ahat}.  
\end{equation}
\end{lemma}
\begin{proof}
The equality \eqref{eq: Transport formulae 5 app ENS} can be easily proved by accounting for \eqref{eq: Transport formulae 4 app ENS} -- \eqref{eq: Transport formulae 77 app ENS}. As for the uniform bound of $\bfd(t;\cdot;\cdot,\cdot)$, one can look at \eqref{eq: Transport formulae d app ENS}, where the smoothness of $\mathcal{M}\t$, the symmetry of $\Tilde{\mathcal{B}}_{\Ga}$ and properties of the spectral and frobenious norms give us
\begin{equation*}
    \begin{aligned}
        d(t;\FlVel;\bfw,\bfv) \leq c \norm{\nbgcov\bfw}_{L^2(\Ga)} \norm{\nbgcov\bfv}_{L^2(\Ga)}.
    \end{aligned}
\end{equation*}
Then a straightforward extension gives the desired result for the rest $d_i$, $i=1,2,3$ \eqref{eq: Transport formulae d1 app ENS} -- \eqref{eq: Transport formulae d3 app ENS}.
\end{proof}

\bibliographystyle{siam}
\bibliography{bibliography}
\end{document}